\documentclass[a4paper,11pt,twoside]{amsart}
\usepackage[english]{babel}
\usepackage[utf8]{inputenc}

\usepackage[a4paper,inner=3cm,outer=3cm,top=4cm,bottom=4cm,pdftex]{geometry}
\usepackage{fancyhdr}
\usepackage{float}
\usepackage{color}
\usepackage{bold-extra}
\usepackage{ mathrsfs }

\usepackage{comment}
\usepackage{graphics}
\usepackage{aliascnt}
\usepackage[pdftex,citecolor=green,linkcolor=red]{hyperref}
\usepackage{mathtools}
\usepackage{cleveref}

\usepackage{amsmath}
\usepackage{amsfonts}
\usepackage{amssymb}
\usepackage{amsthm}
\usepackage{enumerate}
\usepackage{enumitem}
\usepackage{tikz}
\usetikzlibrary{positioning}
\usetikzlibrary{arrows.meta}

\newtheorem{theorem}{Theorem}[section]
\newtheorem{corollary}[theorem]{Corollary}

\newtheorem{lemma}[theorem]{Lemma}
\newtheorem{proposition}[theorem]{Proposition}
\theoremstyle{definition}
\newtheorem{definition}[theorem]{Definition}
\newtheorem{remark}[theorem]{Remark}
\newtheorem{example}[theorem]{Example}

\numberwithin{equation}{section}

\newcommand\R{\mathbb{R}}
\newcommand\Z{\mathbb{Z}}
\newcommand\N{\mathbb{N}}
\newcommand\C{\mathbb{C}}

\newcommand\n{\mathbf{n}}

\newcommand\eps{\varepsilon}
\renewcommand{\Re}{\textnormal{Re}}

\newcommand\Poly{{\operatorname{Poly}}}
\newcommand\TV{{\operatorname{TV}}}
\newcommand\Lip{{\operatorname{Lip}}}

\renewcommand\deg{\textnormal{deg}}
\newcommand{\meas}{\textnormal{meas}}

\let\oldpmod\pmod
\renewcommand{\pmod}[1]{\hspace{-0.12cm}\oldpmod {#1}}
\let\oldmod\mod
\renewcommand{\mod}[1]{\hspace{-0.12cm}\oldmod {#1}}
\newcommand{\asum}{\sideset{}{^{\ast}}\sum}

\begin{document}

\author[K. Matom\"aki]{Kaisa Matom\"aki}
\address{Department of Mathematics and Statistics \\
University of Turku, 20014 Turku\\
Finland}
\email{ksmato@utu.fi}

\author[M. Radziwi{\l\l}]{Maksym Radziwi{\l\l}}
\address{Department of Mathematics, Northwestern University, 2033 Sheridan Road, Evanston, IL 60208, USA}
\email{maksym.radziwill@northwestern.edu}

\author[X. Shao]{Xuancheng Shao}
\address{Department of Mathematics, University of Kentucky\\
715 Patterson Office Tower\\
Lexington, KY 40506\\
USA}
\email{xuancheng.shao@uky.edu}

\author[T. Tao]{Terence Tao}
\address{Department of Mathematics, UCLA\\
405 Hilgard Ave\\
Los Angeles, CA 90095\\
USA}
\email{tao@math.ucla.edu}

\author[J. Ter\"av\"ainen]{Joni Ter\"av\"ainen}
\address{Department of Pure Mathematics and Mathematical Statistics, University of Cambridge\\
Cambridge CB3 0WB\\
UK}
\email{joni.p.teravainen@gmail.com}

\title[Higher uniformity II]{Higher uniformity of arithmetic functions in short intervals II.  Almost all intervals}

\begin{abstract}
We study higher uniformity properties of the von Mangoldt function $\Lambda$, the M\"obius function $\mu$,  and the divisor functions $d_k$ on short intervals $(x,x+H]$ for almost all $x \in [X, 2X]$.

Let $\Lambda^\sharp$ and $d_k^\sharp$ be suitable approximants of $\Lambda$ and $d_k$, $G/\Gamma$ a filtered nilmanifold, and $F \colon G/\Gamma \to \C$ a Lipschitz function. Then our results imply for instance that when $X^{1/3+\varepsilon} \leq H \leq X$ we have, for almost all $x \in [X, 2X]$,
\[
\sup_{g \in \Poly(\Z \to G)} \left| \sum_{x < n \leq x+H} (\Lambda(n)-\Lambda^\sharp(n)) \overline{F}(g(n)\Gamma) \right| \ll H\log^{-A} X
\]
for any fixed $A>0$, and that when $X^{\varepsilon} \leq H \leq X$ we have, for almost all $x \in [X, 2X]$,
\[
\sup_{g \in \Poly(\Z \to G)} \left| \sum_{x < n \leq x+H} (d_k(n)-d_k^\sharp(n)) \overline{F}(g(n)\Gamma) \right| = o(H \log^{k-1} X).
\]

As a consequence, we show that the short interval Gowers norms $\|\Lambda-\Lambda^\sharp\|_{U^s(X,X+H]}$ and $\|d_k-d_k^\sharp\|_{U^s(X,X+H]}$ are also asymptotically small for any fixed $s$ in the same ranges of $H$. This in turn allows us to establish the Hardy--Littlewood conjecture and the divisor correlation conjecture with a short average over one variable.

Our main new ingredients are type $II$ estimates obtained by developing a ``contagion lemma'' for nilsequences and then using this to ``scale up'' an approximate functional equation for the nilsequence to a larger scale. This extends an approach developed by Walsh for Fourier uniformity.
\end{abstract}

\maketitle

\tableofcontents

\section{Introduction}

This paper is a sequel to our previous paper~\cite{MSTT-all}, which was focused on discorrelation of arithmetic functions $f \colon \mathbb{N} \to \mathbb{C}$ with arbitrary nilsequences $n \mapsto F(g(n) \Gamma)$ in all short intervals $(X,X+H]$ with $H = X^\theta$.  In this paper we consider the corresponding question for \emph{almost all} short intervals $(x,x+H]$ with $x \in [X,2X]$, thus we permit a small number of exceptional short intervals on which there might be strong correlation.  It turns out that in this setting we can reduce the exponent $\theta$ significantly.

As in the previous paper we will restrict our attention to the following model examples of functions $f$ and their approximants $f^\sharp$:
\begin{itemize}
\item[$\mu$:]  The \emph{M\"obius function} $\mu(n)$, defined as $(-1)^j$ when $n$ is the product of $j$ distinct primes, and $0$ otherwise, with the approximant $\mu^\sharp$ set equal to zero.
\item[$\Lambda$:]  The \emph{von Mangoldt function} $\Lambda(n)$, defined to equal $\log p$ when $n$ is a power $p^j$ of a prime $p$ for some $j \geq 1$, and $0$ otherwise, with approximant
\begin{equation}\label{eq:Lambdasharpdef}
 \Lambda^\sharp(n) \coloneqq \frac{P(R)}{\varphi(P(R))}1_{(n,P(R))=1}, \quad \text{where}\quad  P(w)\coloneqq\prod_{p<w}p,\quad R \coloneqq \exp( (\log X)^{1/10} ).
\end{equation}
\item[$d_k$:]  The \emph{$k^{\mathrm{th}}$ divisor function} $d_k(n)$, defined to equal to the number of representations of $n$ as the product $n=n_1\cdots n_k$ of $k$ natural numbers, where $k \geq 2$ is a fixed natural number.  (In particular, all implied constants in our asymptotic notation are allowed to depend on $k$.)  The approximant $d_k^\sharp$ will be taken to be
\begin{equation}\label{dks-def}
 d_k^\sharp(n) \coloneqq \sum_{\substack{m \leq R_k^{2k-2}\\ m\mid n}} P_m(\log n),
\end{equation}
where $R_k$ is the parameter
\begin{equation*}
 R_k \coloneqq X^{\frac{\varepsilon}{10k}},
\end{equation*}
where $0 < \varepsilon \leq 1$ is small but fixed\footnote{In~\cite{MSTT-all} we stated our results for $\varepsilon=1$, but as noted in~\cite[Section 1]{MSTT-all}, the arguments work equally well with smaller choices of $\varepsilon$, so long as the constants in the asymptotic notation are allowed to depend on $\varepsilon$.}
and the polynomials $P_m(t)$ (which have degree at most $k-1$) are given by the formula
\begin{align}\label{eq:Pmt} P_m(t) \coloneqq \sum_{j=0}^{k-1} \binom{k}{j} \sum_{\substack{m = n_1 \dotsm n_{k-1}  \\ n_1,\dots,n_j \leq R_k  \\ R_k < n_{j+1},\dots,n_{k-1} \leq R_k^2}} \frac{\left( t - \log(n_1 \dotsm n_j R_k^{k-j})\right)^{k-j-1}}{(k-j-1)! \log^{k-j-1} R_k}.
\end{align}
\end{itemize}
We refer the reader to the discussion in~\cite[Section 3.1]{MSTT-all} for a justification for choosing these approximants $\mu^{\sharp}, \Lambda^\sharp$, $d_k^\sharp$.

By a ``nilsequence'', we mean a function of the form $n \mapsto F(g(n)\Gamma)$, where $G/\Gamma$ is a filtered nilmanifold and $F \colon G/\Gamma \to \C$ is a Lipschitz function.  The precise definitions of these terms will be given in Section~\ref{nilmanifold-sec}, but
for now we remark (as mentioned for example in~\cite[Section 1]{MSTT-all}) that polynomial phases $F(g(n)\Gamma) = e(P(n))$, with $P \colon \Z \to \R$ a polynomial of degree $d$, are a special case of nilsequences.

As in~\cite{MSTT-all}, for technical reasons, it can be beneficial to consider ``maximal discorrelation'' estimates. Loosely following Robert and Sargos~\cite{robert-sargos}, we adopt the notation that, for an interval $I$,
\begin{equation}\label{maximal-sum}
 \left|\sum_{n \in I \cap \Z} f(n)\right|^* \coloneqq \sup_{P \subset I \cap \Z} \left|\sum_{n \in P} f(n)\right|,
\end{equation}
where $P$ ranges over all arithmetic progressions in $I \cap \Z$.

Now we are ready to state our main theorem\footnote{For definitions of terms such as ``filtered nilmanifold'' and $\Poly(\Z \to G)$, see Definitions~\ref{def:filtNilman} and~\ref{def:FiltGroup} below.  For our conventions for asymptotic notation, see Section~\ref{notation-sec}.}.

\begin{theorem}[Discorrelation estimate]\label{discorrelation-thm} Let $X \geq 3$ and $X^{\theta+\varepsilon} \leq H \leq X^{1-\varepsilon}$ for some $0 \leq \theta < 1$ and $\eps > 0$. Let $\delta \in (0, 1/2)$. Let $G/\Gamma$ be a filtered nilmanifold of some degree $d$ and dimension $D$, and complexity at most $1/\delta$, and let $F \colon G/\Gamma \to \C$ be a Lipschitz function of norm at most $1/\delta$.
\begin{itemize}
\item[(i)]  If $\theta = 1/3$ and $A>0$, then
\begin{align*}
\sup_{g \in \Poly(\Z \to G)} \left| \sum_{x < n \leq x+H} \mu(n) \overline{F}(g(n)\Gamma) \right|^* &\leq \frac{H}{\log^{A} X}
\end{align*}
for all $x\in [X,2X]$ outside of a set of measure $O_{A}(\delta^{-O_{d,D,\varepsilon}(1)}X\log^{-A}X)$.
\item[(ii)] If $\theta = 1/3$ and $A>0$, then
\begin{align*}
\sup_{g \in \Poly(\Z \to G)} \left| \sum_{x < n \leq x+H} (\Lambda(n) - \Lambda^\sharp(n)) \overline{F}(g(n)\Gamma) \right|^* &\leq \frac{H}{\log^{A} X}
\end{align*}
for all $x\in [X,2X]$ outside of a set of measure $O_{A}(\delta^{-O_{d,D,\varepsilon}(1)}X\log^{-A}X)$.
\item[(iii)] If $\theta = 1/3$ and $k \geq 2$, then for some constant $c_{k,d,D,\varepsilon}>0$ we have
\begin{align*}
\sup_{g \in \Poly(\Z \to G)}\left| \sum_{x < n \leq x+H} (d_k(n) - d_k^\sharp(n)) \overline{F}(g(n)\Gamma) \right|^* \leq \frac{H}{X^{c_{k,d,D,\varepsilon}}}
\end{align*}
for all $x\in [X,2X]$ outside of a set of measure $O( \delta^{-O_{d,D,\varepsilon}(1)}X^{1-c_{k,d,D,\varepsilon}})$.
\item[(iv)] If $\theta=0$, $A>0,$ and $\eta>0$, and additionally $F$ is $1$-bounded, then
\begin{align*}
\sup_{g \in \Poly(\Z \to G)} \left| \sum_{x < n \leq x+H} \mu(n) \overline{F}(g(n)\Gamma) \right|^*\leq \eta H
\end{align*}
for all $x \in [X,2X]$ outside of a set of measure $O_{A}(\delta^{-O_{d,D,\varepsilon,\eta}(1)} X \log^{-A} X)$.
\item[(v)] If $\theta = 0$, $A>0$, $k \geq 2$ and $\eta>0$, and additionally $F$ is $1$-bounded, then
\begin{align*}
\sup_{g \in \Poly(\Z \to G)}\left| \sum_{x < n \leq x+H} (d_k(n) - d_k^\sharp(n)) \overline{F}(g(n)\Gamma) \right|^* \leq \eta H \log^{k-1} X
\end{align*}
for all $x \in [X,2X]$ outside of a set of measure $O_{A}(\delta^{-O_{d,D,\varepsilon,\eta}(1)} X \log^{-A} X)$.
\end{itemize}
\end{theorem}

By the preceding remarks, Theorem~\ref{discorrelation-thm} works in particular with polynomial phase twists. Hence
we have the following corollary (proven in \Cref{sec:polyphase}), which is new already for linear phases.

\begin{corollary}[Discorrelation of $\mu$ and $\Lambda$ with polynomial phases]\label{cor:discorrelation} Let $X\geq 3$ and $X^{1/3+\varepsilon}\leq H\leq X^{1-\varepsilon}$ for some $\varepsilon>0$. Let $d\geq 1$, and for each $x\in [X,2X]$ let $P_x\colon \mathbb{Z}\to \mathbb{R}$ be a polynomial of degree $d$. Also let $A>0$.
\begin{itemize}
    \item[(i)] We have
    \begin{align*}
     \left|\sum_{x<n\leq x+H}\mu(n)e(P_x(n))\right|\leq \frac{H}{\log^A X}
    \end{align*}
for all $x\in [X,2X]$ outside of a set of measure $O_{A,d,\varepsilon}(X\log^{-A}X)$.

    \item[(ii)]  Either we have
    \begin{align*}
     \left|\sum_{x<n\leq x+H}\Lambda(n)e(P_x(n))\right|\leq \frac{H}{\log^A X}
    \end{align*}
for all $x\in [X,2X]$ outside of a set of measure $O_{A,\varepsilon,d}(X\log^{-A}X)$ or there exists an integer $1\leq q\ll (\log X)^{O_{A,d,\varepsilon}(1)}$ such that
\begin{align*}
\max_{1\leq j\leq d}H^j\|q\alpha_{j,x}\|_{\R/\Z}\leq (\log X)^{O_{A,d,\varepsilon}(1)}
\end{align*}
for measure $\gg_{A,d,\varepsilon}X\log^{-A}X$ of $x\in [X,2X]$,
where $\alpha_{j,x}$ is the degree $j$ coefficient of the polynomial $n\mapsto P_x(x+n)$ and $\|y\|_{\R/\Z}$ denotes the distance from $y$ to the nearest integer(s).
\end{itemize}
\end{corollary}

In case of correlations of $d_2$ with a \emph{fixed linear phase}, we can obtain power-saving estimates for almost all sums of length $X^{\varepsilon}$; see Theorem~\ref{thm_d2} below.

\subsection{Previous results}
We start by discussing the results concerning the von Mangoldt function $\Lambda$. Huxley's classical zero density estimate  can be used to show (see for example~\cite[Section 9.1]{harman-book}) that for any $A>0$ and $X^{1/6+\varepsilon} \leq H \leq X$ we have
\begin{align}\label{eq:lambdashortsum}
 \left| \sum_{x < n \leq x+H} (\Lambda(n) - 1) \right|\leq  H \log^{-A} X
\end{align}
for all $A>0$ and for all $x\in [X,2X]$ outside of an exceptional set of measure $O_{A,\varepsilon}(X\log^{-A}X)$; the exponent $1/6$ was recently improved to $2/15$ by Guth and Maynard~\cite{guth-maynard}. Moreover, under\footnote{One can also obtain similar results assuming just the Lindel\"of hypothesis or the density hypothesis, though perhaps without the power saving.} the Riemann hypothesis, the regime of $H$ can be extended to $X^{\varepsilon} \leq H \leq X$ (in fact, one has a power-saving bound for the left-hand side of~\eqref{eq:lambdashortsum} then). Using the fundamental lemma of the sieve, one can replace the approximant $1$ here by the more refined approximant $\Lambda^\sharp$, and a modification of the arguments used to prove~\eqref{eq:lambdashortsum} also gives for $H$ in the same regime the maximal analogue
\begin{equation}
\label{eq:Lambda1/6}
 \left| \sum_{x < n \leq x+H} (\Lambda(n) - \Lambda^\sharp(n)) \right|^* \leq  H\log^{-A} X
\end{equation}
for all $x\in [X,2X]$ outside of an exceptional set of measure $O_{A,\varepsilon}(X\log^{-A}X)$. 

Next we consider discorrelation of the von Mangoldt function against Fourier phases and nilsequences. In the prequel~\cite{MSTT-all} to this paper we proved that
\begin{equation}
\label{eq:PartILambda}
\sup_{g \in \Poly(\Z \to G)} \left| \sum_{X < n \leq X+H} (\Lambda(n) - \Lambda^\sharp(n)) \overline{F}(g(n) \Gamma) \right|^* \ll_{A,\eps,F,G/\Gamma} H \log^{-A} X
\end{equation}
for $X^{5/8+\varepsilon} \leq H \leq X^{1-\varepsilon}$ for any filtered nilmanifold $G/\Gamma$ and Lipschitz function $F \colon G/\Gamma \to \C$. This generalized the work of Zhan~\cite{Zhan} concerning the linear phase case $\overline{F}(g(n) \Gamma) = e(\alpha n)$ for some $\alpha \in \mathbb{R}$.

These are currently the best known discorrelation estimates for $\Lambda$ that hold for all short intervals. To the best of our knowledge, no further improvement to the range of validity of~\eqref{eq:PartILambda} exists in the literature, even in the case of linear phase functions and almost all short intervals.

Let us next discuss the M\"obius function, for which stronger results are known. The result~\eqref{eq:Lambda1/6} works with $\Lambda-\Lambda^{\sharp}$ replaced by $\mu$ with a very similar proof, but two of us~\cite{mr-annals} (see also~\cite{mr-p}) established unconditionally that in the much wider regime $X^\varepsilon \leq H \leq X$ we have
$$
\left| \sum_{x < n \leq x+H} \mu(n) \right| \leq  H \log^{-c} X$$
for all $x\in [X,2X]$ outside of a set of measure $O_{\varepsilon}(X\log^{-c}X)$
for some absolute constant $c>0$. In fact, in~\cite{mr-annals}  similar results (but with weaker discorrelation and exceptional set bounds) were obtained for any $H = H(X) \leq X$ that goes to infinity as $X \to \infty$. In the sequel~\cite{MR-II}, improved bounds were obtained for the measure of the exceptional set of $x\in [X,2X]$.

Concerning correlations with nilsequences,~\eqref{eq:PartILambda} works with $\Lambda$ replaced by $\mu$, and in~\cite{MSTT-all} we reached intervals of length $H \geq X^{3/5+\varepsilon}$ with weaker error terms, strengthening the earlier works~\cite{MatoTera},~\cite{matomaki-shao} concerning polynomial phases.

The work~\cite{MRTTZ} of four of us together with Ziegler (combined with a routine Fourier expansion to handle the maximal truncation) gives for $X^\varepsilon \leq H \leq X^{1-\varepsilon}$ the bound
\begin{equation}\label{xx}
\sup_{g \in \Poly(\Z \to G)} \left| \sum_{x < n \leq x+H} \mu(n) \overline{F}(g(n)\Gamma)\right|^* = o_{X \to \infty;\varepsilon}(H)
\end{equation}
for all $x\in [X,2X]$ outside of a set of measure $o_{\varepsilon,d, F, G/\Gamma}(X)$, whenever $G/\Gamma$ is a filtered nilmanifold and $F \colon G/\Gamma \to \C$ is Lipschitz. This implies a qualitative version of Theorem~\ref{discorrelation-thm}(iv) where the bound for the measure of the exceptional set is $o_{X \to \infty; \eps,d,D}(X)$.

The bound~\eqref{xx} was a generalization of work of three of us~\cite{mrt-fourier} concerning linear phases. A simplified proof of this linear phase case was recently established by Walsh~\cite{walsh}, based on exploiting the phenomenon of ``contagiousness'' of local frequencies. The proof of~\eqref{xx} given in~\cite{MRTTZ} is quite complicated; in this paper we give a new proof of~\eqref{xx}, using the ideas of Walsh~\cite{walsh} to simplify the arguments, which also gives a slightly stronger outcome in that we obtain better control on the exceptional set.

For applications toward Chowla's conjecture demonstrated in~\cite{TaoEq}, it is important that there is a supremum in~\eqref{xx}. However, the $H$-range would need to be significantly improved in order to deduce the (logarithmic) Chowla conjecture. See~\cite{MRTTZ} and~\cite{Walsh2},\cite{walsh3} for some progress on this in the case of polynomial and linear phases, respectively.

The variant of~\eqref{xx} without the supremum is easier. Regarding this version, X. He and Z. Wang~\cite{HeWang} have established that, for any fixed filtered nilmanifold $G/\Gamma$, $g\in \Poly(\Z\to G)$, and $F\colon G/\Gamma\to \mathbb{C}$ Lipschitz and any $H=H(X)\leq X$ tending to infinity, we have
\begin{equation*}
\left| \sum_{x < n \leq x+H} \mu(n) \overline{F}(g(n)\Gamma)\right|^* = o_{X \to \infty}(H)
\end{equation*}
for all $x \in [X, 2X]$ outside of a set of measure $o(X)$. This  generalized the work of three of us~\cite{MRT} concerning linear phases which in turn generalized~\cite{mr-annals}.

Finally, we review the known results for the divisor functions $d_k$ for fixed $k \geq 2$. As is well known, the mean value of $d_k$ on $[1,X]$ is $\asymp \log^{k-1} X$. Similarly, for $X^{\varepsilon}\leq H\leq X$ we have
\begin{equation*}
\sum_{X < n \leq X+H} d_k(n) \asymp_{\eps} H \log^{k-1} X;
\end{equation*}
the upper bound part of this follows from Shiu's bound~\cite{shiu} for sums of multiplicative functions on short intervals, and the lower bound part follows by considering a minorant of $d_k(n)$ given by restricting $k-1$ of the $k$ variables in the definition of $d_k(n)$ to be $\leq X^{\varepsilon/(2k)}$.

For $k=2$, one can use the second moment estimates of Jutila~\cite{jutila} for the divisor problem in short intervals to show that in the regime $X^\varepsilon \leq H \leq X^{1-\varepsilon}$ one has
$$
\left| \sum_{x < n \leq x+H} (d_2(n) - d_2^\sharp(n))\right| \leq  HX^{-c\eps}$$
for some absolute constant $c>0$ and for all $x\in [X,2X]$ outside a set of measure $\ll X^{1-c\eps}$. Improving on  work of three of the authors~\cite{mrt-correlations-II}, Mangerel~\cite{mangerel} and Y.-C. Sun~\cite{sun} have recently extended the result of~\cite{mr-annals} on bounded multiplicative functions in short intervals to divisor-bounded multiplicative functions, and in particular Sun's work can be used to show that, for any fixed $k \geq 2$, $\varepsilon>0$ and any $(\log X)^{k\log k-k+1+\varepsilon} \leq H \leq X$, we have
$$
\left| \sum_{x < n \leq x+H} (d_k(n) - d_k^\sharp(n))\right|= o_{X \to \infty}(H \log^{k-1} X)$$
for all $x\in [X,2X]$ outside of a set of measure $o_{X\to \infty}(X)$. 

For correlations of $d_k$ with polynomial phases, using the approach introduced in~\cite{Walsh2}, M. Wang~\cite{mengdi-wang} showed that in the regime $\exp(C(\log X)^{1/2}(\log \log X)^{1/2})\leq H\leq X$ with $C$ large enough, we have
for almost all $x\in [X,2X]$ the estimate
\begin{align*}
\sup_{\substack{P\in \mathbb{R}[y]\\\deg(P)\leq s}}\left|\sum_{x<n\leq x+H}(d_k(n)-d_k^{\sharp}(n))e(P(n))\right|=o_s(H\log^{k-1}X).    
\end{align*}
In~\cite{mengdi-wang}, a different approximant than $d_k^{\sharp}$ was used (which is simpler but does not allow as good error terms), but for the sake of comparison we have stated the result with the qualitatively equivalent approximant $d_k^{\sharp}$. The result of~\cite{mengdi-wang} thus applies to shorter intervals than Theorem~\ref{discorrelation-thm}(v), but is only available for polynomial phases. 

Other than the result of~\cite{mengdi-wang}, the only results we are aware of in previous literature apply to all short intervals, rather than almost all short intervals. In particular, in the prequel~\cite{MSTT-all} to this paper, we showed that for $X^{\theta_k+\varepsilon}\leq H\leq X^{1-\varepsilon}$ we have
$$
\sup_{g \in \Poly(\Z \to G)} \left| \sum_{X < n \leq X+H} (d_k(n) - d_k^\sharp(n)) \overline{F}(g(n)\Gamma)\right|^* \ll_{\eps,F,G/\Gamma} HX^{-c_{k,G/\Gamma}\eps}$$
for any filtered nilmanifold $G/\Gamma$ and Lipschitz function $F \colon G/\Gamma \to \C$, some constant $c_{k,G/\Gamma}>0$, with $\theta_2 = 1/3$, $\theta_3 = 5/9$, and $\theta_k = 5/8$ for $k \geq 4$. In particular, the $k=2$ case of Theorem~\ref{discorrelation-thm}(iii) follows from our earlier work, even without the $x$-average.

\subsection{Gowers uniformity}

As in~\cite{MSTT-all}, our discorrelation estimates with nilsequences for arithmetic functions in short intervals imply that their Gowers norms over short intervals are small.

For any non-negative integer $s \geq 1$, and any function $f \colon \Z \to \C$ with finite support, define the (unnormalized) Gowers uniformity norm
$$ \| f \|_{U^{s}(\Z)} \coloneqq \left( \sum_{x,h_1,\dots,h_{s} \in \Z} \prod_{\vec \omega \in \{0,1\}^{s}} \mathcal{C}^{|\vec \omega|} f(x+\omega_1 h_1+\dots+\omega_{s} h_{s}) \right)^{1/2^{s}}$$
where $\vec \omega = (\omega_1,\dots,\omega_{s})$, $|\vec \omega| \coloneqq \omega_1+\dots+\omega_{s}$, and $\mathcal{C} \colon z \mapsto \overline{z}$ is the complex conjugation map. Then for any interval $(x,x+H]$ with $H \geq 1$ and any $f \colon \Z \to \C$ (not necessarily of finite support), define the \emph{Gowers uniformity norm over} $(x,x+H]$ by
\begin{equation*}
 \| f \|_{U^{s}(x,x+H]} \coloneqq \frac{\| f 1_{(x,x+H]} \|_{U^{s}(\Z)}}{\| 1_{(x,x+H]} \|_{U^{s}(\Z)}}
\end{equation*}
where $1_{(x,x+H]} \colon \Z \to \C$ is the indicator function of $(x,x+H]$.

Using the inverse theorem for the Gowers norms (see Proposition~\ref{prop_inverse}) and a construction of pseudorandom majorants in Section~\ref{gowers-sec}, we can deduce a Gowers uniformity estimate from Theorem~\ref{discorrelation-thm}. There and later we use the approximant
\begin{equation}
\label{eq:Lambda_w_def}
\Lambda_{w}(n) \coloneqq  \frac{W}{\varphi(W)}1_{(n,W )=1},
\end{equation}
with $W \coloneqq  \prod_{p\leq w} p$ and $X$ large in terms of $w$.

\begin{theorem}[Gowers uniformity estimate]\label{thm_gowers}
Let $X\geq H\geq 3$ satisfy $X^{\theta+\varepsilon}\leq H\leq X^{1-\varepsilon}$ for some $0 \leq \theta < 1$ and some fixed $\eps > 0$.  Let $s\geq 1$ and $k\geq 2$ be fixed integers, and let $A>0$. Let $w \geq 1$ and let $\Lambda_w$ be as in~\eqref{eq:Lambda_w_def}, and assume that $X$ is large enough in terms of $w$.
\begin{itemize}
\item[(i)]   If $\theta = 1/3$, then
\begin{align*}
\|\Lambda-\Lambda_w\|_{U^s(x,x+H]}&=o_{w\to \infty}(1),\\
\max_{\substack{1\leq a\leq W\\(a,W)=1}}\left\|\frac{\varphi(W)}{W}\Lambda(W\cdot+a)-1\right\|_{U^s(x,x+H]}&=o_{w\to \infty}(1)
\end{align*}
for all $x\in [X,2X]$ outside of an exceptional set of measure $O_{A,\eps}(X\log^{-A}X)$.
\item[(ii)]  If $\theta = 0$, then
\begin{align*}
   \|\mu\|_{U^s(x,x+H]}=o_{X \to \infty}(1)
\end{align*}
for all $x\in [X,2X]$ outside of an exceptional set of measure $O_{A,\eps}(X\log^{-A}X)$.
\item[(iii)]    If $\theta = 0$, $k \geq 2$ and $W'$ satisfies $W\mid W'\mid W^{\lfloor w\rfloor}$, then
\begin{align}\label{eq:gowersdk}\begin{split}
\|d_k-d_k^{\sharp}\|_{U^s(x,x+H]}&=o_{X \to \infty}(\log^{k-1} X),\\
\max_{\substack{1\leq a\leq W'\\(a,W')=1}}\left\|\left(\frac{W}{\varphi(W)}\right)^{k-1}d_k(W'\cdot+a)-\frac{\log^{k-1}X}{(k-1)!}\right\|_{U^s(x,x+H]}&=o_{w \to \infty}(\log^{k-1} X)
\end{split}
\end{align}
for all $x\in [X,2X]$ outside of an exceptional set of measure $O_{A,\eps}(X\log^{-A}X)$.
\end{itemize}
\end{theorem}

Theorem~\ref{thm_gowers}(ii)  was earlier established  in~\cite{MRTTZ} but with a qualitative exceptional set bound of $o_{X \to \infty; \eps}(X)$. The argument used to prove Theorem~\ref{thm_gowers}(ii) (which develops an idea of Walsh~\cite{walsh}) in fact gives a different, simpler proof of the main result in~\cite{MRTTZ}.

In the case $s=2$, the inverse theorem for the Gowers norm has polynomial dependencies by a  standard Fourier-analytic argument, so one could strengthen Theorem~\ref{thm_gowers}(i) to
\begin{align*}
 \|\Lambda-\Lambda_w\|_{U^2(x,x+H]}\leq w^{-c}
\end{align*}
with the same conditions, but now in the larger regime $w\ll_{A,\eps} \log^{A}X$. Similarly, Theorem~\ref{discorrelation-thm}(iii) gives that for $X^{1/3+\varepsilon}\leq H\leq X^{1-\varepsilon}$ we have
\begin{align*}
\|d_k-d_k^{\sharp}\|_{U^2(x,x+H]}\leq X^{-c_{k,\varepsilon}}
\end{align*}
for some constant $c_k>0$ and for all $x\in [X,2X]$ outside of a set of measure $O_{\eps}(X^{1-c_{k,\varepsilon}})$.

\begin{remark} \label{rmk:quantitative}
It is natural to wonder if the quasipolynomial inverse theorem of Leng--Sah--Sawhney~\cite{LSS} would allow quantifying the Gowers norm bounds in Theorem~\ref{thm_gowers}. While we believe a quantification would in principle be possible, the error terms obtained would likely be rather weak (perhaps triple-logarithmic or so), the largest difficulty being tracking the dimension dependence of the ``nilsequence contagion argument'' in Section~\ref{contagion-sec}. Another bottleneck for the M\"obius and divisor functions is that, in order to obtain savings of order $O(\varepsilon)$ using the Tur\'an--Kubilius identity, we need to handle type $II$ sums of length $X^{\varepsilon'}$ where $\varepsilon'=\exp(-\varepsilon^{-2})$. The arguments in Section~\ref{type-ii} involving iterated ``scaling up'' of a phase relation are not quantitatively very effective and would limit the speed at which $\varepsilon$ can go to $0$. For these reasons, and in order not to further complicate the paper, we have chosen to keep the Gowers norm results qualitative. 
\end{remark}

\subsection{Applications}

Theorem~\ref{thm_gowers} allows us to study $\ell$-point correlations with only one averaging variable --- Theorem~\ref{thm_gowers}(i) combined with the generalized von Neumann theorem and (somewhat involved) sieve-theoretic computations implies the following.

\begin{theorem}[$\ell$-point Hardy--Littlewood with one averaging variable]\label{thm:HL} Let $\ell \in \N$ and $\varepsilon>0$ be fixed, and let $X^{1/3+\varepsilon} \leq H \leq X^{1-\varepsilon}$. Then we have
\begin{align*}
\sum_{n\leq X}\Lambda(n)\Lambda(n+h)\cdots \Lambda(n+(\ell-1)h)=(1+o_{X\to \infty}(1))\mathfrak{S}(0,h,\ldots, (\ell-1)h)X
\end{align*}
for proportion $1-o_{X \to \infty}(1)$  of integers $1\leq h\leq H$, where
\begin{align}\label{eq:singular}
\mathfrak{S}(h_1, \dotsc, h_\ell) \coloneqq \prod_{p} \left(1-\frac{|\{h_1, \dotsc, h_\ell\} \pmod{p}|}{p}\right) \left(1-\frac{1}{p}\right)^{-\ell}
\end{align}
is the usual singular series.
\end{theorem}
Allowing $H$ to be fixed in Theorem~\ref{thm:HL} would roughly correspond to the Hardy--Littlewood prime tuples conjecture~\cite{HL}.  The coefficients $0,1,\dots,\ell-1$ here can be replaced by any other fixed, distinct integers if desired. Previously, the case $H=X$ and $\ell$ arbitrary was known by work of Green, the fourth author and Ziegler~\cite{green-tao},\cite{gt-mobius},\cite{gtz}, and the case $H=X^{8/33+\varepsilon}$ and $\ell=2$ was known by work of three of the authors~\cite{mrt-correlationsI}. For $\ell\geq 3$ and $H\leq X^{1-\varepsilon}$, we are not aware of any previous results of this shape in the literature, although using results of part I of this series~\cite{MSTT-all}, together with arguments from Section~\ref{sec:applications} of the current paper,  it would be possible to deduce the case $H=X^{5/8+\varepsilon}$ and $\ell$ arbitrary. Thus we obtain a significant improvement over what the methods of part I would give.

We get a similar result for the divisor function using Theorem~\ref{thm_gowers}(iii), the generalized von Neumann theorem and some divisor correlation computations.

\begin{theorem}[Divisor correlation conjecture with one averaging variable]\label{thm:divisor} \label{app_div} Let $k, \ell \in \N$ and $\varepsilon>0$ be fixed, and let $X^{\varepsilon} \leq H \leq X^{1-\varepsilon}$. Then we have
\begin{align*}
\sum_{n\leq X}d_k(n)d_k(n+h)\cdots d_k(n+(\ell-1)h)=\left(\frac{C(h,k,\ell)}{((k-1)!)^{\ell}}+o_{X\to \infty}(1)\right)(\log X)^{k\ell-\ell}
\end{align*}
for proportion $1-o_{X \to \infty}(1)$  of integers $1\leq h\leq H$, where \begin{align}\label{eq:Chkl}
C(h,k,\ell)\coloneqq \prod_{p}\lim_{Y\to \infty}\frac{\mathbb{E}_{n\leq Y}d_{k,p}(n)d_{k,p}(n+h)\cdots d_{k,p}(n+(\ell-1)h)}{(\mathbb{E}_{n\leq Y}d_{k,p}(n))^{\ell}}
\end{align}
and $d_{k,p}$ is the multiplicative function given on prime powers by $d_{k,p}(q^{a})=\binom{k-1+a}{k-1}$ if $q=p$ and $d_{k,p}(q^{a})=1$ otherwise.
\end{theorem}

\begin{remark}
It is possible to evaluate the constants $C(h,k,\ell)$ explicitly with sufficient effort. See~\cite{tao-computation} for the evaluation of $C(h,k,2)$.
\end{remark}

Allowing $H$ to be fixed in Theorem~\ref{thm:divisor}, and taking $\ell=2$, would correspond to (leading order) verification of the well known divisor correlation conjecture~\cite[Conjecture 1.1]{mrt-correlationsI},~\cite[Conjecture 3]{conrey-gonek},~\cite{ivic-divisor},~\cite{vinogradov-divisor}. Previously, the case $H=X$ and $\ell$ arbitrary was established in the work of Matthiesen~\cite{matthiesen-linear}, and the case $H=(\log X)^{10000k\log k}$ and $\ell=2$ was proven by three of the authors~\cite{mrt-correlations-II}. Moreover, it would be possible to deduce the case  $H=X^{3/5+\varepsilon}$ and $\ell$ arbitrary from part I of this series~\cite{MSTT-all} combined with Section~\ref{sec:applications} of the current paper. Thus, we again obtain a significant improvement over what methods of part I would give.

\subsection{The case of \texorpdfstring{$d_2$}{d2} revisited}

In the case of $d_2$ and a fixed linear phase, we prove in Section~\ref{d2-sec} the following result that has better savings than what Theorem~\ref{discorrelation-thm}(v) implies:

\begin{theorem}[Improved discorrelation for $d_2$ against fixed linear phase]\label{thm_d2}
   Let $X\geq 3$ and $X^{\varepsilon} \leq H \leq X^{1 - \varepsilon}$ for some $\varepsilon \in (0, \tfrac{1}{300})$. Also let $\alpha\in \mathbb{R}$. Then we have
    $$
    \Bigg |\sum_{x < n \leq x + H} (d_2(n) - d_2^{\sharp}(n)) e(n\alpha) \Bigg |^{*} dx \leq HX^{ - \varepsilon/1000}
    $$
for all $x\in [X,2X]$ outside of a set of measure $O_{\varepsilon}(X^{1-\varepsilon/1000})$.
\end{theorem}

The proof method of Theorem~\ref{thm_d2} is rather different from that of Theorem~\ref{discorrelation-thm}, relying on more classical inputs such as the fourth moment of Dirichlet $L$-functions.

\begin{remark} The most notable aspect of Theorem \ref{thm_d2} is the power-saving bound for the discorrelation. It is plausible that, with some work, the methods in~\cite{mrt-correlations-II} could prove a variant of this theorem in which $d_2$ can be replaced by $d_k$ for an arbitrary $k\geq 2$, and $H$ is permitted to be as small as $\log^{C_k} X$ for some $C_k$ depending on $k$, but with the savings $X^{-\eta}$ significantly weakened to $\log^{-c_k} X$ for some $c_k>0$.  However, we do not attempt to establish this claim here.
\end{remark}

\subsection{Methods of proof}

We now describe (in somewhat informal terms) the general strategy of proof of Theorem~\ref{discorrelation-thm}, although for various technical reasons the actual rigorous proof will not quite follow the intuitive plan that is outlined here.  See also \Cref{fig-logic} for a high-level depiction of the proof and \Cref{fig-major,fig-inv,fig-contagion} for the structure of proofs of several key intermediate theorems.  (These figures could also be combined into one giant diagram if desired.)

\begin{figure}[H]
    \centering
\begin{tikzpicture}[node distance=1cm and 1cm]
    \node[draw, rectangle, align=center,fill=green!20] (main) {Discorrelation estimate \\ \Cref{discorrelation-thm}};
    \node[draw, rectangle, left=of main, align=center] (cor) {Discorrelation with poly. \\ \Cref{cor:discorrelation}};
    \node[draw, rectangle, right=of main, align=center,fill=red!20] (d2) {$d_2$ discorrelation \\ \Cref{thm_d2}};
    \node[draw, rectangle, below=of main, align=center] (gowers) {Gowers uniformity \\ \Cref{thm_gowers}};
    \node[draw, rectangle, below left=of gowers, align=center] (hl) {Hardy--Littlewood conj. \\ \Cref{thm:HL}};
    \node[draw, rectangle, below right=of gowers, align=center] (divisor) {Divisor correlations \\ \Cref{thm:divisor}};
    \node[draw, rectangle, above left=of main, align=center,fill=blue!20] (major) {Major arc estimate \\ \Cref{discorrelation-thm-major}};
    \node[draw, rectangle, above=of main, align=center] (inv) {Inverse theorem \\ \Cref{inverse}};
    \node[draw, rectangle, above right=of main, align=center] (comb) {Comb. decompositions \\ \Cref{comb-lambda,comb-mu}};
    \node[draw, rectangle, above left=of inv, align=center] (invold) {Type $I$, $I_2$ inv. theorems \cite{MSTT-all} \\ \Cref{inverse}(i), (iv) };
    \node[draw, rectangle, above right=of inv, align=center,fill=yellow!20] (invnew) {Type $II$ inv. theorem \\ \Cref{inverse}(ii), (iii)};
    \draw[-{Latex[length=4mm, width=2mm]}] (main) -- (cor);
    \draw[-{Latex[length=4mm, width=2mm]}] (main) -- (gowers);
    \draw[-{Latex[length=4mm, width=2mm]}] (gowers) -- (hl);
    \draw[-{Latex[length=4mm, width=2mm]}] (gowers) -- (divisor);
    \draw[-{Latex[length=4mm, width=2mm]}] (major) -- (main);
    \draw[-{Latex[length=4mm, width=2mm]}] (inv) -- (main);
    \draw[-{Latex[length=4mm, width=2mm]}] (comb) -- (main);
    \draw[-{Latex[length=4mm, width=2mm]}] (invold) -- (inv);
    \draw[-{Latex[length=4mm, width=2mm]}] (invnew) -- (inv);
\end{tikzpicture}
    \caption{Logical relationships used to both prove and apply the main estimate, \Cref{discorrelation-thm} (in green). \Cref{discorrelation-thm-major} (in blue) and \Cref{inverse}(ii), (iii) (in yellow) will themselves require lengthy proofs; see \Cref{fig-major,fig-inv} respectively. \Cref{thm_d2} (in red) is proven by different (and more classical) methods.}
    \label{fig-logic}
\end{figure}
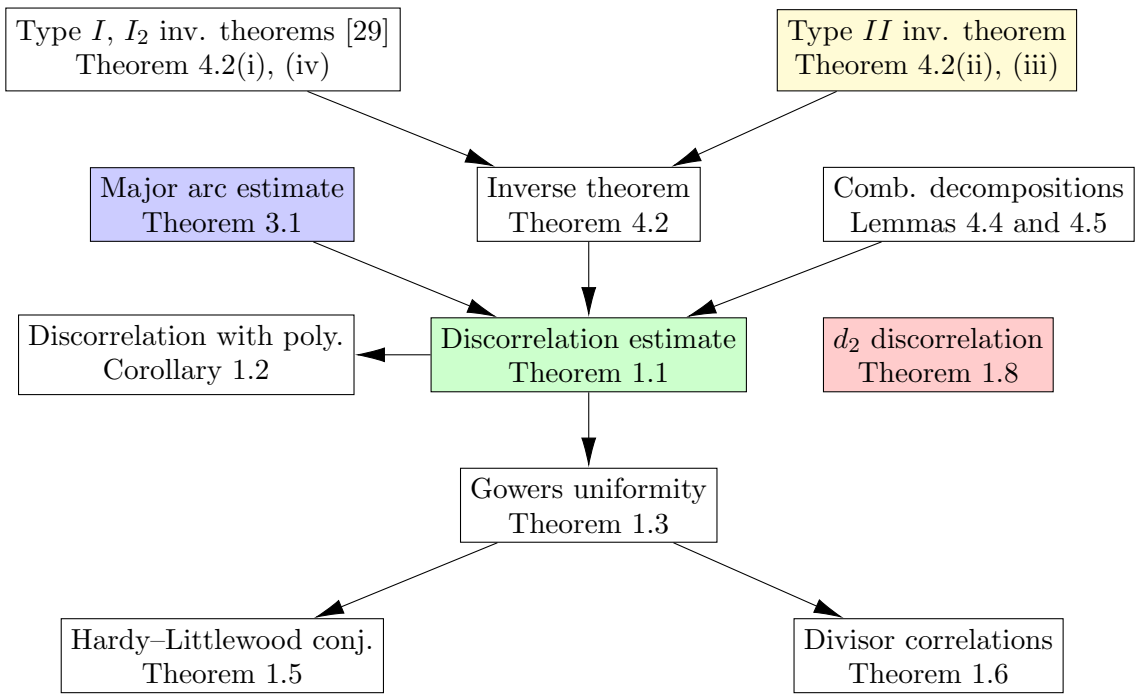

The first step in the case of Theorem~\ref{discorrelation-thm}(i)--(iii), which is standard, is to apply Heath--Brown's identity (Lemma~\ref{hb-identity}) and decompose $\mu, \Lambda, d_k$ and their approximants (up to small error terms) into three standard types of sums for divisor-bounded coefficients $\alpha, \beta$:
\begin{itemize}
\item[($I$)] \emph{Type $I$} sums, which are essentially of the form $\alpha * 1$ for some arithmetic function $\alpha \colon \N \to \C$ supported on the interval $[1, HX^{-\varepsilon/2}]$.
\item[($I_2$)] \emph{Type $I_2$} sums, which are essentially of the form $\alpha * 1 * 1$ for some arithmetic function $\alpha \colon \N \to \C$ supported on the interval $[1, H^{3/2} X^{-1/2-\varepsilon/2}]$.
\item[($II$)] \emph{Type $II$} sums, which are of the form $\alpha * \beta$ for some arithmetic functions $\alpha, \beta \colon \N \to \C$ with $\alpha$ supported on the interval $[X^{\varepsilon/2}, H X^{-\varepsilon/2}]$.
\end{itemize}

On the other hand, the first step in proving Theorem~\ref{discorrelation-thm}(iv)--(v) is to apply an identity of Tur\'an--Kubilius type, which leads only to type $I$ sums in the range $[1,X^{\varepsilon/2}]$ and to type $II$ sums with one variable in the range $[X^{\kappa},X^{\varepsilon/2}]$, with $\kappa>0$ an arbitrarily small constant.  We formalize these decompositions as \Cref{comb-lambda,comb-mu}.

We then split our considerations to the \emph{major arc case} in which the nilsequence
$F(g(n)\Gamma)$ is more-or-less replaced by $n^{iT}$ for some $T = X^{O(1)}$, and the \emph{minor arc case} in which $F$ has mean zero and $g(n)\Gamma$ is
highly equidistributed in the nilmanifold $G/\Gamma$.

The major arc case of Theorem~\ref{discorrelation-thm} is formalized as \Cref{discorrelation-thm-major}, and is proven in Section~\ref{major-arc-sec}. Here we use Parseval's identity to relate the variance of the short interval sum to Dirichlet polynomial estimates. We split this case into two subcases: the case of small $|T|$ (i.e., $|T|\leq X^{1+\varepsilon/2}/H$) and the case of large $|T|$ (i.e., $|T|> X^{1+\varepsilon/2}/H$). If $|T|$ is small, then $n^{iT}$ does not vary much on intervals of the form $[x,x+HX^{-\varepsilon/2}]$, which allows us to effectively reduce to the case $T=0$. This case could be deduced from known estimates on arithmetic functions in almost all short intervals, although with the large $|T|$ case in mind we found it convenient to reprove these estimates.  If instead $|T|$ is large, we split into type $I$ and $II$ sums in the usual fashion. The type $II$ sums are handled using Dirichlet polynomial estimates of Baker--Harman--Pintz~\cite{baker-harman-pintz}. After this, we are left with handling only a very narrow type $I$ range where we have a smooth variable of length $>X^{1-\varepsilon}$. This case in turn can be handled by Taylor expanding $n^{iT} = e(\frac{T}{2\pi} \log n)$ on a short interval around $x$ as a polynomial exponential phase and using van der Corput's exponential sum bound to show that the sequence $\left(\frac{T}{2\pi x}\mod 1\right)_{x\in [X,2X]}$ (corresponding to the linear term in the Taylor expansion) equidistributes.

To prove Theorem~\ref{discorrelation-thm}(i)--(iii) in the remaining, minor arc, cases,
in case of type $I$ and type $I_2$ sums, we simply use our estimates from~\cite{MSTT-all}. Unfortunately, for these contributions we have not found a way to use the additional averaging over $x$ to extend the range where $\alpha$ can be supported.

The remaining task is then to handle the contribution of type $II$ minor arc sums. Let us consider sums of the type
\begin{equation}\label{Expo}
\sum_{\substack{x < ab \leq x + H \\ A_{II} <a\leq 2A_{II}}} \alpha(a)  \beta(b) \overline{F}(g_x(ab) \Gamma),
\end{equation}
with $H^{\varepsilon/2} \leq A_{II} \leq HX^{-\varepsilon/2}$, where $g_x\colon \Z\to G$ is a polynomial sequence, for all $x\in [X,2X]$ outside a small exceptional set.
To treat these sums, we can use a Fourier decomposition and the equidistribution theory of nilmanifolds to reduce (roughly speaking) to treating the following two special cases of these sums:
\begin{itemize}
\item \emph{Abelian type $II$ minor arc sums} in which $F(g_x(n)\Gamma) = e(P_x(n))$ is a polynomial phase that does not ``look like'' a character $n^{iT}$ (or more generally $\chi(n) n^{iT}$ for some Dirichlet character $\chi$ of small conductor) in the sense that the Taylor coefficients of $e(P_x(n))$ around $X$ do not align with the corresponding Taylor coefficients of such characters.

\item \emph{Nonabelian type $II$ minor arc sums}, in which $g_x(n) \Gamma$ is highly equidistributed in the nilmanifold $G/\Gamma$, which is arising from a nonabelian nilpotent group $G$, and $F$ exhibits non-trivial oscillation in the direction of the center $Z(G)$ of $G$ (which one can reduce to be one-dimensional).
\end{itemize}

To handle both these cases, we extend the methodology of Walsh~\cite{walsh} to develop two tools: a theory of ``nilsequence contagion'' and a combinatorial ``scaling'' argument.

To state the nilsequence contagion result (Theorem~\ref{nil-contagion}) that we establish, we need to introduce a relation $\sim_{I,\eta}$ between polynomial sequences on $G$, where $I$ is an interval and $\eta\in (0,1)$. We say that $g_1\sim_{I,\eta}g_2$ for two polynomial sequences $g_1,g_2\colon \Z\to G$ if $g_1=\varepsilon g_2\gamma$ for some $(\eta^{-1},I)$-smooth polynomial sequence $\varepsilon$ and some $\eta^{-1}$-rational polynomial sequence $\gamma$. We are then interested in what can be said of a family of polynomial maps $(g_n)_{n\in [P,2P]\cap \mathbb{N}}, g_n\colon \Z\to G$ if we have the relation
\begin{align}\label{eq:gnn'}
g_{n}(n'\cdot)\sim_{\frac{1}{nn'}I,\eta}g_{n'}(n\cdot)
\end{align}
for proportion $\geq \eta$ of pairs $(n,n')\in ([P,2P]\cap \N)^2$. We claim that, under suitable conditions on the parameters $P, |I|, \eta$, there must exist a polynomial sequence $g_{*}$ such that
\begin{align*}
 g_n\sim_{\frac{1}{n}I,\eta}g_{*}(n\cdot)
\end{align*}
for $\gg \eta^{O(1)}P$ choices of $n\in [P,2P]\cap \N$  (and these indeed are clearly solutions to~\eqref{eq:gnn'} if we adjust the value of $\eta$ by a polynomial amount). 

Walsh showed in~\cite{walsh} (see also~\cite{Walsh2},~\cite{walsh3}) that an analogous claim holds in the case where $g_n(y)=\alpha_ny$ is a linear polynomial on $\R/\Z$, and this was one key ingredient that enabled him to give a simpler proof of the main result of~\cite{mrt-fourier}. Our result extends this to the case of arbitrary polynomial sequences on Lie groups and in addition, importantly for the proof of Theorem~\ref{discorrelation-thm}, our result allows extra flexibility already in the case of linear polynomials (we can improve the dependencies between the different parameters $P, |I|, \eta$, resulting in improved exponents for Theorem~\ref{discorrelation-thm}). 

The proof of the nilsequence contagion result involves first showing a version for classical polynomials (Theorem~\ref{thm_contagious}) by some combinatorial arguments that give iteratively stronger conclusions until we reach the desired conclusion. We then prove the general case by considering a certain two-variable family of subspaces of $\Poly(\Z\to G)$ and performing a double induction with respect to this family, where the starting case is trivial, the final case is the full nilsequence contagion result, and the passage from one case to the next is enabled by the polynomial case.

The scaling argument executed in Section~\ref{type-ii} is a combinatorial argument that together with the nilsequence contagion lemma shows that if the type $II$ sum~\eqref{Expo} is large, then the phase function $x\mapsto g_x$, initially defined on $[X,2X]$, can be extended to a larger scale $[XA_{II}^K,C_KXA_{II}^K]$ for any fixed $K$, such that the phases $g_x$ at scale $X$ and  the phases $g_z$ at scale $XA_{II}^K$ are related by an ``approximate functional equation'' roughly of the form
\begin{align*}
g_z(m\cdot)\sim_{\frac{1}{m}(z,z+HA_{II}^{K}],(\log X)^{-O_K(1)}} g_x \end{align*}
for ``many'' $m\asymp A_{II}^K$. Since $A_{II}\geq H^{\varepsilon/2}\geq X^{\kappa}$ for some constant $\kappa>0$, in all cases of Theorem~\ref{discorrelation-thm} we can find $K\ll_{\varepsilon}1$ such that $A_{II}^{K}\geq X$; thus, we can reach a very large scale where the phase functions ``communicate'' with those at scale $[X,2X]$.  We then employ a technical proposition (Proposition~\ref{prop:Furstenberg-Weiss}), related to the non-abelian type $II$ estimate established in the prequel paper~\cite{MSTT-all}, to show that the functional equation~\eqref{eq:gnn'} has no solutions in the non-abelian minor arc case where $G$ is non-abelian with one-dimensional center and $g_x\Gamma$ is highly equidistributed. In the complementary abelian case, we can reduce to $G/\Gamma=\R/\Z$. We then show that the only solutions are essentially those where $g_x$ is the truncated Taylor series around $x$ of the function $n\mapsto T\log n$ for some $T\ll X^{O(1)}$;  for this a key ingredient is Proposition~\ref{prop:logarithm} that relates approximately dilation-invariant polynomials to Taylor polynomials of the logarithm function. This then brings us back to the major arc case that was discussed above.

\begin{remark}
There is an obstacle in improving the exponent $\theta$ in Theorem~\ref{discorrelation-thm}(i) below $\theta=1/3$. Namely, for $H=X^{1/3 - \varepsilon}$, we do not know how to handle $d_3$-type sums
\[
\sum_{\substack{X < abc \leq X + H \\ X^{1/3-\varepsilon/2} < a, b, c \leq X^{1/3+\varepsilon/2}}} \overline{F}(g_x(abc) \Gamma).
\]
arising from Heath-Brown's identity.  Our type $II$ estimates are not applicable anymore since none of the variables is supported on numbers $\leq HX^{-\eta}$ for any $\eta>0$. On the other hand, our
type $I_2$ estimates from our previous paper~\cite{MSTT-all} rely on a certain decomposition of hyperbolic neighborhoods into arithmetic progressions which is only non-trivial in the regime $\theta \geq 1/3$.
\end{remark}

\subsection{Acknowledgments}

KM was supported by Academy of Finland grant no. 285894. MR acknowledges the support of NSF grant DMS-2401106. XS was supported by NSF grant DMS-2200565. TT was supported by a Simons Investigator grant, the James and Carol Collins Chair, the Mathematical Analysis \& Application Research Fund Endowment, and by NSF grant DMS-1764034. JT was supported by Academy of Finland grant no. 362303, a von Neumann Fellowship (NSF grant DMS-1926686), and funding from European Union's Horizon
Europe research and innovation programme under Marie Sk\l{}odowska-Curie grant agreement no.
101058904 and ERC grant agreement no. 101162746.  

 We are grateful to the anonymous referee for a careful reading of the paper and for numerous helpful comments and corrections. We thank Jiseong Kim for a correction.

\subsection{Notation}\label{notation-sec}

We use $Y \ll Z$, $Y = O(Z)$, or (when $Y \geq 0$) $Z \gg Y$ to denote the estimate $|Y| \leq CZ$ for some constant $C$.  If we permit this constant to depend on one or more parameters we shall indicate this by appropriate subscripts, thus for instance $O_{\eps,A}(Z)$ denotes a quantity bounded in modulus by $C_{\eps,A} Z$ for some quantity $C_{\eps,A}$ depending only on $\eps,A$.  We write $Y \asymp Z$ for $Y \ll Z \ll Y$, and we write $y\sim Y$ for $Y<y\leq 2Y$. When working with $d_k$, all implied constants are permitted to depend on $k$.  We also write $X = o_{w \to \infty}(Y)$ (respectively $X=o_{\delta\to 0}(Y)$) to denote an estimate of the form $|X| \leq c(w) Y$ where $c(w) \to 0$ as $w \to \infty$ (respectively, $|X| \leq c(\delta) Y$ where $c(\delta) \to 0$ as $\delta \to 0$); again, if the implied decay rate $c$ needs to depend on additional parameters, we indicate this with appropriate subscripts.

If $x$ is a real number (respectively an element of $\R/\Z$), we write $e(x) \coloneqq e^{2\pi i x}$ and let $\|x\|_{\R/\Z}$ denote the distance of $x$ to the nearest integer (respectively to zero).  We let $\{x\}$ be the \emph{signed} fractional part of $x$, that is to say the unique real number in $(-1/2,1/2]$ that is equal to $x$ modulo $1$.  In particular $\|x\|_{\R/\Z} = |\{x\}|$. We use $\log$ to denote the natural logarithm function $x \mapsto \log x$.

We use $1_E$ to denote the indicator of an event $E$, thus $1_E$ equals $1$ when $E$ is true and $0$ otherwise.  If $S$ is a set, we write $1_S$ for the indicator function $1_S(n) \coloneqq 1_{n \in S}$.

We write $|E|$ for the cardinality of a finite set $E$. If $I$ is an interval, we use $|I|$ to denote its length. We write $\meas(\mathcal{B})$ for the Lebesgue measure of a set $\mathcal{B}\subset \mathbb{R}$.
If $\vec \omega=(\omega_1,\ldots, \omega_s)\in \{0,1\}^{s}$ is a vector, we write $|\vec \omega|\coloneqq\omega_1+\cdots+\omega_s$.

Unless otherwise specified, all sums range over natural number values, except for sums over $p$ which are understood to range over primes.  We use $d\mid n$ to denote the assertion that $d$ divides $n$; $(n,m)$ to denote the greatest common divisor of $n$ and $m$, $n = a \ (q)$ to denote the assertion that $n$ and $a$ have the same residue mod $q$; and $(f*g)(n) \coloneqq \sum_{d\mid n} f(d) g(n/d)$ to denote the Dirichlet convolution of two arithmetic functions $f,g \colon \N \to \C$.  We use $d \mid W^\infty$ to denote the claim that $d\mid W^m$ for sufficiently large $m$, or equivalently that all prime factors of $d$ divide $W$. We also sometimes write $P(w)=\prod_{p<w}p$.

The \emph{height} of a rational number $a/b$ with $a,b$ coprime is defined as $\max(|a|, |b|)$.

\section{Basic tools}\label{sec:lemmas}

In this section we state various basic tools and definitions that we will need in this paper. Most of these are quoted from our previous paper~\cite{MSTT-all}, but some (in particular, Proposition~\ref{prop:logarithm}) were not present in that work.

\subsection{Total variation}

The notion of maximal summation, which we defined in~\eqref{maximal-sum}, interacts well with the notion of total variation.

\begin{definition}[Total variation]\label{tv-def}  Given any function $f\colon P \to \C$ on an arithmetic progression $P$, the \emph{total variation norm} $\|f\|_{\TV(P)}$ is defined by
$$ \|f\|_{\TV(P)} \coloneqq \sup_{n \in P} |f(n)| + \sup_{n_1 < \dots < n_k} \sum_{j=1}^{k-1} |f(n_{j+1})-f(n_j)|$$
where the second supremum ranges over all increasing finite sequences $n_1 < \dots < n_k$ in $P$ and all $k \geq 1$.  Note that in this finitary setting one can simply take $n_1,\dots,n_k$ to be the elements of $P$ in increasing order, if desired. We define $\|f\|_{\TV(P)}=0$ when $P$ is empty.  For any natural number $q \geq 1$, we also set
$$ \|f\|_{\TV(P;q)} \coloneqq \sum_{a \in \Z/q\Z} \|f\|_{\TV(P \cap (a+q\Z))}.$$
\end{definition}

Informally, if $f$ is bounded in $\TV(P;q)$ norm, then $f$ does not vary much on each residue class modulo $q$ in $P$.  Applying the fundamental theorem of calculus we see that if $f \colon I \to \C$ is a continuously differentiable function on an interval $I$ and $P\subset I$ is an arithmetic progression, then
\begin{equation*}
 \|f\|_{\TV(P)} \ll \sup_{t \in I} |f(t)| + \int_I |f'(t)|\ dt.
\end{equation*}
  Furthermore, from the identity $ab-a'b' = (a-a')b + (b-b')a'$ we see that
\begin{equation*}
\| fg \|_{\TV(P;q)} \ll \|f\|_{\TV(P;q)} \|g\|_{\TV(P;q)}
\end{equation*}
for any functions $f,g \colon P \to \C$ defined on an arithmetic progression, and any $q \geq 1$.

We recall some basic properties of maximal summation:

\begin{lemma}[Basic properties of maximal sums]\label{basic-prop}~\cite[Lemma 2.2]{MSTT-all}\ 
\begin{itemize}
\item[(i)]  (Triangle inequalities) For any subprogression $P'$ of an arithmetic progression $P$, and any $f \colon P \to \C$ we have
$$ \left| \sum_{n \in P} f(n) 1_{P'}(n) \right|^* = \left| \sum_{n \in P'} f(n) \right|^* \leq \left| \sum_{n \in P} f(n) \right|^*$$
and
$$ \left|\sum_{n \in P} f(n)\right| \leq\left |\sum_{n \in P} f(n)\right|^* \leq \sum_{n \in P} |f(n)|.$$
If $P$ can be partitioned into two subprogressions as $P = P_1 \uplus P_2$, then
\begin{equation*}
\left |\sum_{n \in P} f(n)\right|^* \leq \left|\sum_{n \in P_1} f(n)\right|^* + \left|\sum_{n \in P_2} f(n)\right|^*.
\end{equation*}
Finally, the map $f \mapsto |\sum_{n \in P} f(n)|^*$ is a norm.
\item[(ii)]  (Summation by parts) Let $P$ be an arithmetic progression, and let $f,g \colon P \to \C$ be functions.  Then we have
\begin{equation*}
 \left| \sum_{n \in P} f(n) g(n) \right|^* \leq \|g\|_{\TV(P)} \left| \sum_{n \in P} f(n) \right|^*
\end{equation*}
and more generally
\begin{equation*}
 \left| \sum_{n \in P} f(n) g(n) \right|^* \leq \|g\|_{\TV(P;q)} \left| \sum_{n \in P} f(n) \right|^*
\end{equation*}
for any $q \geq 1$.
\end{itemize}
\end{lemma}

\subsection{Vinogradov lemma}

If $f \colon \Z \to \R/\Z$ is a function, and $I$ is an interval of some length $|I| \geq 1$, we define the \emph{smoothness norm}
\begin{align}\label{eq:smoothf} \| f \|_{C^d(I)} \coloneqq \sup_{0 \leq j \leq d} \sup_{n \in I} |I|^j \| \partial^j_1 f(n) \|_{\R/\Z}
\end{align}
where $\partial_1$ is the difference operator\footnote{In our previous paper \cite{MSTT-all} we used the convention $\partial_1 f(n) = f(n)-f(n-1)$ instead, but this makes little difference to the final statements, and we have switched to this convention to align with other literature in the subject.} $\partial_1 f(n) \coloneqq f(n+1) - f(n)$. We make a similar definition if $f$ takes values in $\R$ instead of $\R/\Z$ (where now we replace $\| \|_{\R/\Z}$ by the absolute value). In the case where $f$ is a polynomial of degree $d$, we have $\partial^j_1 f = 0$ for all $j>d$, and in this case we write $\|f\|_{C^d(I)}$ as $\|f\|_{C^\infty(I)}$. We remark that this definition of $\|f\|_{C^\infty(I)}$ deviates slightly from that in~\cite[Definition 2.7]{green-tao-ratner}; in particular, we allow the index $j$ to equal zero and we allow $n$ to range over $I$ instead of setting $n$ to equal $0$. If $I=[x,x+H]$ with $x,H\geq 1$ and $f(n)=\sum_{k=0}^{d}\alpha_k \binom{n-n_0}{k}$ for some integer $n_0 \in I$, then
\begin{align}\label{eq:smoothcomparison}
\|f\|_{C^{\infty}(I)}\asymp_d \max_{0\leq j\leq d}H^j\|\alpha_j\|_{\R/\Z}.  
\end{align}
The lower bound is clear (specialize $n$ to $n_0$ in~\eqref{eq:smoothf}).  The upper bound comes from using the Taylor expansion identity
$$ \partial_1^j f(n) = \sum_{k=j}^d \binom{n-n_0}{k-j} \partial_1^k f(n_0)$$
and the triangle inequality.  

 If $f$ is a polynomial of degree $d$, for any intervals $I\subset I'$ we have
\begin{align}\label{eq:smoothness}
\|f\|_{C^{\infty}(I)}\leq \|f\|_{C^{\infty}(I')}\ll_d \left(\frac{|I'|}{|I|}\right)^d \|f\|_{C^{\infty}(I)};   
\end{align}
indeed, the first inequality is trivial, and for the second one we can use~\eqref{eq:smoothcomparison}. We also claim the dilation property
\begin{equation}\label{eq:dil}
 \|f(a \cdot) \|_{C^\infty(\frac{1}{a}I)} \ll_d \|f\|_{C^\infty(I)}
\end{equation}
for any $1 \leq a \leq |I|$, where here and later $\frac{1}{a}I := \{\frac{t}{a} \colon t \in I\}$.  In view of~\eqref{eq:smoothcomparison}, it suffices to show that for every $0 \leq k \leq d$, the dilated binomial coefficient $\binom{an}{k}$ has an expansion of the form $\sum_{j=0}^k c_{j,k} \binom{n}{j}$ with $c_{j,k} = O_d(a^k)$.  But this is clear since $\binom{an}{k}$ is a degree $k$ polynomial in $n$ with coefficients $O_d(a^k)$.

We have the following convenient form of the Vinogradov lemma:

\begin{lemma}[Vinogradov lemma]\label{vin}\cite[Lemma 2.3]{MSTT-all}  Let $\eps, \delta \in (0, 1/2)$, $d \geq 0$, and let $P \colon \Z \to \R/\Z$ be a polynomial of degree at most $d$. Let $I$ be an interval with $|I| \geq 1$, and suppose that
$$ \| P(n) \|_{\R/\Z} \leq \eps$$
for at least $\delta |I|$ integers $n \in I$.  Then either we have $\delta \ll_d \eps$, or else 
$$ \| qP \|_{C^\infty(I)} \ll_d \delta^{-O_d(1)} \eps $$
for some integer $1 \leq q \ll_d \delta^{-O_d(1)}$.
\end{lemma}

As a useful corollary of Lemma~\ref{vin} we have, roughly speaking, that if many dilates of a polynomial are smooth, then the polynomial itself is smooth, providing a converse of sorts to \eqref{eq:dil}.

\begin{corollary}[Concatenating dilated smoothness]\label{smooth-dilate}\cite[Corollary 2.4]{MSTT-all}  Let $\delta \in (0, 1/2)$, $d \geq 0$, and let $P \colon \Z \to \R/\Z$ be a polynomial of degree at most $d$.  Let $A \geq 1$, let $I$ be an interval with $|I| \geq 2A$, and suppose that
\begin{equation*}
 \| P(a \cdot) \|_{C^\infty(\frac{1}{a} I)} \leq \frac{1}{\delta}
\end{equation*}
for at least $\delta A$ integers $a$ in $[A,2A]$. Then either we have $|I| \ll_d \delta^{-O_d(1)} A$, or else 
$$ \| q P \|_{C^\infty(I)} \ll_d \delta^{-O_d(1)}$$
for some integer $1 \leq q \ll_d \delta^{-O_d(1)}$.
\end{corollary}

For the proof of our type $II$ estimates, we also need the following proposition (which can be seen as an abstraction of~\cite[Proposition 2.2]{matomaki-shao}) that relates polynomials $P$ satisfying a certain dilation invariance to multiples $T \log$ of the logarithm function.

\begin{proposition}[Approximately dilation-invariant polynomials behave like the logarithm]\label{prop:logarithm} Let $d\geq 0$. Let $X\geq H, A \geq 2$ and $0 < \delta \leq \frac{1}{\log X}$ with $H \geq \delta^{-C_{d}} \max(A, X/A)$ for a large enough constant $C_{d}$. Let $P \colon \mathbb{Z}\to \mathbb{R}/\mathbb{Z}$ be a polynomial of degree $d$. Cover $(A,2A]$ by $O(X/H)$ intervals $I_{A'}= (A',A'\left(1+\frac{H}{X}\right)]$ with $A\leq A'\leq 2A$, with each point belonging to at most $O(1)$ of these intervals. Suppose that for at least $\delta X/H$ of the intervals $I_{A'}$ there exist at least $\delta|I_{A'}|^2$ integer pairs $(a_1,a_2)\in I_{A'}^2$ for which we have  
\begin{align}\label{eq:P1P2}
 \|P(a_1\cdot)-P(a_2\cdot)\|_{C^{\infty}(X/A',X/A'+H/A']}\leq \frac{1}{\delta}.   
\end{align}
Then there exist a real number $T$ satisfying $|T|\ll_{d} \delta^{-O_{d}(1)}(X/H)^{d+1}$ and an integer $1\leq q\ll_{d}\delta^{-O_{d}(1)}$ such that 
\begin{align*}
  \|qP(\cdot)-T\log(\cdot)\|_{C^d(X,X+H]}\ll_{d} \delta^{-O_{d}(1)}.    
\end{align*}
\end{proposition}

Note that the function $T \log(\cdot)$ is dilation-invariant up to constants:
$$ T \log(a_1 \cdot) - T \log(a_2 \cdot) = T \log \frac{a_1}{a_2}.$$
As a consequence, any polynomial $P$ which is close to $T \log$ in $C^d(X,X+H]$ norm would also obey an estimate similar to \eqref{eq:P1P2} for all $a_1,a_2 \in I_{A'}$, thanks to \eqref{eq:dil}.  The above proposition can be viewed as a sort of converse to this assertion.

\begin{proof}
We may assume that $\delta>0$ is smaller than any fixed absolute constant. We allow implied constants to depend on $d$. By rounding, we may assume that $X$ is an integer.

By the pigeonhole principle, we can find $a_2 \in I_{A'}$ and a subinterval $I'_{A'} = (A'', A'' \left(1+\delta^3 \frac{H}{X}\right)]$ such that \eqref{eq:P1P2} holds for $\gg \delta^4 |I_{A'}|$ choices of $a_1 \in I'_{A'}$.  For each such $a_1$, we write 
\begin{align*}
  P(a_1n)-P(a_2n) - (P(a_1X/A'')-P(a_2X/A'')) = \sum_{j=1}^d \beta_j(a_1,a_2)(n-X/A'')^j
\end{align*}
for some real numbers $\beta_j(a_1,a_2)$. From~\eqref{eq:P1P2} and~\eqref{eq:smoothcomparison} we see that $(H/A')^j\|d! \beta_j(a_1,a_2)\|_{\R/\Z}\ll 1/\delta$ for all $1\leq j\leq d$. 
By the triangle inequality, this implies that for all $n = X/A'' + O( \delta^2H/A' )$ we have
\begin{align*}
\|d! (P(a_1n)-P(a_2 n)-(P(a_1X/A'')-P(a_2X/A'')))\|_{\R/\Z}\ll \frac{1}{\delta}\sum_{j=1}^d \delta^{2j}\ll \delta.     
\end{align*}
We conclude that for any $a_1$ as above we have
\begin{align}\label{eq:ePa}
\left|\sum_{X<a_1 n \leq X+\delta^2 H}e(d!P(a_1n)-d!P(a_2n))\right|\gg \delta^2H/A.    
\end{align}
Summing over the possible choices of $a_1$, this implies that
$$
\left|\sum_{a_1 \in I'_{A'}} \sum_{X<a_1 n \leq X+\delta^2 H} \alpha_{a_1} \beta_n e(d!P(a_1n))\right|\gg \delta^6 H    
$$
for some bounded coefficients $\alpha_{a_1}, \beta_n$. 

We can then invoke \cite[Proposition 2.2]{matomaki-shao} and conclude that if we write
\begin{align*}
P(y)=\sum_{j=0}^{d}\alpha_j(y-X)^j  \mod 1
\end{align*}
with $\alpha_j\in \mathbb{R}/\mathbb{Z}$, then for some integer $1\leq q\ll \delta^{-O(1)}$ we have
\begin{align}\label{eq:alphaj}
\|q(j\alpha_{j}+X(j+1)\alpha_{j+1})\|_{\R/\Z}\ll \delta^{-O(1)}X/H^{j+1},    
\end{align}
for all $1\leq j\leq d$, with the convention that $\alpha_{d+1}=0$. By backward induction on $j$, we may thus lift from $\R/\Z$ to $\R$ and write $(-1)^j qj\alpha_j = \beta_j \mod 1$ for some real numbers $\beta_j$ obeying
$$ \beta_j - X \beta_{j+1} = O( \delta^{-O(1)} X / H^{j+1} )$$
for all $j=1,\dots,d$, with the convention that $\beta_{d+1} = 0$.  If we write $\beta_1 = T/X$ for a real number $T$, then a forward induction in $j$ gives
$$ \beta_j = \frac{T}{X^j} + O( \delta^{-O(1)} / H^j )$$
for all $1 \leq j \leq d+1$; in particular substituting $j=d+1$ we have $T \ll \delta^{-O(1)} (X/H)^{d+1}$.  By construction of $\beta_j$, we now have
$$ qd! \alpha_j = d! T \frac{(-1)^j}{jX^j} + O( \delta^{-O(1)} / H^j ) \mod 1$$
for $j=1,\dots,d$. 

Let $\partial_1 f(h) = f(h)-f(h-1)$ be the differencing operator in the $h$ variable.  Then, by the Taylor series of the logarithm function, for all $h = O(H)$ and $1 \leq i \leq d$ we have
\begin{align*}
\partial_1^i [d!qP(X+h)]&=
\sum_{j=1}^d d! q \alpha_j \partial_1^i h^j \\
&= \sum_{j=1}^d \frac{d! (-1)^{j}}{j} \frac{T}{X^j} \partial_1^i h^j + O( \delta^{-O(1)} / H^i ) \mod 1\\
&= \sum_{j=1}^\infty \frac{d! (-1)^{j}}{j} \frac{T}{X^j} \partial_1^i h^j + O( \delta^{-O(1)} / H^i ) \mod 1\\
&= \partial_1^i [d! T \log(1+h/X)]  + O( \delta^{-O(1)} / H^i ) \mod 1\\
&= \partial_1^i [d! T \log(X+h)]  + O( \delta^{-O(1)} / H^i ) \mod 1.
\end{align*}

Making the substitution $n \coloneqq X+h$, we conclude that
$$ \| d! q P(\cdot) - d! T \log(\cdot) \|_{C^d(X,X+H]} \ll \delta^{-O(1)} $$
and the claim follows after some relabeling.  
\end{proof}

\subsection{Equidistribution on nilmanifolds}\label{nilmanifold-sec}

We now recall some of the basic notation and results from~\cite{green-tao-ratner} on equidistribution of polynomial maps on nilmanifolds.

\begin{definition}[Filtered group]\label{def:FiltGroup}  Let $d \geq 1$. A \emph{filtered group} is a group $G$ (which we express in multiplicative notation $G = (G,\cdot)$ unless explicitly indicated otherwise) together with a filtration $G_\bullet = (G_i)_{i=0}^\infty$ of nested groups $G \geq G_0 \geq G_1 \geq \dots$ such that $[G_i,G_j] \leq G_{i+j}$ for all $i,j \geq 0$. We say that this group has degree at most $d$ if $G_i$ is trivial for all $i>d$.  Given a filtered group of degree at most $d$, a \emph{polynomial map} $g \colon \Z \to G$ from $\Z$ to $G$ is a map of the form $g(n) = g_0 g_1^{\binom{n}{1}} \dots g_d^{\binom{n}{d}}$ where $g_i \in G_i$ for all $0 \leq i \leq d$; the collection of such maps will be denoted by $\Poly(\Z \to G)$.  
\end{definition}

The well known Lazard--Leibman theorem (see e.g.~\cite[Proposition 6.2]{green-tao-ratner}) states that $\Poly(\Z \to G)$ is a group under pointwise multiplication; also, from~\cite[Corollary 6.8]{green-tao-ratner} it follows that if $g \colon \Z \to G$ is a polynomial map then so is $n \mapsto g(an+b)$ for any integers $a,b$.

If $G$ is a simply connected nilpotent Lie group, we write $\log G$ for the Lie algebra.  From the Baker--Campbell--Hausdorff formula\footnote{More precisely, the Baker--Campbell--Hausdorff formula, which has only finitely many non-zero terms in the connected case, provides an associative polynomial multiplication law on $\log G$ that agrees locally with the multiplication law on $G$ after conjugating by the exponential map, thus providing a local isomorphism of Lie groups; see e.g., \cite[\S 1.2, 2.4]{tao-hilbert}.  As $\log G$ and $G$ are both simply connected, one can extend this local isomorphism to a global one by monodromy.}  it follows that the exponential map $\exp \colon \log G \to G$ is a homeomorphism and hence has an inverse $\log \colon G \to \log G$.

\begin{definition}[Filtered nilmanifolds] \label{def:filtNilman} Let $d, D \geq 1$ and $0 < \delta < 1$.  A \emph{filtered nilmanifold} $G/\Gamma$ of degree at most $d$, dimension $D$, and complexity at most $1/\delta$ consists of the following data:
\begin{itemize}
\item A filtered simply connected nilpotent Lie group $G$ of dimension $D$ equipped with a filtration $G_\bullet = (G_i)_{i=0}^\infty$ of degree at most $d$, with $G_0=G_1=G$ and all $G_i$ closed connected subgroups of $G$.
\item A lattice (i.e., a discrete cocompact subgroup $\Gamma$) of $G$, with the property that $\Gamma_i \coloneqq \Gamma \cap G_i$ is a lattice of $G_i$ for all $i \geq 0$. 
\item  A linear basis $X_1,\dots,X_D$ (which we call a \emph{Mal'cev basis}) of $\log G$.
\end{itemize}
Furthermore we assume the following axioms:
\begin{itemize}
\item[(i)] For all $1 \leq i,j \leq D$ we have $[X_i,X_j] = \sum_{i,j < k \leq D} c_{ijk} X_k$ for some rational numbers $c_{ijk}$ of height at most $1/\delta$.
\item[(ii)]  For all $0 \leq i \leq D$, the vector space $G_i$ is spanned by the $X_j$ with $D - \dim G_i < j \leq D$.
\item[(iii)]  We have $\Gamma = \{ \exp(n_1 X_1) \cdots \exp(n_D X_D)\colon n_1,\dots,n_D \in \Z \}$.
\end{itemize}
One easily sees that $G/\Gamma$ has the structure of a smooth compact $D$-dimensional manifold, to which we can associate a probability Haar measure $d\mu_{G/\Gamma}$.  We define the metric $d_G$ on $G$ to be the largest right-invariant metric such that $d_G( \exp(t_1 X_1) \dots \exp(t_D X_D), 1) \leq \sup_{1 \leq i \leq D} |t_i|$ for all $t_1,\dots,t_D \in \R$.  We then define a metric $d_{G/\Gamma}$ on $G/\Gamma$ by the formula $d_{G/\Gamma}(x, y) \coloneqq \inf_{g\Gamma = x, h \Gamma = y} d_G(g,h)$.  The Lipschitz norm of a function $F \colon G/\Gamma \to \C$ is defined as
$$ \sup_{x \in G/\Gamma} |F(x)| + \sup_{x,y \in G/\Gamma\colon x \neq y} \frac{|F(x)-F(y)|}{d_{G/\Gamma}(x,y)}.$$

A \emph{horizontal character} $\eta$ associated to a filtered nilmanifold is a continuous homomorphism\footnote{This is a lift of the convention in \cite{green-tao-ratner}, in which a horizontal character is a continuous homomorphism from $G$ to $\R/\Z$ that annihilates $\Gamma$.} $\eta \colon G \to \R$ that maps $\Gamma$ to the integers. 

An element $\gamma$ of $G$ is said to be \emph{$M$-rational} for some $M \geq 1$ if one has $\gamma^r \in \Gamma$ for some natural number $1 \leq r \leq M$.  A subnilmanifold $G'/\Gamma'$ of $G/\Gamma$ (thus $G'$ is a closed connected subgroup of $G$ with $\Gamma'_i \coloneqq G'_i \cap \Gamma$ cocompact in $G'_i$ for all $i$) is said to be \emph{$M$-rational} if each element $X'_1,\dots,X'_{\dim G'}$ of the Mal'cev basis associated to $G'$ is a linear combination of the $X_i$ with all coefficients being rational numbers of height at most $M$.
\end{definition}

One easily sees that every horizontal character takes the form $\eta(g) = \lambda( \log g)$ for some linear functional $\lambda \colon \log G \to \R$ that annihilates $\log [G,G]$ and maps $\log \Gamma$ to the integers.  From this one can verify that the number of horizontal characters of Lipschitz norm at most $1/\delta$ is at most $O_{d,D}( \delta^{-O_{d,D}(1)} )$. 

\begin{remark}
\label{rem:RZhorchar}
In particular, if $G=\mathbb{R}$ and $\Gamma=\mathbb{Z}$, all horizontal characters are of the form $y\mapsto ry$ with $r\in \Z$. 
\end{remark}

From several applications of Baker--Campbell--Hausdorff formula (see e.g.~\cite[Appendix B]{MRTTZ})
we see that if $G$ has degree at most $d$ and $\gamma_1, \gamma_2$ are $M$-rational, then $\gamma_1 \gamma_2$ is $O_d(M^{O_d(1)})$-rational.

The following theorem gives a basic dichotomy between equidistribution and smoothness:

\begin{theorem}[Quantitative Leibman theorem]\label{qlt}\cite[Theorem 2.7]{MSTT-all}  Let $0 < \delta < 1/2$, let $d,D \geq 1$, let $I$ be an interval with $|I| \geq 1$, and let $G/\Gamma$ be a filtered nilmanifold of degree at most $d$, dimension at most $D$, and complexity at most $1/\delta$.  Let $F \colon G/\Gamma \to \C$ be Lipschitz of norm at most $1/\delta$ and of mean zero (i.e., $\int_{G/\Gamma} F\ d\mu_{G/\Gamma} = 0$).  Suppose that $g \colon \Z \to G$ is a polynomial map with
$$ \left|\sum_{n \in I} F(g(n)\Gamma)\right|^* \geq \delta |I|.$$
Then there exists a non-trivial horizontal character $\eta \colon G \to \R$ of Lipschitz norm $O_{d,D}(\delta^{-O_{d,D}(1)})$ such that
$$ \| \eta \circ g\|_{C^\infty(I)} \ll_{d,D} \delta^{-O_{d,D}(1)}.$$
\end{theorem}

Let $G/\Gamma$ be a filtered nilmanifold of dimension $D$ and complexity at most $1/\delta$.  A \emph{rational subgroup} $G'$ of complexity at most $1/\delta$ is a closed connected subgroup of $G$ with the property that $\log G'$ admits a linear basis consisting of $\dim G'$ vectors of the form $\sum_{i=1}^D a_i X_i$, where each $a_i$ is a rational of height at most $1/\delta$.  In~\cite[Proposition A.10]{green-tao-ratner} it is shown that $G'/\Gamma'$ can be equipped with the structure of a filtered nilmanifold of complexity $O_{d,D}(\delta^{-O_{d,D}(1)})$, where $\Gamma' \coloneqq \Gamma \cap G'$, $G'_i \coloneqq G_i \cap G'$, and the metrics $d_G, d_{G'}$ are comparable on $G'$ up to factors of $O_{d,D}(\delta^{-O_{d,D}(1)})$; one can view $G'/\Gamma'$ as a subnilmanifold of $G/\Gamma$.	

One can easily verify from basic linear algebra and the Baker--Campbell--Hausdorff formula that the following groups are rational subgroups of $G$ of complexity $O_{d,D}(\delta^{-O_{d,D}(1)})$:
\begin{itemize}
\item The groups $G_i$ in the filtration for $0 \leq i \leq d$.
\item The kernel $\ker \eta$ of any horizontal character $\eta$ of Lipschitz norm $O_{d,D}(\delta^{-O_{d,D}(1)})$.
\item The center $Z(G) = \{ \exp(X)\colon X \in \log G; [X,Y] = 0\,\, \forall Y \in \log G \}$ of $G$. 
\item The intersection $G' \cap G''$ or commutator $[G',G'']$ of two rational subgroups $G',G''$ of $G$ of complexity $O_{d,D}(\delta^{-O_{d,D}(1)})$.
\item The product $G' N$ of two rational subgroups $G',N$ of $G$ of complexity $O_{d,D}(\delta^{-O_{d,D}(1)})$, with $N$ normal.
\end{itemize}

If we quotient out a filtered nilmanifold by a normal subgroup, we obtain another filtered nilmanifold, with polynomial bounds on complexity:

\begin{lemma}[Quotienting by a normal subgroup]\label{quotient-normal}\cite[Lemma 2.8]{MSTT-all}  Let $G/\Gamma$ be a filtered nilmanifold of degree at most $d$, dimension $D$ and complexity at most $1/\delta$.  Let $N$ be a normal rational subgroup of $G$ of complexity at most $1/\delta$, and let $\pi \colon G \mapsto G/N$ be the quotient map.  Then $\pi(G)/\pi(\Gamma)$ can be given the structure of a filtered nilmanifold of degree at most $d$, dimension $D - \dim N$, and complexity $O_{d,D}(\delta^{-O_{d,D}(1)})$, such that
$$d_{\pi(G)}( \pi(g), \pi(h) ) \asymp_{d,D} \delta^{-O_{d,D}(1)} \inf_{n \in N} d_G(g, nh)$$
for any $g,h \in G$.
\end{lemma}

A \emph{central frequency} is a continuous homomorphism $\xi \colon Z(G) \to \R$ which maps $Z(G) \cap \Gamma$ to the integers $\Z$ (that is to say, a horizontal character on $Z(G)$, or a Fourier character of the central torus $Z(G) /(Z(G) \cap \Gamma)$).  A function $F \colon G/\Gamma \to \C$ is said to \emph{oscillate with central frequency} $\xi$ if one has the identity
$$ F(zx) = e(\xi(z)) F(x)$$
for all $x \in G/\Gamma$ and $z \in Z(G)$.  Similarly to horizontal characters, the number of central frequencies $\xi$ of Lipschitz norm at most $1/\delta$ is $O_{d,D}(\delta^{-O_{d,D}(1)})$.  If $\xi$ is such a central frequency, one can easily verify that the kernel $\ker \xi \leq Z(G)$ is a rational normal subgroup of $G$ of complexity $O_{d,D}(\delta^{-O_{d,D}(1)})$.

We have the following convenient decomposition:

\begin{proposition}[Central Fourier approximation]\label{central}\cite[Proposition 2.9]{MSTT-all}  Let $d,D \geq 1$ and $0 < \delta < 1$.  Let $G/\Gamma$ be a filtered nilmanifold of degree at most $d$, dimension $D$, and complexity at most $1/\delta$.  Let $F \colon G/\Gamma \to \C$ be a Lipschitz function of norm at most $1/\delta$.  Then we can decompose
$$ F = \sum_\xi F_\xi + O(\delta)$$
where $\xi$ ranges over central frequencies of Lipschitz norm at most $O_{d,D}(\delta^{-O_{d,D}(1)})$, and each $F_\xi$ has Lipschitz norm
$O_{d,D}(\delta^{-O_{d,D}(1)})$ and oscillates with central frequency $\xi$.  Furthermore, if $F$ has mean zero, then the same holds for all of the $F_\xi$.
\end{proposition}

We now recall some properties of polynomial sequences.  

\begin{definition}[Smoothness, total equidistribution, rationality]  Let $G/\Gamma$ be a filtered nilmanifold, $g \in \Poly(\Z \to G)$ be a polynomial sequence, $I \subset \R$ be an interval of length $|I| \geq 1$, and $M>0$.
\begin{itemize}
\item[(i)] We say that $g$ is \emph{$(M,I)$-smooth} if one has
$$ d_G(g(n), 1_G) \leq M; \quad d_G(g(n), g(n-1)) \leq M/|I|$$
for all $n \in I$.
\item[(ii)]  We say that $g$ is \emph{totally $1/M$-equidistributed} in $G/\Gamma$ if one has
$$ \left| \frac{1}{|P|} \sum_{n \in P} F(g(n)\Gamma) - \int_{G/\Gamma} F \right| \leq \frac{1}{M} \|F\|_{\Lip}$$
whenever $F \colon G/\Gamma \to \C$ is Lipschitz and $P$ is an arithmetic progression in $I$ of cardinality at least $|I|/M$.
\item[(iii)]  We say that $g$ is \emph{$M$-rational} if for some $1 \leq r \leq M$ one has $g(n)^r \in \Gamma$ for each $n \in \Z$.  
\end{itemize}
\end{definition}

From Taylor expansion and the Baker--Campbell--Hausdorff formula it is not difficult to see that if $G/\Gamma$ has degree at most $d$ and $g$ is $M$-rational, then the map $n \mapsto g(n) \Gamma$ is $q$-periodic for some period $1 \leq q \ll_d M^{O_d(1)}$; see, e.g., \cite[Lemma A.12]{green-tao-ratner}.  

We also have the following basic facts about smooth and rational polynomials for any interval $I$:

\begin{lemma}\label{mult}  Suppose that $G/\Gamma$ is a filtered nilmanifold of degree and dimension $O(1)$ and complexity at most $M$ for some $M \geq 1$.
\begin{itemize}
\item[(i)] (Smooth polynomials form an approximate group) The constant polynomial $1$ is $M$-smooth on $I$, and if $g,h$ are $M$-smooth on $I$, then $g^{-1}$ and $gh$ are $M^{O(1)}$-smooth.
\item[(ii)] (Rational polynomials form an approximate group) The constant polynomial $1$ is $M$-rational, and if $g,h$ are $M$-rational, then $g^{-1}$ and $gh$ are $M^{O(1)}$-rational.
\end{itemize}
\end{lemma}

\begin{proof}
Claim (i) follows from the triangle inequality, the right-invariance of the metric $d_G$, and the approximate left-invariance of $d_G$ (see~\cite[Lemma A.5]{green-tao-ratner}). Claim (ii) is~\cite[Lemma A.11(v)]{green-tao-ratner}.    
\end{proof}

\begin{lemma}\label{factor-simple}~\cite[Lemma 2.11]{MSTT-all}
Let $d,D \geq 1$ and $0 < \delta < 1$.  Let $G/\Gamma$ be a filtered nilmanifold of degree at most $d$, dimension $D$, and complexity at most $1/\delta$.  Let $g \in \Poly(\Z \to G)$, and let $I$ be an interval with $|I| \geq 1$.  Suppose that
$$ \|\eta \circ g\mod 1\|_{C^{\infty}(I)} \leq 1/\delta $$
for some non-trivial horizontal character $\eta\colon G \rightarrow \R$ of Lipschitz norm at most $1/\delta$.
Then there is a decomposition $g = \eps g' \gamma$ into polynomial maps $\eps, g', \gamma \in \Poly(\Z \to G)$ such that
\begin{itemize}
\item[(i)] $\eps$ is $(\delta^{-O_{d,D}(1)},I)$-smooth;
\item[(ii)]  $g'$ takes values in $G' = \ker\eta$;
\item[(iii)] $\gamma$ is $\delta^{-O_{d,D}(1)}$-rational.
\end{itemize}
\end{lemma}

We also rely on the following large sieve inequality for nilsequences, which is a more quantitative variant of the one in~\cite[Proposition 4.11]{MRTTZ} and will be deduced from~\cite[Proposition 2.15]{MSTT-all}.

\begin{proposition}[Large sieve]\label{large-sieve}
Let $d,D \geq 1$ and $0 < \delta < 1$.  Let $G/\Gamma$ be a filtered nilmanifold of degree at most $d$, dimension $D$, and complexity at most $1/\delta$, whose center $Z(G)$ is one-dimensional. Let $I$ be an interval with $|I| \geq 1$, and let $F \colon G/\Gamma \to \C$ be Lipschitz of norm at most $1/\delta$ and having a non-zero central frequency $\xi$.  Let $f \colon \Z \to \C$ be a function with $\sum_{n \in I} |f(n)|^2 \leq \frac{1}{\delta} |I|$. Then there exist $K\ll_{d,D} \delta^{-O_{d,D}(1)}$ and elements $g_1,\ldots, g_K\in \Poly(\mathbb{Z}\to G)$ such that the following holds. Whenever $g\in \Poly(\mathbb{Z}\to G)$ satisfies
\begin{equation}\label{eq:sumf}
 \left|\sum_{n \in I} f(n) \overline{F}(g(n) \Gamma)\right|^* \geq \delta |I|,
\end{equation}
at least one of the following holds.
\begin{itemize}
\item[(i)]  There exists a non-trivial horizontal character $\eta \colon G \to \R$ having  Lipschitz norm $O_{d,D}(\delta^{-O_{d,D}(1)})$ such that $\| \eta \circ g \|_{C^\infty(I)} \ll_{d,D} \delta^{-O_{d,D}(1)}$.
\item[(ii)] There exists $i\in \{1, \dotsc, K\}$ such that we can write
$$ g = \eps g_i \gamma$$
where $\eps$ is $(O_{d,D}(\delta^{-O_{d,D}(1)}),I)$-smooth and $\gamma$ is $O_{d,D}(\delta^{-O_{d,D}(1)})$-rational.
\end{itemize}
\end{proposition}

Informally speaking, this proposition asserts that a given function $f$ can only correlate with a bounded number of essentially distinct non-trivial nilsequences $F(g(\cdot)\Gamma)$; it is a higher-order analogue of the classical large sieve inequality,  which asserts (roughly speaking) that $f$ will only have a bounded number of essentially distinct frequencies where the Fourier transform is large.

\begin{proof}
This follows quickly from~\cite[Proposition 2.15]{MSTT-all}. We may assume that $\delta<1/2$. Let $C$ be large enough in terms of $d, D$, and let $C'$ be large enough in terms of $C$. Let $\mathcal{G}=\{g_1,\ldots, g_J\}$ be a maximally large collection of elements of $\Poly(\mathbb{Z}\to G)$ such that
\begin{align*}
\left|\sum_{n\in I}f(n)\overline{F}(g_i(n)\Gamma)\right|^{*}\geq \delta I    
\end{align*}
for all $1\leq j\leq J$ and such that for any $h,h'\in \mathcal{G}$ with $h\neq h'$ we have the
following properties:
\begin{enumerate}
    \item[($P_C$)] For any non-trivial horizontal character $\eta\colon G\to \mathbb{R}$ of Lipschitz norm $\leq \delta^{-C}$ we have $\|\eta\circ h\mod 1\|_{C^{\infty}(I)}> \delta^{-C}$;

    \item[($Q_C$)] If $h=\varepsilon h'\gamma$ with $\varepsilon, \gamma\in \Poly(\mathbb{Z}\to G)$, then either $\varepsilon$ is not $(\delta^{-C},I)$-smooth or $\gamma$ is not $\delta^{-C}$-rational.
\end{enumerate}

We may assume that the collection $\{g_1,\ldots, g_J\}$ is non-empty. If $C$ is large enough, by~\cite[Proposition 2.15]{MSTT-all} we have $J\leq \delta^{-C}$. Now, let $g\in \Poly(\mathbb{Z}\to G)$ be any polynomial sequence satisfying~\eqref{eq:sumf}. If $g$ fails property $P_C$, we are done. If the pair $(g,g_{j'})$ fails property $Q_{C'}$ for some $1\leq j'\leq J$, we are again done. Finally, we are left with the case where $g$ satisfies $P_C$ and $(g,g_{j'})$ satisfies $Q_{C'}$ for any $1\leq j'\leq J$ (and therefore by Lemma~\ref{mult} also $(g_{j},g)$ satisfies $Q_C$ for any $1\leq j\leq J$, thanks to the assumption that $C'$ is large in terms of $C$). But since property $Q_{C'}$ implies property $Q_C$, this contradicts the maximality of the collection $\{g_1,\ldots, g_J\}$. The claim follows. 
\end{proof}

The following proposition from~\cite{MSTT-all} will be an important ingredient in the proof of our type $II$ estimate, although it will be applied in a different context to the one in~\cite{MSTT-all}, namely in order to analyze the outcome of Walsh's contagion argument.  It will play a role in the nonabelian type $II$ theory similar to the role Proposition \ref{prop:logarithm} will play in the abelian type $II$ theory.

\begin{proposition}[Abstract non-abelian type $II$ inverse theorem]~\cite[Proposition 6.1]{MSTT-all}\label{prop:Furstenberg-Weiss} Let $C\geq 1$, $d,D \geq 1$, $X \geq H\geq 2$, $X\geq A \geq 1$, $0 < \delta < \frac{1}{\log X}$, and let $G/\Gamma$ be a filtered nilmanifold of degree at most $d$, dimension at most $D$, and complexity at most $1/\delta$, with $G$ non-abelian.  Let $g\colon \Z \to G$ be a polynomial map.  Cover $(A,2A]$ by at most $C X/H$ intervals $I_{A'} = (A',(1+\frac{H}{X}) A']$ with $A \leq A' \leq 2A$, with each point belonging to at most $C$ of these intervals.  Suppose that for at least $\frac{1}{C} \delta^{C} X/H$ of the intervals $I_{A'}$, there exist at least $\frac{1}{C} \delta^C |I_{A'}|^2$ integer pairs $(a,a') \in I_{A'}^2$ for which there exists a factorization
$$ g(a' \cdot) = \eps_{aa'} g(a \cdot) \gamma_{aa'}$$
where $\eps_{aa'}$ is $(C \delta^{-C}, J_{A'})$-smooth and $\gamma_{aa'}$ is $C\delta^{-C}$-rational, with $J_{A'}$ denoting the interval
$$J_{A'} \coloneqq \left(\left(1-\frac{10 H}{X}\right) \frac{X}{A'}, \left(1+\frac{10H}{X}\right) \frac{X}{A'}\right].$$
Then either
$$ H \ll_{d,D,C} \delta^{-O_{d,D,C}(1)} \max( A, X/ A)$$
or there exists a non-trivial horizontal character $\eta \colon G \to \R$ of Lipschitz norm $O_{d,D,C}(\delta^{-O_{d,D,C}(1)})$ such that 
$$ \| \eta \circ g \|_{C^\infty(X,X+H]} \ll_{d,D,C} \delta^{-O_{d,D,C}(1)}.$$
\end{proposition}

\subsection{Combinatorial lemmas}

The following lemma is Heath-Brown's identity in the form of~\cite[Lemma 2.16]{MSTT-all}, but with one additional case.

\begin{lemma}\label{hb-identity}
Let $X \geq 2$, and let $L \in \mathbb{N}$ be fixed. We may find a collection $\mathcal{F}$ of $O((\log X)^{O(1)})$ functions $f \colon \mathbb{N} \to \mathbb{R}$, such that
$$ \Lambda(n) = \sum_{f \in \mathcal{F}} f(n) $$
for each $X/2 \leq n \leq 4X$, and each $f \in \mathcal{F}$ takes the form
$$ f = a^{(1)}* \cdots * a^{(\ell)} $$
for some $\ell \leq 2L$, where $a^{(i)}$ is supported on $(N_i, 2N_i]$ for some $N_i \geq 1/2$, and each $a^{(i)}(n)$ is either $1_{(N_i, 2N_i]}(n)$, $(\log n)1_{(N_i, 2N_i]}(n)$, or $\mu(n)1_{(N_i, 2N_i]}$. Further, $N_1N_2\cdots N_{\ell} \asymp X$, and $N_i \ll X^{1/L}$ for each $i$ with $a^{(i)}(n) = \mu(n) 1_{(N_i, 2N_i]}(n)$. The same statement holds with $\mu$ in place of $\Lambda$ (but $(\log n)1_{(N_i, 2N_i]}(n)$ does not appear) and with $d_k$ in place of $\Lambda$ (but $(\log n)1_{(N_i, 2N_i]}(n)$ and $\mu(n)1_{(N_i, 2N_i]}(n)$ do not appear).
\end{lemma}

\begin{proof}
The cases of $\mu$ and $\Lambda$ are~\cite[Lemma 2.16]{MSTT-all}. For the case of $d_k$, note that $d_k$, which is the $k$-fold convolution of $1$, can be decomposed into a sum of  $O(\log^k X)$ terms, each
of which takes the form
$$f = 1_{(N_1,2N_1]} *\cdots * 1_{(N_k,2N_k]}$$
for some $N_i \geq  1/2$ with $N_1N_2\cdots N_k \asymp X$.
\end{proof}

We shall also frequently use Shiu's bound (see~\cite[Theorem 1]{shiu}).

\begin{lemma}{\cite[Lemma 2.17]{MSTT-all}}\label{le:shiu} Let $A\geq 1$ and $\varepsilon>0$ be fixed. Let $X\geq H\geq X^{\varepsilon}$ and let $1\leq q\leq H^{1-\varepsilon}$ be an integer. Let $f$ be
a non-negative multiplicative function such that $f(p^{\ell})\leq A^{\ell}$ for every prime $p$ and every $\ell\in \mathbb{N}$, and suppose that
$f(n)\ll_c n^{c}$ for every $c > 0$. Then for any integer $a$ coprime to $q$ we have
\begin{align*}
\sum_{\substack{X<n\leq X+H\\n\equiv a\pmod q}}f(n)\ll \frac{H}{\varphi(q)\log X}\exp\Bigg(\sum_{\substack{p\leq 2X\\p\nmid q}}\frac{f(p)}{p}\Bigg).    
\end{align*}
\end{lemma}

We will also need the fact that the approximants $\Lambda^{\sharp}$ and $d_k^{\sharp}$ are essentially type $I$ sums.

\begin{lemma}\label{le:dkdecompose} Let $k\geq 2$ be a fixed integer, let $\varepsilon>0$ be fixed, and let $X\geq 3$.
\begin{itemize}
    \item[(i)] There exists a sequence $a(n)$ supported on $[1, X^{\varepsilon/5}]$ with $a(n) \ll \log X$ such that, for every $1\leq n\leq 3X$, we have $\Lambda^{\sharp}(n)=(a*1)(n)+E(n)$, where $E$ is negligible in the sense that for any $X^{\varepsilon}\leq H\leq x\leq 2X$ we have
    \begin{align*}
      \sum_{x<n\leq x+H}|E(n)|\ll H\exp(-\log^{1/20} X).   
    \end{align*}
\item[(ii)] There exists $J \ll 1$ and sequences $a_1(n), \dotsc, a_J(n)$ supported on $[1, X^{\varepsilon/5}]$ with $|a_j(m)| \leq d_{k-1}(m)$ such that, for each $1\leq n\leq 3X$, we have $d_k^{\sharp}(n)=\sum_{1\leq j\leq J}(a_j*\psi_j)(n)$, where $\psi_j(n)=(\log^{\ell_j} n)/(\log^{\ell_j} X)$ for some integer $0\leq \ell_j\leq k-1$. 
\end{itemize}
\end{lemma}

\begin{proof}(i) Recall the definitions $R=\exp((\log X)^{1/10})$ and $P(w)=\prod_{p<w}p$ from~\eqref{eq:Lambdasharpdef}. Defining
$$E(n)=\Lambda^{\sharp}(n)-\frac{P(R)}{\varphi(P(R))}\sum_{\substack{d\mid n\\d\leq X^{\varepsilon/5}\\d\mid P(R)}}\mu(d),$$
the function $\Lambda^{\sharp}(n)-E(n)$ is of the desired form $a*1$, and by M\"obius inversion and~\cite[Lemma 2.18]{MSTT-all} for $X^{\varepsilon}\leq H\leq x\leq 2X$ we have
\begin{align}\label{eq:lambdaI}
\sum_{x<n\leq x+H}|E(n)|\leq \frac{P(R)}{\varphi(P(R))}\sum_{\substack{x<dn\leq x+H\\d>X^{\varepsilon/5}\\d\mid P(R)}}1\ll_{\varepsilon} H\exp(-\log^{1/20}X).  
\end{align}

(ii) By~\eqref{dks-def} and~\eqref{eq:Pmt}, we have
\begin{align*}
d_k^{\sharp}(n)=\sum_{j=0}^{k-1}\frac{\binom{k}{j}}{(k-j-1)!}\sum_{\substack{n=n_1\cdots n_{k-1}d\\n_1,\ldots, n_{j-1}\leq R_k\\R_k<n_{j+1},\ldots, n_{k-1}\leq R_k^2}}\bigg(\frac{\log(R_k^{k-j}n_{j+1}\cdots n_{k-1})+\log d}{\log R_k}\bigg)^{k-j-1}.    
\end{align*}
Applying the binomial formula to the inner sum and merging the variables $n_1,\ldots, n_{k-1}$, the claim follows.
\end{proof}

\section{Major arc estimates}\label{major-arc-sec}

The purpose of this section is to handle the case of Theorem~\ref{discorrelation-thm} where $F(g(n)\Gamma)$ is major arc in the sense that it behaves like $n^{iT}$ in arithmetic progressions to small modulus (the case $T=0$ would correspond to the classical major arcs in the circle method), for some $T$ of polynomial size. To this end, we prove the following theorem.

\begin{theorem}[Major arc estimate]\label{discorrelation-thm-major}  Let $X\geq 3$ and $X^{\theta+\varepsilon} \leq H \leq X$ for some $\theta\in (0,1)$ and some small $\eps\in (0,1)$.  Let $T$ be a real number with $|T| \leq X^C$ for some $C>0$.
\begin{itemize}
\item[(i)]  (Major arc discorrelation) Let $\theta = 1/3$. Then, for any $A>0$, one has
\begin{align*}
\left| \sum_{x < n \leq x+H} \mu(n) n^{-iT} \right|^*\leq H \log^{-A} X
\end{align*}
and
\begin{align*}
\left| \sum_{x < n \leq x+H} (\Lambda(n) - \Lambda^\sharp(n)) n^{-iT} \right|^* \leq H \log^{-A} X
\end{align*}
for all $x\in [X,2X]$ outside of a set of measure $O_{A,\eps, C}( X\log^{-A}X)$. Furthermore,  for any integer $k\geq 2$ and some constant $c_{k,C}>0$ depending only on $k,C$, one has
\begin{equation*}
\left| \sum_{x < n \leq x+H} (d_k(n) - d_k^\sharp(n)) n^{-iT} \right|^* \leq HX^{-c_{k,C} \eps}
\end{equation*}
for all $x\in [X,2X]$ outside of a set of measure $O_{\eps,C}(X^{1-c_{k,C}\eps})$.
\item[(ii)]  (Major arc discorrelation with weight on prime factors)  Let $\theta=0$ and $\kappa \in (0, \varepsilon/20)$. Then, for any $A>0$ and any integer $k\geq 2$, one has
\begin{equation*}
\left| \sum_{x < n \leq x+H} \mu(n)n^{-iT} \sum_{X^{\kappa}<p\leq X^{\varepsilon/10}}1_{p\mid n}  \right|^* \leq H  \log^{-A} X
\end{equation*}
and
\begin{equation*}
\left| \sum_{x < n \leq x+H} (d_k(n) - d_k^\sharp(n))n^{-iT}  \sum_{X^{\kappa}<p\leq X^{\varepsilon/10}}1_{p\mid n} \right|^*\leq H  \log^{-A} X
\end{equation*}
for all  $x\in [X,2X]$ outside of a set of measure $O_{A,\eps,\kappa, C}(X\log^{-A}X)$.
\end{itemize}
\end{theorem}

The proof strategy of Theorem~\ref{discorrelation-thm-major} is to apply combinatorial decompositions to $\mu, \Lambda-\Lambda^{\sharp}, d_k-d_k^{\sharp}$ (possibly with an additional weight as in part (ii)) and to prove estimates for the resulting type $I$ and $II$ major arc sums. These sums will be handled in Subsections~\ref{sub:typeImajor} and~\ref{sub:typeIImajor}, respectively, and then we are left with some reductions (involving the removal of maximal summation) and combinatorial considerations in Subsections~\ref{sub:smallT}--\ref{sub:largeT}.  For the logic of the proof, see \Cref{fig-major}.

\begin{figure}[H]
    \centering
\begin{tikzpicture}[node distance=1cm and 1cm]
    \node[draw, rectangle, align=center, fill=blue!20] (main) {Major arc estimate\\ \Cref{discorrelation-thm-major}};
    \node[draw, rectangle, above =of main, align=center] (small) {Small $|T|$ case of \\ \Cref{discorrelation-thm-major}, \S \ref{sub:smallT}};
    \node[draw, rectangle, above left=of main, align=center] (large-i) {Large $|T|$ case of \\ \Cref{discorrelation-thm-major}(i), \S \ref{sub:largeT}};
    \node[draw, rectangle, above right=of main, align=center] (large-ii) {Large $|T|$ case of \\ \Cref{discorrelation-thm-major}(ii), \S \ref{sub:largeT}};
    \node[draw, rectangle, above=of large-i, align=center] (type-i) {Type $I$ estimate \\ \Cref{le:typeImajor}};
    \node[draw, rectangle, above=of small, align=center] (type-ii) {Type $II$ estimate \\ \Cref{le:typeIImajor}};
    \node[draw, rectangle, above left=of type-ii, align=center] (parseval) {Parseval bound \cite{mr-annals}\\ \Cref{le:perron}};
    \node[draw, rectangle, above right=of type-ii, align=center] (bhp) {Parallelogram lemma \cite{baker-harman-pintz}\\ \Cref{le:BHP}};

    \draw[-{Latex[length=4mm, width=2mm]}] (small) -- (main);
    \draw[-{Latex[length=4mm, width=2mm]}] (large-i) -- (main);
    \draw[-{Latex[length=4mm, width=2mm]}] (large-ii) -- (main);
    \draw[-{Latex[length=4mm, width=2mm]}] (type-i) -- (large-i);
    \draw[-{Latex[length=4mm, width=2mm]}] (type-ii) -- (large-i);
    \draw[-{Latex[length=4mm, width=2mm]}] (type-ii) -- (large-ii);
    \draw[-{Latex[length=4mm, width=2mm]}] (type-ii) -- (small);
    \draw[-{Latex[length=4mm, width=2mm]}] (parseval) -- (type-ii);
    \draw[-{Latex[length=4mm, width=2mm]}] (bhp) -- (type-ii);
\end{tikzpicture}
    \caption{Proof of \Cref{discorrelation-thm-major}. The Vinogradov--Korobov bound (\Cref{le:vin-kor}) and  Heath--Brown identity (\Cref{hb-identity}) are also used at several points in the proof, but are not depicted on this diagram to reduce clutter.}
    \label{fig-major}
\end{figure}
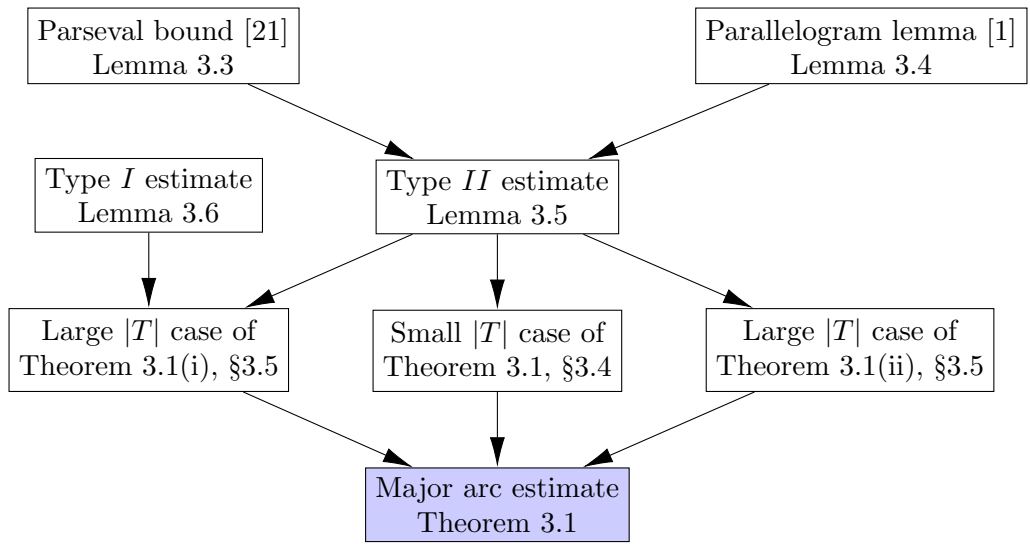

\subsection{Lemmas on Dirichlet polynomials}

We begin with the following standard pointwise bounds for Dirichlet polynomials of $\mu, \Lambda, d_k$.

\begin{lemma}[Vinogradov--Korobov bound]\label{le:vin-kor}
Let $k\in \mathbb{N}$ be fixed. Let  $f\in \{\mu, \Lambda, d_k\}$. Let $X\geq 3$,  $A>0$, and define $W_f$ by the formulae
$$W_{\mu}\coloneqq 0; \quad W_{\Lambda} \coloneqq \log^{A} X; \quad W_{d_k} \coloneqq X^{c_k},$$
where $c_k>0$ is a small enough constant. Then we have
\begin{align*}
\sup_{W_f\leq |T|\leq X^{A}}\bigg|\sum_{X<n\leq 2X}\frac{f(n)}{n^{1+iT}}\bigg|^{*}\ll_A \frac{\log^{O(1)}X}{(W_f+\log^{A}X)^{1/2}}.  \end{align*}
\end{lemma}

\begin{proof}
For some interval $I\subset [X,2X]$ and some integers $q\geq r\geq 1$, we have
\begin{align}\label{eq:maxsum}
\bigg|\sum_{X<n\leq 2X}\frac{f(n)}{n^{1+iT}}\bigg|^{*}=\Bigg|\sum_{\substack{n\in I\\n\equiv r\pmod q}} \frac{f(n)}{n^{1+iT}}\Bigg|.
\end{align}
We may assume that $q\ll (W_f+\log^{A}X)^{1/2}$, as otherwise the claimed estimate follows directly from the triangle inequality (and Lemma~\ref{le:shiu} in the case $f=d_k$). Let $d=(r,q)$ and $r'=r/d, q'=q/d$. Then using the orthogonality of characters, the right-hand side of~\eqref{eq:maxsum} becomes
\begin{align}\label{eq:fdn}
\leq \frac{1}{\varphi(q')}\sum_{\chi\pmod{q'}} \bigg|\sum_{\substack{m\in d^{-1}I\\}} \frac{f(dm)\chi(m)}{m^{1+iT}}\bigg|.
\end{align}
Now the claim follows in the case of $f\in \{\mu, \Lambda\}$ from~\cite[Lemma 3.9]{MSTT-all}(ii)--(iii) (which is an application of the Vinogradov--Korobov zero free region) and partial summation. The case $f=d_k$ follows from~\cite[Lemma 3.9]{MSTT-all}(i) and partial summation after first writing $d_k(n)=\sum_{n=n_1\cdots n_k}1$ in~\eqref{eq:fdn} and pulling out all but the largest of the variables $n_i$ outside the absolute values.
\end{proof}

We will also need the following standard estimate, which relates the variance of a sequence in short intervals to the mean square of its Dirichlet polynomial.

\begin{lemma}[Parseval-type bound for variance]\label{le:perron} Let $X\geq 3$ and $1\leq W\leq X^{1/10}$. Let $C>0$ and $|a(n)|\leq d_2(n)^{C}$. Let\footnote{For some applications where $a$ is a convolution of two dyadically supported sequences, it is convenient that $A(s)$ involves a sum over an interval somewhat longer than $[X,2X]$; therefore we use the interval $[X/8,8X]$ here.} $A$ denote the Dirichlet polynomial
$$A(s) \coloneqq \sum_{X/8< n\leq 8X}a(n)n^{-s}.$$
\begin{itemize}
    \item[(i)] For $1\leq H_1\leq H_2\leq X/W^3$, one has
    \begin{align*}
    &\frac{1}{X}\int_{X}^{2X}\left|\frac{1}{H_1}\sum_{x<n\leq x+H_1}a(n)-\frac{1}{H_2}\sum_{x< n\leq x+H_2}a(n)\right|^2\, dx\\
    \ll_C& \left(\frac{\log^{2^C-1}X}{W}\right)^2+\max_{T\geq \frac{X}{H_1}}\frac{X/H_1}{T}\int_{W\leq |t|\leq T}|A(1+it)|^2\, dt.   \end{align*}

    \item[(ii)] For $1\leq H_1\leq X$, one has
     \begin{align*}
    \frac{1}{X}\int_{X}^{2X}\bigg|\frac{1}{H_1}\sum_{x<n\leq x+H_1}a(n)\bigg|^2\, dx\ll \max_{T\geq \frac{X}{H_1}}\frac{X/H_1}{T}\int_{|t|\leq T}|A(1+it)|^2\, dt.
    \end{align*}
\end{itemize}

\end{lemma}

\begin{proof}
This is very similar to~\cite[Lemma 14]{mr-annals}. Let us give some details for part (i); part (ii) is similar but easier.

For any  $x\geq 1$ such that $x,x+H_1,x+H_2$ are not integers, Perron's formula gives
\begin{align*}
&\frac{1}{H_1}\sum_{x<n\leq x+H_1}a(n)-\frac{1}{H_2}\sum_{x<n\leq x+H_2}a(n)\\
=&\frac{1}{2\pi}\int_{\mathbb{R}}A(1+it)\left(\frac{(x+H_1)^{1+it}-x^{1+it}}{H_1(1+it)}-\frac{(x+H_2)^{1+it}-x^{1+it}}{H_2(1+it)}\right)\,dt\\
\coloneqq &I_1(x)+I_2(x),
\end{align*}
where $I_1(x)$ is the integral over $|t|\leq W$ and $I_2(x)$ is the integral over $|t|>W$. By Taylor approximation, for any real number $|y|\leq 1$ and complex number $s$, we have $(1+y)^{s}=1+sy+O((|s|^2+1)y^2)$. Substituting this into the integrand in $I_1(x)$ and recalling that $H_2\leq X/W^3$, we obtain
\begin{align*}
 |I_1(x)|^2&\ll \left(\int_{|t|\leq W}|A(1+it)|(|t|+1)\frac{H_2}{x}\,dt\right)^2 \ll \left(W^2\frac{H_2}{X}\max_{|t|\leq W}|A(1+it)|\right)^2\\
 &\ll \left(W^{-1}\sum_{X/8<n\leq 8X}\frac{|a(n)|}{n}\right)^2 \ll \left(\frac{\log^{2^C-1}X}{W}\right)^2,
\end{align*}
using in the last step the standard divisor sum bound $\sum_{X/8< n\leq  8X}d_2(n)^{C}/n\ll \log^{2^C-1} X$.

To handle the contribution of $I_2(x)$, it suffices to show for $j\in \{1,2\}$ that
\begin{align*}
\int_{X}^{2X}\left|\int_{|t|>W}A(1+it)\frac{(x+H_j)^{1+it}-x^{1+it}}{H_j(1+it)}\, dt\right|^2\, dx\ll \max_{T\geq \frac{X}{H_1}}\frac{X^2/H_1}{T}\int_{W\leq |t|\leq T}|A(1+it)|^2\, dt.
\end{align*}
This estimate is contained in the proof of~\cite[Lemma 14]{mr-annals} (and is proven by inserting a smooth majorant for $1_{[X,2X]}(x)$ into the integral over $x$ and opening the square).
\end{proof}

For proving type $II$ major arc estimates, we need a lemma of Baker, Harman and Pintz~\cite{baker-harman-pintz} on mean values of products of Dirichlet polynomials.

\begin{lemma}[Baker--Harman--Pintz parallelogram lemma]\label{le:BHP} Let $\theta\in (1/2,1)$, $\varepsilon>0$ and $K\geq 2$ be fixed. Let $X\geq 3$. Let $|a(n)|,|b(n)|,|c(n)|\leq d_2(n)^{K}$, and define $$A(s)=\sum_{M<m\leq KM}a(m)m^{-s},\quad  B(s)=\sum_{N<n\leq KN}b(n)n^{-s},\quad C(s)=\sum_{R<r\leq KR}c(r)r^{-s},$$
where $M,N,R\geq 1$ satisfy $MNR\asymp X$. Suppose that $M=X^{\alpha_1}$, $N=X^{\alpha_2}$, where $\alpha_1,\alpha_2\in [0,1]$ satisfy
\begin{align*}
|\alpha_1-\alpha_2|<2\theta-1-\frac{\varepsilon}{10},\quad \alpha_1+\alpha_2>1-\min\left\{4\theta-2,\frac{4\theta-1}{3},\frac{24\theta-13}{3}\right\}+\frac{\varepsilon}{10}.  \end{align*}
Suppose also that for some $0\leq T_0\leq X^{1-\theta}$ and $1\leq W\leq X^{\varepsilon/1000}$ we have
\begin{align}\label{eq:Cbound}
\max_{T_0\leq |t|\leq X^{1-\theta}}|C(1+it)|\leq \frac{1}{W^{1/3}}.  \end{align}
Then we have
\begin{align}\label{eq:ABCintegral}
\int_{T_0}^{X^{1-\theta}}|A(1+it)||B(1+it)||C(1+it)|\, dt\ll \frac{\log^{O(1)} X}{W^{1/3}}.
\end{align}
\end{lemma}

\begin{proof}
This is essentially~\cite[Lemma 7.3]{harman-book} with $g=1$ (alternatively, see~\cite[Lemma 9]{baker-harman-pintz}), the proof of which is based on mean and large value estimates for Dirichlet polynomials. One minor difference is that in that lemma one had $T_0=\exp(\log^{1/3} X)$ and $W=\log^{-A} X$ (and $\varepsilon=o(1)$), but the same proof works with the more general choices above. Another minor difference is that in~\cite[Lemma 7.3]{harman-book} the integral in~\eqref{eq:ABCintegral} was over the $1/2$-line rather than the $1$-line (and similarly~\eqref{eq:Cbound} was replaced with an analogous estimate on the $1/2$-line), but this also makes no difference in the proof.
\end{proof}

\subsection{Type \texorpdfstring{$II$}{II} major arc estimates}\label{sub:typeIImajor}

In this subsection, we consider the kind of type $II$ major arc sums that arise in the proof of Theorem~\ref{discorrelation-thm-major}.
The following lemma is used in the proof of Theorem~\ref{discorrelation-thm-major} to control type $II$ sums.

\begin{lemma}[Type $II$ major arc estimate]\label{le:typeIImajor}
 Let $\varepsilon>0$ be small and let  $C>0$.   Also let $|a(m)|,|b(m)|\leq d_2(m)^{C}$, $X\geq 3$ and  $1 \leq W\leq X^{\varepsilon/1000}$.
\begin{itemize}
\item[(i)] If $X^{\eps}\leq M\leq X^{1/2}$, $X/W^4 \geq H_2 \geq H_1 \geq X^{1/3+\varepsilon}$, and
\begin{align}\label{eq:supbound}
  \sup_{W\leq |t|\leq XW/H_1}\bigg|\sum_{m\sim M}\frac{a(m)}{m^{1+it}}\bigg|\leq \frac{1}{W^{1/3}},\quad  \sup_{W\leq |t|\leq XW/H_1}\bigg|\sum_{X/(3M)<n\leq 3X/M}\frac{b(n)}{n^{1+it}}\bigg|  \leq \frac{1}{W^{1/3}},
 \end{align}
we have
\begin{align*}
\bigg|\sum_{\substack{x<mn\leq x+H_1\\m\sim M}} a(m)b(n)-\frac{H_1}{H_2}\sum_{\substack{x<mn\leq x+H_2\\m\sim M}} a(m)b(n)\bigg|\leq \frac{H_1}{W^{1/10}}  \end{align*}
for all $x\in [X,2X]$ outside of a set of measure $O_{\eps,C}( X(\log^{O_C(1)} X)/W^{1/10})$. 

\item[(ii)] The same conclusion also holds assuming $X/W^4 \geq H_2 \geq H_1 \geq X^{2\varepsilon}$, and $X^{\varepsilon}\leq M\leq H_1X^{-\varepsilon}$, and the first estimate in~\eqref{eq:supbound} (without assuming $H_1 \geq X^{1/3+\eps}$ or the second estimate in~\eqref{eq:supbound}).

\item[(iii)] If $X^{\eps}\leq M\leq X^{1/2}$, $X\geq H_1\geq X^{1/3+\varepsilon}$, and
\begin{align}\label{eq:supbound2}
  \sup_{|t|\leq XW/H_1}\bigg|\sum_{m\sim M}\frac{a(m)}{m^{1+it}}\bigg| \leq \frac{1}{W^{1/3}},\quad   \sup_{ |t|\leq XW/H_1}\bigg|\sum_{X/(3M)<n\leq 3X/M}\frac{b(n)}{n^{1+it}}\bigg|  \leq \frac{1}{W^{1/3}},
 \end{align}
then we have
\begin{align*}
\bigg|\sum_{\substack{x<mn\leq x+H_1\\m\sim M}} a(m)b(n)\bigg|\leq \frac{H_1}{W^{1/10}}
\end{align*}
for all $x\in [X,2X]$ outside of a set of measure $O_{\varepsilon,C}( X(\log^{O_C(1)} X)/W^{1/10})$. 

\item[(iv)]The same conclusion also holds assuming $X \geq H_1 \geq X^{2\varepsilon}$, and $X^{\varepsilon}\leq M\leq H_1X^{-\varepsilon}$, and the first estimate in~\eqref{eq:supbound2} (without assuming $H_1 \geq X^{1/3+\eps}$ or the second estimate in~\eqref{eq:supbound2}).
\end{itemize}
\end{lemma}

\begin{proof}[Proof of Lemma~\ref{le:typeIImajor}]
Let $K_C>0$ be a large enough constant which may depend on $C$. We may assume that $W\geq \log^{K_C} X$, since otherwise, by adjusting the $O_C(1)$ constants in the lemma, we may take the exceptional set there to contain all of $[X,2X]$. Applying Chebyshev's inequality, it suffices to show in the case of Lemma~\ref{le:typeIImajor}(i)--(ii) that
\begin{align}\label{eq:perrongoal1}\begin{split}
\frac{1}{X}\int_{X}^{2X}\bigg|\sum_{\substack{x<mn\leq x+H_1\\m\sim M}} a(m)b(n)-\frac{H_1}{H_2}\sum_{\substack{x<mn\leq x+H_2\\m\sim M}} a(m)b(n)\bigg|^2\, dx\ll H_1^2\frac{\log^{O_C(1)}X}{W^{3/10}},\quad
\end{split}
\end{align}
and in the case of Lemma~\ref{le:typeIImajor}(iii)--(iv) that
\begin{align}\label{eq:perrongoal2}
\frac{1}{X}\int_{X}^{2X}\bigg|\sum_{\substack{x<mn\leq x+H_1\\m\sim M}} a(m)b(n)\bigg|^2\, dx\ll H_1^2\frac{\log^{O_C(1)}X}{W^{3/10}}.
\end{align}

Let $N \coloneqq X/M$ and
\begin{align*}
F(s)=\sum_{\substack{X/8<mn\leq 8X\\m\sim M\\N/3<n\leq 3N}}a(m)b(n)(mn)^{-s}.
\end{align*}
Then $F(s)=A(s)B(s)$, where
\begin{align*}
A(s)\coloneqq \sum_{m\sim M}a(m)m^{-s},\quad B(s)\coloneqq \sum_{N/3<n\leq 3N}b(n)n^{-s}.
\end{align*}
By Lemma~\ref{le:perron}(i),  the claim~\eqref{eq:perrongoal1} reduces to showing that
\begin{align}\label{eq:integralbound1}
\int_{W\leq |t|\leq T}|A(1+it)B(1+it)|^2\, dt\ll \frac{TH_1}{X}\cdot \frac{\log^{O_C(1)} X}{W^{3/10}}
\end{align}
uniformly for all $T\geq X/H_1$. Similarly, by Lemma~\ref{le:perron}(ii), the claim~\eqref{eq:perrongoal2} reduces to showing that
\begin{align}\label{eq:integralbound2}
\int_{-T}^{T}|A(1+it)B(1+it)|^2\, dt\ll \frac{TH_1}{X}\cdot \frac{\log^{O_C(1)} X}{W^{3/10}}
\end{align}
uniformly for all $T\geq X/H_1$.

 Note that by the mean value theorem for Dirichlet polynomials (see e.g. \cite[Theorem 9.1]{ik}) and a standard bound for moments of the divisor function we have
\begin{align*}
\int_{-T}^{T}|A(1+it)B(1+it)|^2\, dt&\ll(T+X)\sum_{X/8<\ell\leq 8X}\frac{|\sum_{\ell=mn}a(m)b(n)|^2}{\ell^2}\\
&\ll (T+X)\sum_{X/8<\ell\leq 8X}\frac{d_2(\ell)^{4C+2}}{\ell^2} \ll \left(1+\frac{T}{X}\right) \log^{O_C(1)} X.
\end{align*}
Since $W\leq H_1^{1/20}$, it follows that~\eqref{eq:integralbound1} and~\eqref{eq:integralbound2} hold in the case $T\geq XW^{3/10}/H_1$. We may therefore assume from now on that
\begin{align}\label{eq:Tupperbound}
T< XW^{3/10}/H_1.
\end{align}

If $M\leq H_1/W^{3/10}$, we can use the pointwise bound from~\eqref{eq:supbound} on $|A(1+it)|$ and the mean value theorem for Dirichlet polynomials to bound
\begin{align*}
\int_{W \leq |t| \leq T}|A(1+it)B(1+it)|^2\, dt&\ll \frac{1}{W^{2/3}}\int_{-T}^{T}|B(1+it)|^2\, dt\\
&\ll \frac{T+\frac{X}{M}}{W^{2/3}}\sum_{X/(8M)<n\leq 8X/M}\frac{|b(n)|^2}{n^2} \ll \frac{\log^{O_C(1)} X}{W^{2/3}},
\end{align*}
so~\eqref{eq:integralbound1} holds. We obtain~\eqref{eq:integralbound2} similarly. Hence parts (ii) and (iv) of the lemma follow (since there we assume that $X^{\varepsilon}\leq M\leq H_1X^{-\varepsilon}$).

We are left with proving parts (i) and (iii) (where we assume $X\geq H_1\geq X^{1/3+\varepsilon}$ and $X^{\varepsilon}\leq M\leq X^{1/2}$). By~\eqref{eq:Tupperbound} we have $T<X^{2/3}$. It now suffices to prove that
\begin{align}\label{eq:F2}
\int_{W\leq |t|\leq X^{2/3}}|A(1+it)B(1+it)|^2\, dt\ll \frac{\log^{O_C(1)} X}{W^{3/10}}
\end{align}
in the case of Lemma~\ref{le:typeIImajor}(i) and
\begin{align}\label{eq:F2b}
\int_{|t|\leq X^{2/3}}|A(1+it)B(1+it)|^2\, dt\ll \frac{\log^{O_C(1)} X}{W^{3/10}}
\end{align}
in the case of Lemma~\ref{le:typeIImajor}(iii).

In what follows, denote $M=X^{\alpha_1}, N=X^{\alpha_2}$, so that the coefficients of $A(s)$ are supported on $[X^{\alpha_1},2X^{\alpha_1}]$ and the coefficients of $B(s)$ are supported on $[X^{\alpha_2}/3,3X^{\alpha_2}]$. The estimates~\eqref{eq:F2} and~\eqref{eq:F2b} follow from\footnote{Alternatively, one could use~\cite[(9.2.3) in Lemma 9.4]{harman-book}. We have chosen to reduce matters to Lemma~\ref{le:BHP} as that lemma is better known.}
Lemma~\ref{le:BHP}. Indeed, we can apply that lemma with $X^2$ in place of $X$, and with $\theta=2/3$ and $T_0\in \{W,0\}$. Applying it to the Dirichlet polynomials $A(s), A(s), B(s)^2$, and recalling~\eqref{eq:supbound},~\eqref{eq:supbound2}, we obtain the claim if
\begin{align*}
2\alpha_1>1-\frac{4\cdot \frac{2}{3}-1}{3}+\frac{\varepsilon}{10}=\frac{4}{9}+\frac{\varepsilon}{10}.
\end{align*}
We are left with the case $\alpha_1\leq 2/9+\varepsilon/20$. But in this case applying Lemma~\ref{le:BHP} to the Dirichlet polynomials $B(s), B(s), A(s)^2$ instead, we obtain the claim if
\begin{align*}
2\alpha_2>1-\frac{4\cdot \frac{2}{3}-1}{3}+\frac{\varepsilon}{10}=\frac{4}{9}+\frac{\varepsilon}{10},
\end{align*}
which must now hold since $\alpha_1+\alpha_2=1$. This completes the proof.
\end{proof}

\subsection{Type \texorpdfstring{$I$}{I} major arc estimates}\label{sub:typeImajor}

Next we shall prove some type $I$ major arc estimates that are useful in the case when $|T|$ is large.

\subsubsection{Type $I$ major arc estimates without weights}

The following lemma asserts, roughly speaking, that type $I$ major arc sums on short intervals are small as soon as $n^{iT}$ exhibits non-trivial oscillation on the interval.  

\begin{lemma}[Type $I$ major arc estimate]\label{le:typeImajor} Let $C>0$ and $\varepsilon>0$.   Let $2\leq X^{\varepsilon}\leq H\leq X^{1-\varepsilon}$, $0<\delta<1/\log X$ and $1\leq M\leq HX^{-\varepsilon/2}$.  Let $|a(m)|\leq d_2(m)^{C}$, and suppose that $a$ is supported on $[M,2M]$. Suppose that for some real number $T \in [-X^C,X^C]$ we have
\begin{align*}
\bigg|\sum_{x<n\leq x+H}(a*1)(n)n^{-iT}\bigg|^*\geq \delta H
\end{align*}
for all $x\in [X,2X]$ in a set of measure $\geq \delta X$.
Then we have
\begin{align}\label{eq:OC1}
|T|\ll_{\varepsilon,C} \delta^{-O_{\varepsilon,C}(1)}\frac{X}{H}.
\end{align}
\end{lemma}

 In fact, only the case $M\leq X^{\varepsilon}$ of this lemma will be needed, since the type $II$ region in Lemma~\ref{le:typeIImajor} is very wide.

\begin{proof}
    We allow all implied constants in this proof to depend on $\eps,C$.  On increasing the $O_{\varepsilon,C}(1)$ constant in~\eqref{eq:OC1} if necessary, we may assume that  $X$ is large enough in terms of $\varepsilon,C$ and that
\begin{equation}\label{delta-inv}
    \delta^{-1}\leq X^{c_{C,\eps}}
\end{equation}
 for a suitably small constant $c_{C,\eps}>0$.

By Taylor approximation, for $y\in [x,x+H]$ and for any $x\in [X,2X]$ and integer $J\geq 1$, we can write
\begin{align}\label{eq:Taylor-yiT}
y^{iT}=e\left(\frac{T}{2\pi}\log y\right)=e(P_x(y))+O_J\bigg(T\bigg(\frac{H}{X}\bigg)^{J+1}\bigg)
\end{align}
where
\begin{align*}
    P_x(y) \coloneqq \frac{T}{2\pi} \log x + \frac{T}{2\pi}\sum_{j=1}^{J}\frac{(-1)^{j-1}}{jx^j}(y-x)^j
\end{align*}
is the Taylor polynomial of degree $J$ of the function $\frac{T}{2\pi}\log(\cdot)$ at $x$.
We choose $J = O(1)$ to be large enough so that the error term in~\eqref{eq:Taylor-yiT} becomes $O(X^{-1})$.
Then, for all $x\in [X,2X]$ in a set $E$ of measure $\geq \delta X$, we have
\begin{align*}
\bigg|\sum_{x<n\leq x+H} (a*1)(n)e(-P_x(n))\bigg|^{*}\gg \delta H.
\end{align*}
Since $H\in \mathbb{N}$, without loss of generality we may assume that $E$ is a union of $\gg \delta X$ intervals of the form $[m,m+1)$ with $m\in \mathbb{N}$. By assumptions, the sum above is a $(\delta^{O(1)}, 2M)$ type $I$ sum in the notation of~\cite[Definition 4.1(i)]{MSTT-all}. Since $M \leq HX^{-\eps/2}$, it follows from the type $I$ estimate of~\cite[Theorem 4.2(i)]{MSTT-all} (together with Remark~\ref{rem:RZhorchar}) that for all $x\in E$ there is some integer $1\leq k_x\ll \delta^{-O(1)}$ such that
\begin{align}\label{eq:rxj}
\left\|k_x\frac{T}{2\pi jx^j}\right\|_{\R/\Z} \ll \frac{\delta^{-O(1)}}{H^j}    \end{align}
for all $1\leq j\leq J$. Actually we will only need the $j=1$ case of this bound.  By the pigeonhole principle, we can find an integer $1\leq k\ll \delta^{-O(1)}$ such that
\begin{equation}\label{eq:kt}
    \left\|\frac{kT}{2\pi x}\right\|_{\R/\Z}\ll \frac{\delta^{-O(1)}}{H}
\end{equation}
for $\gg \delta^{O(1)} X$ integers $x$ in $[X,2X]$.

If we have
$$ \frac{k|T|}{2\pi X} < \frac{1}{2},$$
then \eqref{eq:kt} simplifies to
$$\frac{k|T|}{2\pi x} \ll \frac{\delta^{-O(1)}}{H},$$
and \eqref{eq:OC1} follows. Thus we may assume that
$$ \frac{k|T|}{2\pi X} \geq \frac{1}{2}.$$
If we then define $h \coloneqq \lfloor \sqrt{H} \rfloor$, then $X^{1+\eps/4} \ll hk|T| \ll X^{O_C(1)}$ and
\begin{equation}\label{eq:hkt}
    \left\|\frac{hkT}{2\pi x}\right\|_{\R/\Z}\ll \frac{\delta^{-O(1)}}{H^{1/2}}
\end{equation}
for $\gg \delta^{O(1)} X$ integers $x$ in $[X,2X]$.

From van der Corput's exponential sum bound~\cite[Theorem 8.20]{ik} we have
\begin{align}\label{eq:vdc} \sum_{X < n \leq 2X} e \left( \frac{mhkT}{2\pi n} \right) \ll X^{1-c'_{C,\eps}}
\end{align}
for all $1 \leq m \leq X^{c'_{C,\eps}}$ and some $c'_{C,\eps} > 0$.  Applying the Erd\H{o}s--Tur\'an inequality~\cite[Corollary 1.1 in Section 1]{Montgomery} and~\eqref{eq:vdc}, we conclude that the bound \eqref{eq:hkt} can only hold for at most
$$ \ll \frac{\delta^{-O(1)}}{H^{1/2}} X + X^{1-c'_{C,\eps}} \log X$$
integers $x \in [X,2X]$. But this is a contradiction if the constant $c_{C,\eps}$ in \eqref{delta-inv} is small enough.
\end{proof}

\subsection{Theorem~\ref{discorrelation-thm-major} in the case of small \texorpdfstring{$|T|$}{|T|}}\label{sub:smallT} The proof of Theorem~\ref{discorrelation-thm-major} splits into the ``small $|T|$ case''
\begin{align}\label{eq:smallT}
    |T|\leq \frac{X}{H} X^{\varepsilon/2}
\end{align}
and the ``large $|T|$ case''
\begin{align}\label{eq:largeT}
    |T| > \frac{X}{H} X^{\varepsilon/2}.
\end{align}

In this section we treat the case of small $|T|$; the large $|T|$ case will be handled in \Cref{sub:largeT}.

Let $\varepsilon,\kappa>0$ be small, let $A,C>0$ be large, and let $c_{k,C}>0$ be small enough. To have a unified notation to treat all cases, we set $l \coloneqq 1$ when proving \Cref{discorrelation-thm-major}(i) and $l \coloneqq 2$ when proving \Cref{discorrelation-thm-major}(ii), fix $f\in \{\mu, \Lambda,d_k\}$ for $l=1$ and $f \in \{\mu,d_k\}$ for $l=2$, and introduce the parameters and weights
\begin{align}\label{eq:deltachoice}
\delta&\coloneqq \begin{cases}
 \log^{-A}X\quad  &\textnormal{if } f\in \{\mu, \Lambda\} \textnormal{ or } l=2,\\
 X^{-c_{k,C}\varepsilon} \quad  &\textnormal{if } f=d_k \textnormal{ and } l=1,
\end{cases}\\
g(n)&\coloneqq \begin{cases}
1\quad   &\textnormal{if } l=1,\\\nonumber
\sum_{X^{\kappa}<p\leq X^{\varepsilon/10}}1_{p\mid n}\quad   &\textnormal{if } l=2,
\end{cases}\\\nonumber
H'&\coloneqq H / X^{\eps/2} \geq \begin{cases}
    X^{1/3+\varepsilon/2} & \textnormal{if } l=1, \\
    X^{\varepsilon/2} & \textnormal{if } l=2,
\end{cases}\\\nonumber
H^{*}& \coloneqq X^{0.99}.
\end{align}
The condition \eqref{eq:smallT} can now be written as
\begin{align}\label{eq:smallT-alt}
|T|\leq \frac{X}{H'}.
\end{align}

By Chebyshev's inequality, it suffices to show that
\begin{align}\label{eq:majorarcgoal}
\frac{1}{X}\int_{X}^{2X}\bigg(\bigg|\sum_{x<n\leq x+H}(f(n)-f^{\sharp}(n))g(n)n^{-iT}\bigg|^{*}\bigg)^2 \, dx\ll_{A,\eps,\kappa}  \delta^{3}H^2.
\end{align}
 From~\eqref{eq:smallT-alt} we see that the function $n\mapsto n^{-iT}$ has total variation $O(1)$ on any interval $I\subset [X,4X]$ of length $H'$.  Thus, by covering $(x,x+H]$ by $O(H/H')$ intervals of length $H'$ and using \Cref{basic-prop}(i)--(ii), we have
$$
\bigg|\sum_{x<n\leq x+H}(f(n)-f^{\sharp}(n))g(n)n^{-iT}\bigg|^{*}
\ll \sum_{0 \leq \ell \leq \frac{H}{H'}}
\bigg|\sum_{x+\ell H'<n\leq x+(\ell+1) H'}(f(n)-f^{\sharp}(n))g(n)\bigg|^{*}
$$
and hence by Cauchy--Schwarz
$$
\bigg(\bigg|\sum_{x<n\leq x+H}(f(n)-f^{\sharp}(n))g(n)n^{-iT}\bigg|^{*}\bigg)^2
\ll \frac{H}{H'} \sum_{0 \leq \ell \leq \frac{H}{H'}}
\bigg(\bigg|\sum_{x+\ell H'<n\leq x+(\ell+1) H'}(f(n)-f^{\sharp}(n))g(n)\bigg|^{*}\bigg)^2.
$$
Integrating and applying a translation, it thus suffices to show that
\begin{align}\label{eq:goal}
\frac{1}{X}\int_{X}^{3X}\bigg(\bigg|\sum_{x<n\leq x+H'}(f(n)-f^{\sharp}(n))g(n)\bigg|^{*}\bigg)^2 \, dx\ll_{A,\eps,\kappa}  \delta^{3}(H')^2. 
\end{align}
Thus we have eliminated the $n^{-iT}$ factor, at the slight cost of worsening $H$ to $H'$.

Now we deal with the maximal truncation $|\cdot|^*$.  Observe the bounds $g(n)\ll_{\kappa} 1$ for $n\leq 4X$ and $d_k^{\sharp}(n)\ll d_k(n)\ll n^{o(1)}$ (see~\cite[(3.14)]{MSTT-all}). Hence, by Shiu's bound (Lemma~\ref{le:shiu}), we have
\begin{align}\label{eq:shiubound}
\sum_{\substack{x<n \leq x+H' \\ n\equiv r \pmod q}} |f(n) - f^\sharp(n)|g(n) \ll \left(\frac{H'}{\varphi(q)}+(H')^{1/100}\right) \log^{O(1)} X 
\end{align}
(say) for any $1 \leq r \leq q$.  Thus by \eqref{maximal-sum} we have
\begin{align*}
& \bigg|\sum_{x<n\leq x+H'}(f(n)-f^{\sharp}(n))g(n)\bigg|^{*}\\
&\quad \ll \sum_{1 \leq r \leq q \leq \delta^{-4}} \sup_{0 \leq H'' \leq H'}
\bigg|\sum_{\substack{x<n \leq x+H'' \\ n\equiv r \pmod q}} (f(n) - f^\sharp(n)) g(n)\bigg| + \delta^2 H' \log^{O(1)} X.
\end{align*}
By a further application of Shiu's bound, one can round $H''$ to the nearest multiple of $\delta^{10} H'$ without significantly worsening the error term.  Thus we have
\begin{align}\label{eq:H'sum}\begin{split}
    & \bigg|\sum_{x<n\leq x+H'}(f(n)-f^{\sharp}(n))g(n)\bigg|^{*}\\
&\quad \ll \sum_{1 \leq r \leq q \leq \delta^{-4}} \sum_{\substack{0 \leq H'' \leq H' \\ H'' \in (\delta^{10} H') \cdot \mathbb{Z}}}
\bigg|\sum_{\substack{x<n \leq x+H'' \\ n\equiv r \pmod q}} (f(n) - f^\sharp(n)) g(n)\bigg| \\
&\quad\quad + \delta^2 H' \log^{O(1)} X.
\end{split}
\end{align}

The contribution of the $\delta^2 H' \log^{O(1)} X$ error to~\eqref{eq:goal} is acceptable if $A$ is large enough.  As there are $O(\delta^{-18})$ summands in~\eqref{eq:H'sum}, it will suffice by the Cauchy--Schwarz inequality to show that
\begin{equation}\label{eq:majorarcgoal2}
 \frac{1}{X}\int_{X}^{3X} \bigg|\sum_{\substack{x<n\leq x+H'' \\ n\equiv r \pmod q}}(f(n)-f^{\sharp}(n))g(n)\bigg|^2 \, dx\ll_{A,\eps,\kappa}  \delta^{39} (H')^2
\end{equation}
for all $1 \leq r \leq q \leq \delta^{-4}$ and all $H'' \in [0,H']$.  By~\eqref{eq:shiubound} we may assume that $H'' \geq \delta^{100}H'\gg H/X^{2\varepsilon/3}$.

Fix $r,q,H''$.
From~\cite[Theorem 3.1 with $\varepsilon = 1/200$]{MSTT-all} (and the fact that $g(n)$ is a sum of $O(X^{\varepsilon/10})$ indicators of arithmetic progressions, and we assume that our $\varepsilon$ is small), we have the long interval estimate 
\begin{align*}
 \bigg|\sum_{x < n \leq x+H^{*}} (f(n)-f^{\sharp}(n)) g(n)\bigg|^{*}\ll_{A,\varepsilon} \delta^{100}H^{*}.
 \end{align*}
for all $x \in [X,3X]$. Hence, by the triangle inequality,~\eqref{eq:majorarcgoal2} follows if we show that
\begin{align}\label{eq:majorarcgoal3}\begin{split}
&\frac{1}{X}\int_X^{3X} \Bigg| \sum_{\substack{x < n \leq x+H''\\n\equiv r \pmod q}} (f(n)-f^{\sharp}(n))g(n)-\frac{H''}{H^{*}}\sum_{\substack{x< n \leq x+H^{*}\\n\equiv r \pmod q}} (f(n)-f^{\sharp}(n))g(n) \Bigg|^2\ dx\\ \ll&_{A,\eps,\kappa} \delta^{39} (H'')^2.
\end{split}
\end{align}

We now claim the estimate
\begin{align}\label{eq:fsharp1}
 \bigg|\sum_{\substack{x<n\leq x+H''\\n\equiv r \pmod q}}f^{\sharp}(n) g(n)-\frac{H''}{H^{*}}\sum_{\substack{x< n \leq x+H^{*}\\n\equiv r \pmod q}} f^{\sharp}(n) g(n)\bigg|\ll_{A,\varepsilon} \delta^{100}H''
 \end{align}
 for all $x \in [X,3X]$. In the case $f=\mu$, this bound is trivial since $\mu^{\sharp}$ vanishes.  In the case $f\in \{\Lambda, d_k\}$, $l=1$ (so that $g=1$), this follows from~\cite[Lemma 3.3]{MSTT-all}, noting that $H'' \geq X^{1/4}$ when $l=1$.  The only remaining case to verify is when $f=d_k$, $l=2$ (note that we do not permit $f=\Lambda$ when $l=2$). In that case, by Lemma~\ref{le:dkdecompose}(ii), $d_k^{\sharp}$ is a sum of $O(1)$ type $I$ sums $a_j*\psi$ where $\psi=\log^{\ell_j}$ for some integers $0\leq \ell_j\ll 1$, some divisor-bounded sequences $a_j$ supported on $[1,X^{\varepsilon/5}]$. Moreover,  $g$ is a sum of $O(X^{\varepsilon/10})$ indicators of arithmetic progressions. Hence,  \eqref{eq:fsharp1} follows in this case by using 
 \begin{align*}
\bigg| \sum_{\substack{n\in I\\n\equiv r \pmod q}}1_{m\mid n}-\frac{|I|}{|J|}\sum_{\substack{n\in J\\n\equiv r \pmod q}}1_{m\mid n}\bigg|\ll 1
 \end{align*}
 for any nonempty intervals $I,J$ with $|I|\leq |J|$, and for various integers $1\leq m\leq X^{\varepsilon/5+\varepsilon/10}$, and applying partial summation. 
 
In view of \eqref{eq:fsharp1} and the triangle inequality, we have reduced~\eqref{eq:majorarcgoal3} to showing that
\begin{align}\label{eq:majorarcgoal3a}\begin{split}
\frac{1}{X}\int_X^{3X} \Bigg| \sum_{\substack{x < n \leq x+H''\\n\equiv r \pmod q}} f(n) g(n)-\frac{H''}{H^{*}}\sum_{\substack{x< n \leq x+H^{*}\\n\equiv r \pmod q}} f(n) g(n) \Bigg|^2\ dx\ll_{A,\eps,\kappa} \delta^{39} (H'')^2
\end{split}
\end{align}
We now treat the $l=1,2$ cases separately.

\subsubsection{Proof of~\eqref{eq:majorarcgoal3a} for $l=1$}\label{sub:HB}
By Lemma~\ref{hb-identity} with $L=3$, for $X/2\leq n\leq 4X$ we can decompose $f$ into the sum of $O(\log^{O(1)} X)$ functions $h$, each of which is bounded by $d_2^{O(1)}$ and is of one of the following forms:

\begin{itemize}
    \item[Type $I$:] $h=a*\psi$, with $a$ supported on a dyadic subinterval of $[1,X^{\varepsilon}]$, and $\psi\equiv 1$ or $\psi\equiv \log^j$ with $0\leq j\ll 1$.
    \item[Type $II$:] $h=a*b$, with each of $a,b$ a convolution of at most five functions, each of which is a restriction of one of the functions $1,\log,\mu$ to an interval, with $a$ and $b$ each being supported on some dyadic subinterval of $[c_0X^{\varepsilon},c_0^{-1}X^{1-\varepsilon}]$ for some constant $c_0>0$.
\end{itemize}

By the triangle inequality, to prove~\eqref{eq:majorarcgoal3a}, it will suffice to show that
\begin{align}\label{eq:majorarcgoal3b}\begin{split}
    \frac{1}{X}\int_X^{3X} \Bigg| \sum_{\substack{x < n \leq x+H''\\n\equiv r \pmod q}} h(n)-\frac{H''}{H^{*}}\sum_{\substack{x< n \leq x+H^{*}\\n\equiv r \pmod q}} h(n) \Bigg|^2\ dx\ll_{A,\eps,\kappa} \delta^{40} (H'')^2.
    \end{split}
    \end{align}

Recall that by the choice of $\delta$ in~\eqref{eq:deltachoice}, the bound in~\eqref{eq:majorarcgoal3b} amounts to a power-saving if $f=d_k$ and $l=1$, and to an arbitrary power of logarithm saving otherwise. By Lemma~\ref{le:vin-kor}, if $h=a*b$ and we are in the type $II$ case, the sequences $a$ and $b$ satisfy~\eqref{eq:supbound} with $W=\log^{1000A}X$ in the case $f\in \{\Lambda,\mu\}$ and with $W=X^{c_k}$ for some sufficiently small $c_k$ in the case $f=d_k$. Hence, if $h$ is of type $II$, then the claim~\eqref{eq:majorarcgoal3b} follows from Lemma~\ref{le:typeIImajor}(i). 

Consider then the case where $h=a*\psi$ is of type $I$ with $a$ supported on $[M,2M]$ with $1\leq M\leq X^{\varepsilon}$ and $\psi \equiv \log^j$ for some $j\ll 1$. Note that if $j\geq 1$, we have
\begin{align*}
\psi(\ell)=j\int_{1}^{X}\frac{(\log u)^{j-1}}{u}1_{u\leq m}\,d u.    
\end{align*}
Using this decomposition and the estimate $\sum_{\ell\in I\cap P}1=|I\cap P|+O(1)$ for any interval $I$ and arithmetic progression $P$,  for any $M\in [1,X^{\varepsilon}]$ we have
\begin{align*}
\sum_{\substack{x<n\leq x+H''\\n\equiv r \pmod q}}h(n)&=\sum_{m\sim M}a(m)\sum_{\substack{x/m<\ell\leq (x+H'')/m\\\ell m\equiv r \pmod q}}\psi(\ell)\\
=&\sum_{m\sim M}a(m)\frac{H''}{H^{*}}\sum_{\substack{x/m<\ell\leq (x+H^{*})/m\\\ell m\equiv r \pmod q}}\psi(\ell)+O\left(\left(\frac{H''}{H^{*}}+1\right)MX^{o(1)}\right)\\
=&\frac{H''}{H^{*}} \sum_{\substack{x<n\leq x+H^{*}\\n\equiv r \pmod q}}h(n)+O\left(\left(\frac{H''}{H^{*}}+1\right)MX^{o(1)}\right).
\end{align*}
This implies~\eqref{eq:majorarcgoal3b} and hence concludes the proof of~\eqref{eq:majorarcgoal3a} in the $l=1$ case.

\subsubsection{Proof of~\eqref{eq:majorarcgoal3a} for $l=2$}

Now let $f\in\{\mu, d_k\}$. Since $f$ is multiplicative, we have the decomposition
\begin{align}\label{eq:ramare1}\begin{split}
f(n)g(n)&=\sum_{\substack{n=pm\\X^{\kappa}<p\leq X^{\varepsilon/10}}}f(p)f(m)+O(d_k(n)(\log X)1_{p^2\mid n\textnormal{ for some } p>X^{\kappa}})\\
&\coloneqq (a*b)(n)+O(d_k(n)(\log X)1_{p^2\mid n\textnormal{ for some } p>X^{\kappa}}),
\end{split}
\end{align}
where $a(n)=f(n)1_{n\in \mathbb{P}\cap (X^{\kappa},X^{\varepsilon/10}]}$.
We substitute this into~\eqref{eq:majorarcgoal3a}, noting that the $O(\cdot)$ error term contributes $\ll (H''X^{-\kappa/3})^2 \log^{O(1)} X$ by the triangle inequality and Shiu's bound (\Cref{le:shiu}). Hence, it suffices to prove~\eqref{eq:majorarcgoal3a} with $f\cdot g$ replaced by the type $II$ sum $a*b$. Now~\eqref{eq:majorarcgoal3a} follows from Lemma~\ref{le:typeIImajor}(ii) (after decomposing $a$ into $O(\log X)$ dyadically supported sequences), using Lemma~\ref{le:vin-kor} to verify the first estimate in~\eqref{eq:supbound}.

\subsection{Theorem~\ref{discorrelation-thm-major} in the case of large \texorpdfstring{$|T|$}{|T|}}\label{sub:largeT}

Now we handle the large $T$ case $|T|  \geq \frac{X}{H} X^{\varepsilon/2}$.

 We recall the notions of type $I$ and type $II$ sums from Subsection~\ref{sub:HB}.
We begin by noting that $d_k^{\sharp}$ and $\Lambda^{\sharp}$ can be approximated by suitable type $I$ sums.  By Lemma~\ref{le:dkdecompose}(i), for $n\in [1,3X]$ we can approximate $$\Lambda^{\sharp}(n)=(a*1)(n)+E(n),$$
 where $|a(n)|\ll \log X$, $a$ is supported on $[1,X^{\varepsilon/5}]$, and for $X^{\varepsilon}\leq H\leq x\leq 2X$ we have $\sum_{x<n\leq x+H}|E(n)|\ll_{\varepsilon} H\exp(-\log^{1/20}X)$. By Lemma~\ref{le:dkdecompose}(ii), for $n\in [1,3X]$ we can write as a sum of $J = O(1)$ type $I$ sums in the form
\begin{align}\label{eq:dksharp}
d_k^{\sharp}(n)=\sum_{j=1}^{J}(a_j*\psi_j)(n),
\end{align}
where $a_j$ are divisor-bounded sequences supported on $[1,X^{\varepsilon/5}]$, and $\psi_j(n)=\log^{j} n$.

We begin with Theorem~\ref{discorrelation-thm-major}(i). We apply 
Lemma~\ref{hb-identity} and the triangle inequality to reduce to proving Theorem~\ref{discorrelation-thm-major}(i) with $f-f^{\sharp}$ for $f\in \{\mu,\Lambda,d_k\}$ replaced by a type $I$ sum $a*\psi$ or by a type $II$ sum $a*b$ (adjusting $A$ or $c_{k,C}$ if necessary). 

 Now, in the case of type $I$ sums, the claim follows from Lemma~\ref{le:typeImajor} (we can get rid of the possible $\log$ factor in the type $I$ convolution by writing 
 $$a*\log^{m}=m\int_{1}^{3X}(a*1_{[1,\cdot]}\log^{m-1})(t)\frac{dt}{t}$$
 and applying the triangle inequality.). In the case of type $II$ sums, if $f=d_k$, we can apply Lemma~\ref{le:typeIImajor}(iii) with $W=X^{\varepsilon/4}$ (recalling that $H\geq X^{1/3+\varepsilon}$), since the assumption~\eqref{eq:supbound} with $W=X^{\varepsilon/4}$ holds  for the sequences $a(n)n^{iT}$ and $b(n)n^{iT}$ by Lemma~\ref{le:vin-kor}. If instead $f\in \{\mu, \Lambda\}$, we can similarly apply Lemma~\ref{le:typeIImajor}(iii) but with $W=\log^{100A} X$, again using Lemma~\ref{le:vin-kor} to verify the assumption~\eqref{eq:supbound}.  This completes the proof of Theorem~\ref{discorrelation-thm-major}(i).

We then turn to Theorem~\ref{discorrelation-thm-major}(ii). Recall that $g(n)=\sum_{X^{\kappa}<p\leq X^{\varepsilon/10}}1_{p\mid n}$  Applying~\eqref{eq:ramare1} and Shiu's bound (Lemma~\ref{le:shiu}), it suffices to prove Theorem~\ref{discorrelation-thm-major}(ii) with $f\cdot g$ for $f\in \{\mu,d_k\}$ replaced by a type $II$ sum $a*b$, where $a(n)=f(n)1_{n\in \mathbb{P}\cap (X^{\kappa},X^{\varepsilon/10}]}$ and $|b(n)|\ll d_2(n)^{O(1)}$. We also want to show that the term $d_k^{\sharp}\cdot g$ arising in the case $f=d_k$ is a linear combination of type $I$ sums. To this end, note that by~\eqref{eq:dksharp} we can write
\begin{align}\label{eq:dksharpg}\begin{split}
d_k^{\sharp}(n)\sum_{X^{\kappa}<p\leq X^{\varepsilon/10}}1_{p\mid n}&=\sum_{X^{\kappa}<p\leq X^{\varepsilon/10}}\sum_{j=1}^{J}\sum_{n=m\ell}a_j(m)\psi_j(\ell)(1_{p\mid m}+1_{p\mid \ell})\\
&+O(X^{o(1)}1_{p^2\mid n\textnormal{ for some } X^{\kappa}<p\leq X^{\varepsilon/10}})\\
&=\sum_{j=1}^{J}\sum_{X^{\kappa}<p\leq X^{\varepsilon/10}}\left(\sum_{n=pm'\ell}a_j(pm')\psi_j(\ell)+\sum_{n=pm\ell'}a_j(m)\psi_j(p\ell')\right)\\
&+O(X^{o(1)}1_{p^2\mid n\textnormal{ for some } X^{\kappa}<p\leq X^{\varepsilon/10}}),   
\end{split}
\end{align}
where $a_j$ is a divisor-bounded sequence supported on $[1,X^{\varepsilon/5}]$ and $\psi_j(n)=\log^j n$. The contribution of the $O(\cdot)$ error term here is negligible by Lemma~\ref{le:shiu}. Also note that by the binomial formula $\psi_j(p\ell')$ is a linear combination of functions of the form $\psi_r(p)\psi_{j-r}(\ell')$ with $0\leq r\leq j$. We conclude that $d_k^{\sharp}\cdot g$ is a sum of $O(1)$ type $I$ sums.

We then use Lemma~\ref{le:typeIImajor}(ii) to handle the type $II$ sums arising from $f\cdot g$  (noting that the first estimate in~\eqref{eq:supbound} holds with $W=\log^{100A} X$ for $a(n)n^{-iT}$ by the assumption $|T|\geq \frac{X}{H}X^{\varepsilon/2}$, Lemma~\ref{le:vin-kor} and partial summation). Furthermore, we use Lemma~\ref{le:typeImajor} for the type $I$ sums arising from $d_k^\sharp\cdot g$. This concludes the proof of Theorem~\ref{discorrelation-thm-major}(ii).

\section{Reduction to type \texorpdfstring{$II$}{II} estimates}

To complement the major arc estimates in Theorem~\ref{discorrelation-thm-major}, we will establish some ``inverse theorems'' that yield discorrelation between an arithmetic function $f$ and a nilsequence $F(g(n)\Gamma)$ assuming that $f$ is of ``type $I$'', ``type $II$'', or ``type $I_2$'', and the nilsequence is ``minor arc'' in a suitable sense.  To make this precise, we recall some definitions from~\cite{MSTT-all}:

\begin{definition}[Type $I$, $II$, $I_2$ sums]\label{struct-sum}  Let $0 < \delta < 1$ and $A_I, A_{II}^-, A_{II}^+, A_{I_2} \geq 1$.  
\begin{itemize}
\item[(i)]  (Type $I$ sum) A \emph{$(\delta,A_I)$ type $I$ sum} is an arithmetic function of the form $f = \alpha *\beta$, where $\alpha$ is supported on $[1,A_I]$, and one has
\begin{equation}\label{abound}
\sum_{n \leq A} |\alpha(n)|^2 \leq \frac{1}{\delta} A
\end{equation}
and
\begin{equation}\label{btv}
\| \beta \|_{\TV(\N; q)} \leq \frac{1}{\delta}
\end{equation}
for all $A \geq 1$ and some $1 \leq q \leq \frac{1}{\delta}$.
\item[(ii)]  (Type $II$ sum) A \emph{$(\delta, A_{II}^-, A_{II}^+)$ type $II$ sum} is an arithmetic function of the form $f = \alpha * \beta$, where $\alpha$ is supported on $[A_{II}^-,A_{II}^+]$, and one has the bound~\eqref{abound} and
\begin{equation*}
\sum_{n \leq B} |\beta(n)|^2 \leq \frac{1}{\delta} B \quad \text{and} \quad \sum_{n \leq B} |\beta(n)|^4 \leq \frac{1}{\delta^2} B
\end{equation*}
for all $A,B \geq 1$.  
(The type $II$ sums become vacuous if $A_{II}^- > A_{II}^+$.)
\item[(iii)]  (Type $I_2$ sum)  A \emph{$(\delta, A_{I_2})$ type $I_2$ sum} is an arithmetic function of the form $f = \alpha * \beta_1 * \beta_2$, where $\alpha$ is supported on $[1,A_{I_2}]$ and satisfies~\eqref{abound} for all $A \geq 1$, and $\beta_1, \beta_2$ satisfy the bound~\eqref{btv} for some $1 \leq q \leq \frac{1}{\delta}$.
\end{itemize}
\end{definition}

We are now ready to state the inverse theorem.

\begin{theorem}[Inverse theorem]\label{inverse}  Let $d,D \geq 1$, $3 \leq H \leq X$, $0 < \delta < \frac{1}{\log X}$, and let $G/\Gamma$ be a filtered nilmanifold of degree at most $d$, dimension at most $D$, and complexity at most $1/\delta$.  Let $F \colon G/\Gamma \to \C$ be Lipschitz of norm at most $1/\delta$ and mean zero.  Let $f \colon \N \to \C$ be an arithmetic function such that
\begin{equation*}
\left| \sum_{x < n \leq x+H} f(n) F(g_x(n) \Gamma) \right|^* \geq \delta H
\end{equation*}
for all $x$ in a subset $E$ of $[X,2X]$ of measure at least $\delta X$ where, for each $x \in E$, $g_x \colon \Z \to G$ is a polynomial map.
\begin{itemize}
\item[(i)]  (Type $I$ inverse theorem)  If $f$ is a $(\delta,A_I)$ type $I$ sum for some $A_I \geq 1$, then either
\begin{equation}
\label{eq:TypeIcond} 
H \ll_{d,D} \delta^{-O_{d,D}(1)} A_I
\end{equation}
or else there exists a non-trivial horizontal character $\eta \colon G \to \R$ of Lipschitz norm $O_{d,D}( \delta^{-O_{d,D}(1)})$ such that
$$ \| \eta \circ g_x \|_{C^\infty(x, x+H]} \ll_{d,D} \delta^{-O_{d,D}(1)}$$
for all $x$ in a subset of $E$ of measure $\gg_{d,D} \delta^{O_{d,D}(1)} X$.
\item[(ii)]  (Type $II$ inverse theorem, non-abelian case)  If $\varepsilon > 0$, $f$ is a $(\delta,A_{II}^-, A_{II}^+)$ type $II$ sum for some $A_{II}^+ \geq A_{II}^- \geq X^\eps$, $G$ is non-abelian with one-dimensional center, and $F$ oscillates with a central frequency $\xi$ of Lipschitz norm at most $1/\delta$, then either
\begin{equation}\label{H-bound}
 H \ll_{d,D,\eps} \delta^{-O_{d,D,\eps}(1)} A_{II}^+
\end{equation}
or else there exists a non-trivial horizontal character $\eta \colon G \to \R$ of Lipschitz norm $O_{d,D,\eps}( \delta^{-O_{d,D,\eps}(1)})$ such that
\begin{equation*}
 \| \eta \circ g_x \|_{C^\infty(x, x+H]} \ll_{d,D,\eps} \delta^{-O_{d,D,\eps}(1)}
\end{equation*}
for all $x$ in a subset of $E$ of measure $\gg_{d,D,\eps} \delta^{O_{d,D,\eps}(1)} X$.
\item[(iii)]  (Type $II$ inverse theorem, abelian case)  If $\varepsilon > 0$, $f$ is a $(\delta,A_{II}^-, A_{II}^+)$ type $II$ sum for some $A_{II}^+ \geq A_{II}^- \geq X^\eps$, and for every $x \in E$ one has $F(g_x(n)\Gamma) = e(P_x(n))$ for some polynomial $P_x \colon \Z \to \R$ of degree at most $d$, then either~\eqref{H-bound} holds,
or else there exists a real number $|T| \ll_{d,\eps} \delta^{-O_{d,\eps}(1)} (X/H)^{d+1}$ such that
$$ \| e(P_x(n)) n^{-iT} \|_{\TV( (x, x+H] \cap \Z; q)} \ll_{d,\eps} \delta^{-O_{d,\eps}(1)} $$
for some integer $1 \leq q \ll_{d,\eps} \delta^{-O_{d,\eps}(1)}$ and
for all $x$ in a subset of $E$ of measure $\gg_{d,\eps} \delta^{O_{d,\eps}(1)} X$.
\item[(iv)] (Type $I_2$ inverse theorem) If $f$ is a $(\delta,A_{I_2})$ type $I_2$ sum for some $A_{I_2} \geq 1$, then either
\begin{equation*}
 H \ll_{d,D} \delta^{-O_{d,D}(1)} X^{1/3} A_{I_2}^{2/3}
\end{equation*}
or else there exists a non-trivial horizontal character $\eta \colon G \to \R$ of Lipschitz norm $O_{d,D}( \delta^{-O_{d,D}(1)})$ such that
$$ \| \eta \circ g_x \|_{C^\infty(x, x+H]} \ll_{d,D} \delta^{-O_{d,D}(1)}$$
for all $x$ in a subset of $E$ of measure $\gg_{d,D} \delta^{O_{d,D}(1)} X$.
\end{itemize} 
\end{theorem}

\begin{remark}\label{rmk:measurability}
In the proof of Theorem~\ref{inverse}, we may clearly assume that $H\in \mathbb{N}$ and then that the map $x\mapsto g_x$ is constant on all intervals of the form $[m,m+1)$ with $m\in \mathbb{N}$. This makes it easy to check the measurability of various subsets of $E$ defined in terms of $g_x$. 
\end{remark}

We observe that parts (i), (iv) of Theorem~\ref{inverse} follow easily from parts (i), (iv) of~\cite[Theorem 4.2]{MSTT-all}. For instance, if the hypotheses of Theorem~\ref{inverse}(i) hold and~\eqref{eq:TypeIcond} does not hold, then from~\cite[Theorem 4.2(i)]{MSTT-all}, we see that for each $x \in E$ there exists a non-trivial horizontal character $\eta_x \colon G \to \R$ of Lipschitz norm $O_{d,D}(\delta^{-O_{d,D}(1)})$ such that
$$ \| \eta_x \circ g_x \|_{C^\infty(x, x+H]} \ll_{d,D} \delta^{-O_{d,D}(1)}.$$
There are only $O_{d,D}(\delta^{-O_{d,D}(1)})$ possible choices of $\eta_x$, so Theorem~\ref{inverse}(i) follows from the pigeonhole principle and Remark~\ref{rmk:measurability}, which guarantees the measurability of those subsets of the set of $x\in E$ for which $\eta_x$ takes a given value.  Theorem~\ref{inverse}(iv) follows similarly from~\cite[Theorem 4.2(iv)]{MSTT-all}. 

Thus it only remains to establish parts (ii) and (iii) of Theorem~\ref{inverse}, which do not follow from their counterparts in~\cite[Theorem 4.2]{MSTT-all}, since the conclusion 
$$ H \ll_{d,D} \delta^{-O_{d,D}(1)} \max(A_{II}^+, X/A_{II}^-)$$
in~\cite[Theorem 4.2]{MSTT-all}(ii)--(iii) is always satisfied when $H \leq X^{1/2}$. We shall prove Theorem~\ref{inverse}(ii)--(iii) in Section~\ref{type-ii}, based on work in Section~\ref{contagion-sec}.

In this section we show how Theorem~\ref{inverse}, when combined with the major arc estimates from the previous section, implies Theorem~\ref{discorrelation-thm}.

\subsection{Combinatorial decompositions}

We start by describing the combinatorial decompositions (Lemmas~\ref{comb-lambda} and~\ref{comb-mu} below) that allow us to reduce sums involving $\mu,\Lambda,d_k$ to type $I$, type $II$, and type $I_2$ sums. Lemma~\ref{comb-lambda} will be used to prove Theorem~\ref{discorrelation-thm}(i)--(iii) and Lemma~\ref{comb-mu} will be used to prove Theorem~\ref{discorrelation-thm}(iv)--(v).

As in~\cite{MSTT-all}, we notice that while the model function $\Lambda^{\sharp}$ is not a type $I$ sum, we can approximate it well by the type $I$ sum
\begin{equation}\label{lambdasharp-i-def}
\Lambda_I^\sharp(n) \coloneqq  \frac{P(R)}{\varphi(P(R))}\sum_{\substack{d \leq X^{\varepsilon/5}\\d\mid (n, P(R))}} \mu(d),
\end{equation}
where we recall that $R=\exp((\log X)^{1/10})$.
Indeed, by the proof of Lemma~\ref{le:dkdecompose}(i) we see that for $x\in [X,2X]$ we have
\begin{equation}\label{eq:lambdasharp2}
 \sum_{x < n \leq x+H} |\Lambda^\sharp_I(n)-\Lambda^{\sharp}(n)|\ll H\exp(-(\log X)^{1/20}).
\end{equation}
In practice, this bound allows us to substitute $\Lambda^\sharp$ with the type $I$ sum $\Lambda^\sharp_I$ with negligible cost.

We state two combinatorial decompositions, one that applies to all of our functions of interest and another one that is more flexible (in the sense of allowing a more restricted type $II$ range) but applies only to the functions $\mu$ and $d_k$ (with an additional weight). 

\begin{lemma}[Combinatorial decompositions of $\mu,\Lambda,\Lambda^\sharp_I, d_k$ and $d_k^\sharp$]\label{comb-lambda}
 Let $\eps > 0$ and $k \geq 2$ be fixed. Let $g \in \{\mu, \Lambda, \Lambda^\sharp_I, d_k, d_k^\sharp\}$. There is a set $\mathcal{F}$ of size $O((\log X)^{O(1)})$ consisting of functions $f \colon \mathbb{N} \to \mathbb{R}$ such that, for each $n \in [X/2, 3X]$, we have
\[
g(n) = \sum_{f \in \mathcal{F}} f(n),
\]
and each component $f \in \mathcal{F}$ satisfies one of the following:
\begin{itemize}
\item[(i)] $f$ is a $(\log^{-O(1)} X, O(X^{1/3+\varepsilon/2}))$ type $I$ sum;
\item[(ii)] $f$ is a $(\log^{-O(1)} X, O(X^{\varepsilon/2})$ type $I_2$ sum.
\item[(iii)] $f$ is a $(\log^{-O(1)} X, X^{\varepsilon/10}, O(X^{1/3}))$ type $II$ sum.
\end{itemize}
\end{lemma}

\begin{lemma}[Flexible combinatorial decompositions of $\mu$ and $d_k$]\label{comb-mu}
Let $k \geq 1$ be fixed. Let $A>0$, $\varepsilon>0$, and $\eta\in (0,\varepsilon/(10k)]$. Let $X \geq H \geq X^\eps$ and $g\in\{\mu, d_k\}$, and write $c_g=0$ if $g=\mu$ and $c_g=k-1$ if $g=d_k$. Also let $$Y_g=\sum_{X^{\exp(-\eta^{-2})}<p\leq X^{\varepsilon/10}}\frac{c_g+1}{p}.$$ 

There exist $K\ll 1$ and functions $f_j\colon \mathbb{N}\to \mathbb{R}$ for $1\leq j\leq K$ such that each $f_j$ is either a $(\log^{-O(1)} X, X^{\varepsilon/2})$ type $I$ sum or a $(\log^{-O(1)} X, X^{\exp(-\eta^{-2})}, X^{\varepsilon/10})$ type $II$ sum, and such that the following holds. For any sequence $\{\omega_n\}$ with $|\omega_n| \leq 1$ and any $x \in [X, 2X]$, we have
$$
\sum_{x < n \leq x+H}(g(n)-g^{\sharp}(n))\omega_n=\frac{1}{Y_g}\sum_{x < n \leq x+H}(g(n)-g^{\sharp}(n))\sum_{X^{\exp(-\eta^{-2})}<p\leq X^{\varepsilon/10}}1_{p\mid n}\omega_n+O(\eta H\log^{c_g}X) 
$$
for all $x\in [X,2X]$ outside of a set of measure $O_{A,\varepsilon,\eta}(X\log^{-A}X)$, 
and additionally 
$$ 
\sum_{x < n \leq x+H}(g(n)-g^{\sharp}(n))\sum_{X^{\exp(-\eta^{-2})}<p\leq X^{\varepsilon/10}}1_{p\mid n}\omega_n = \sum_{j=1}^K\sum_{x < n \leq x+H} f_j(n) \omega_n + O(HX^{-\exp(-\eta^{-2})/2}) 
$$
for all $x\in [X,2X]$.
\end{lemma}

Lemma~\ref{comb-lambda} is proved below, and the proof of Lemma~\ref{comb-mu} is given in the next subsection.

Unlike in~\cite[Section 4.1]{MSTT-all}, we do not need to impose in our combinatorial decompositions conditions on discorrelation of type $II$ sums with $n^{iT}$. This is thanks to our major arc estimate (Theorem~\ref{discorrelation-thm-major}) which shows that functions such as $\Lambda(n)-\Lambda^\sharp(n)$ are discorrelated with $n^{iT}$. For further discussion on this matter, see~\cite[Remark 4.6]{MSTT-all}.

We will prove Lemma~\ref{comb-lambda} by first decomposing the relevant functions into certain Dirichlet convolutions (using Lemma~\ref{hb-identity} in the proof of Lemma~\ref{comb-lambda} for $\mu$ and $\Lambda$. It is then easy to see that we get appropriate sums.

\begin{proof}[Proof of Lemma~\ref{comb-lambda}]
The function $\Lambda^\sharp_I$ is  clearly a $(\log^{-O(1)} X, X^{\eps/5})$ type $I$ sum by its definition~\eqref{lambdasharp-i-def}. Similarly, the function $d_k^{\sharp}$ is a sum of $O(1)$ such sums by Lemma~\ref{le:dkdecompose}(ii).

For $\Lambda$, $\mu$, and $d_k$, we apply Lemma~\ref{hb-identity} with $L = \lceil 10/\varepsilon \rceil$. Each of the $O((\log X)^{O(1)})$ components $f \in \mathcal{F}$ takes the form
\begin{equation*}
f = a^{(1)}* \cdots * a^{(\ell)} 
\end{equation*}
for some $\ell \leq 2L$, where each $a^{(i)}$ is supported on $(N_i, 2N_i]$ for some $N_i \geq 1/2$, and each $a^{(i)}(n)$ is either $1_{(N_i, 2N_i]}(n)$, $(\log n)1_{(N_i, 2N_i]}(n)$, or $\mu(n)1_{(N_i, 2N_i]}(n)$. Moreover, 
\begin{equation}
\label{eq:N1prod}
N_1N_2\cdots N_{\ell} \asymp X,
\end{equation}
and $N_i \leq X^{\varepsilon/10}$ for each $i$ with $a^{(i)}(n) = \mu(n) 1_{(N_i, 2N_i]}(n)$. Consequently whenever $N_i > X^{\eps/10}$, we have $\|a^{(i)}\|_{\TV(\N)} \ll \log X$.

Without loss of generality we can assume that
\begin{equation}
\label{eq:Niorder}
N_1 \geq N_2 \geq \dotsb \geq N_\ell.
\end{equation}

Consider first the case $N_1 > X^{2/3-\eps/2}$. Since $2/3-\eps/2 > \eps/10$ and thus $\|a^{(1)}\|_{\TV(\N)} \ll \log X$, from~\eqref{eq:N1prod} we see that $f$ is a $(\log^{-O(1)} X, O(X^{1/3+\eps/2}))$ type $I$ sum of the form $\alpha*\beta$ with $\beta = a^{(1)}$ and $\alpha = a^{(2)}*\cdots*a^{(\ell)}$. 

Henceforth we may assume that $N_j \leq X^{2/3-\eps/2}$ for each $j$. Next consider the case that $N_1 N_2 > X^{1-\varepsilon/2}$. Since $N_1, N_2 \leq X^{2/3-\eps/2}$, this implies that $N_1, N_2 > X^{1/3} > X^{\varepsilon/10}$ and thus $\|a^{(1)}\|_{\TV(\N)}, \|a^{(2)}\|_{\TV(\N)}  \ll \log X$. Hence, by~\eqref{eq:N1prod}, the function $f$ is a $(\log^{-O(1)} X,\allowbreak O(X^{\eps/2}))$ type $I_2$ sum of the form $f = \alpha*\beta_1*\beta_2$, with $\beta_1 = a^{(1)}$, $\beta_2 = a^{(2)}$.

In the remaining case, we have $N_1 N_2 \leq X^{1-\varepsilon/2}$ and hence necessarily $\ell \geq 3$ and $X/(N_1 N_2) \gg X^{\varepsilon/2}$. By~\eqref{eq:Niorder} we have $N_3 \ll X^{1/3}$ and so there exists $j \in \{3, \dotsc, \ell\}$ such that $X^{\eps/10}\leq N_3 \dotsm N_j \ll  X^{1/3}$ and thus we have a $(\log^{-O(1)} X, X^{\eps/10}, O(X^{1/3}))$  type $II$ sum of the form $f = \alpha * \beta$, where $\alpha = a^{(3)} * \dotsb * a^{(j)}$.
\end{proof}

\subsection{Divisor sums in short intervals}\label{subsec:divsum}

In this subsection, we prove Lemma~\ref{comb-mu}. We first need an auxiliary result about sums of $d_k$ in almost all short intervals, which will also be needed later in Section~\ref{gowers-sec}. 

\begin{lemma}[Divisor functions in almost all short intervals]\label{le:dkshort}
Let $X\geq 3$, $\varepsilon,\eta>0$, $k\in \mathbb{N}$, and $A>0$.  Let $H\in [X^{\varepsilon},X^{1-\varepsilon}]$. Then
\begin{align}\label{eq:dkmaximum}
\max_{\substack{a,q\in \mathbb{N}\\a, q \leq \log X \\ (a,q)=1}}\left|\sum_{x<n\leq x+H}\left(\frac{q}{\varphi(q)}\right)^{k-1}d_k(qn+a)-\frac{1}{(k-1)!}H\log^{k-1}X\right|\leq \eta H\log^{k-1}X    \end{align}
for all $x\in [X,2X]$ outside of a set of measure $O_{A,\varepsilon,\eta,k}(X\log^{-A}X)$.
\end{lemma}

 We remark that the strong bound on the exceptional set in this lemma is important in what follows, and hence we cannot apply the results of~\cite{sun},~\cite{mangerel} on short sums of $d_k$ that work in much shorter intervals. We also note that in this section we only need the $q=1$ case of the lemma, whereas in Section~\ref{gowers-sec} we need the general case.   

\begin{proof}[Proof of Lemma~\ref{le:dkshort}]
We can clearly assume that $\eta$ is fixed and small in terms of $\varepsilon$. By the union bound, it suffices to show that~\eqref{eq:dkmaximum} holds without the maximum for any given coprime $a,q$ in this range.

Fix large $A\geq 1$ and small $\varepsilon>0$. Let $H_1=X/\log^{1000A} X$. 
We first claim that for $x\in [qX,3qX]$ we have
\begin{align}\label{eq:dkclaim}
\left|\sum_{\substack{x<n\leq x+qH_1\\n\equiv a\pmod q}}\left(\left(\frac{q}{\varphi(q)}\right)^{k-1}d_k(n)-\frac{\log^{k-1}X}{(k-1)!}\right)\right|\leq \frac{\eta}{100}H_1\log^{k-2+\varepsilon}X.   
\end{align}
This estimate follows from standard estimates for the $d_k$ divisor function in arithmetic progressions; for example~\cite{FI-divisor} gives for coprime $1\leq a\leq q\leq x^{\delta_k}$ the estimate 
\begin{align} \label{eq:twopoints}
\sum_{\substack{n\leq x\\n\equiv a\pmod q}} d_k(n)=\frac{x}{\varphi(q)} \sum_{\substack{n \leq x \\ (n,q) = 1}} d_k(n) + O(x^{1 - \delta_k}) = \frac{x}{\varphi(q)} P_{k,q}(\log x) + O(x^{1 - \delta_k})
\end{align}
for some $\delta_k>0$ and $P_{k,q}$ given by, 
$$
P_{k,q}(t) = \text{Res}_{s = 1} \Big ( \prod_{p | q} \Big (1 - \frac{1}{p^s} \Big )^k \zeta(s)^k e^{(s - 1) t} \Big ).
$$
Since $\zeta(s) =(s - 1)^{-1} + O(1)$ in the neighborhood of $s = 1$, using Cauchy's formula on the circle $|z - 1| = 1 / t$ we conclude that $P_{k,q}$ is a polynomial of degree $k-1$, and uniformly in $1 \leq q \leq x$ and $t \geq 1$, 
$$
P_{k,q}(t) = \Big ( \frac{\varphi(q)}{q} \Big )^{k} \frac{t^{k - 1}}{(k - 1)!} + \sum_{j = 0}^{k - 2} c_{j,q} t^{j} \ , \ |c_{j,q}| \ll \Big ( \frac{q}{\varphi(q)} \Big )^{10 k}. 
$$
The claim~\eqref{eq:dkclaim} follows by subtracting \eqref{eq:twopoints} at two points, and using the above formula for $P_{k,q}(t)$ together with the fact that $q / \varphi(q) \ll \log\log x$ for any $1 \leq q \leq x$. 

Now, by~\eqref{eq:dkclaim}, it suffices to show that 
\begin{align*}
\left|\sum_{\substack{x<n\leq x+H\\n\equiv a\pmod q}}d_k(n)-\frac{H}{H_1}\sum_{\substack{x<n\leq x+H_1\\n\equiv a\pmod q}}d_k(n)\right|\leq \frac{\eta}{2} H\log^{k-1}X    
\end{align*}
for all $x\in [X,2X]$ outside of a set of measure $O_{A,\varepsilon,\eta,k}(X\log^{-A}X)$. Let $\mathcal{S}$ be the set of integers having at least one prime factor in $(H^{\eta^2},H^{1/2}]$. 
 Applying Shiu's bound (Lemma~\ref{le:shiu}) to the multiplicative function $d_k\cdot (1-1_{\mathcal{S}})$ (recalling that we can assume that $\eta$ is small in terms of $\varepsilon$ and thus small compared to the implied constant in Shiu's bound), we reduce to showing that 
\begin{align*}
\left|\sum_{\substack{x<n\leq x+H\\n\equiv a\pmod q}}d_k(n)1_{\mathcal{S}}(n)-\frac{H}{H_1}\sum_{\substack{x<n\leq x+H_1\\n\equiv a\pmod q}}d_k(n)1_{\mathcal{S}}(n)\right|\leq \frac{\eta}{3} H\log^{k-1}X    
\end{align*}
with the same exceptional set bound. Using the orthogonality of characters, it is enough to show the estimate
\begin{align}\label{eq:dkshort}
\left|\sum_{x<n\leq x+H}d_k(n)\chi(n)1_{\mathcal{S}}(n)-\frac{H}{H_1}\sum_{x<n\leq x+H_1}d_k(n)\chi(n)1_{\mathcal{S}}(n)\right|\leq \frac{\eta}{3} H\log^{k-1}X    
\end{align}
for all Dirichlet characters $\chi\pmod q$.

Denoting by $\omega_{(P_1,P_2]}(n)$ the number of prime factors of $n$ from $(P_1,P_2]$, we can write
\begin{align}\label{eq:dk1S2}\begin{split}
d_k(n)\chi(n)1_{\mathcal{S}}(n)&=\sum_{\substack{n=pm\\H^{\eta^2}<p\leq H^{1/2}}}\frac{k\chi(p)d_k(m)\chi(m)}{1+\omega_{(H^{\eta^2},H^{1/2}]}(m)}+O(d_k(n)1_{p^2\mid n\,\, \textnormal{for some}\,\, p\in (H^{\eta^2},H^{1/2}]})\\
&\coloneqq (\alpha*\beta)(n)+O(d_k(n)1_{p^2\mid n\,\, \textnormal{for some}\,\, p\in (H^{\eta^2},H^{1/2}]}), 
\end{split}
\end{align}
where $\alpha(n)=k\chi(n)1_{\mathbb{P}\cap (H^{\eta^2},H^{1/2}]}$ and $|\beta(n)|\ll d_k(n)$. The error term in~\eqref{eq:dk1S2} has a negligible contribution to the left-hand side of~\eqref{eq:dkshort} by the divisor bound, so it suffices to show that
\begin{align*}
\left|\sum_{x<n\leq x+H}(\alpha*\beta)(n)-\frac{H}{H_1}\sum_{x<n\leq x+H_1}(\alpha*\beta)(n)\right|\leq \frac{\eta}{4} H\log^{k-1}X    
\end{align*}
for all $x\in [X,2X]$ outside of a set of measure $O_{A,\varepsilon,\eta,k}(X\log^{-A}X)$.
But this follows immediately from our type $II$ major arc estimate (Lemma~\ref{le:typeIImajor}(i)), using Lemma~\ref{le:vin-kor} and partial summation to verify the assumption there.  
\end{proof}

Utilizing Lemma~\ref{le:dkshort}, we can prove the following $d_k$-weighted (dual) Tur\'an--Kubilius inequality in short intervals, which quickly leads to Lemma~\ref{comb-mu}. 

\begin{lemma}[Tur\'an--Kubilius in short intervals with divisor function weight]\label{le:turan} Let $X\geq 3$,   $k\in \mathbb{N}$, $A>0$ and $\varepsilon,\varepsilon_1,\varepsilon_2\in (0,1/10)$. Also let $X^{\varepsilon}\leq H\leq X^{1-\varepsilon}$ and suppose that $\varepsilon_1\leq \varepsilon_2 e^{-2k}$. Then we have
\begin{align}
\label{eq:T-KClaim}
\sum_{x<n\leq x+H}d_k(n)\left|\sum_{X^{\varepsilon_1}<p\leq X^{\varepsilon_2}}1_{p\mid n}-\sum_{X^{\varepsilon_1}<p\leq X^{\varepsilon_2}}\frac{k}{p}\right|\leq 20\sqrt{k}H(\log^{k-1} X)\left(\log \frac{\varepsilon_2}{\varepsilon_1}\right)^{1/2}
\end{align}
for all $x\in [X,2X]$ outside of a set of measure $O_{A,\varepsilon,\varepsilon_1,\varepsilon_2,k}(X\log^{-A}X)$.    
\end{lemma}

\begin{proof}
We may assume that $X$ is large enough in terms of $A,\varepsilon,\varepsilon_1,\varepsilon_2,k$ and that $A$ is large enough in terms of $k$.

For any $Y\geq H'\geq 2$, let $\mathcal{E}(Y,H')$ be the set of $x\in (Y,2Y]$ for which
\begin{align*}
\left|\sum_{x<n\leq x+H'}d_k(n)-\frac{1}{(k-1)!}H'\log^{k-1}x\right|\geq \frac{1}{100k^2}\left(\log \frac{\varepsilon_2}{\varepsilon_1}\right)^{-1} H'\log^{k-1}Y. \end{align*}
Then, by Lemma~\ref{le:dkshort} with $q=1$, for any $Y \geq 2$ and $H'\in [Y^{\varepsilon/10},Y]$ we have
\begin{align}
\label{eq:EY}\meas(\mathcal{E}(Y,H'))\ll_{A,\varepsilon,\varepsilon_1,\varepsilon_2,k}Y\log^{-100A} Y.     
\end{align}
Define
\begin{align*}
\mathcal{E}=\{x\in [X,2X]\colon \max_{\frac{1}{2}\leq M\leq H^{0.9}}\frac{1}{M}|\{r\in (M,2M]\cap \mathbb{N}\colon \frac{x}{r}\in \mathcal{E}(X/r,H/r)\}|>\log^{-10A}X\}.     
\end{align*}
By the union bound, Markov's inequality and~\eqref{eq:EY}, we have
\begin{align*}
\meas(\mathcal{E})&\ll (\log X)\max_{M\in [\frac{1}{2},H^{0.9}]}\frac{\log^{10A}X}{M}\sum_{M<r\leq 2M}\meas(\{x\in [X,2X]\colon x/r\in \mathcal{E}(X/r,H/r)\})\\
&\ll (\log^{10A+1} X)\max_{r\in [1,H^{0.9}]}r\cdot \meas\left(\mathcal{E}\left(\frac{X}{r},\frac{H}{r}\right)\right)\\
&\ll_{A,\varepsilon,\varepsilon_1,\varepsilon_2,k}X\log^{-A}X.\end{align*}

Now we work with $x \in [X,2X]\setminus \mathcal{E}$. Let us first dispose of the large primes on the left hand side of~\eqref{eq:T-KClaim}. We let 
\[
\varepsilon_2' = \frac{k\varepsilon_2}{\log \frac{\varepsilon_2}{\varepsilon_1}}. 
\]
Note that since $\varepsilon_1 \leq \varepsilon_2 e^{-2k},$ we have $\varepsilon_2' \leq \varepsilon_2$. We have chosen $\varepsilon_2'$ in such a way that (using the inequality $\log t \leq t^{1/2}$ for $t>0$ on the second line)
\begin{align}
\label{eq:vareps2b1} 
\varepsilon_2' &= \frac{k\varepsilon_2}{\log \frac{\varepsilon_2}{\varepsilon_1}} \leq \frac{k}{10\log \frac{\varepsilon_2}{\varepsilon_1}}, \\
\label{eq:vareps2b2} 
\log \frac{\varepsilon_2}{\varepsilon_2'} &= \log\frac{\log \frac{\varepsilon_2}{\varepsilon_1}}{k} \leq \frac{1}{k^{1/2}} \left(\log \frac{\varepsilon_2}{\varepsilon_1}\right)^{1/2}.
\end{align}
Now, using that $x \in [X, 2X] \setminus \mathcal{E}$, Mertens' theorem and~\eqref{eq:vareps2b2} as well as Shiu's bound for $p$ such that $x/p \in \mathcal{E}(X/p, H/p)$, we see that the contribution of primes $p \in (X^{\varepsilon_2'}, X^{\varepsilon_2}]$ to the left hand side of~\eqref{eq:T-KClaim} is at most
\begin{align*}
&3k \sum_{\substack{X^{\varepsilon_2'} < p \leq X^{\varepsilon_2}}} \frac{H}{p} \log^{k-1} X + O_\varepsilon\left(k \sum_{\substack{X^{\varepsilon_2'} < p \leq X^{\varepsilon_2} \\ x/p \in \mathcal{E}(X/p, H/p)}} \frac{H}{p} \log^{k-1} X\right) \\
&\leq 4kH\log^{k-1} X \log \frac{\varepsilon_2}{\varepsilon_2'} \leq 4\sqrt{k}H\log^{k-1} X \left(\log \frac{\varepsilon_2}{\varepsilon_1}\right)^{1/2}.
\end{align*}
Notice that we have to be careful when we apply Shiu's bound since the implied constant depends on $\varepsilon$, but this is not a problem when we save powers of $\log X$ as we can assume that $X$ is large in terms of $\varepsilon$.

Now we turn to the contribution of primes $p \in (X^{\varepsilon_1}, X^{\varepsilon_2'}]$ to~\eqref{eq:T-KClaim}. We apply the Cauchy--Schwarz inequality (using again that $x \in [X, 2X] \setminus{\mathcal{E}}$) to reduce matters to proving
\begin{align*}
\sum_{x<n\leq x+H}d_k(n)\left|\sum_{X^{\varepsilon_1}<p\leq X^{\varepsilon_2'}}1_{p\mid n}-\sum_{X^{\varepsilon_1}<p\leq X^{\varepsilon_2'}}\frac{k}{p}\right|^2\leq 100kH(\log^{k-1} X)\log \frac{\varepsilon_2'}{\varepsilon_1}.
\end{align*}
Opening the square, the claimed estimate becomes
\begin{align*}
\sum_{X^{\varepsilon_1}<p_1,p_2\leq X^{\varepsilon_2'}}\sum_{x<n\leq x+H}d_k(n)\left(1_{p_1\mid n}-\frac{k}{p_1}\right)\left(1_{p_2\mid n}-\frac{k}{p_2}\right)\leq 100 kH(\log^{k-1} X)\log \frac{\varepsilon_2'}{\varepsilon_1}.
\end{align*}
For any prime $p$ dividing $n$, we have $d_k(n)=kd_k(n/p)+O(d_k(n)1_{p^2\mid n})$. Using this together with the inequality $(a-b)^2\leq 2(a^2+b^2)$ and the assumption $x\in [X,2X]\setminus \mathcal{E}$, the contribution of $p_1=p_2$ to the sum above is crudely
\begin{align*}
&\leq 2\sum_{X^{\varepsilon_1}<p\leq X^{\varepsilon_2'}}\sum_{x<n\leq x+H}d_k(n)\left(1_{p\mid n}+\frac{k^2}{p^2}\right)\\
&=2k\sum_{X^{\varepsilon_1}<p\leq X^{\varepsilon_2'}}\sum_{x/p<m\leq (x+H)/p}d_k(m)+O_k\left(HX^{-\varepsilon_1+o(1)}\right)\\
&\leq 3kH(\log^{k-1}X)\sum_{X^{\varepsilon_1}<p\leq X^{\varepsilon_2'}}\frac{1}{p}+O_\varepsilon\left(k \sum_{\substack{X^{\varepsilon_1} < p \leq X^{\varepsilon_2'} \\ x/p \in \mathcal{E}(X/p, H/p)}} \frac{H}{p} \log^{k-1}X\right)+O_k\left(HX^{-\varepsilon_1+o(1)}\right)\\
&\leq 4kH(\log^{k-1}X) \log\frac{\varepsilon_2'}{\varepsilon_1}.    
\end{align*}

In view of this and Mertens' theorem, it suffices to show that, for any $P_1,P_2\in [X^{\varepsilon_1}, X^{\varepsilon_2'}]$, we have have for all but $\ll P_1P_2\log^{-A}X$ pairs of distinct primes $p_1\sim P_1, p_2\sim P_2$ the estimate 
\begin{align}\label{eq:dkp1p2}
\sum_{x<n\leq x+H}d_k(n)\left(1_{p_1\mid n}-\frac{k}{p_1}\right)\left(1_{p_2\mid n}-\frac{k}{p_2}\right)\leq 95\left(\log \frac{\varepsilon_2'}{\varepsilon_1}\right)^{-1}\frac{H}{p_1p_2}\log^{k-1} X,
\end{align}
say. Using again $d_k(n)=kd_k(n/p)+O(d_k(n)1_{p^2\mid n})$ for a prime $p\mid n$, we see that the left-hand side of~\eqref{eq:dkp1p2} is
\begin{align*}
&k^2\sum_{x/(p_1p_2)<m\leq (x+H)/(p_1p_2)}d_k(m)-\frac{k^2}{p_1}\sum_{x/p_2<m\leq (x+H)/p_2}d_k(m)-\frac{k^2}{p_2}\sum_{x/p_1<m\leq (x+H)/p_1}d_k(m)\\
&+\frac{k^2}{p_1p_2}\sum_{x<m\leq x+H}d_k(m)+O(HX^{o(1)-\varepsilon_1}). \end{align*}
Since $x\in [X,2X]\setminus \mathcal{E}$, for all but $\ll P_1P_2\log^{-A}X$ pairs of distinct primes $p_1\sim P_1, p_2\sim P_2$ the main term is
\begin{align*}
\leq k^2 \frac{H}{(k-1)!p_1p_2} \left(\log^{k-1} \frac{X}{p_1 p_2} - \log^{k-1} \frac{X}{p_2} - \log^{k-1} \frac{X}{p_1} + \log^{k-1} X\right) + \left(\log \frac{\varepsilon_2'}{\varepsilon_1}\right)^{-1} \frac{H}{p_1p_2}\log^{k-1}X. \end{align*}
By the mean value theorem and~\eqref{eq:vareps2b1}, here
\begin{align*}
&\log^{k-1} \frac{X}{p_1 p_2} - \log^{k-1} \frac{X}{p_2} - \log^{k-1} \frac{X}{p_1} + \log^{k-1} X \leq 2(k-1) \log X^{\varepsilon_2'} \log^{k-2} X \\
&\leq 2k^{2} \left(\log \frac{\varepsilon_2'}{\varepsilon_1}\right)^{-1} \log^{k-1} X,
\end{align*}
and the desired bound follows since $\frac{2k^{2+2}}{(k-1)!} \leq 86$ for all $k \in \mathbb{N}$.
\end{proof}

We are now ready to prove Lemma~\ref{comb-mu}.

\begin{proof}[Proof of Lemma~\ref{comb-mu}] Let $\kappa=\exp(-\eta^{-2})$, and let $c_g$ and $Y_g$ be as in Lemma~\ref{comb-mu}. Note that $Y_g\asymp \eta^{-2}$. Then, by Lemma~\ref{le:turan}, for any complex numbers $\omega_n$ with $|\omega_n|\leq 1$ we have
\begin{align*}
\sum_{x<n\leq x+H}(g(n)-g^{\sharp}(n))\omega_n=\frac{1}{Y_g}\sum_{x<n\leq x+H}(g(n)-g^{\sharp}(n))\sum_{X^{\kappa}<p\leq X^{\varepsilon/10}}1_{p\mid n}\omega_n+O(\eta H\log^{c_g}X)   
\end{align*}
for all $x\in [X,2X]$ outside of an exceptional set of measure $O_{A,\varepsilon,\eta, k}(X\log^{-A}X)$, proving the first claim of the lemma.

From~\eqref{eq:ramare1} and the divisor bound, we have
\begin{align*}
\sum_{x<n\leq x+H}g(n)\sum_{X^{\kappa}<p\leq X^{\varepsilon/10}}1_{p\mid n}=\sum_{x<n\leq x+H}(a*b)(n)+O(HX^{-\kappa+o(1)}),    
\end{align*}
where $a,b$ are divisor-bounded sequences and $a$ is supported on $[X^{\kappa},X^{\varepsilon/10}]$. On the other hand, from~\eqref{eq:dksharpg} we have
\begin{align*}
\sum_{x<n\leq x+H}d_k^{\sharp}(n)\sum_{X^{\kappa}<p\leq X^{\varepsilon/10}}1_{p\mid n}=\sum_{j=1}^K\sum_{x<n\leq x+H}(a_j*b_j)(n)+O(HX^{-\kappa+o(1)}),  \end{align*}
where $K\ll 1$ and each $a_j, b_j$ is divisor-bounded, $a_j$ is supported on $[1,X^{\varepsilon/5+\varepsilon/10}]$ for all $1\leq j\leq K$, and $b_j = \log^j n$. This gives the second claim of the lemma.
\end{proof}

\subsection{Deduction of Theorem~\ref{discorrelation-thm}}

In this subsection we deduce Theorem~\ref{discorrelation-thm} from Theorem~\ref{inverse}. We present this deduction in a somewhat more general framework, in order to treat all the cases at once and ease potential applications to other sequences of arithmetic interest.

\begin{theorem}[Reduction to discorrelation with Archimedean characters]\label{thm_generalseq}
Let $X \geq 3$ and $X^{\theta+\varepsilon} \leq H \leq X^{1-\varepsilon}$ for some $0 \leq \theta < 1$ and let $\eps > 0$. Let $\kappa>0$, $A\geq 1$, and $\delta \in (0, \frac{1}{(\log X)^{A}})$. Let $G/\Gamma$ be a filtered nilmanifold of some degree $d$ and dimension $D$, and complexity at most $1/\delta$, and let $F \colon G/\Gamma \to \C$ be a Lipschitz function of norm at most $1/\delta$. Let $f\colon \mathbb{N}\to \mathbb{R}$ be a function with $|f(n)|\ll (\log n)^{O(1)}d_2(n)^{O(1)}$, and suppose that one of the following holds.
\begin{itemize}
\item[(i)] $\theta\geq  1/3$ and $f$ is a sum of $O((\log X)^{A})$ functions, each of which is a $(\delta^{-1},X^{\theta+\varepsilon/2})$ type $I$ sum, a $(\delta
^{-1},X^{(3\theta-1)/2+\varepsilon/2})$ type $I_2$ sum, or a $(\delta^{-1},X^{\kappa},X^{\theta+\varepsilon/2})$ type $II$ sum.
\item[(ii)] $\theta=0$ and $f$ is a sum of $O((\log X)^{A})$ functions, each of which is a $(\delta^
{-1},X^{\theta+\varepsilon/2})$ type $I$ sum or a $(\delta^
{-1},X^{\kappa},X^{\theta+\varepsilon/2})$ type $II$ sum.
\end{itemize}
Then we have
\begin{align}\label{eq:deltaA}\begin{split}
&\sup_{g \in \Poly(\Z \to G)} \left| \sum_{x < n \leq x+H} f(n) \overline{F}(g(n)\Gamma) \right|^*\\
&\quad \ll_{A,d,D,\varepsilon,\kappa} \delta^{-O_{d,D,\kappa}(1)} \sup_{|t|\leq \delta^{-O_{d,D,\kappa}(1)}(X/H)^{d+1}}\left|\sum_{x<n\leq x+H}f(n)n^{it}\right|^{*}+\delta H
\end{split}
\end{align}
for all $x\in [X,2X]$ outside of a set of measure $O_{A,d,D,\varepsilon,\kappa}(\delta X)$.
\end{theorem}

\begin{proof}[Proof of Theorem~\ref{discorrelation-thm}] We may assume that $A>0$ is large in terms of $\varepsilon,d,D$ in Theorem~\ref{discorrelation-thm}. Now, parts (i)--(iii) of the theorem  follow by applying Theorem~\ref{thm_generalseq}(i) (with $\delta=\log^{-A} X$ or $\delta=X^{-c_{k,d,D,\varepsilon}}$), Lemma~\ref{comb-lambda},~\eqref{eq:lambdasharp2}, and Theorem~\ref{discorrelation-thm-major}(i). Parts (iv)--(v), in turn, follow by  applying Theorem~\ref{thm_generalseq}(ii) (with $\delta=\log^{-A} X$), Lemma~\ref{comb-mu}, and Theorem~\ref{discorrelation-thm-major}(ii).
\end{proof}

\begin{proof}[Proof of Theorem~\ref{thm_generalseq}]
Let $X^{\theta + \eps} \leq H \leq X^{1-\eps}$ and $\eps > 0$, and let $A>0$ be large in terms of $d,D,\varepsilon,\kappa$. For brevity, we allow all implied constants in $O(1)$ notation to depend on $d,D,\kappa$, and in $\gg$ notation to depend on $A,d,D,\varepsilon,\kappa$. We may assume that $H$ is an integer.  

We use induction on the dimension $D$ of $G/\Gamma$. In view of the triangle inequality, we may assume that $F$ has mean zero (by decomposing $F=(F-\int F)+\int F$). In particular, this shows that the case $D=0$ of the theorem holds, so we now assume that $D>0$ and that the theorem has already been established for smaller values of $D$. In view of Proposition~\ref{central} and Lemma~\ref{le:shiu}, we may assume that $F$ oscillates with a central frequency $\xi \colon Z(G) \rightarrow \R$ of Lipschitz norm at most $\delta^{-O(1)}$. If the center $Z(G)$ has dimension larger than $1$, or $\xi$ vanishes, then the conclusion follows from induction hypothesis applied to $G/\ker\xi$ (via Lemma~\ref{quotient-normal}). Henceforth we assume that $G$ has one-dimensional center and that $\xi$ is non-zero. (A zero-dimensional center is not possible since $G$ is nilpotent and non-trivial.) 

After these reductions (and adjusting the value of $A$), we may assume for the sake of contradiction that the following holds:  There exists a nilmanifold $G/\Gamma$ of complexity at most $1/\delta$ and $F$ of Lipschitz norm at most $1/\delta$ such that
\begin{equation}\label{eq:discorLamClaimRev}
\left| \sum_{x < n \leq x+H} f(n) F(g_x(n) \Gamma) \right|^*\gg \delta H
\end{equation}
for all $x$ in a subset $E$ of $[X,2X]$ of measure $\gg \delta X$ where, for each $x \in E$, $g_x \colon \Z \to G$ is a polynomial map. Thanks to Remark~\ref{rmk:measurability}, we may assume that $x\mapsto g_x$ is constant on intervals of the form $[m,m+1)$ with $m\in \mathbb{N}$.  

Using the pigeonhole principle, for some type $I$, type $I_2$ or type $II$ sum $h$ as in the theorem, one has the bound
\begin{equation}\label{eq:discorLamClaim-3}
\left| \sum_{x < n \leq x+H} h(n) F(g_x(n) \Gamma) \right|^* \gg \delta^{2} H
\end{equation}
for all $x$ in a subset $E'$ of $E$ of measure $\gg \delta^{2} X$.

Consider first the case when $h$ is a $(\delta^{-1}, X^{\kappa}, O(X^{\theta+\varepsilon/2}))$ type $II$ sum and $G$ is abelian, hence one-dimensional since $G=Z(G)$.  Then we may identify $G/\Gamma$ with the standard circle $\R/\Z$ (increasing the Lipschitz constants for $F$, $\xi$ by $O(\delta^{-O(1)})$ if necessary) and $\xi$ with an element of $\Z$ of magnitude $O(\delta^{-O(1)})$, and we can write
$$ F(y) = C e(\xi y)$$
for some $C = O(\delta^{-O(1)})$ and all $y \in \R/\Z$.  We can write $\xi \cdot g_x(n) \Gamma = P_x(n) \hbox{ mod } 1$ for some polynomial $P_x \colon \Z \rightarrow \R$ of degree at most $d$, thus by~\eqref{eq:discorLamClaimRev},~\eqref{eq:discorLamClaim-3} we have
\begin{equation}\label{fepn}
\left| \sum_{x < n \leq x+H} f(n) e(-P_x(n)) \right|^* \gg \delta^{O(1)} H
\end{equation}
and
\begin{equation}\label{fepn-2}
\left| \sum_{x < n \leq x+H} h(n) e(-P_x(n)) \right|^* \gg \delta^{O(1)} H
\end{equation}
for all $x$ in $E'$.
Theorem~\ref{inverse}(iii) and~\eqref{fepn-2} imply that there exists a real number $T \ll \delta^{-O(1)} (X/H)^{d+1}$ and an integer $1\leq q \ll \delta^{-O(1)}$ such that
\begin{equation*}
\| e(P_x(n)) n^{-iT} \|_{\TV((x, x+H] \cap \Z; q)} \ll\delta^{-O(1)}
\end{equation*}
for all $x$ in a subset of $E'$ of measure $\gg \delta^{O(1)} X$. By Lemma~\ref{basic-prop}(ii) and \eqref{fepn}, we thus obtain
\begin{equation*}
\left|\sum_{x < n \leq x+H} f(n) n^{-iT} \right|^* \gg \delta^{O(1)} H
\end{equation*}
for all such $x$. But now (after adjusting the value of $A$), from~\eqref{eq:discorLamClaimRev} we obtain the desired bound~\eqref{eq:deltaA}. 

Now, in all the remaining cases, $h$ is either a type $I$ sum, a type $I_2$ sum, or a type $II$ sum (with the relevant parameters), and in the last case $G$ is non-abelian with one-dimensional center. Hence, Theorem~\ref{inverse} and~\eqref{eq:discorLamClaim-3} imply that there exists a non-trivial horizontal character $\eta \colon G \rightarrow \R$ of Lipschitz norm $\delta^{-O(1)}$ such that
\begin{equation}\label{g-smooth}
\|\eta \circ g_x\mod 1\|_{C^{\infty}(x, x+H]} \ll \delta^{-O(1)}
\end{equation}
for all $x$ in a subset of $E$ of measure at least $\delta^{O(1)} X$.

Now that we have~\eqref{g-smooth}, we can reduce the dimension and apply the induction hypothesis to contradict~\eqref{eq:discorLamClaimRev} and hence conclude the proof as follows. By~\eqref{g-smooth} and Lemma~\ref{factor-simple}, we have a decomposition $g_x = \eps_x g'_x\gamma_x$ for some $\eps_x, g'_x, \gamma_x \in \Poly(\Z \to G)$ such that
\begin{itemize}
\item[(i)] $\eps_x$ is $(\delta^{-O(1)}, (x, x+H])$-smooth;
\item[(ii)]  There is a $\delta^{-O(1)}$-rational proper subnilmanifold $G'_x/\Gamma'_x$ of  $G/\Gamma$ such that $g'_x$ takes values in $G'_x$; and
\item[(iii)] $\gamma_x$ is $\delta^{-O(1)}$-rational.
\end{itemize}

At present, the objects $\eps_x, G'_x, \Gamma'_x, \gamma_x$ in the above decomposition depend on $x$, which is undesirable.  However, we can improve this situation by repeated application of the pigeonhole principle as follows (refining the set of $x$ as necessary).  
\begin{itemize}
\item Observe from Lemma~\ref{factor-simple} that $G'_x$ is the kernel of $\eta$ and is thus automatically independent of $x$, thus we may write $G'_x = G'$. 
\item Since $G'_x/\Gamma'_x$ is a $\delta^{-O(1)}$-rational subnilmanifold of $G/\Gamma$, the generators of $\Gamma'_x$ are $\delta^{-O(1)}$-rational in $G$, thus there are only $O(\delta^{-O(1)})$ many possible choices for $\Gamma'_x$.  Applying the pigeonhole principle, and refining $x$ to a smaller subset of $E$ of measure $\delta^{O(1)} X$, we may assume that $\Gamma'_x = \Gamma'$ is independent of $x$.
\item Since $\gamma_x$ is $\delta^{-O(1)}$-rational, $\gamma_x \Gamma$ has a period $1 \leq q_x \leq \delta^{-O(1)}$.  By the pigeonhole principle as before, we may refine the set of $x$ and assume that $q_x = q$ is independent of $x$.  
\end{itemize}
In particular we now have
\begin{equation}\label{gx-decomp}
g_x \Gamma = \eps_x g'_x \gamma_x \Gamma.
\end{equation}

We can form a partition $(x, x+H] = P_{x,1} \cup \cdots \cup P_{x,r}$ for some $1\leq r \leq \delta^{-O(1)}$, where each $P_{x,i}$ is an arithmetic progression of modulus $q$ and, for each $x$, $d_G(\eps_x(n), \eps_x(n')) \leq \delta^4$ whenever $n,n' \in P_{x,i}$ (which can be ensured by the smoothness of $\eps_x$ as long as $|P_{x,i}| \leq \delta^CH$ for some sufficiently large constant $C=C_{d,D,\kappa}$).  By the triangle inequality for maximal sums (Lemma~\ref{basic-prop}(i)), we have
$$ 
\left| \sum_{x < n \leq x+H} f(n) F(g_x(n)\Gamma) \right|^* \leq \sum_{i=1}^r \left| \sum_{n \in P_{x,i}} f(n) F(g_x(n)\Gamma) \right|^*.
$$
Thus, by the pigeonhole principle, there exists $i$ such that 
$$ 
\left| \sum_{x < n \leq x+H} f(n) F(g_x(n)\Gamma) \right|^* \ll \delta^{-O(1)} \left| \sum_{n \in P_{x,i}} f(n) F(g_x(n)\Gamma) \right|^*.
$$
for all $x$ in a subset $E'\subset E$ of measure at least $\delta^{O(1)} X$.

Fix such $i$.  The function $n\mapsto \gamma_x(n) \Gamma$ is constant on $P_{x,i}$ and $\delta^{-O(1)}$-rational, hence we may write $\gamma_x(n) \Gamma = \gamma_{x,i} \Gamma$ for all $n \in P_{x,i}$ and some $\gamma_{x,i} \in G$ which is $\delta^{-O(1)}$-rational.  There are only $O(\delta^{-O(1)})$ possibilities for $\gamma_{x,i}$, so by pigeonholing in $x$ once again we may assume that $\gamma_{x,i} = \gamma_i$ for some $\delta^{-O(1)}$-rational $\gamma_i \in G$ independent of $x$.

On $P_{x,i}$, the function $\eps_x$ ranges in a $\delta^4$-ball in $G$ at distance $\delta^{-O(1)}$ to the origin.  Since the ball of radius $\delta^{-O(1)}$ around the origin can be covered by $O(\delta^{-O(1)})$ balls of radius $\delta^4$, we may pigeonhole to conclude (after once again refining the set of $x$) that there exists $\eps_{i,0}$ at distance $O(\delta^{-O(1)})$ from the origin such that $d_G(\eps_x(n), \eps_{i,0}) \ll \delta^4$ for all $n \in P_{x,i}$ and all $x$ in a subset of $E$ of measure at least $\delta^{O(1)} X$.

Let $g_{x, i} \in \Poly(\Z\to G)$ be the polynomial sequence defined by
$$ g_{x,i}(n) = \gamma_i^{-1} g'_x(n) \gamma_i, $$
which takes values in $\gamma_i^{-1}G'\gamma_i$.  For each $n \in P_{x,i}$ and $x\in E$ we have from~\eqref{gx-decomp} and the right-invariance of $d_G$ that
\begin{align*} 
|F(g_x(n)\Gamma) - F(\eps_{i,0} \gamma_i g_{x,i}(n)\Gamma)|&\leq \|F\|_{\Lip} \cdot d_G(\eps_x(n) g_x'(n)\gamma_i, \eps_{i,0} g'_x(n) \gamma_i) \\ 
&= \|F\|_{\Lip} \cdot d_G(\eps_x(n), \eps_{i,0}) \ll \delta^3. 
\end{align*}
It follows from this and Lemma~\ref{le:shiu} that
\begin{align}
\label{eq:LamF-Fi}
\left| \sum_{x < n \leq x+H} f(n) F(g_x(n)\Gamma) \right|^* &\leq \delta^{-O(1)} \left| \sum_{n \in P_{x,i}} f(n)  F(\eps_{i,0} \gamma_i g_{x,i}(n)\Gamma) \right|^* + O(\delta^2H).
\end{align}
 Note that since $F$ is $\delta^{-1}$-Lipschitz, the function $F(\varepsilon_{i,0}\gamma_i\cdot)$ is $\delta^{-O(1)}$-Lipschitz.  Let $B$ be large in terms of $d,D,\kappa$. Then, by Lemma~\ref{basic-prop}(i) and the induction hypothesis, we have, for each $i = 1, \dotsc, r$,
\begin{align*}
\left| \sum_{n \in P_{x,i}} f(n)  F(\eps_{i,0} \gamma_i g_{x,i}(n)\Gamma) \right|^* &\leq \sup_{g \in \Poly(\Z \to G')} \left| \sum_{x < n \leq x+H} f(n) F(\eps_{i,0} \gamma_i g(n)\Gamma) \right|^* \\
&\ll \delta^B H
\end{align*}
for $x \in E'$ with $E'\subset E$ having measure $\gg \delta^{O(1)}X$. Combining this with~\eqref{eq:LamF-Fi}, we obtain
$$ \left| \sum_{x< n \leq x+H} f(n) F(g(n)\Gamma) \right|^* \ll \delta^{2} H $$
for $x\in E'$, contradicting our assumption~\eqref{eq:discorLamClaimRev}. This completes the proof of Theorem~\ref{thm_generalseq}.
\end{proof}

\section{Contagion lemmas}\label{contagion-sec}

In this section we develop the theory of ``contagiousness'',  introduced recently by Walsh~\cite{walsh}.  This will culminate in a ``nilsequence contagion lemma'', \Cref{nil-contagion}, which will be a key ingredient in the proof of the type $II$ inverse theorems (Theorem~\ref{inverse}(ii)--(iii)).

\begin{figure}[H]
    \centering
\begin{tikzpicture}[node distance=1cm and 1cm]
    \node[draw, rectangle, align=center,fill=gray!20] (nil-contagion) {Nilsequence contagion \\ \Cref{nil-contagion}};
    \node[draw, rectangle, above=of nil-contagion, align=center] (mono) {Monomial contagion \\ \Cref{thm_contagious}};
    \node[draw, rectangle, above left=of nil-contagion, align=center] (qg) {Coefficient extraction \\ \Cref{qg}};
    \node[draw, rectangle, above right=of mono, align=center] (contag-2) {Making $\alpha'_1$ smaller\\ \Cref{le_contagious2}};
    \node[draw, rectangle, above left=of mono, align=center] (contag-1) {Making $\alpha'_1$ small\\ \Cref{le_contagious1}};
    \node[draw, rectangle, above left=of contag-2, align=center] (cube) {Finding $s$-cubes\\ \Cref{le_comb_cube}};
    \node[draw, rectangle, above =of cube, align=center] (sfold) {$s$-fold differencing\\ \Cref{s-fold}};

    \draw[-{Latex[length=4mm, width=2mm]}] (mono) -- (nil-contagion);
    \draw[-{Latex[length=4mm, width=2mm]}] (qg) -- (nil-contagion);
    \draw[-{Latex[length=4mm, width=2mm]}] (contag-2) -- (mono);
    \draw[-{Latex[length=4mm, width=2mm]}] (contag-1) -- (mono);
    \draw[-{Latex[length=4mm, width=2mm]}] (sfold) -- (contag-1);
    \draw[-{Latex[length=4mm, width=2mm]}] (cube) -- (contag-1);
    \draw[-{Latex[length=4mm, width=2mm]}] (sfold) -- (contag-2);
    \draw[-{Latex[length=4mm, width=2mm]}] (cube) -- (contag-2);
\end{tikzpicture}
    \caption{The proof of \Cref{nil-contagion}. This result will then be used to prove the type $II$ inverse theorems (Theorem~\ref{inverse}(ii)--(iii)); see \Cref{fig-inv}.}
    \label{fig-contagion}
\end{figure}
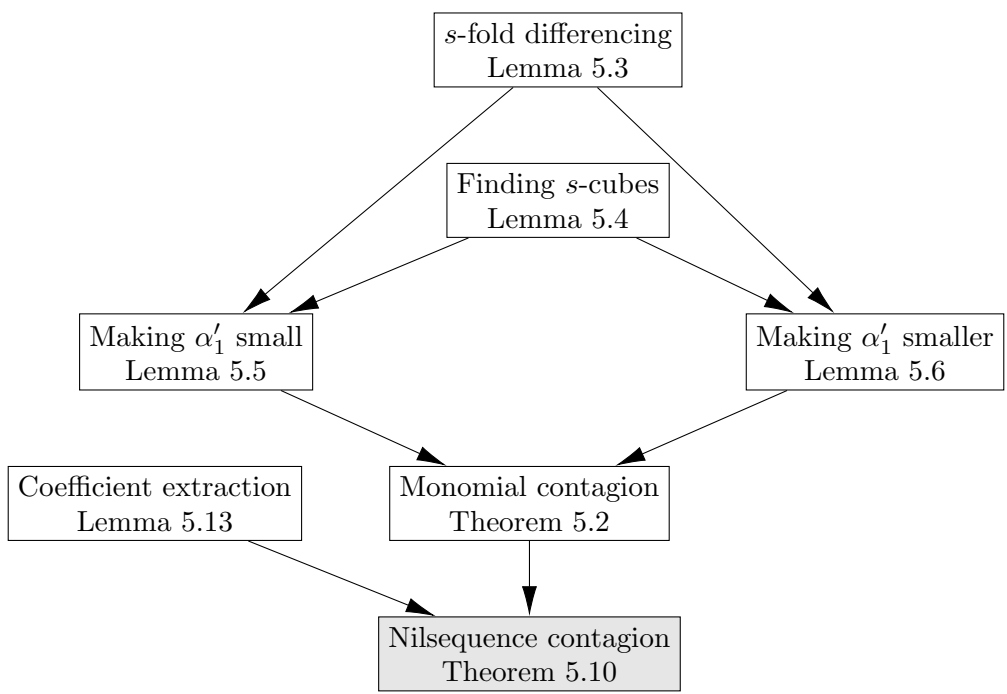

\subsection{Walsh's contagion lemma}

Let $P \geq 1$, let $S$ be a set of integers in $[P,2P]$, and suppose that $\alpha \colon S \to \R$ is a function obeying the bounds
$$ \| \alpha(n) - \alpha_* n \|_{\R/\Z} \leq \eps / P$$
for every $n \in S$, some real number $\alpha_*$, and some $\eps>0$.  Then from the triangle inequality we have
\begin{align}\label{eq:contagious} \| n_1 \alpha(n_2) - n_2 \alpha(n_1) \|_{\R/\Z} \ll \eps
\end{align}
for all $n_1,n_2 \in S$.  

In the recent paper~\cite{walsh}, Walsh established the following partial converse to this observation, encapsulating a phenomenon referred to in that paper as ``contagiousness''\footnote{The name for the terminology relates to the fact, demonstrated by Lemma~\ref{Walsh}, that if we have the relation~\eqref{eq:contagious} for ``many'' pairs $(n_1,n_2)\in S^2$, then it holds for \emph{all} $(n_1,n_2)\in S'^2$ for a ``large'' subset $S' \subseteq S$, so the relations~\eqref{eq:contagious} are in a sense contagious.}:

\begin{lemma}[Walsh's contagion lemma]\label{Walsh}~\cite[Lemma 2.4]{walsh} Let $\eta \in (0, 1/2)$, $P \geq 1$, and $\eps \in (0, c/P^2)$ for a sufficiently small absolute constant $c$. Let $S$ be a set of primes in $[P,2P]$, and let $\alpha \colon S \to \R$ be a function satisfying
$$ \| p_1 \alpha(p_2) - p_2 \alpha(p_1) \|_{\R/\Z} \leq \eps$$
for $\geq \eta |S|^2$ pairs $(p_1,p_2) \in S^2$. Then there is a real number $\alpha_*$ such that
$$ \| \alpha(p) - \alpha_* p \|_{\R/\Z} \ll \eps / P$$
for $\gg_\eta |S|$ primes $p \in S$.
\end{lemma}

This lemma was then used in~\cite{walsh} to obtain a simplified proof of the local Fourier uniformity estimate 
\begin{align*}
\int_{X}^{2X}\sup_{\alpha\in \mathbb{R}}\left|\sum_{x<n\leq x+H}\mu(n)e(-\alpha n)\right|\, dx=o(HX),    
\end{align*}
where $H=X^{\theta}$ for fixed $\theta>0$.

In order to adapt Lemma~\ref{Walsh} to our current applications, we will extend and modify it in several ways, which are essential for our application to Theorem~\ref{inverse}(ii)--(iii):

\begin{itemize}
\item[(i)] We will remove the restriction that $S$ consists solely of primes.
\item[(ii)] We will make the dependence on the parameter $\eta$ polynomial in nature.
\item[(iii)] We will relax the bound $\eps < c/P^2$ to $\eps < c/P^{\kappa}$ for arbitrarily small exponents $\kappa>0$.
\item[(iv)] We will replace the linear polynomials $p_1,p_2,p$ in Lemma~\ref{Walsh} with monomials $p_1^s, p_2^s, p^s$ of arbitrary degree $s$.
\item[(v)] We will then replace the frequency functions $\alpha$ by nilsequences.
\end{itemize}

In addition, for technical reasons and to facilitate an induction argument, we will replace the frequency function $\alpha$ with a pair $\alpha_1, \alpha_2$ of frequency functions.

\subsection{Monomial contagion}

We begin by establishing a ``monomial contagion lemma'' that achieves the first four (i)--(iv) of the five objectives stated above.

\begin{theorem}[Monomial contagion lemma]\label{thm_contagious}
Let $s\in \N$ and $\kappa>0$. There exists $C_{s,\kappa}\geq 1$ such that the following holds. Let $0 < \eta < 1/2$, $P\geq \eta^{-C_{s,\kappa}}$ and $0<\eps<\eta^{C_{s,\kappa}}/P^{\kappa}$. Let $S_1,S_2$ be subsets of $[P,2P] \cap \Z$ and let $\alpha_1 \colon S_1 \to \R$, $\alpha_2 \colon S_2 \to \R$  be two functions satisfying 
\begin{align}\label{eq0}
\|n_1^s \alpha_2(n_2)-n_2^s\alpha_1(n_1)\|_{\R/\Z} \leq \eps    
\end{align} 
for all $(n_1,n_2)$ in a subset ${\mathcal N}_0$ of $S_1 \times S_2$ of cardinality at least $\eta P^2$. Then there is $\alpha_* \in \R$ and an integer $1 \leq q \ll \eta^{-O_{s,\kappa}(1)}$ such that, for $\gg \eta^{O_{s,\kappa}(1)} P^2$ pairs $(n_1,n_2) \in {\mathcal N}_0$, one has
\begin{equation}\label{q-factor}
\|q(\alpha_i(n_i)-\alpha_* n_i^s)\|_{\R/\Z} \ll \eta^{-O_{s,\kappa}(1)}\eps/P^s
\end{equation}
for $i=1,2$.
\end{theorem}

In fact, in our proof the factor of $q$ in~\eqref{q-factor} is only needed for $i=2$, but we will not exploit this refinement of the theorem here.  Because $n_1,n_2$ are no longer restricted to be prime, the factor of $q$ cannot be eliminated entirely; a simple counterexample is when $S_1=S_2=[P,2P]\cap \mathbb{Z}$, $s=1$, $\alpha_1(n_1)=0$, $\alpha_2(n_2) = 1/2$, $\eta \asymp 1$, $P$ is large and $\eps$ is infinitesimal, and ${\mathcal N}_0$ is the set of pairs $(n_1,n_2) \in S_1 \times S_2$ with $n_1,n_2 = 2 \hbox{ mod } 4$. In this example, there is no way to subtract off a common linear factor $\alpha_* n_i$ from $\alpha_i(n_i)$ for $i=1,2$ to make them both close to zero for many $(n_1,n_2) \in {\mathcal N}_0$.

\begin{remark} Strictly speaking, Theorem \ref{thm_contagious} is not a strengthening of Lemma \ref{Walsh}, both because of the introduction of the new quantity $q$ discussed above, and also because the density parameter $\eta$ is relative to the integers in $[P,2P]$ rather than the primes in $[P,2P]$, thus in principle incurring additional losses of $\log P$.  However, these losses turn out to not be significant for our current application.    
\end{remark}

In the rest of this subsection we prove Theorem~\ref{thm_contagious}.  Let $s, \kappa$ be fixed, and assume $C_{s,\kappa}$ is sufficiently large.  We allow implied constants to depend on $s,\kappa$, but not on $C_{s,\kappa}$.  Let $\eta, \eps, P, S_1, S_2, \alpha_1, \alpha_2, \mathcal{N}_0$ be as in the statement of the theorem. In all the lemmas in this subsection, the assumptions are as in Theorem~\ref{thm_contagious}.

We will make a number of modifications to the pair $(\alpha_1,\alpha_2)$ in order to improve the estimates on the individual functions $\alpha_1,\alpha_2$.  To keep track of these modifications, it is convenient to introduce the set ${\mathcal A}$ of all pairs $(\alpha'_1, \alpha'_2)$ of functions $\alpha'_1, \alpha'_2 \colon S \to \R$ which are ``equivalent'' to the original pair $(\alpha_1,\alpha_2)$ in the sense that there exists a real number $\alpha_*$ such that
$$ \alpha'_i(n_i) = \alpha_i(n_i) - \alpha_* n_i^s \hbox{ mod } 1 $$
for all $n_i \in S_i$ and $i=1,2$.  Clearly $(\alpha_1,\alpha_2)$ lies in ${\mathcal A}$.  Also, if $(\alpha'_1, \alpha'_2)$ lies in ${\mathcal A}$, and $\beta$ is a real number, then the pair $(\alpha''_1, \alpha''_2)$ defined by
$$ \alpha''_i(n_i) \coloneqq \{ \alpha'_i(n_i) - \beta n_i^s \}$$
also lies in ${\mathcal A}$ (recall that $\{x\} \in (-1/2,1/2]$ is the signed fractional part of $x$).  Finally, observe that if $(\alpha'_1, \alpha'_2)$ lies in ${\mathcal A}$, we have
$$ n_1^s \alpha'_2(n_2)-n_2^s\alpha'_1(n_1) = n_1^s \alpha_2(n_2)-n_2^s\alpha_1(n_1) \hbox{ mod } 1$$
for any $n_1 \in S_1, n_2 \in S_2$, and hence from~\eqref{eq0} we have
\begin{equation}\label{eq0-alt}
 \| n_1^s \alpha'_2(n_2)-n_2^s\alpha'_1(n_1)\|_{\R/\Z} \leq \eps
\end{equation}
for all $(n_1,n_2) \in {\mathcal N}_0$.

Our strategy will be to locate pairs $(\alpha'_1, \alpha'_2) \in {\mathcal A}$ for which one has good bounds on the magnitude of $\alpha'_1$; our first attempt to do so will only locate such a pair obeying a relatively weak bound
$$ |\alpha'_1(n_1)| \leq 2^s \eps$$
for many $n_1$, but by an iterative process we will eventually be able to locate a pair obeying the stronger bound
$$ |\alpha'_1(n_1)| \ll \frac{\eta^{-O(1)}}{P^s} \eps$$
for many $n_1$.  Once we have obtained this control on $\alpha'_1(n_1)$, we will be able to obtain similar control on $\alpha'_2(n_2)$ for many $n_2$ (after multiplying by a small integer $q$), at which point we can establish the theorem.

We turn to the details. To gain control on $\alpha_1$ we will ``differentiate'' the inequality~\eqref{eq0} (or~\eqref{eq0-alt}) $s$ times in the $n_2$ variable in order to mostly eliminate the $n_2^s$ term.  More precisely, we can take $s$-fold differences via the following lemma (where, as per the convention above, the assumptions are as in Theorem~\ref{thm_contagious}).

\begin{lemma}[Taking $s$-fold differences]\label{s-fold} Let $(\alpha'_1,\alpha'_2) \in {\mathcal A}$, and let ${\mathcal N}$ be a subset of ${\mathcal N}_0$. Let $n_{2,0}$ be an integer, and let $\vec h = (h_1,\dots,h_s)$ be a tuple of integers, with the property that
\begin{equation}\label{n1n2}
 (n_1, n_{2,0} + \vec \omega \cdot \vec h) \in {\mathcal N}_0
\end{equation}
whenever $(n_1,n_{2,0}) \in {\mathcal N}$ and $\vec \omega \in \{0,1\}^s$, with $\cdot$ denoting the usual inner product.  Then one has
$$ \| h \alpha'_1(n_1) - \beta n_1^s \|_{\R/\Z}\leq 2^s \eps$$
whenever $(n_1,n_{2,0}) \in {\mathcal N}$, where $h$ is the integer
$$ h \coloneqq s! h_1 \cdots h_s$$
and $\beta$ is the real number
\begin{equation}\label{beta-def}
 \beta \coloneqq \sum_{\vec \omega\in \{0,1\}^s}(-1)^{|\vec \omega|} \alpha'_2(n_{2,0} + \vec \omega \cdot \vec{h}).
\end{equation}
\end{lemma}

\begin{proof}
By~\eqref{n1n2} and~\eqref{eq0-alt}, we have
$$ \|n_1^s \alpha'_2(n_{2,0} + \vec \omega \cdot \vec h)-(n_{2,0} + \vec \omega \cdot \vec h)^s\alpha'_1(n_1)\|_{\R/\Z} \leq \eps $$
whenever $(n_1,n_{2,0}) \in {\mathcal N}$ and $\vec \omega \in \{0,1\}^s$. Hence by the triangle inequality and~\eqref{beta-def}
$$ \left\|\beta n_1^s -\sum_{\vec \omega\in \{0,1\}^s}(-1)^{|\vec \omega|} (n_{2,0} + \vec \omega \cdot \vec h)^s\alpha'_1(n_1)\right\|_{\R/\Z} \leq 2^s \eps.$$
From binomial expansion of $(n_{2,0}+\vec \omega\cdot \vec{h})^s$ (paying particular attention to the $\omega_1 \cdots \omega_s$ terms) we observe the identity
\begin{align*}
\sum_{\vec \omega\in \{0,1\}^s}(-1)^{|\vec \omega|} (n_{2,0}+\vec \omega\cdot \vec{h})^s = s! h_1 \cdots h_s = h
\end{align*}
and the claim follows.
\end{proof}

In order to use the above lemma, we need to be able to ensure a plentiful supply of pairs $(n_1,n_{2,0})$ obeying~\eqref{n1n2}, with a value of $h$ that we can ``tune'' to be of a specified size.  To achieve this we establish the following version of the Hilbert cube lemma~\cite{hilbert-cube},~\cite[Exercise 1.3.2]{tao-higher}. 

\begin{lemma}[Finding combinatorial cubes in dense sets] \label{le_comb_cube} Let $s\geq 1$ be fixed and let $\eta>0$. 
Let ${\mathcal N}$ be a subset of ${\mathcal N}_0$ of cardinality $\gg \eta^{O(1)} P^2$, and let $1 \leq H \leq P$. Then either $H \leq \eta^{-O(1)}$, or else (if the constant $C_{s,\kappa}$ is large enough depending on the implied constants in the preceding hypothesis) there exists a tuple $\vec h = (h_1,\dots,h_s)$ of positive integers $\eta^{O(1)} H\ll h_1,\dots,h_s \leq H$ such that one has the containment
\begin{equation}\label{e7}
(n_1,n_2+\vec \omega \cdot \vec h)_{\vec \omega \in \{0,1\}^s} \subset {\mathcal N}^{\{0,1\}^s} 
\end{equation}
for $\gg \eta^{O(1)} P^2$ choices of pairs $(n_1,n_2) \in {\mathcal N}_0$, where $\cdot$ denotes the usual dot product.
\end{lemma}

\begin{proof} Let $C$ be a large constant (depending on $s,\kappa$, and the implied constants in the hypotheses) to be chosen later.  We may assume that $H \geq \eta^{-C}$, since otherwise the conclusion follows trivially. Given a tuple $(f_{\vec \omega})_{\vec \omega \in \{0,1\}^s}$ of functions $f_{\vec \omega} \colon \Z^2 \to \R$ supported on a finite set, we define the Gowers box inner product
$$ \langle (f_{\vec \omega})_{\vec \omega \in \{0,1\}^s} \rangle_{\Box^s} 
\coloneqq \sum_{n_1,n_2 \in \Z}\, \sum_{\vec h^0, \vec h^1 \in [1,H]^s}\,
\prod_{\vec \omega \in \{0,1\}^s} f_{\vec \omega}( n_1, n_2 + h^{\omega_1}_1 + \cdots + h^{\omega_s}_s )$$
where $\vec \omega = (\omega_1,\dots,\omega_s)$ and $\vec h^0 = (h^0_1,\dots,h^0_s)$, $\vec h^1 = (h^1_1,\dots,h^1_s)$ are understood to have integer components.  We then define the associated Gowers box norm
$$ \|f\|_{\Box^s} \coloneqq \langle (f)_{\vec \omega \in \{0,1\}^s} \rangle_{\Box^s}^{1/2^s}.$$
From the Gowers--Cauchy--Schwarz inequality (see e.g.,~\cite[Lemma B.2]{green-tao}) we have
\begin{equation}\label{gcz}
 |\langle (f_{\vec \omega})_{\vec \omega \in \{0,1\}^s} \rangle_{\Box^s}| \leq \prod_{\vec \omega' \in \{0,1\}^s} \|f_{\vec \omega'}\|_{\Box^s}.
\end{equation}

Let ${\mathcal Q} \subset \N$ be the set of pairs $(n_1,n_2) \in \Z^2$ with $n_1 \in [P,2P]$ and $|n_2| \leq (s+2) P$. Then a routine computation gives
$$ \| 1_{\mathcal Q} \|_{\Box^s} \ll (H^{2s} P^2)^{1/2^s}.$$
We apply~\eqref{gcz} with $f_{0,\dots,0} \coloneqq 1_{\mathcal N}$ and $f_{\vec \omega} \coloneqq 1_{\mathcal Q}$ for all other $\vec \omega\in \{0,1\}^s$, to obtain
$$ \langle (f_{\vec \omega})_{\vec \omega \in \{0,1\}^s} \rangle_{\Box^s} \ll \| 1_{\mathcal N} \|_{\Box^s} (H^{2s} P^2)^{1-1/2^s}.$$
On the other hand, if $(n_1,n_2 + h^0_1 + \cdots + h^0_s) \in {\mathcal N}$ for some $n_1,n_2 \in \Z$ and $\vec h^0, \vec h^1 \in [1,H]^s$, then $( n_1, n_2 + h^{\omega_1}_1 + \cdots + h^{\omega_s}_s ) \in {\mathcal Q}$ for all $\vec \omega \in \{0,1\}^s$.  Thus we can expand
\begin{align*}
\langle (f_{\vec \omega})_{\vec \omega \in \{0,1\}^s} \rangle_{\Box^s}  &= 
 \sum_{n_1,n_2 \in \Z} \sum_{\vec h^0, \vec h^1 \in [1,H]^s} 1_{\mathcal N}(n_1,n_2 + h^0_1 + \cdots + h^0_s) \\
 &= 
 \sum_{n_1,n_2 \in \Z} \sum_{\vec h^0, \vec h^1 \in [1,H]^s} 1_{\mathcal N}(n_1,n_2)\\ 
&\gg H^{2s} |\mathcal{N}| \\
&\gg \eta^{O(1)} H^{2s} P^2.
\end{align*}
Combining these inequalities, we conclude that
$$ \| 1_{\mathcal N} \|_{\Box^s}^{2^s} \gg \eta^{O(1)} H^{2s} P^2$$
and thus
$$ \sum_{n_1,n_2 \in \Z} \sum_{\vec h^0, \vec h^1 \in [1,H]^s}
\prod_{\vec \omega \in \{0,1\}^s} 1_{\mathcal N}( n_1, n_2 + h^{\omega_1}_1 + \cdots + h^{\omega_s}_s ) \gg \eta^{O(1)} H^{2s} P^2.$$
By swapping coefficients of $\vec h^0$ and $\vec h^1$ as necessary and using symmetry, we may reduce to the contribution in which $h^1_i \geq h^0_i$ for all $i=1,\dots,s$.  Replacing $n_2$ by $n_2 + h^0_1 + \cdots + h^0_s$ and writing $\vec h^1 - \vec h^0 = \vec h$, we conclude
$$ \sum_{n_1,n_2 \in \Z} \sum_{\vec h \in [0,H]^s}
\prod_{\vec \omega \in \{0,1\}^s} 1_{\mathcal N}( n_1, n_2 + \vec \omega \cdot \vec h ) \gg \eta^{O(1)} H^{s} P^2.$$
The contribution of those $\vec h$ with at least one coefficient less than $\eta^C H$ can be easily bounded by $O( \eta^C H^s P^2)$.  Thus for $C$ large enough we have
$$ \sum_{n_1,n_2 \in \Z} \sum_{\vec h \in [\eta^C H,H]^s}
\prod_{\vec \omega \in \{0,1\}^s} 1_{\mathcal N}( n_1, n_2 + \vec \omega \cdot \vec h ) \gg \eta^{O(1)} H^{s} P^2.$$
By the pigeonhole principle, there thus exists $\vec h \in [\eta^C H,H]^s$ such that
$$ \sum_{n_1,n_2 \in \Z} 
\prod_{\vec \omega \in \{0,1\}^s} 1_{\mathcal N}( n_1, n_2 + \vec \omega \cdot \vec h ) \gg \eta^{O(1)} P^2$$
and the claim follows.
\end{proof}

As a first application of the above two lemmas, we locate a pair $(\alpha'_1,\alpha'_2) \in {\mathcal A}$ with $\alpha'_1$ somewhat small.

\begin{lemma}[Making $\alpha'_1(n)$ somewhat small]\label{le_contagious1}  There exists $(\alpha'_1, \alpha'_2) \in {\mathcal A}$ such that
\begin{align*}
|\alpha_1'(n_1)| \leq 2^s \eps
\end{align*}
for $\gg \eta^{O(1)}P^2$ pairs $(n_1,n_2) \in {\mathcal N}_0$.
\end{lemma}

\begin{proof}
Applying Lemma~\ref{le_comb_cube} with $H=\eta^{-C}$ for some sufficiently large constant $C$, we can find a tuple $\vec h = (h_1,\dots,h_s)$ of positive integers $h_1,\dots,h_s \ll \eta^{-O(1)}$ and a subset ${\mathcal N}$ of ${\mathcal N}_0$ of cardinality $\gg \eta^{O(1)} P^2$ such that
$$ (n_1, n_2 + \vec \omega \cdot \vec h) \in {\mathcal N}_0$$
for all $(n_1,n_2) \in {\mathcal N}$ and $\vec \omega \in \{0,1\}^s$. 

By pruning all sparse columns from ${\mathcal N}$, we may assume that whenever $(n_1,n_2) \in {\mathcal N}$, then one also has $(n_1,n'_2) \in {\mathcal N}$ for $\gg \eta^{O(1)} P$ choices of $n'_2 \in S_2$.

By the pigeonhole principle, we may find $n_2^0 \in S_2$ such that $(n_1,n_2^0) \in {\mathcal N}$ for $\gg \eta^{O(1)} P$ choices of $n_1 \in S_1$.
Applying Lemma~\ref{s-fold}, we conclude that
\begin{align}\label{eq10}
 \|h\alpha_1(n_1)-\beta n_1^s\|_{\R/\Z} \leq 2^s \eps
\end{align}
whenever $(n_1,n^0_2)\in \mathcal{N}$, for some real number $\beta$ (not depending on $n_1$) and with
$$
h \coloneqq s!h_1\cdots h_s.$$
In particular $1 \leq h \ll \eta^{-O(1)}$.  From~\eqref{eq10} we then see that whenever $(n_1,n^0_2) \in {\mathcal N}$, we have
\begin{equation}\label{alphab}
 \left\|\alpha_1(n_1)-\frac{\beta}{h} n_1^s - \frac{a}{h} \right\|_{\R/\Z} \leq \frac{2^s \eps}{h}
\end{equation}
for some integer $1 \leq a \leq h$. By the pigeonhole principle, we can then find a single $1 \leq a \leq h$ and a subset $S'_1$ of $S_1$ of size $\gg \eta^{O(1)} P$ such that for all $n_1 \in S'_1$, one has $(n_1,n_2^0) \in {\mathcal N}$, and~\eqref{alphab} holds for this fixed choice of $a$.

For technical reasons we need to analyze how $n_1$ shares primes in common with $h$.  Let ${\mathcal H}$ be the multiplicative semigroup generated by the primes dividing $h$.  We observe that
\begin{align*}
\sum_{q \in {\mathcal H}\cap \mathbb{N}} \frac{1}{q^{1/2}} &= \prod_{p|h} \left(1 - \frac{1}{p^{1/2}}\right)^{-1}\ll \exp\left( \sum_{p|h} \frac{2}{p^{1/2}}  \right) \\
&\ll \exp\left(\sum_{p|h} \log p \right) \ll \exp(\log h ) \ll \eta^{-O(1)}.
\end{align*}
In particular, for any $Q \geq  1$, we have
$$ \sum_{\substack{q \in {\mathcal H}\cap \mathbb{N}\\ q \geq Q}} \frac{1}{q} \ll \eta^{-O(1)} Q^{-1/2}.$$
From this and the union bound, we see that if $1 \leq Q \leq P$, then the number of those $n_1 \in S'_1$ that are divisible by some $q \in {\mathcal H}\cap \mathbb{N}$ with $q \geq Q$ is at most $O( \eta^{-O(1)} Q^{-1/2} P)$.  Thus, for a suitable choice of $Q = \eta^{-O(1)}$, such $n_1$ only occupy at most half (say) of the set $S'_1$.  By the pigeonhole principle, we may thus find $q \in {\mathcal H}\cap\mathbb{N}$ with $1 \leq q < Q$ and a subset $S''_1$ of $S'_1$ of cardinality $\gg \eta^{O(1)} P$ with the property that for all $n_1 \in S'_1$, $q$ is the largest element of ${\mathcal H}\cap \mathbb{N}$ that divides $n_1$, thus $n_1 = q n'_1$ where $n'_1$ is coprime to $h$.  By further application of the pigeonhole principle, we may also restrict $n'_1$ to a single primitive residue class $n'_1 = b \hbox{ mod } h$ of $h$.  Letting $k$ be a positive integer that solves $kb^s\equiv a\pmod h$, we then have
$$ \frac{k}{q^s h} n_1^s = \frac{k (n'_1)^s}{h} = \frac{a}{h} \hbox{ mod } 1$$
for all $n_1 \in S''_1$.  If we then define
$$ \beta' \coloneqq \frac{\beta}{h} + \frac{k}{q^s h}$$
we conclude from~\eqref{alphab} that
$$ \|\alpha_1(n_1)-\beta' n_1^s  \|_{\R/\Z} \leq \frac{2^s \eps}{h}$$
for all $n_1 \in S''_1$.  Because we have previously pruned all sparse columns from ${\mathcal N}$, we conclude that
$$ \|\alpha_1(n_1)-\beta' n_1^s  \|_{\R/\Z} \leq \frac{2^s \eps}{h}$$
for $\gg \eta^{O(1)} P^2$ of the pairs $(n_1,n_2) \in {\mathcal N}_0$.  Setting $\alpha'_i(n_i) \coloneqq \{ \alpha_i(n_1) - \beta' n_i^s\}$ for $i=1,2$, we obtain the claim.
\end{proof}

Next, we show that if we have already found a pair $(\alpha'_1,\alpha'_2) \in {\mathcal A}$ in which $\alpha'_1$ is often of some small size $O(\mu)$, we can find a new pair $(\alpha''_1,\alpha''_2)$ in which $\alpha''_1$ is often as small as $O\left(\eta^{-O(1)} \eps \left( \mu + \frac{1}{P^s} \right) \right)$.

\begin{lemma}[Making $\alpha'_1(n)$ even smaller]\label{le_contagious2}  Let $\mu > 0$, and suppose that there is $(\alpha'_1, \alpha'_2) \in {\mathcal A}$ such that
\begin{equation}\label{alpha-in}
|\alpha'_1(n_1)|\leq \mu    
\end{equation} 
for $\gg \eta^{O(1)} P^2$ of the pairs $(n_1,n_2) \in {\mathcal N}_0$.  Then (if $C_{s,\kappa}$ is sufficiently large depending on the implied constants in the hypothesis) there exists $(\alpha''_1, \alpha''_2) \in {\mathcal A}$ such that
\begin{equation*}
|\alpha''_1(n_1)| \ll \eta^{-O(1)} \eps \left( \mu + \frac{1}{P^s} \right )
\end{equation*}
for $\gg \eta^{O(1)}P^2$ pairs $(n_1,n_2)\in {\mathcal N}_0$.
\end{lemma}

\begin{proof}  
Let $C_*$ be a large constant depending only on $s, \kappa$ to be chosen later; until this constant is selected, we do \emph{not} allow implied constants to depend on $C_*$. The constant $C_{s, \kappa}$ will be large in terms of $C_*$.  We may assume that $\mu \leq \eta^{C_*}$, since otherwise the claim follows from Lemma~\ref{le_contagious1}.

Applying Lemma~\ref{le_comb_cube} with $H= \eta^{C_*/(2s)} \min(\mu^{-1/s},P)$, we can find a tuple $\vec h = (h_1,\dots,h_s)$ of positive integers $\eta^{O(1)}H\ll h_1,\dots,h_s \leq H$ and a subset ${\mathcal N}$ of ${\mathcal N}_0$ of cardinality $\gg \eta^{O(1)} P^2$ such that
$$ (n_1, n_2 + \vec \omega \cdot \vec h) \in {\mathcal N}_0$$
for all $(n_1,n_2) \in {\mathcal N}$ and $\vec \omega \in \{0,1\}^s$.  By pruning as in the proof of Lemma~\ref{le_contagious1}, we may assume that whenever $(n_1,n_2) \in {\mathcal N}$, then one also has $(n_1,n'_2) \in {\mathcal N}$ for $\gg \eta^{O(1)} P$ choices of $n'_2 \in S_2$.

By the pigeonhole principle, we may find $n_2^0 \in S_2$ such that $(n_1,n_2^0) \in {\mathcal N}$ for $\gg \eta^{O(1)} P$ choices of $n_1 \in S_1$.
Applying Lemma~\ref{s-fold} as before, we have
\begin{equation}\label{ho}
 \|h\alpha'_1(n_1)-\beta n_1^s\|_{\R/\Z} \leq 2^s \eps
\end{equation}
whenever $(n_1,n^0_2)\in \mathcal{N}$, where $\beta$ is a real number and $h$ is a positive integer of size
\begin{equation*}
\eta^{O(1)}Y\ll h\ll Y,\quad \textnormal{where}\quad Y\coloneqq  (\eta^{C_*/(2s)} \min(\mu^{-1/s},P))^s = \eta^{C_*/2 } \min(\mu^{-1},P^s).
\end{equation*}
From~\eqref{alpha-in} we have
\begin{equation}\label{hi}
 |h\alpha'_1(n_1)| \ll \eta^{C_*/2}
\end{equation}
and hence
$$ \| \beta n_1^s \|_{\R/\Z} \ll \eta^{C_*/2}$$
for $\gg \eta^{O(1)} P$ choices of $n_1 \in S_1$.  Applying Lemma~\ref{vin}, we may find a positive integer $1 \leq q \ll \eta^{-O(1)}$ such that
\begin{equation}\label{qbeta}
 \| q\beta  \|_{\R/\Z} \ll \eta^{C_*/2-O(1)} / P^s.
\end{equation}
From~\eqref{ho} we have
\begin{equation}\label{qha}
 \| qh \alpha'_1(n_1) - \{q\beta\} n_1^s \|_{\R/\Z} \ll \eta^{-O(1)} \eps
\end{equation}
whenever $(n_1,n_2^0) \in {\mathcal N}$.
On the other hand from~\eqref{hi} and~\eqref{qbeta} we have
$$ |qh \alpha'_1(n_1) - \{q\beta\} n_1^s| \ll \eta^{C_*/2-O(1)}.$$
For $C_*$ large enough, this implies that $qh \alpha'_1(n_1) - \{q\beta\} n_1^s$ has magnitude less than $1/2$, and so from~\eqref{qha} we now have
$$|qh \alpha'_1(n_1) - \{q\beta\} n_1^s| \ll \eta^{-O(1)} \eps$$
whenever $(n_1,n_2^0) \in {\mathcal N}$. Thus
$$ \left|\alpha'_1(n_1) - \frac{\{q\beta\}}{qh} n_1^s\right| \ll \eta^{-O(1)} \frac{\eps}{qh} \ll \eta^{-C_{*}/2-O(1)} \eps \left( \mu + \frac{1}{P^s} \right).$$
Setting
$$ \alpha''_i(n_i) \coloneqq \alpha'_i(n_i) - \frac{\{q\beta\}}{qh} n_i^s$$
for $i=1,2$, we obtain the claim. 
\end{proof}

If we apply  Lemma~\ref{le_contagious1} followed by $j$ applications of Lemma~\ref{le_contagious2}, we see that for any fixed $j$ (depending on $s,\kappa$) we can find $(\alpha'_1,\alpha'_2) \in {\mathcal A}$ such that
$$ |\alpha'_1(n_1)| \ll \eta^{-O(1)} \eps \left( \eps^{j} + \frac{1}{P^s} \right )
$$
for $\gg \eta^{O(1)}P^2$ pairs $(n_1,n_2)\in {\mathcal N}_0$.  Since by hypothesis $0<\eps<\eta^{C_{s,\kappa}}/P^{\kappa}$, we can choose $j$ large enough depending on $\kappa$ so that the $\eps^{j}\leq \frac{1}{P^s}$, and thus we now have
\begin{equation}\label{alpha1-small}
 |\alpha'_1(n_1)| \ll \frac{\eta^{-O(1)}}{P^s} \eps
\end{equation}
for all pairs $(n_1,n_2)$ in a subset ${\mathcal N}$ of ${\mathcal N}_0$ of size $\gg \eta^{O(1)} P^2$.  

Having obtained good control on $\alpha'_1$, we turn to $\alpha'_2$.  From~\eqref{eq0-alt} and~\eqref{alpha1-small} we have
\begin{equation}\label{n1s}
\| n_1^s \alpha'_2(n_2)\|_{\R/\Z} \ll \eta^{-O(1)} \eps
\end{equation}
for all $(n_1,n_2) \in {\mathcal N}$.  Thus we can find a subset $S'_2$ of $S_2$ of cardinality $\gg \eta^{O(1)} P$ such that for every $n_2 \in S'_2$,~\eqref{n1s} holds for $\gg \eta^{O(1)} P$ choices of $n_1$ with $(n_1,n_2) \in {\mathcal N}$.  Applying Lemma~\ref{vin}, we thus see that for each $n_2 \in S'_2$ there is an integer $1 \leq q \ll \eta^{-O(1)}$ such that
$$ \| q\alpha'_2(n_2)\|_{\R/\Z} \ll \eta^{-O(1)} \eps / P^s.$$
By pigeonholing and refining $S'_2$ as necessary, we may assume without loss of generality that $q$ is independent of $n_2$.  We conclude that for $\gg \eta^{O(1)} P^2$ pairs $(n_1,n_2) \in {\mathcal N}_0$, we have
$$ \| q\alpha'_i(n_i)\|_{\R/\Z} \ll \eta^{-O(1)} \eps / P^s$$
for $i=1,2$, and Theorem~\ref{thm_contagious} follows.

\subsection{Nilsequence contagion}

In this subsection we bootstrap the monomial contagion lemma (Theorem~\ref{thm_contagious}) to prove a nilsequence contagion lemma. 

Given a filtered nilmanifold $G/\Gamma$, we introduce the following relation between polynomial sequences that are equivalent up to smooth and rational components. 

\begin{definition}\label{def:nilseqequiv}
Let $G/\Gamma$ be a filtered nilmanifold, $I\subset\mathbb{R}$ be an interval with $|I|\geq 1$, and let $\eta\in (0,1)$. Let $g, g' \in  \Poly(\mathbb{Z} \to G)$. We write $g \sim_{I,\eta} g'$ if 
\begin{align*}
g = \varepsilon g'\gamma,    
\end{align*}
for some  $\varepsilon, \gamma \in \Poly(\mathbb{Z} \to G)$, with $\varepsilon$ being $(\eta^{-1},I)$-smooth and $\gamma$ being $\eta^{-1}$-rational.    
\end{definition}

\begin{remark}
Note that the relation $\sim_{I,\eta}$ depends also on the choice of the filtered nilmanifold $G/\Gamma$. However, as the choice of the filtered nilmanifold will always be clear from context, we omit this data from the notation.     
\end{remark}

Note that while the relation $\sim_{I,\eta}$ is not quite an equivalence relation, it is an equivalence relation up to polynomial losses in $\eta$. It is also scalable with respect to the interval $I$. These facts are summarized in the following lemma.

\begin{lemma}[Basic properties of the $\sim_{I,\eta}$ relation]\label{le:equivalence} 
Let $B, d,D\geq 0$. Let $\eta \in (0,1/2)$, and let $G/\Gamma$ be a filtered nilmanifold of dimension $D$, degree $d$ and complexity $\leq \eta^{-B}$. Let $I\subset \mathbb{R}$ be an interval of length $\geq 1$, and let $g_1,g_2,g_3\in \Poly(\mathbb{Z}\to G)$.  
\begin{enumerate}
    \item[(i)] We have 
    \begin{align*}
g_1&\sim_{I,\eta}g_1,\\
g_1&\sim_{I,\eta} g_2 \implies g_2\sim_{I,\eta^{O_{B,d,D}(1)}}g_1,\\
g_1&\sim_{I,\eta}g_2\quad \textnormal{ and }\quad g_2\sim_{I,\eta}g_3\implies g_1\sim_{I,\eta^{O_{B,d,D}(1)}} g_3.  
\end{align*}
\item[(ii)] Let $I_1, I_2, I_3\subset \mathbb{R}$ be intervals of length at least $d+1$ with $I_1\subset I_2\subset I_3$ and $|I_3|\leq \eta^{-B}|I_2|$. Then 
\begin{align*}
g_1&\sim_{I_2,\eta}g_2\implies g_1\sim_{I_1,\eta} g_2,\\
g_1&\sim_{I_2,\eta}g_2\implies g_1\sim_{I_3,\eta^{O_{B,d,D}(1)}} g_2.    
\end{align*}

\end{enumerate}

\end{lemma}

\begin{proof}  We allow implied constants to depend on $B,d,D$.

The first claim in part (i) is trivial. The second part follows from the fact that if $\varepsilon$ is $(\eta^{-1},I)$-smooth and $\gamma$ is $\eta^{-1}$-rational, then by Lemma~\ref{mult} also $\varepsilon^{-1}$ is $(\eta^{-O(1)},I)$-smooth and $\gamma^{-1}$ is $\eta^{-O(1)}$-rational. The third part follows by noting that if $g_1=\varepsilon_1 g_2\gamma_1$ and $g_2=\varepsilon_2 g_3\gamma_2$ with $\varepsilon_i$ being $(\eta^{-1},I)$-smooth and $\gamma_i$ being $\eta^{-1}$-rational for $i=1, 2$, then  $g_1=\varepsilon_1\varepsilon_2g_3\gamma_2\gamma_1$ and by Lemma~\ref{mult} we know that $\varepsilon_1\varepsilon_2$ is $(\eta^{-O(1)},I)$-smooth and $\gamma_2\gamma_1$ is $\eta^{-O(1)}$-rational.

The first claim in part (ii) is trivial. We proceed to prove the second part. It suffices to show that if $\varepsilon\in \Poly(\mathbb{Z}\to G)$ is $(\eta^{-1},I_2)$-smooth, then it is $(\eta^{-O(1)},I_3)$-smooth. This amounts to showing that  for all $n\in I_3$ we have $d_G(\varepsilon(n),1_G)\leq \eta^{-O(1)}$ and \begin{align}\label{eq:dG}
d_G(\varepsilon(n),\varepsilon(n-1))\leq \frac{\eta^{-O(1)}}{|I_3|}.  \end{align}
Note that if we have~\eqref{eq:dG}, and if $n_0\in I_2$ is chosen arbitrarily, then by repeated application of the triangle inequality, for any $n\in I_3$, we have 
\begin{align*}
 d_G(\varepsilon(n),1_G)\leq d_G(\varepsilon(n_0),1_G)+|n-n_0|\frac{\eta^{-O(1)}}{|I_3|}\ll  \eta^{-O(1)}  
\end{align*}
by the fact that $\varepsilon$ is $(\eta^{-1},I_2)$-smooth. Therefore, it suffices to prove~\eqref{eq:dG}. 

Let $V=\log G$. Then $V$ is a vector space of dimension $D=\dim G$, so it is isomorphic to $\mathbb{R}^D$. Equip $V$ with a Euclidean norm $\|\cdot\|$. Consider the function $P(n)=\log(\varepsilon(n)\varepsilon(n-1)^{-1})$. The sequence $n \mapsto \varepsilon(n) \varepsilon(n-1)^{-1}$ is a polynomial map from $\Z$ to $G$, hence by the Baker--Campbell--Hausdorff formula, $P$ is a polynomial of degree at most $d$ from $\Z$ to $\log G\cong\mathbb{R}^{D}$.  By hypothesis, $\|P(n)\| = O(\eta^{-O(1)}/|I_2|)$ for all $n \in I_2$, hence by Lagrange interpolation (and the assumption $|I_2|\geq d+1$) we have $\|P(n)\| = O(\eta^{-O(1)}/|I_3|)$ for all $n \in I_3$, and \eqref{eq:dG} follows.
\end{proof}

With the above notation, we are ready to state the nilsequence contagion lemma.

\begin{theorem}[Nilsequence contagion lemma]\label{nil-contagion} Let $\kappa \in (0, 1)$ and $\eta \in (0, 1/2)$.  Let $G/\Gamma$ be a filtered nilmanifold of dimension $D$, degree $d$ and complexity at most $1/\eta$. 
Let $C$ be a constant that is sufficiently large depending on $\kappa, D, d$.  Let $P \geq \eta^{-C}$ and let $I$ be an interval with $|I| \geq \eta^{-C} P^{2+\kappa}$. Let $S_1, S_2$ be subsets of $[P,2P]$, and suppose that we have polynomial maps $g_n, g'_{n'} \colon \Z \to G$ for $n \in S_1, n' \in S_2$ such that
\begin{equation*}
 g_n( n' \cdot) \sim_{\frac{1}{nn'} I,\eta} g'_{n'}(n \cdot)
\end{equation*}
for $\geq \eta P^2$ pairs $(n,n') \in S_1 \times S_2$.  Then there is a polynomial map $g_*\colon \Z \to G$ such that
$$ g_n \sim_{\frac{1}{n} I,\eta^{O_{\kappa,D,d}(1)}} g_*(n \cdot)$$ 
for $\gg \eta^{O_{\kappa,D,d}(1)} P$ values of $n \in S_1$ and
$$ g'_{n'} \sim_{\frac{1}{n'} I,\eta^{O_{\kappa,D,d}(1)}} g_*(n' \cdot)$$ 
for $\gg \eta^{O_{\kappa,D,d}(1)} P$ values of $ n' \in S_2$. 
\end{theorem}

It is instructive to see why Theorem~\ref{nil-contagion} contains Theorem~\ref{thm_contagious} as a special case, so we pause to outline this implication. Let the assumptions be as in Theorem~\ref{thm_contagious}, and allow implied constants to depend on $s,\kappa$.  Apply Theorem~\ref{nil-contagion} with $I=[\frac{1}{2}\varepsilon^{-1/s}P^2,\varepsilon^{-1/s}P^2]$ and with $G/\Gamma=\mathbb{R}/\mathbb{Z}$, equipped with a filtration of degree $s$. Take $g_n(n')=\alpha_1(n)(n')^{s}, g_{n'}'(n)=\alpha_2(n')n^s$. Then for $(n,n')\in \mathcal{N}_0$ we have $g_{n}(n'\cdot)\sim_{\frac{1}{nn'}I,\eta}g'_{n'}(n\cdot)$, since by~\eqref{eq0} for $y\in \frac{1}{nn'}I$ the function $\epsilon(y)\coloneqq g_{n}(n'y)-g_{n'}'(ny)$ satisfies
$$\|\epsilon(y)-\epsilon(y-1)\|_{\R/\Z}\ll \|(\alpha_1(n)(n')^s-\alpha_2(n')n^s)\|_{\R/\Z}y^{s-1}\ll \frac{1}{|\tfrac{1}{nn'}I|}\ll \varepsilon^{1/s},$$
where we used the triangle inequality estimate $\|\beta y^s-\beta(y-1)^s\|_{\mathbb{R}/\mathbb{Z}}\ll \|\beta\|_{\mathbb{R}/\mathbb{Z}}y^{s-1}$ for $\beta\in \mathbb{R}$.
Theorem~\ref{nil-contagion} now tells us that there exists a polynomial $g_{*}$ of degree at most $s$ such that $g_n\sim_{\frac{1}{n}I,\eta^{O(1)}}g_{*}(n\cdot)$ for $n\in S_1^{*}$, where $S_1^{*}\subset S_1$ has cardinality $\gg \eta^{O(1)}P$. Hence, there exists an integer $1\leq q\ll \eta^{-O(1)}$ such that the polynomial $y\mapsto g_n(qy)-g_{*}(nqy)$ is $\eta^{-O(1)}$-smooth on $\frac{1}{qn}I$. Now, denoting by $\alpha_{*}$ the degree $s$ coefficient of $g_{*}$, the highest degree coefficient of this polynomial is $q^s\alpha_1(n)-\alpha_{*}q^sn^s$, so we in particular have  
$$\|q^s(\alpha_1(n)-\alpha_{*}n^s)\|_{\R/\Z}\ll \eta^{-O(1)}\left(\frac{|I|}{qn}\right)^{-s}\ll \eta^{-O(1)}\varepsilon/P^s$$
for $n\in S_1^{*}$, as desired. The same argument gives the analogous conclusion for $\alpha_2$. 

We now prove Theorem~\ref{nil-contagion}.  Let $\kappa, \eta, G/\Gamma$, $d$, $D$, $C$, $P$, $I$ be as in this theorem; we allow implied constants to depend on $\kappa,D,d$.  The strategy is to inductively apply Theorem~\ref{thm_contagious} to certain ``coefficients'' of the $g_n, g'_{n'}$, until all of the coefficients of $g_n, g'_{n'}$ are related to a single polynomial map $g_*$.  These coefficients will be indexed by a pair $(j,k)$ of natural numbers, and are related to a certain way to express the group $\Poly(\Z \to G)$ as a tower of abelian extensions (i.e., to describe $\Poly(\Z \to G)$ as a solvable\footnote{In fact $\Poly(\Z \to G)$ is actually a nilpotent group, but for this argument the solvable structure is more relevant.} group). 

In order to define these coefficients properly, we will need to introduce additional filtration structures on $G$.  Namely, we define additional filtrations $G_{[j,k],\bullet} = (G_{[j,k],i})_{i \geq 0}$ for $j,k \geq 0$ by
$$ G_{[j,k],i}  \coloneqq  G_{j+i+1_{i<k}}.$$
One can check that these are indeed filtrations.  Furthermore we have that $G_{[j,k],i}$ is monotone decreasing with respect to lexicographical ordering on $[j,k]$, thus
$$ G_{[j',k'],i} \leq G_{[j,k],i}$$
whenever $[j',k'] \geq [j,k]$, in the sense that either $j' > j$ or $j'=j$ and $k' > k$.  Also we have
$$ [G_{[j_1,k_1],i_1}, G_{[j_2,k_2],i_2}]  \leq G_{[j_1+j_2,k_1+k_2],i_1+i_2}$$
for any $i_1,i_2,j_1,j_2,k_1,k_2 \geq 0$.  Finally we have the identities
$$ G_\bullet =  G_{[0,0],\bullet} = G_{[0,1],\bullet}$$
and
$$ G_{[j,k],\bullet} = G_{[j+1,0],\bullet}$$
when $j+k > d$.

If we then define $\Poly(\Z \to G)_{[j,k]}$ to be the group of polynomial maps $g$ from $\Z$ to the filtered group $(G, G_{[j,k],\bullet})$, then
$\Poly(\Z \to G)_{[j,k]}$ is monotone decreasing with respect to $[j,k]$ in the sense that
$$ \Poly(\Z \to G)_{[j',k']} \leq \Poly(\Z \to G)_{[j,k]}$$
whenever $[j',k'] \geq [j,k]$, and also we have the filtration property
$$ [ \Poly(\Z \to G)_{[j_1,k_1]}, \Poly(\Z \to G)_{[j_2,k_2]}] \leq \Poly(\Z \to G)_{[j_1+j_2,k_1+k_2]}$$
for any $[j_1,k_1], [j_2,k_2]$; in particular, the quotient groups $\Poly(\Z \to G)_{[j,k]} / \Poly(\Z \to G)_{[j,k+1]}$ are all abelian.  Finally we have the identities
\begin{equation}\label{poly-1}
 \Poly(\Z \to G) = \Poly(\Z \to G)_{[0,0]} = \Poly(\Z \to G)_{[0,1]}
\end{equation}
and
\begin{equation}\label{poly-2}
 \Poly(\Z \to G)_{[j,k]} = \Poly(\Z \to G)_{[j+1,0]}
\end{equation}
when $j+k > d$.  We refer to elements of $\Poly(\Z \to G)_{[j,k]}$ as $[j,k]$-polynomial maps; informally, these are polynomial maps in which a certain number of lower order coefficients vanish, leaving only higher order terms.

An element $g$ of $\Poly(\Z \to G)_{[j,k]}$ can be extended by~\cite[Lemma 4.2]{MRTTZ} to a polynomial map from $\R$ to $G_{[j,k], \bullet}$, thus 
$$ g(n) = \exp( \sum_{\ell=0}^d X_\ell n^\ell )$$
for all $n \in \Z$ and some coefficients $X_\ell \in \log G_{[j,k],\ell}$.  By the Baker--Campbell--Hausdorff formula and induction, for every $m \in \{-1, 0, \dotsc, d\}$ we can obtain representations
$$ g(n) = g'_0 (g'_1)^n \dots (g'_m)^{n^m} \exp( \sum_{\ell=m+1}^d X_{\ell,m} n^\ell )$$
for all $n \in \Z$ and some coeffficients $g'_\ell \in G_{[j,k],\ell}$ and $X_{\ell,m} \in \log G_{[j,k],\ell}$.  Setting $m=d$ and relabeling, we conclude that we have the representation
$$ g(n) = \tilde g_{j+1} \tilde g_{j+2}^n \cdots \tilde g_{j+k}^{n^{k-1}} g_{j+k}^{n^k} \cdots g_{j+d}^{n^d}$$
for all $n \in \Z$, where $\tilde g_i \in G_i$ for $j+1 \leq i \leq j+k$ and $g_i \in G_i$ for $j+k \leq i \leq j+d$, in particular $g$ is equal to the monomial $n \mapsto g_{j+k}^{n^k}$ modulo $\Poly(\Z \to G)_{[j,k+1]}$.  We define the $[j,k]$-frequency map 
$$\log_{[j,k]}\colon \Poly(\Z \to G)_{[j,k]} \to \log G_{j+k} / \log G_{j+k+1}$$
by the formula
$$ \log_{[j,k]} g  \coloneqq  \log g_{j+k} \hbox{ mod } \log G_{j+k+1}.$$
One can verify that this is a surjective group homomorphism.  One can think of $\log_{[j,k]} g$ as the ``lowest order'' or ``most important'' coefficient of the $[j,k]$-polynomial map $g$. 

\begin{example} The following table for $d=2$ shows the form of $n\mapsto g(n)$ for $g \in \Poly(\Z \to G)_{[j,k]}$, as well as the value of $\log_{[j,k]} g$, for various choices of $[j,k]$, where for each $i$, $g_i, \tilde g_i$ denote arbitrary elements of $G_i$:

\begin{center}
\begin{tabular}{ l l l }
$[j,k]$ & $g(n)$ & $\log_{[j,k]} g$ \\
\hline
$[0,0]$ & $g_0 g_1^n g_2^{n^2}$ & $0$ \\
$[0,1]$ & $\tilde g_1 g_1^n g_2^{n^2}$ & $\log g_1 \hbox{ mod } \log G_2$ \\
$[0,2]$ & $\tilde g_1 \tilde g_2^n g_2^{n^2}$ & $\log g_2$ \\
$[0,3]$ & $\tilde g_1 \tilde g_2^n$ & $0$ \\
$[1,0]$ & $g_1 g_2^n$ & $\log g_1 \hbox{ mod } \log G_2$ \\
$[1,1]$ & $\tilde g_2 g_2^n$ & $\log g_2$ \\
$[1,2]$ & $\tilde g_2$ & $0$ \\
$[2,0]$ & $g_2$ & $\log g_2$ \\
$[2,1]$ & $1$ & $0$ \\
$[3,0]$ & $1$ & $0$
\end{tabular}
\end{center}

Observe that we have a subnormal series
\begin{align*}
&\Poly(\Z \to G)_{[0,0]} = \Poly(\Z \to G)_{[0,1]} \geq \Poly(\Z \to G)_{[0,2]} \geq \Poly(\Z \to G)_{[0,3]} \\
= &\Poly(\Z \to G)_{[1,0]} \geq \Poly(\Z \to G)_{[1,1]} \geq \Poly(\Z \to G)_{[1,2]} \\
= &\Poly(\Z \to G)_{[2,0]} \geq \Poly(\Z \to G)_{[2,1]} \\
= &\Poly(\Z \to G)_{[3,0]}
\end{align*}
that starts at $\Poly(\Z \to G) =  \Poly(\Z \to G)_{[0,0]}$ and ends at the trivial group $\{1\} = \Poly(\Z \to G)_{[3,0]}$; the maps $\log_{[j,k]}$ are essentially the quotient maps between adjacent groups in this subnormal series.  In particular, if $g \in \Poly(\Z \to G)_{[j,k]}$ and $[j',k']$ is the successor to $[j,k]$ in the lexicographical ordering, then $g$ lies in the next group $\Poly(\Z \to G)_{[j',k']}$ in the subnormal series if and only if the coefficient $\log_{[j,k]} g$ vanishes. 
\end{example}

\begin{example}  In the case that $G$ is the Heisenberg group generated by two generators $e_1,e_2$ with central commutator $[e_1,e_2]$, then we have $g(n) \Gamma = e_1^{\{P_1(n)\}} e_2^{\{P_2(n)\}} [e_1,e_2]^{\{P_{12}(n)\}} \Gamma$ for some bracket polynomials $P_1(n),P_2(n),P_{12}(n)$; for instance, if $$g(n) = e_1^{\alpha n + \kappa} e_2^{\beta n + \sigma} [e_1,e_2]^{\gamma n^2 + \delta n + \epsilon}$$ 
for various real numbers $\alpha,\beta,\gamma,\delta,\eps,\sigma,\kappa$ then we can take $P_1(n)  \coloneqq  \alpha n + \kappa$, $P_2(n)  \coloneqq  \beta n + \sigma$, and
$$ P_{12}(n)  \coloneqq  \lfloor \alpha n + \kappa \rfloor (\beta n + \sigma) + \gamma n^2 + \delta n + \epsilon.$$
The above table then specializes as follows:

\begin{center}
\begin{tabular}{ l l l l l }
$[j,k]$ & $P_1(n)$ & $P_2(n)$ & $P_{12}(n)$ & $\log_{[j,k]} g$ \\
\hline
$[0,0]$ & $\alpha n + \kappa$ & $\beta n + \sigma$ & $\lfloor \alpha n + \kappa \rfloor (\beta n + \sigma) + \gamma n^2 + \delta n + \epsilon$ & $0$ \\
$[0,1]$ & $\alpha n + \kappa$ & $\beta n + \sigma$ & $\lfloor \alpha n + \kappa \rfloor (\beta n + \sigma) + \gamma n^2 + \delta n + \epsilon$ & $(\alpha,\beta)$ \\
$[0,2]$ & $\kappa$ & $\sigma$ & $\lfloor \kappa \rfloor \sigma + \gamma n^2 + \delta n + \eps$ & $\gamma$ \\
$[0,3]$ & $\kappa$ & $\sigma$ & $\lfloor \kappa \rfloor \sigma + \delta n + \eps$ & $0$ \\
$[1,0]$ & $\kappa$ & $\sigma$ & $\lfloor \kappa \rfloor \sigma + \delta n + \eps$ & $(\kappa,\sigma)$ \\
$[1,1]$ & $0$ & $0$ & $\delta n + \eps$ & $\delta$ \\
$[1,2]$ & $0$ & $0$ & $\eps$ & $0$ \\
$[2,0]$ & $0$ & $0$ & $\eps$ & $\eps$ \\
$[2,1]$ & $0$ & $0$ & $0$ & $0$ \\
$[3,0]$ & $0$ & $0$ & $0$ & $0$
\end{tabular}
\end{center}

On each row of this table, $\log_{[j,k]} g$ informally captures the ``leading coefficients'' of the triple $(P_1(n), P_2(n), P_{12}(n))$ of bracket polynomials, and the kernel of this map produces the slightly smaller class of bracket polynomials in the row below.  If one specializes to the case $\alpha=\beta=\kappa=\sigma=0$ one recovers the classical situation of a quadratic polynomial $\gamma n^2 + \delta n + \epsilon$ and its coefficients $\gamma, \delta, \epsilon$, with the degree $k$ coefficient being extracted by the map $\log_{[d-k,k]}$.
\end{example}

For any interval $I$ of length $\geq 1$, any $\eta\in (0,1/2)$ and any $[j,k]$, we define a relation
$$ g \sim_{I, [j,k],\eta} g'$$
on polynomials $g,g'\colon \Z \to G$ if one has a factorisation
$$ g = \eps g' \gamma g_{[j,k]}$$
where $\eps$ is $\eta^{-1}$-smooth on $I$, $\gamma$ is $\eta^{-1}$-rational, and $g_{[j,k]} \in \Poly(\Z \to G)_{[j,k]}$, that is to say $g$ is equal to $\eps g' \gamma$ modulo $[j,k]$-polynomial maps.  From Lemma~\ref{mult}, we see that $\sim_{I,[j,k],\eta}$ is an equivalence relation (up to polynomial losses in $\eta$). Note that the relation $\sim_{I,[d+1,0],\eta}$ corresponds to the relation $\sim_{I,\eta}$ defined in Definition~\ref{def:nilseqequiv}.

For $i \geq 0$, define the \emph{$i$-torus} $T_i$ to be $T_i  \coloneqq  \log G_i / \log( G_{i+1} \Gamma_i )$; $T_1$ is the horizontal torus when $G_2 = [G_1,G_1]$ and $T_d$ is the vertical torus, while $T_i$ is trivial for $i=0$ or $i>d$.  Given an element $\xi_i$ of $\log G_i / \log G_{i+1}$, we define $\|\xi_i\|_{T_i}$ to be the distance to the origin of the projection of $\xi_i$ to $T_i$.

The relation $\sim_{I, [j,k+1],\eta}$ implies a relation on the $[j,k]$-frequencies of two $[j,k]$-polynomial maps $g,g'$, even if the relation is conjugated by an arbitrary additional polynomial map $g_*$.

\begin{lemma}\label{qg}  Suppose that $g_*\colon \Z \to G$ is a polynomial map and $g,g'\colon \Z \to G$ are $[j,k]$-polynomial maps, and suppose that
$$ g_* g \sim_{I,[j,k+1],\eta}  g_* g'$$
for some interval $I$ and some $\eta\in (0,1/2)$.  Then we have
$$ \| q (\log_{[j,k]}(g) - \log_{[j,k]}(g')) \|_{T_{j+k}} \ll \eta^{-O(1)} |I|^{-j-k}$$
for some integer $1 \leq q \ll \eta^{-O(1)}$.
\end{lemma}

\begin{proof}  We can assume that $|I| \geq C \eta^{-C}$ for some large constant $C$ since the claim is trivial otherwise.  We can also assume $k+j>0$ since the conclusion is trivial for $k=j=0$. 

By hypothesis, we have
$$ g_* g = \eps g_* g' \gamma \hbox{ mod } \Poly(\Z \to G)_{[j,k+1]}$$
for some $(\eta^{-1},I)$-smooth $\eps \in \Poly(\Z \to G)$ and $\eta^{-1}$-rational $\gamma \in \Poly(\Z \to G)$.  Thus
\begin{equation}\label{geq}
 g = [g_*^{-1},\eps] \eps g' \gamma \hbox{ mod } \Poly(\Z \to G)_{[j,k+1]},
\end{equation}
where $[a,b]\coloneqq aba^{-1}b^{-1}$.
We now claim inductively that for all $[0,1] \leq [j',k'] \leq [j,k]$, one has
$$ \eps, \gamma \in \Poly(\Z \to G)_{[j',k']}.$$
This is already established for $[j',k']=[0,1]$ by~\eqref{poly-1}.  Suppose it is true for some $[0,1] \leq [j',k'] < [j,k]$ (so in particular $j+k \geq 1$).  Then we have
$$ g = \eps g' \gamma \hbox{ mod } \Poly(\Z \to G)_{[j',k'+1]}$$
and then on applying $\log_{[j',k']}$ (which annihilates $g,g'$) we have
$$ 0 = \log_{[j',k']}(\eps) +  \log_{[j',k']}(\gamma).$$ 
Since $\eps$ is $\eta^{-O(1)}$-smooth on $I$, $\log \eta$ takes values in a ball of radius $O(\eta^{-O(1)})$ on $I$, so by the Lagrange interpolation theorem, the degree $d$ coefficient of $\log \eps$ has size $O(\eta^{-O(1)} |I|^{-d})$ for all $d$.  In particular, $\log_{[j',k']}(\eps)$ is of size $O(\eta^{-O(1)} |I|^{-j'-k'})$. Since $\gamma$ is $\eta^{-O(1)}$-rational, $\log_{[j',k']}(\gamma)$ is $\eta^{-O(1)}$-rational also.  Recalling that $|I| \geq C \eta^{-C}$, the only way these two statements are consistent are if $\log_{[j',k']}(\eps) = \log_{[j',k']}(\gamma) = 0$.  This implies that $\eps, \gamma \in \Poly(\Z \to G)_{[j',k'+1]}$.  Inducting (using~\eqref{poly-2}) we obtain the claim.  In particular, $\eps, \gamma \in \Poly(\Z \to G)_{[j,k]}$.  Applying $\log_{[j,k]}$ to~\eqref{geq} we now get
$$ \log_{[j,k]} g = \log_{[j,k]} \eps + \log_{[j,k]} g' + \log_{[j,k]} \gamma.$$
Since $\log_{[j,k]}(\eps) = O(\eta^{-O(1)} |I|^{-j-k})$ and $\log_{[j,k]}(\gamma)$ is $\eta^{-O(1)}$-rational, we obtain the claim.\end{proof}

Now we can prove Theorem~\ref{nil-contagion}.

\begin{proof}[Proof of Theorem~\ref{nil-contagion}] Initialize $\mathcal{N}=S_1$, $\mathcal{N}'=S_2$. Then we have 
\begin{align}\label{gp-gpp2}
g_n(n'\cdot)\sim_{\frac{1}{nn'}I,\eta^{O(1)}}g'_{n'}(n\cdot)    
\end{align}
for $\gg \eta^{O(1)}P^2$ pairs $(n,n')\in \mathcal{N}\times \mathcal{N}'$ (in fact, the $\eta^{O(1)}$ term above could be replaced with $\eta$).

We claim inductively that for $[0,1] \leq [j,k] \leq [d+1,0]$, we can refine $\mathcal{N}, \mathcal{N}'$ by a factor of $\eta^{O(1)}$ such that the relation~\eqref{gp-gpp2} still holds for $\gg \eta^{O(1)} P^2$ pairs, and there is a polynomial map $g_*\colon \Z \to G$ such that
\begin{align*} g_n &\sim_{\frac{1}{n} I,[j,k],\eta^{O(1)}} g_*(n \cdot),\\
 g'_{n'} &\sim_{\frac{1}{n'} I,[j,k],\eta^{O(1)}} g_*(n' \cdot)
\end{align*}
for all $n \in \mathcal{N}$, $n' \in \mathcal{N}'$; setting $[j,k] = [d+1,0]$ will give the claim.

For $[j,k] = [0,1]$ this follows by setting $g_*$ to be the identity and using~\eqref{poly-1}. By~\eqref{poly-2} and induction\footnote{Note that the induction proceeds in a ``staircase'' fashion: from case $[0,1]$ we can eventually get to case $[0,d+1]$, which by~\eqref{poly-2} is the same as case $[1,0]$, and then repeating the same procedure we eventually get to case $[1,d+1]$, which by~\eqref{poly-2} is the same as case $[2,0]$, and so on.}, it thus suffices to show that if the claim holds for some $[0,1] \leq [j,k] < [d+1,0]$ then it also holds for $[j,k+1]$. By hypothesis we have
\begin{align*} g_n &\sim_{\frac{1}{n} I, [j,k+1],\eta^{O(1)}} g_*(n \cdot) \tilde g_n\\ 
g'_{n'} &\sim_{\frac{1}{n'} I, [j,k+1],\eta^{O(1)}} g_*(n' \cdot) \tilde g'_{n'}
\end{align*}
for some $[j,k]$-polynomials $\tilde g_n, \tilde g'_{n'}$.  By~\eqref{gp-gpp2} we then have
$$ g_*(nn' \cdot) \tilde g_n( n' \cdot) \sim_{\frac{1}{nn'} I, [j,k+1],\eta^{O(1)}} g_*(nn' \cdot) \tilde g'_{n'}(n \cdot)$$
for $\gg \eta^{O(1)} P^2$ pairs $n \in \mathcal{N}, n' \in \mathcal{N}'$.
Applying Lemma~\ref{qg}, we conclude that
$$ \| q ( \log_{[j,k]}( \tilde g_{n}(n' \cdot) ) - \log_{[j,k]}(\tilde g'_{n'}(n \cdot) )  \|_{T_{j+k}} \ll \eta^{-O(1)} (|I|/P^2)^{-j-k}$$
for some integer $1 \leq q \ll \eta^{-O(1)}$, which we can pigeonhole to be independent of $n,n'$.  From Taylor expansion we have
$$ \log_{[j,k]} g( n \cdot ) = n^{j+k} \log_{[j,k]} g$$
and hence
$$ \| q ( (n')^{j+k} \log_{[j,k]}( \tilde g_{n} ) - n^{j+k} \log_{[j,k]}(\tilde g'_{n'} )  \|_{T_{j+k}} \ll \eta^{-O(1)} (|I|/P^2)^{-j-k}.$$
For $k+j>0$, we may apply Theorem~\ref{thm_contagious} once for each coordinate of the torus $T_{j+k}$, and modify $q$ as necessary, to refine $\mathcal{N}, \mathcal{N}'$ so that
$$ \| q ( \log_{[j,k]}( \tilde g_{n} ) - \beta n^{j+k} ) \|_{T_{j+k}} \ll \eta^{-O(1)} (|I|/P)^{-j-k}$$
and
$$ \| q ( \log_{[j,k]}( \tilde g'_{n'} ) - \beta (n')^{j+k} ) \|_{T_{j+k}} \ll \eta^{-O(1)} (|I|/P)^{-j-k}$$
for all $n \in \mathcal{N}, n' \in \mathcal{N}'$ and some $\beta \in \log G_{j+k}/\log G_{j+k+1}$, while still keeping~\eqref{gp-gpp2} for $\gg \eta^{O(1)} P^2$ pairs.  In the $k=0$ case we can trivially achieve this just by setting $\beta=0$.  From this we may factor
$$ g_n \sim_{\frac{1}{n} I, [j,k+1],\eta^{O(1)}} g_{*}(n \cdot) g_{[j,k]}^{(n\cdot)^{j+k}} $$
and
$$ g'_{n'} \sim_{\frac{1}{n'} I, [j,k+1],\eta^{O(1)}} g_{*}(n' \cdot) g_{[j,k]}^{(n'\cdot)^{j+k}},$$
where $g_{[j,k]}$ is any element of $G_{j+k}$ with $\log_{[j,k]} g_{[j,k]} = \beta$. This allows us to close the induction, replacing $g_*$ with $g_* g_{[j,k]}^{(\cdot)^{j+k}}$.  This completes the proof of Theorem~\ref{nil-contagion}.
\end{proof}

\section{Proof of type II case}\label{type-ii}

In this section we establish the type $II$ cases (ii), (iii) of Theorem~\ref{inverse}, using a version of Walsh's contagion argument~\cite{walsh}.

\begin{figure}[H]
  \centering
\begin{tikzpicture}[node distance=1cm and 1cm]
  \node[draw, rectangle, align=center,fill=yellow!20] (invnew) {Type $II$ inv. theorem \\ \Cref{inverse}(ii), (iii)};
  \node[draw, rectangle, above left=of invnew, align=center] (scaledown) {Scaling down \\ \Cref{prop:scaledown}};
  \node[draw, rectangle, above =of invnew, align=center] (scaleup) {Scaling up \\ \Cref{prop:iterate}};
  \node[draw, rectangle, above right=of invnew, align=center] (conclude) {Scaling conclusion \\ \Cref{prop:conclusion}};
  \node[draw, rectangle, above =of scaleup, align=center,fill=gray!20] (nil) {Nilsequence contagion \\ \Cref{nil-contagion}};
  \node[draw, rectangle, above =of conclude, align=center] (log) {Approx. dilation invar. \cite{matomaki-shao}\\ \Cref{prop:logarithm} ; \\ Abstract type II est. \cite{MSTT-all}\\ \Cref{prop:Furstenberg-Weiss} };
  \node[draw, rectangle, above =of scaledown, align=center] (sieve) {Nilseq. large sieve \cite{MSTT-all} \\ \Cref{large-sieve} };
  \draw[-{Latex[length=4mm, width=2mm]}] (scaledown) -- (invnew);
  \draw[-{Latex[length=4mm, width=2mm]}] (scaleup) -- (invnew);
  \draw[-{Latex[length=4mm, width=2mm]}] (conclude) -- (invnew);
  \draw[-{Latex[length=4mm, width=2mm]}] (log) -- (conclude);
  \draw[-{Latex[length=4mm, width=2mm]}] (nil) -- (scaleup);
  \draw[-{Latex[length=4mm, width=2mm]}] (sieve) -- (scaledown);
\end{tikzpicture}
  \caption{Proof of the type II inverse theorem, \Cref{inverse}(ii), (iii). For the proof of \Cref{nil-contagion}, see \Cref{fig-contagion}.}
  \label{fig-inv}
\end{figure}
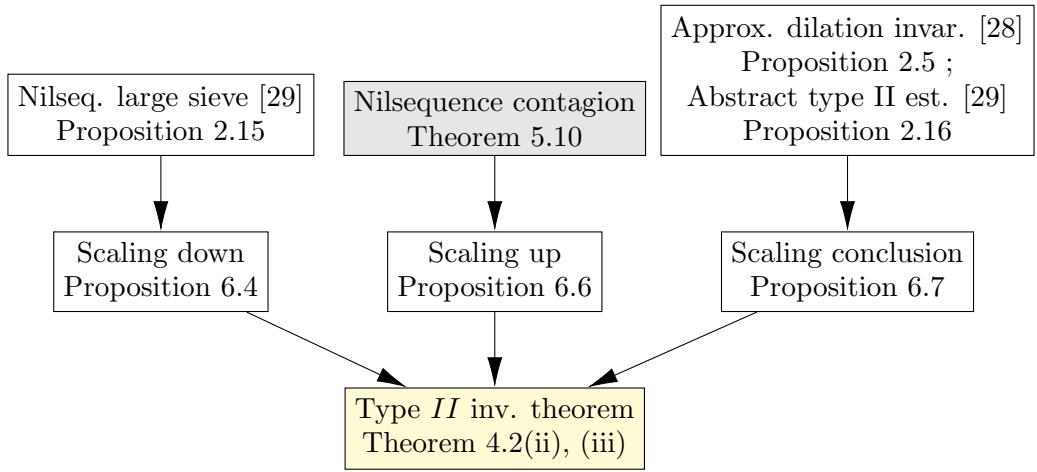

Recall the definition of the relation $\sim_{I,\eta}$ from Definition~\ref{def:nilseqequiv}. The following two definitions will also be important for expressing the main argument in this section. 

\begin{definition}[Configurations of points and polynomial sequences]
Let $I$ be an interval and $\sigma>0$. We say that a collection $\mathcal{J}\subset I\times \Poly(\mathbb{Z}\to G)$ is a \emph{$(\sigma,H)$-configuration in $I$}, if $\mathcal{J}$ is of the form $\{(x,g_x)\colon x\in \mathcal{X}\}$, where $|\mathcal{X}|\geq \sigma|I|/H$, $g_x\in \Poly(\Z \to G)$, and the points of $\mathcal{X}\subset I$ are $H$-separated (that is, $x_1,x_2\in \mathcal{X}$ and $x_1\neq x_2$ imply $|x_1-x_2|\geq H$).   
\end{definition}

Informally, a $(\sigma,H)$-configuration abstracts (and discretizes) the notion of having correlations with nilsequences $F(g_x(n) \Gamma)$ on a large family of intervals $(x,x+H]$ in $I$; compare with the definition of an $(X,H)$-family of intervals in~\cite[Section 2]{mrt-fourier} for linear phases.

\begin{definition}[Upwards scaling]\label{def:upscale}
Let $C\geq 1$, $\eta, \sigma\in (0,1/2)$, and $A,H,Y\geq 1$. Let $\mathcal{J}_0$ be a $(\sigma,H)$-configuration in $[Y,CY]$, and let $\mathcal{J}_1$ be a $(\sigma,AH)$-configuration in $[AY,CAY]$.

We write
\begin{align*}
\mathcal{J}_0\longrightarrow_{\eta,C}\mathcal{J}_1   
\end{align*}
if for each $(y,g_y)\in \mathcal{J}_1$ we have
$$
\sum_{A < a \leq CA} \sum_{(x,g_x) \in \mathcal{J}_0} \mathrm{I}_{a,C,AH,\sigma}(y,x) \geq \eta A$$
where $\mathrm{I}_{a,C,AH,\sigma}(y,x)$ is the indicator function
\begin{equation}\label{I-notation}
  \mathrm{I}_{a,C,AH,\sigma}(y,x) \coloneqq 1_{|y-ax| \leq CAH} 1_{g_y(a \cdot) \sim_{\frac{1}{a} (y,y+AH],\sigma} g_x}.
\end{equation}
\end{definition}

Note that the relation $\mathcal{J}_0\longrightarrow_{\eta,C}\mathcal{J}_1$ gets weaker if $\eta$ is decreased or $C$ is increased.  Informally, the relation $\mathcal{J}_0\longrightarrow_{\eta,C}\mathcal{J}_1$ indicates that the nilsequences associated to the intervals $(y,y+AH]$ in $\mathcal{J}_1$ are essentially rescaled versions of the nilsequences associated to the intervals $(x,x+H]$ in $\mathcal{J}_0$.  There are multiple ways to connect these intervals by rescaling; the idea of the contagion argument is to use many existing rescaling relationships at one pair of scales to create additional scaling relationships at a larger pair of scales.

The next lemma shows that the notion of upwards scaling is ``transitive''. 

\begin{lemma}[Transitivity of upwards scaling]\label{le:transitive} Let $X\geq H\geq A_2\geq A_1\geq 3$. Let $\sigma\in \left(0,\frac{1}{2\log(A_1A_2)}\right)$ and $C \in [1,\sigma^{-1}]$. Let $G/\Gamma$ be a filtered nilmanifold of degree and dimension $O(1)$ and complexity $\leq \sigma^{-1}$. Let $\mathcal{J}_0$ be a $(\sigma,H)$-configuration in $[Y,CY]$, let $\mathcal{J}_1$ be a $(\sigma,A_1H)$-configuration in $[A_1Y,CA_1Y]$, and let $\mathcal{J}_2$ be a $(\sigma,A_1A_2H)$-configuration in $[A_1A_2Y,CA_1A_2Y]$. Suppose that 
\begin{align*}
\mathcal{J}_0\longrightarrow_{\sigma,C} \mathcal{J}_1\quad \textnormal{and}\quad\mathcal{J}_1\longrightarrow_{\sigma,C} \mathcal{J}_2. 
\end{align*}
Then we have
\begin{align}\label{eq:transitiverelation}
\mathcal{J}_0\longrightarrow_{\sigma^{O(1)},2C^2} \mathcal{J}_2.
\end{align}
\end{lemma}

\begin{proof}  
From \Cref{def:upscale}, we see that for every $(z,g_z)\in \mathcal{J}_2$ we have
\begin{align*}
\sum_{A_2<a_2\leq CA_2}\sum_{(y,g_y)\in \mathcal{J}_1} \mathrm{I}_{a_2,C,A_1A_2 H,\sigma}(z,y) \geq \sigma A_2
\end{align*}
and for every $(y,g_y)\in \mathcal{J}_1$ we have
\begin{align*}
\sum_{A_1<a_1\leq CA_1}\sum_{(x,g_x)\in \mathcal{J}_0} \mathrm{I}_{a_1,C,A_1 H,\sigma}(y,x) \geq \sigma A_1.
\end{align*}
Combining the previous two estimates, for every $(z,g_z)\in \mathcal{J}_2$ we obtain
\begin{equation}\label{eq:zsum}
  \sum_{\substack{A_1<a_1\leq CA_1\\ A_2<a_2\leq C A_2}}\sum_{(y,g_y)\in \mathcal{J}_1}\mathrm{I}_{a_2,C,A_1A_2 H,\sigma}(z,y) \sum_{(x,g_x)\in \mathcal{J}_0}\mathrm{I}_{a_1,C,A_1 H,\sigma}(y,x) \geq \sigma^2A_1A_2.
\end{equation}

Note that for every non-zero term in~\eqref{eq:zsum} we have 
$|z-a_2y|\leq CA_1A_2H$ and $|y-a_1x|\leq CA_1H$, so we must have $|z-a_1a_2x|\leq 2C^2A_1A_2H$. Applying Lemma~\ref{le:equivalence}(ii), this implies that for every non-zero term in~\eqref{eq:zsum} we have 
$$g_z(a_1a_2\cdot)\sim_{\frac{1}{a_1a_2}(z,z+A_1A_2H],\sigma}g_y(a_1\cdot) \quad\textnormal{ and  }\quad g_y(a_1\cdot)\sim_{\frac{1}{a_1a_2}(z,z+A_1A_2H],\sigma^{O(1)}}g_x,$$
and by the transitivity of the relation $\sim_{I,\eta}$ (Lemma~\ref{le:equivalence}(i)), these give
$$g_z(a_1a_2\cdot)\sim _{\frac{1}{a_1a_2}(z,z+A_1A_2H],\sigma^{O(1)}}g_x.$$
In other words,
$$ \mathrm{I}_{a_1 a_2, 2C^2, A_1 A_2 H, \sigma^{O(1)}}(z,x) = 1.$$
Hence, we obtain 
$$
\sum_{\substack{A_1<a_1\leq CA_1\\ A_2<a_2\leq C A_2}}\sum_{(x,g_x)\in \mathcal{J}_0} \mathrm{I}_{a_1 a_2, 2C^2, A_1 A_2 H, \sigma^{O(1)}}(z,x) \sum_{\substack{(y,g_y)\in \mathcal{J}_1\\|y-a_1x|\leq CA_1H}} 1 \geq \sigma^2A_1A_2.  
$$
Since the sum over $y$ is bounded by $\ll C$, we may make the substitution $a=a_1 a_2$ and conclude that for every $(z,g_z)\in \mathcal{J}_2$  we have
\begin{align}\label{eq:d_2sum}
 \sum_{A_1A_2<a\leq C^2A_1A_2}\, \sum_{(x,g_x)\in \mathcal{J}_0} d_2(a)\mathrm{I}_{a, 2C^2, A_1 A_2 H, \sigma^{O(1)}}(z,x)
 \gg \frac{\sigma^2}{C}A_1A_2. 
\end{align}

Applying to~\eqref{eq:d_2sum} the Cauchy--Schwarz inequality and the standard divisor sum bound $$\sum_{a\leq X}d_2(a)^2\ll X\log^3 X,$$
for all $(z,g_z)\in \mathcal{J}_2$ we obtain
\begin{align*}
 \sum_{A_1A_2<a\leq C^2A_1A_2}\left(\sum_{(x,g_x)\in \mathcal{J}_0}\mathrm{I}_{a, 2C^2, A_1 A_2 H, \sigma^{O(1)}}(z,x)\right)^2
 &\gg \frac{\sigma^4}{C^4}\frac{A_1A_2}{\log^3(C^2A_1A_2)}\\
 &\gg  \frac{\sigma^{8}}{C^4}A_1A_2.  
\end{align*}
Since the sum inside the square has at most $O(C^2)$ non-zero terms, we conclude that
$$
  \sum_{A_1A_2<a\leq C^2A_1A_2}\,\sum_{(x,g_x)\in \mathcal{J}_0}\mathrm{I}_{a_1 a_2, 2C^2, A_1 A_2 H, \sigma^{O(1)}}(z,x) \gg  \frac{\sigma^{8}}{C^6}A_1A_2.   
$$
The claim~\eqref{eq:transitiverelation} now follows.
\end{proof}

The following proposition importantly tells us that either Theorem~\ref{inverse} holds, or the map $x\mapsto g_x$ can be extended to a smaller scale. 

\begin{proposition}[Scaling down]\label{prop:scaledown}
Let $d,D\geq 1$, $\varepsilon>0$, and let $X \geq H \geq X^{\varepsilon}\geq 3$ and $\delta \in (0, \frac{1}{\log X})$. Let $G/\Gamma$ be a filtered nilmanifold of degree at most $d$, dimension at most $D$, and complexity at most $1/\delta$.  Let $F \colon G/\Gamma \to \C$ be Lipschitz of norm at most $1/\delta$ and mean zero.  Let $f \colon \N \to \C$ be an arithmetic function such that $f$ is a $(\delta,A_{II}^-, A_{II}^+)$ type $II$ sum for some $A_{II}^+ \geq A_{II}^- \geq X^\eps$. Suppose that
\begin{equation}\label{invo2}
\left| \sum_{x < n \leq x+H} f(n) F(g_x(n) \Gamma) \right|^* \geq \delta H
\end{equation}
for all $x$ in a subset $E$ of $[X,2X]$ of measure at least $\delta X$ where, for each $x \in E$, $g_x \colon \Z \to G$ is a polynomial map.  Then at least one of the following holds. 
\begin{itemize}
\item[(i)] We have
\begin{equation*}
 H \ll_{d,D,\varepsilon} \delta^{-O_{d,D,\eps}(1)} A_{II}^+.
\end{equation*}

\item[(ii)] There exists a non-trivial horizontal character $\eta \colon G \to \R$ having Lipschitz norm $O_{d,D,\eps}( \delta^{-O_{d,D,\eps}(1)})$ such that
\begin{equation*}
 \| \eta \circ g_x \|_{C^\infty(x, x+H]} \ll_{d,D,\eps} \delta^{-O_{d,D,\eps}(1)}
\end{equation*}
for all $x$ in a subset of $E$ of measure $\gg_{d,D,\eps} \delta^{O_{d,D,\eps}(1)} X$.

\item[(iii)] 
  There exists $A\in [A_{II}^{-},A_{II}^{+}]$, a $(\delta^{O_{d,D,\varepsilon}(1)},H)$-configuration $\mathcal{J}_0$ in $[X,2X]$ and a $(\delta^{O_{d,D,\varepsilon}(1)},H/A)$-configuration $\mathcal{J}_{-1}$ in $[X/(2A),2X/A]$, such that $\mathcal{J}_0=\{(x,g_x)\colon x\in E'\}$ for some set $E'\subset E$ and 
 \begin{align}\label{eq:relation1}
\mathcal{J}_{-1}\longrightarrow_{\delta^{O_{d,D,\eps}(1)},O_{d,D,\varepsilon}(1)} \mathcal{J}_0.
 \end{align}
 \end{itemize}
\end{proposition}

\begin{remark}
\label{rem:charRZ}
   Note that in the case of Theorem~\ref{inverse}(iii) (where $F(g_x(n))=e(P_x(n))$ for some polynomial $P_x\colon \mathbb{Z}\to \mathbb{R}$), if conclusion (ii) of Proposition~\ref{prop:scaledown} holds, then $\eta(u)=qu$ for some integer $1\leq |q|\ll \delta^{-O_{d,\varepsilon}(1)}$, so 
   \begin{align*}
     \|qP_x\mod 1\|_{C^{\infty}(x,x+H]}\ll_{d,\varepsilon} \delta^{-O_{d,\varepsilon}(1)}.  
   \end{align*}
   If $\alpha_{x,j}$ denotes the degree $j$ coefficient of $P_x(x+\cdot)$, this implies that
   \begin{align*}
    \|q\alpha_{x,j}\|_{\R/\Z}\ll \delta^{-O_{d,\varepsilon}(1)}/H^j   
   \end{align*}
   for all $1\leq j\leq d$. This in turn implies that
   \begin{align*}
\| e(P_x(\cdot))\|_{\TV( (x, x+H] \cap \Z; q)} \ll_{d,\eps} \delta^{-O_{d,\eps}(1)}. 
   \end{align*}
   Hence, in the proof of Theorem~\ref{inverse}(ii)--(iii) we may assume that conclusion (iii) above holds. 
\end{remark}

\begin{proof}
We let all implied constants to depend on $d,D,\eps$. As in Remark~\ref{rmk:measurability}, we may assume that $x\mapsto g_x$ is constant on intervals of the form $[m,m+1)$ with $m\in \mathbb{N}$. 

We induct on the dimension $D$ of $G/\Gamma$. In view of Proposition~\ref{central} and Lemma~\ref{le:shiu}, we may assume that $F$ oscillates with a central frequency $\xi \colon Z(G) \rightarrow \R$ of Lipschitz norm at most $\delta^{-O(1)}$. If the center $Z(G)$ has dimension larger than $1$, or $\xi$ vanishes, then the conclusion follows from induction hypothesis applied to $G/\ker\xi$ (via Lemma~\ref{quotient-normal}). Henceforth we assume that $G$ has one-dimensional center and that $\xi$ is non-zero. 

Let $x\in E$. By~\eqref{invo2},  Lemma~\ref{basic-prop}(i) and the assumption that $f=\alpha*\beta$ is a $(\delta, A_{II}^{-}, A_{II}^{+})$ type $II$ sum, we have
 \begin{align*}
\sum_{A_{II}^{-}\leq a\leq A_{II}^{+}}|\alpha(a)|\left|\sum_{x/a< b\leq (x+H)/a}\beta(b)F(g_x(ab)\Gamma)\right|^{*}\geq \delta H. 
\end{align*}
By the pigeonhole principle and the assumption $\delta\leq 1/\log X$, we conclude that there is some $A\in [A_{II}^{-},A_{II}^{+}]$ such that
\begin{align}\label{eq:alphabeta}
\sum_{A<a\leq 2A}|\alpha(a)|\left|\sum_{x/a< b\leq x/a+H/A}\beta(b)F(g_x(ab)\Gamma)\right|^{*}\gg \delta^{2} H.    \end{align}

 Now, let $\mathcal{B}$ be the set of $x\in [X,2X]$ for which we have
\begin{align*}
\sum_{A<a\leq 2A}|\alpha(a)|1_{|\alpha(a)|\geq \delta^{-10}}\left|\sum_{x/a<b\leq x/a+H/A}\beta(b)F(g_x(ab)\Gamma)\right|^{*}\geq  \delta^{3} H.  
\end{align*}
Since $F$ is Lipschitz of norm $\leq \delta^{-1}$, $|F|$ is pointwise bounded by $\delta^{-1}$. Then, by Markov's inequality, the Cauchy--Schwarz inequality, and the assumed $L^2$ bounds on $\alpha$ and $\beta$,  we have 
\begin{align*}
 \delta^{3}H\meas(\mathcal{B})&\leq \delta^{-1} \int_{X}^{2X}\, \sum_{A<a\leq 2A}|\alpha(a)|1_{|\alpha(a)|\geq \delta^{-10}}\sum_{x/a<b\leq x/a+H/A}|\beta(b)|\, dx\\
 &\ll \delta^{-1} H\sum_{A<a\leq 2A}|\alpha(a)|1_{|\alpha(a)|\geq \delta^{-10}} \sum_{b\leq (2X+H)/A}|\beta(b)|\\
 &\ll  \delta^{9} H\sum_{A<a\leq 2A}|\alpha(a)|^2\sum_{b\leq (2X+H)/A}|\beta(b)|\\
 &\ll \delta^{9-3/2}HX,
\end{align*}
so $\meas(\mathcal{B})\ll \delta^{9-3/2-3}X\ll \delta^4 X$. 

From now on, let $x\in E^*\coloneqq E\setminus \mathcal{B}$; the set $E^{*}$ has measure $\gg \delta X$. From~\eqref{eq:alphabeta} we conclude that for $x\in E^*$ we have 
\begin{align*}
\sum_{A<a\leq 2A}|\alpha(a)|1_{|\alpha(a)|\leq \delta^{-10}}\left|\sum_{x/a< b\leq x/a+H/A}\beta(b)F(g_x(ab)\Gamma)\right|^{*}\gg \delta^{2} H,   \end{align*}
so
\begin{align}\label{eq:invo2.5}
\sum_{A<a\leq 2A}\left|\sum_{x/a< b\leq x/a+H/A}\beta(b)F(g_x(ab)\Gamma)\right|^{*}\gg \delta^{12} H.  \end{align}

Let 
\begin{align*}
S(x,a)\coloneqq \left|\sum_{x/a< b\leq x/a+H/A}\beta(b)F(g_x(ab)\Gamma)\right|^{*},    
\end{align*}
and let $\mathcal{B}^{*}$ be the set of $x\in E^{*}$ for which
\begin{align*}
\sum_{A<a\leq 2A}S(x,a)1_{S(x,a)\geq \delta^{-20}\frac{H}{A}}\geq \delta^{13}H.    
\end{align*}
Then, by the Cauchy--Schwarz inequality and the pointwise bound on $|F|$, we have
\begin{align*}
\delta^{13}H\meas(\mathcal{B}^{*})&\leq  \delta^{20}\frac{A}{H}\int_{X}^{2X}\sum_{A<a\leq 2A}S(x,a)^2\, dx\\
&\ll  \delta^{18}\int_{X}^{2X}\sum_{A<a\leq 2A} \sum_{x/a<b\leq x/a+H/A}|\beta(b)|^2\, dx.
\end{align*}
Exchanging the order of integration and summation and using the assumed $L^2$ bound on $\beta$, we see that this is $\ll \delta^{17}HX$. Hence $\meas(\mathcal{B}^{*})\ll \delta^4 X$. 

From now on, let $x\in E^{**}\coloneqq E^{*}\setminus \mathcal{B}^{*}$; the set $E^{**}$ has measure $\gg \delta X$. From~\eqref{eq:invo2.5} we conclude that for $x\in E^{**}$ we have
\begin{align*}
\sum_{A<a\leq 2A}S(x,a)1_{S(x,a)\leq \delta^{-20}\frac{H}{A}}\gg \delta^{12}H.    
\end{align*}  
By Markov's inequality, we deduce that there is a set $\mathcal{S}\subset ((X,2X]\times (A,2A])\cap \mathbb{N}^2$ with $|\mathcal{S}|\gg \delta^{33}AX$ such that for all $(x,a)\in \mathcal{S}$ we have
\begin{align}\label{eq:nilseqlargesieve}
S(x,a) = \left|\sum_{x/a< b\leq x/a+H/A}\beta(b)F(g_x(ab)\Gamma)\right|^{*}\gg \delta^{32} \frac{H}{A}.
\end{align}

Let us say that $(x,a)\in \mathcal{S}$ is \emph{good} if $x\in E^{**}$ and  
\begin{align}\label{eq:good}
\sum_{x/a<b\leq x/a+\delta^{100}H/A}|\beta(b)|\leq \delta^{50}\frac{H}{A}\quad \textnormal{and}\quad \sum_{x/a<b\leq x/a+2H/A}|\beta(b)|^2\leq \delta^{-50}\frac{H}{A}.    
\end{align}
The number of $(x,a)\in \mathcal{S}$ that are not good is by Markov's inequality
\begin{align*}
 &\ll \delta^{-50}\frac{A}{H}\sum_{A<a\leq 2A}\sum_{X<x\leq 2X}\sum_{x/a<b\leq x/a+\delta^{100}H/A}|\beta(b)|\\
 &\quad +\delta^{50}\frac{A}{H}\sum_{A<a\leq 2A}\sum_{X< x\leq 2X}\sum_{x/a<b\leq x/a+2H/A}|\beta(b)|^2\\
 &\ll \delta^{100-50}A^2\sum_{b\leq (2X+H)/A}|\beta(b)|+\delta^{50}A^2\sum_{b\leq (2X+2H)/A}|\beta(b)|^2\ll \delta^{49}AX.
\end{align*}
Thus, there are $\gg \delta^{33}AX$ good pairs $(x,a)\in \mathcal{S}$.

We need to perform some discretization in the possible summation lengths. To this end, for 
 $z>0$, let 
\begin{align*}
Q(z)\coloneqq \min\left([z,\infty)\cap \left(\delta^{100}\frac{H}{A}\mathbb{Z}\right)\right).    
\end{align*}
Now, by~\eqref{eq:nilseqlargesieve} and~\eqref{eq:good}, for any good $(x,a)\in \mathcal{S}$ we have
\begin{align}\label{eq:nilseqlargesieve3}
\left|\sum_{Q(x/a)< b\leq Q(x/a)+H/A}\beta(b)F(g_x(ab)\Gamma)\right|^{*}\gg \delta^{32} \frac{H}{A}.
\end{align}

Let $C_1$ be a large constant, and let $C_2$ be large in terms of $C_1$. We may assume that $H/A\geq \delta^{-C_2}$, since otherwise conclusion (i) of the proposition holds. Since for any good $(x,a)\in \mathcal{S}$ we have~\eqref{eq:nilseqlargesieve3}, we may apply the nilsequence large sieve (Proposition~\ref{large-sieve}) to~\eqref{eq:nilseqlargesieve3} to deduce that there exists a collection $(\widetilde{g}_{Q(x/a),j})_{1\leq j\leq J}$ of elements of $\Poly(\mathbb{Z}\to G)$ such that $J\ll \delta^{-O(1)}$ and such that at least one of the following holds for any good $(x,a)\in \mathcal{S}$:
\begin{enumerate}
    \item[(a)] There is a non-trivial horizontal character $\eta_{x,a}\colon G\to \mathbb{R}$ of Lipschitz norm $\leq \delta^{-C_1}$ such that
    \begin{align*}
     \|\eta_{x,a}\circ g_x(a\cdot )\mod 1\|_{C^{\infty}{(Q(x/a),Q(x/a)+H/A]}}\leq \delta^{-C_1}.   
    \end{align*}
    \item[(b)] For some $1\leq j\leq J$ (depending on $(x,a)$) we have 
      \begin{align*}
 g_{x}(a\cdot)\sim_{(Q(x/a),Q(x/a)+H/A],\delta^{C_1}} \widetilde{g}_{Q(x/a),j}. \end{align*}  
\end{enumerate}

We now split into two cases.

\textbf{Case 1: (a) holds for at least half of the good pairs in $\mathcal{S}$.}

By~\eqref{eq:smoothness}, whenever $(x,a)$ has property (a) we have
\begin{align}\label{eq:etaCinfty}
\|\eta_{x,a}\circ g_x(a\cdot)\mod 1\|_{C^{\infty}(x/a,(x+H)/a]}\ll \delta^{-C_1}.  \end{align}
Now, by the pigeonhole principle there is some non-trivial horizontal character $\eta'\colon G\to \mathbb{R}$ of Lipschitz norm $\leq \delta^{-C_1}$ and some set of integers in $(X,2X]$ of size $\gg \delta^{C_1+O(1)}X$ such that for each such $x$ there exist $\gg \delta^{C_1+O(1)}A$ choices of $a\in (A,2A]\cap \mathbb{N}$ for which~\eqref{eq:etaCinfty} holds with $\eta_{x,a}=\eta'$.
By Corollary~\ref{smooth-dilate} (and the assumption $H/A\geq \delta^{-C_2}$), we deduce that for some integer $1\leq |q|\ll \delta^{-O_{C_1}(1)}$ the horizontal character $\eta''=q\eta'$ satisfies
\begin{align*}
\|\eta''\circ g_x\mod 1\|_{C^{\infty}(x,x+H]}\ll \delta^{-O_{C_1}(1)}
\end{align*} 
for $\gg \delta^{C_1+O(1)}X$ integers $x\in (X,2X]$. Conclusion (ii) follows, since $x\mapsto g_x$ is constant on intervals of the form $[m,m+1)$ with $m\in \mathbb{N}$.

\textbf{Case 2: (b) holds for at least half of the good pairs in $\mathcal{S}$.}

By Lemma~\ref{le:equivalence}(ii), for any pair $(x,a)$ satisfying (b),
for some $1\leq j\leq J$ we have 
\begin{align*}
 g_{x}(a\cdot)\sim_{\frac{1}{a}(x,x+H],\delta^{O(C_1)}} \widetilde{g}_{Q(x/a),j}.
 \end{align*}
By the pigeonhole principle, there exists some $1\leq j_0\leq J$ such that for $\gg \delta^{C_1+O(1)}AX$ pairs $(x,a)\in \mathcal{S}$ we have 
\begin{align}\label{eq:qxa2}
 g_{x}(a\cdot)\sim_{\frac{1}{a}(x,x+H],\delta^{O(C_1)}} \widetilde{g}_{Q(x/a),j_0}.
 \end{align}

For $t\in [0,1)$, let 
$$\mathcal{Y}_t\coloneqq \left(\frac{H}{A}\mathbb{Z}+t\frac{H}{A}\right)\cap  \left[\frac{X}{2A},\frac{2X}{A}\right].$$
Then $\mathcal{Y}_t$ is $H/A$-separated and has size $\gg \frac{X}{H}$. 
By the pigeonhole principle, there is some $t_0\in [0,1)$ such that the set 
\begin{align*}
\{(x,a)\in \mathcal{S}\colon (x,a)\textnormal{ satisfies~\eqref{eq:qxa2} and}\,\, Q(x/a)=Q(y) \textnormal{ for some } y\in \mathcal{Y}_{t_0}\}   
\end{align*} 
has size $\gg \delta^{C_1+O(1)}AX$. By the pigeonhole principle again, this implies that there exists an $H$-separated set $E'\subset (X,2X]\cap \mathbb{N}$ of size $\gg \delta^{C_1+O(1)}X/H$ such that for every $x\in E'$ there exist $\gg  \delta^{C_1+O(1)}A$ choices of $a\in (A,2A]\cap \mathbb{N}$  for which $(x,a)$ satisfies~\eqref{eq:qxa2} and $Q(x/a)=Q(y)$ for some unique $y\in \mathcal{Y}_{t_0}$. 

 For $y\in \mathcal{Y}_{t_0}$, define $\widetilde{g}_y\coloneqq \widetilde{g}_{Q(y),j_0}$. Then for every $x\in E'$ there exist $\gg \delta^{C_1+O(1)}A$ choices of $a\in (A,2A]\cap \mathbb{N}$ such that for some $y\in \mathcal{Y}_{t_0}$ we have $Q(x/a)=Q(y)$ (so that in particular $|y-x/a|\leq \delta^{100}H/A$) and 
\begin{align*}
g_{x}(a\cdot )\sim_{\frac{1}{a}(x,x+H],\delta^{O(C_1)}}\widetilde{g}_y.   \end{align*}
Taking $\mathcal{J}_{-1}=\{(y,\widetilde{g}_{y})\colon y\in \mathcal{Y}_{t_0}\}$ and $\mathcal{J}_0=\{(x,g_x)\colon x\in E'\}$, we now have the desired relation~\eqref{eq:relation1}. 
\end{proof}

Using the improved contagiousness result (Theorem~\ref{nil-contagion}), we have the following result that extends~\cite[Proposition 3.7]{walsh} from the linear phase case to nilsequences and also, crucially, works with smaller $H$.\footnote{Writing~\cite[Proposition 3.7]{walsh} in the notation of Proposition~\ref{prop:iterate}, the condition~\eqref{eq:HA} would be replaced with a more restrictive condition of the type $H\geq \sigma^{-K}A^2$.} This  step crucially allows us to ``scale up'' the map $x\mapsto g_x$ of Proposition~\ref{prop:scaledown} to a larger scale where it is easier to analyze.

\begin{proposition}[Scaling up]\label{prop:iterate} Let $Y\geq H\geq A\geq 2$, $0 < \sigma < \frac{1}{2 \log Y}$,  and $2\leq C\leq \sigma^{-1}$.   Let $G/\Gamma$ be a filtered nilmanifold of degree $d$,  dimension $D$ and complexity $\leq \sigma^{-1}$. Let $\mathcal{J}_0$ be a $(\sigma,H)$-configuration in $[Y,CY]$, and let $\mathcal{J}_1$ be a $(\sigma,AH)$-configuration in $[AY,CAY]$ such that
\begin{align*}
  \mathcal{J}_0 \longrightarrow_{\sigma,C}\mathcal{J}_1.     
\end{align*}
Then, if 
\begin{align}\label{eq:HA}
H\geq \sigma^{-K}A,\, \textnormal{ and }\, A\geq \sigma^{-K}  
\end{align} for a large enough constant $K$ (depending on $d,D$), there exists a $(\sigma^{O(1)},A^2H)$-configuration $\mathcal{J}_2$ in $[A^2Y,2CA^2Y]$ such that 
\begin{align*}
 \mathcal{J}_1 \longrightarrow_{\sigma^{O_{d,D}(1)},O_{d,D}(C^2)}\mathcal{J}_2.  
\end{align*}
\end{proposition}

\begin{proof}  We allow implied constants to depend on $d,D$.
By \Cref{def:upscale}, we have
\begin{align*}
 \sum_{(y,g_y)\in \mathcal{J}_1}\sum_{A<a\leq CA}\sum_{(x,g_x)\in \mathcal{J}_0} \mathrm{I}_{a,C,AH,\sigma}(y,x)\gg \sigma \frac{Y}{H}\cdot \sigma A.    
\end{align*}
On using the Cauchy--Schwarz inequality, this implies that
\begin{align*}
\sum_{(x,g_x)\in \mathcal{J}_0}\left(\sum_{(y,g_y)\in \mathcal{J}_1}\sum_{A<a\leq CA}\mathrm{I}_{a,C,AH,\sigma}(y,x) \right)^2 &\gg \frac{(\sigma^2 AY/H)^2}{|\mathcal{J}_0|}\\
&\gg \sigma^4C^{-1} A^2\frac{Y}{H}.   
\end{align*}
Opening the square, we obtain
$$
\sum_{(x,g_x)\in \mathcal{J}_0}\sum_{\substack{(y_1,g_{y_1})\in \mathcal{J}_1 \\ (y_2,g_{y_2})\in \mathcal{J}_1}}\sum_{\substack{A<a_1\leq CA\\A<a_2\leq C A}} \quad \mathrm{I}_{a_1,C,AH,\sigma}(y_2,x) \mathrm{I}_{a_2,C,AH,\sigma}(y_1,x) \gg \sigma^4C^{-1} A^2\frac{Y}{H}. 
$$

Note that for every non-zero term in this sum we have $|y_2-a_1x|\leq CAH$ and $|y_1-a_2x|\leq CAH$, hence $|a_1y_1-a_2y_2|\leq 2C^2A^2H$.  In particular, there exists $z\in C^2A^2H\mathbb{Z}$ with $z\in [A^2Y,2C^2A^2Y]$ such that $|z-a_1y_1|\leq 2C^2A^2H$ and $|z-a_2y_2|\leq 2C^2A^2H$. We
conclude that
\begin{align*}
&\sum_{(x,g_x)\in \mathcal{J}_0}\,\sum_{\substack{(y_1,g_{y_1})\in \mathcal{J}_1 \\ (y_2,g_{y_2})\in \mathcal{J}_1}}\,\sum_{\substack{A<a_1\leq CA\\A<a_2\leq C A}}\,\sum_{\substack{z\in[A^2Y, 2C^2A^2Y]\\ z\in C^2A^2H\mathbb{Z}}}1_{|z-a_1y_1|, |z-a_2 y_2|\leq 2C^2A^2H} \mathrm{I}_{a_1,C,AH,\sigma}(y_2,x) \mathrm{I}_{a_2,C,AH,\sigma}(y_1,x) \\
 &\gg \sigma^ 4C^{-1}A^2\frac{Y}{H}.
\end{align*}
Moving the $z$ sum outside and applying the pigeonhole principle, we see that there exists a $C^2A^2H$-separated set $\mathcal{J}_2^{*}$ in $[A^2Y,2C^2A^2Y]$ of size $\gg \sigma^{O(1)} Y/H$ such that for each $z\in \mathcal{J}_2^{*}$ there is a set $\mathcal{Q}_z$ of $\gg \sigma^{O(1)} A^2$ quadruples $(a_1,a_2,(y_1,g_{y_1}),(y_2,g_{y_2}))\in ([A,CA]\cap \mathbb{N})^2\times \mathcal{J}_1^2$ such that  
$$
|z-a_1y_1|, |z-a_2y_2|\leq 2C^2A^2H
$$
and
$$ \mathrm{I}_{a_1,C,AH,\sigma}(y_2,x) = \mathrm{I}_{a_2,C,AH,\sigma}(y_1,x) = 1 $$
for some $(x,g_x) \in \mathcal{J}_0$.  For fixed $z$, Each $a_1$ is associated to $O(\sigma^{-O(1)})$ many choices of $y_1$, so by pruning we may assume that $y_1$ is determined by $a_1$; similarly we may assume $y_2$ is determined by $a_2$.  By further pruning, we may assume that whenever $(a_1,a_2,(y_1,g_{y_1}),(y_2,g_{y_2}))\in \mathcal{Q}_z$ holds for some $(a_1,a_2)$, it holds for $\gg \sigma^{O(1)}A^2$ such tuples. 
By Lemma~\ref{le:equivalence}(ii), we now obtain
$$ g_{y_1}(a_2\cdot)\sim_{\frac{1}{a_1a_2}(z,z+A^2H], \sigma^{O(1)}}g_x,$$
and similarly
$$ g_{y_2}(a_1\cdot)\sim_{\frac{1}{a_1a_2}(z,z+A^2H], \sigma^{O(1)}}g_x,$$
and hence by Lemma~\ref{le:equivalence}(i) for $(a_1,a_2,(y_1,g_{y_1}),(y_2,g_{y_2}))\in \mathcal{Q}_z$ we have
$$ g_{y_1}(a_2\cdot)\sim_{\frac{1}{a_1a_2}(z,z+A^2H], \sigma^{O(1)}}g_{y_2}(a_1\cdot).$$

Let $z\in \mathcal{J}_2^{*}$, and let $(y,g_y)\in \mathcal{J}_1$ be the last coordinate of some element of $\mathcal{Q}_z$. Since $H/A\geq \sigma^{-K}$ and $K$ is a large enough constant, from the nilsequence contagion result (Theorem~\ref{nil-contagion}) and Lemma~\ref{le:equivalence}(i) we conclude that there exists $g_z\in \Poly(\mathbb{Z}\to G)$ and $\gg \sigma^{O(1)}A$ values of $a\in [A,CA]\cap \mathbb{N}$ such that 
\begin{align*}
g_z(a\cdot)\sim_{\frac{1}{a}(z,z+A^2H], \sigma^{O(1)}}g_{y}. 
\end{align*}
We select such a $g_z$ for each $z\in \mathcal{J}_2^{*}$, and set $\mathcal{J}_2=\{(z,g_z)\colon z\in \mathcal{J}_2^{*}\}$.  The desired claim follows.
\end{proof}

We are going to apply Proposition~\ref{prop:iterate} iteratively. The next proposition tells us how we can obtain information on $x\mapsto g_x$ in the original scale after passing to a suitably large scale.

\begin{proposition}[Conclusion of repeated scaling]\label{prop:conclusion}  Let $X\geq H\geq 2$ and $M\geq X$. Let $0 < \sigma <\frac{1}{2\log M}$ and $2\leq C\leq \sigma^{-1}$. Let $G/\Gamma$ be a filtered nilmanifold of degree $d$,  dimension $D$ and complexity $\leq \sigma^{-1}$. Suppose that there exists a $(\sigma,H)$-configuration $\mathcal{J}_0=\{(x,g_x)\colon x\in E'\}$ in $[X,2X]$ and a $(\sigma,MH)$-configuration $\mathcal{J}_1$ in $[MX,CMX]$ such that
\begin{align*}
\mathcal{J}_0\longrightarrow_{\sigma,C} \mathcal{J}_1.
\end{align*}
Then, if $H\geq \sigma^{-K}$ for a large enough constant $K$ (in terms of $d,D$), the following hold. 
\begin{itemize}
    \item[(i)] If $G/\Gamma=\mathbb{R}/\mathbb{Z}$ and $g_x=P_x$ for some polynomials $P_x\colon \mathbb{Z}\to \mathbb{R}/\mathbb{Z}$ of degree at most $d$, then there exists a real number $|T| \ll \sigma^{-O_{d,D,C}(1)} (X/H)^{d+1}$ and a positive integer $1 \leq q' \ll \sigma^{-O_{d,D,C}(1)}$ such that
    \begin{align*}
       \| e(P_x(n))n^{-iT}\|_{\TV( (x, x+H] \cap \Z; q')} \ll \sigma^{-O_{d,D,C}(1)}.  
    \end{align*} 
		for all $x$ in a subset of $E'$ of cardinality $\gg \sigma^{O_{d,D,C}(1)} X/H$.
    \item[(ii)] If $G$ is non-abelian with one-dimensional center, then there exists a non-trivial horizontal character $\eta \colon G \to \R$ of Lipschitz norm $O(\sigma^{-O_{d,D,C}(1)})$ such that
\begin{equation*}
 \| \eta \circ g_x \|_{C^\infty(x, x+H]} \ll \sigma^{-O_{d,D,C}(1)}
\end{equation*}
for all $x$ in a subset of $E'$ of cardinality $\gg \sigma^{O_{d,D,C}(1)} X/H$.
\end{itemize}
\end{proposition}

\begin{proof} We allow implied constants to depend on $d,D,C$. We arbitrarily fix an element $(z,g_z)$ of $\mathcal{J}_1$.  By assumption, we have
\begin{align*}
\sum_{(x,g_x)\in \mathcal{J}_0} \sum_{M<m\leq CM} \mathrm{I}_{m,C,MH,\sigma}(z,x) \geq \sigma M.    
\end{align*}
Applying the pigeonhole principle and the fact that $|\mathcal{J}_0|\leq \frac{X}{H}$, we see that there is a $(\sigma/(2C),H)$-configuration $\mathcal{J}_{0}'\subset \mathcal{J}_0$ such that for every $(x,g_x)\in \mathcal{J}_{0}'$ we have
\begin{align*}
\sum_{M<m\leq CM}\mathrm{I}_{m,C,MH,\sigma}(z,x) \gg \sigma \frac{MH}{X}.    
\end{align*}
Let $\mathcal{M}_{x}$ be the set of integers $M < m \leq CM$ for which the summand is positive. Then $|\mathcal{M}_{x}|\gg \sigma MH/X$. Note that 
\begin{align}\label{eq:Mx}
\mathcal{M}_x\subset \left[\frac{z}{x}-\frac{CMH}{X},\frac{z}{x}+\frac{CMH}{X}\right]   
\end{align}
and that for all $m\in \mathcal{M}_{x}$ we have
\begin{align}\label{eq:gzmi2}
 g_z(m\cdot)\sim_{\frac{1}{m}(z,z+MH],\sigma}g_x.
\end{align}
Using Lemma~\ref{le:equivalence}(i)--(ii) and the fact that $|\frac{z}{m_i}-\frac{z}{(m_1+m_2)/2}|\ll CH$ for $i=1,2$ and $m_i\in \mathcal{M}_x$, for any $x$ that is the first coordinate of some element of $\mathcal{J}_0'$ and any $(m_1,m_2)\in \mathcal{M}_x^2$ we have 
\begin{align*}
 g_z(m_1\cdot)\sim_{\frac{2}{m_1+m_2}(z,z+MH],\sigma^{O(1)}}g_z(m_2\cdot).  
\end{align*}

For any $M'\geq 1$, set $J_{M'}\coloneqq ((1-\frac{10MH}{z})\frac{z}{M'}, (1+\frac{10MH}{z})\frac{z}{M'}]$. In view of the pigeonhole principle, Lemma~\ref{le:equivalence}(ii) and~\eqref{eq:Mx} and~\eqref{eq:gzmi2}, we can find a $10CMH/X$-separated set $\mathcal{M}\subset [M,CM]$ (which is the union of $\mathcal{M}_x$ for some $x\in \mathcal{J}_0'$) of cardinality $\gg \sigma^{O(1)}X/H$ such that, denoting $I_{M'}=(M',(1+\frac{H}{X})M')$, for each $M'\in \mathcal{M}$ there exist $\gg \sigma^{O(1)}|I_{M'}|^2$ pairs $(m_1,m_2)\in I_{M'}^2$ for which
\begin{align}\label{eq:gzmi2a}
 g_z(m_1\cdot)\sim_{J_{M'},\sigma^{O(1)}}g_z(m_2\cdot),  
\end{align}
and additionally 
\begin{align}\label{eq:gzmi2b}
 g_z(m_1\cdot)\sim_{\frac{1}{m_1}(z,z+MH],\sigma}g_x
\end{align}
for some $x$ that is the first coordinate of an element of $\mathcal{J}_0'$ with $|z/x-m_1|\leq 2CMH/X$.
Unwrapping the definition of the relation $\sim_{I,\eta}$, from~\eqref{eq:gzmi2a} we conclude that for each $M'\in \mathcal{M}$ there exist $\gg \sigma^{O(1)}|I_{M'}|^2$ pairs $(m_1,m_2)\in I_{M'}^2$ for which there exists a factorization
\begin{align}\label{eq:key}
g_z(m_1\cdot)=\varepsilon_{m_1,m_2}g_z(m_2\cdot )\gamma_{m_1,m_2},
\end{align}
where $\varepsilon_{m_1,m_2}$ is $(\sigma^{-O(1)},J_{M'})$-smooth and $\gamma_{m_1,m_2}$ is $\sigma^{-O(1)}$-rational. The argument now splits into two cases. 

\textbf{Proof of part (ii).} Applying Proposition~\ref{prop:Furstenberg-Weiss} (with $(z,M,MH)$ in place of $(X,A,H)$) and using the fact that $G$ is non-abelian, we conclude that either
\begin{align}\label{eq:MH}
 MH\ll \sigma^{-O(1)}\max\left(M,X\right)   
\end{align}
or there exists a non-trivial horizontal character $\eta \colon G \to \R$ of Lipschitz norm at most $O(\sigma^{-O(1)})$ such that 
\begin{align}\label{eq:geta} \| \eta \circ g_z \|_{C^\infty(z,z+MH]} \ll \sigma^{-O(1)}.
\end{align}
The estimate~\eqref{eq:MH} cannot hold if $K$ is chosen large enough in the statement of the proposition. Hence we have~\eqref{eq:geta}. Recalling that~\eqref{eq:gzmi2b} holds for any $m_1 \in \mathcal{M}$, by Lemma~\ref{le:equivalence}(ii) we see that for any $m\in \mathcal{M}$ there is some $x$ that is the first coordinate of an element of $\mathcal{J}_0'$ with $|z/x-m|\leq 2CMH/X$ such that
\begin{align}\label{eq:gzmi2c}
 \| \eta \circ g_x \|_{C^\infty(x,x+H]} \ll \sigma^{-O(1)}.   
\end{align}
Since $\mathcal{M}$ has cardinality $\gg \sigma^{O(1)}X/H$ and is $10CMH/X$-separated, each $m\in \mathcal{M}$ corresponds to a different $x$, so there must exist $\gg \sigma^{O(1)}X/H$ elements $x\in E'$ for which~\eqref{eq:gzmi2c} holds. Now part (ii) follows. 

\textbf{Proof of part (i).} Now we have $G/\Gamma=\mathbb{R}/\mathbb{Z}$.  Then by~\eqref{eq:key} for any $M'\in \mathcal{M}$ there exist $\gg\sigma^{O(1)}|I_{M'}|^2$ pairs $(m_1,m_2)\in I_{M'}^2$ such that
\begin{align*}
  \|q_{m_1,m_2,z}(g_z(m_2\cdot) - g_z(m_1\cdot))\mod 1\|_{C^{\infty}(J_{M'})}\ll \sigma^{-O(1)}.
  \end{align*}
for some integer $1\leq q_{m_1,m_2,z}\leq \sigma^{-O(1)}$. Pigeonholing in the possible choices of $q_{m_1,m_2,z}$, we can assume that this quantity is independent of $m_1,m_2,z$. Now Proposition~\ref{prop:logarithm} (with $(z,M,MH)$ in place of $(X,A,H)$) implies that there exists some integer $1\leq q\leq \sigma^{-O(1)}$ and some real number $T\ll (X/H)^{d+1}\sigma^{-O(1)}$ such that 
\begin{align}\label{eq:gz}
\|q\cdot g_z-T\log(\cdot)\|_{C^d(z,z+MH]}\ll \sigma^{-O(1)}.
\end{align}
Let $m\in \mathcal{M}$.  Dilating~\eqref{eq:gz} by $m$ (using~\eqref{eq:dil} and the fact that $T\log(\cdot)$ is close to a polynomial of degree $d$ by Taylor expansion), we obtain
$$ \|q\cdot g_z(m \cdot)-T\log(\cdot)\|_{C^d(\frac{z}{m},\frac{z}{m}+\frac{MH}{m}]}\ll \sigma^{-O(1)}.$$
On the other hand, from~\eqref{eq:gzmi2b} there is some $x$ that is the first coordinate of some element of $\mathcal{J}_0'$ with $|z/x-m|\leq 2CMH/X$ for which we have 
$$ \| q'_{z,m,x} \cdot g_z(m \cdot) - q'_{z,m,x} \cdot g_x \|_{C^d(\frac{z}{m},\frac{z}{m}+\frac{MH}{m}]}
\ll \sigma^{-O(1)}$$
for some integers $1 \leq q'_{z,m,x} \ll \sigma^{-O(1)}$.  Hence by the triangle inequality and~\eqref{eq:smoothness} (and the fact that $T\log()$ is close to a polynomial of degree $d$) we have
$$ \| qq'_{z,m,x} \cdot g_x  -q'_{z,m,x} T \log(\cdot) \|_{C^d(x,x+H]} \ll \sigma^{-O(1)}.$$
By the pigeonhole principle, we may thus find a natural number $1 \leq q'' \ll \sigma^{-O(1)}$ such that
\begin{equation}\label{eq:Qx1} \| qq'' \cdot g_x  -q'' T \log(\cdot) \|_{C^d(x,x+H]} \ll \sigma^{-O(1)}.
\end{equation}
Since $\mathcal{M}$ is $10CMH/X$-separated, each $m\in \mathcal{M}$ with $|z/x-m|\leq 2CMH/X$ corresponds to a different choice of $x$, so the above holds for all $x$ in a subset of $E'$ of cardinality $\gg \sigma^{O(1)}X/H$. 

 By Taylor expansion, we can find a polynomial $Q_x$ of degree $d' = O(1)$ such that
\begin{equation}\label{eq:Qx2} \| Q_x - \frac{T}{q} \log(\cdot) \|_{C^d(x,x+H]} \ll 1
\end{equation}
and thus by~\eqref{eq:Qx1},~\eqref{eq:Qx2} and the triangle inequality
$$ \| qq'' (g_x - Q_x) \|_{C^d(x,x+H]} \ll \sigma^{-O(1)}.$$
This implies that
$$ qq''(g_x(n)-Q_x(n)) = \sum_{j=0}^{\max(d,d')} \alpha_j \binom{n-\lfloor x\rfloor}{j}\mod 1$$
for some coefficients $\alpha_j$ of size $O( \sigma^{-O(1)} / H^j )$.  One can then calculate
$$ \| e(g_x - Q_x) \|_{\TV((x,x+H]; \max(d,d')! qq'')} \ll \sigma^{-O(1)}$$
and thus by~\eqref{eq:Qx2} and the triangle inequality
$$ \| e(g_x - \frac{T}{q} \log(\cdot)) \|_{\TV((x,x+H]; \max(d,d')! qq'')} \ll \sigma^{-O(1)}.$$
This completes the proof of part (i) after some relabeling.
\end{proof}

We are now ready to conclude the proof of our type $II$ estimates by combining the results of this section.

\begin{proof}[Proof of Theorem~\ref{inverse}(ii)--(iii)] 
  Let the notation and assumptions be as in Theorem~\ref{inverse}(ii)--(iii). Let $B$ be a large enough constant (in terms of $d,D,\varepsilon$), and let us say that $x$ is \emph{bad} if there is no non-trivial horizontal character $\eta$ of Lipschitz norm $\leq \delta^{-B}$ such that $\|\eta\circ g_x\mod 1\|_{C^{\infty}(x,x+H]}\leq \delta^{-B}$. We may assume that all $x\in E$ are bad (since if the set $E'$ of non-bad $x$ has measure $\geq \delta X/2$, say, we are done, and otherwise we may consider the set $E\setminus E'$, which has measure $\geq \delta X/2$). 

 We apply  Proposition~\ref{prop:scaledown}. If conclusion (i) or (ii) of that proposition holds, we are done. Suppose then that conclusion (iii) holds. Then there exists $A\in [A_{II}^{-},A_{II}^{+}]$, a $(\delta^{O_{d,D,\eps}(1)},H)$-configuration $\mathcal{J}_0$ in $[X,2X]$ and a $(\delta^{O_{d,D,\eps}(1)},H/A)$-configuration $\mathcal{J}_{-1}$ in $[X/(2A),2X/A]$, such that $\mathcal{J}_0=\{(x,g_x)\colon x\in E'\}$ for some set $E'\subset E$ and 
 \begin{align*}
\mathcal{J}_{-1}\longrightarrow_{\delta^{O_{d,D,\eps}(1)},O_{d,D,\varepsilon}(1)} \mathcal{J}_0.
 \end{align*}
 
  Let $K\in \mathbb{N}$ be the least integer such that $A^{K}\geq X$. Note that as $A_{II}^{-}\geq X^{\varepsilon}$ we have $K\ll_{\varepsilon} 1$. Applying Proposition~\ref{prop:iterate} $K$ times, and noting that for any given constant $C'$ we may assume that $H\geq \delta^{-C'}A$, we see that for all $2\leq k\leq K$ there is some $C_k=O_{d,D,\varepsilon}(1)$ and some $(\delta^{O_{d,D,\varepsilon}(1)},HA^k)$-configuration $\mathcal{J}_k$ in $[A^kX/2,C_kA^k X/2]$ such that we have
  \begin{align*}
\mathcal{J}_{k-1} \longrightarrow_{\delta^{O_{d,D,k}(1)},C_k} \mathcal{J}_{k}.     
  \end{align*}
  By applying Lemma~\ref{le:transitive} repeatedly, this implies that
  \begin{align*}
\mathcal{J}_0\longrightarrow_{\delta^{O_{d,D,\varepsilon}(1)},C_K'} \mathcal{J}_K     
  \end{align*} 
  for some $C_K'=O_{d,D,\varepsilon}(1)$.
  But now by Proposition~\ref{prop:conclusion} (and the assumption that $x$ is bad for all $x\in E$), we obtain one of the conclusions of Theorem~\ref{inverse}(ii)--(iii).
\end{proof}

\section{Controlling the Gowers uniformity norms}\label{gowers-sec}

In order to deduce our Gowers uniformity result in almost all short intervals (Theorem~\ref{thm_gowers}) from Theorem~\ref{discorrelation-thm}, we wish to apply the inverse theorem for the Gowers norms to $\Lambda-\Lambda^{\sharp}$, $d_k-d_{k}^{\sharp}$, $\mu$. However, before we can apply the inverse theorem, we need to show that the functions $\Lambda-\Lambda^{\sharp}$, $d_k-d_{k}^{\sharp}$ possess pseudorandom majorants even when restricted to \emph{short} intervals\footnote{In some applications of Gowers norm theory, the recent quasipolynomial inverse theorem of Leng--Sah--Sawhney~\cite{LSS} has been able to dispose of the need for pseudorandom majorants, when combined with strong dimension-uniform bounds on the nilsequence side. As discussed in Remark~\ref{rmk:quantitative}, we do not know how to obtain good dimension dependencies in our main theorems, and therefore we need to construct pseudorandom majorants.}. In the case of long intervals, the existence of pseudorandom majorants for these functions follows from earlier works~\cite{green-tao},~\cite{matthiesen-linear}. In the case of short intervals $(X, X+H]$, the existence of these pseudorandom majorants is established in the prequel~\cite{MSTT-all} for $\Lambda$ when $H \geq X^{\eps}$, and for $d_k$ when $H \geq X^{1/5+\eps}$. The main purposes of this section are to show that the majorant for $d_k$ constructed in~\cite[Section 9]{MSTT-all} is pseudorandom in $(x, x+H]$  for almost all $x \in [X, 2X]$ in the larger regime $H \geq X^{\eps}$ and to deduce Theorem~\ref{thm_gowers} from this. For this deduction in the case of $d_k$, we will need a logarithm-free type $I$ estimate given as Proposition~\ref{prop:sharpinverse}.

We begin by recalling from~\cite{MSTT-all} the relevant definition of pseudorandomness.

\begin{definition}[Pseudorandomness over short intervals]
Let $x,H\geq 1$. Let $D\in \mathbb{N}$ and $0<\eta<1$. We say that a function $\nu\colon \mathbb{Z}\to \mathbb{R}_{\geq 0}$ is \emph{$(D,\eta)$-pseudorandom at location $x$ and scale $H$} if the function $\nu_x(n)\coloneqq \nu(x+n)$ satisfies the following. Let $d, t \leq D$ and $\psi_1,\ldots, \psi_t$ be affine-linear forms, where each $\psi_i\colon \mathbb{Z}^d\to \mathbb{Z}$ has the form $\psi_i(\mathbf{x})=\dot{\psi_i}\cdot \mathbf{x}+\psi_i(0)$, with $\dot{\psi_i}\in \mathbb{Z}^d$ and $\psi_i(0)\in \mathbb{Z}$ satisfying $|\dot{\psi_i}|\leq D$ and $|\psi_i(0)|\leq DH$, and with $\dot{\psi_i}$ and $\dot{\psi_j}$ linearly independent whenever $i\neq j$. Then, for any convex body $K\subset [-H,H]^d$, 
\begin{align*}
\left|\sum_{\mathbf{n}\in K}\nu_x(\psi_1(\mathbf{n}))\cdots \nu_x(\psi_t(\mathbf{n}))-\textnormal{vol}(K)\right|\leq \eta H^d.    \end{align*}
\end{definition}

\begin{remark}
\label{rem:Pseudorandom->GowersUniform}
Note that if a function $\nu \colon \mathbb{Z} \to \mathbb{R}_{\geq 0}$ is $(D, \eta)$-pseudorandom at location $x$ and scale $H$, then in particular
\[
\Vert \nu_x - 1 \Vert _{U^s(x, x+H]} \ll_D \eta^{1/2^s}.
\]
for every $s \leq D$. 
\end{remark}
We then state the inverse theorem for unbounded functions that we are going to use.

\begin{proposition}[An inverse theorem for pseudorandomly bounded functions]\label{prop_inverse}\cite[Proposition 9.4]{MSTT-all}
Let $s\in \mathbb{N}$ and $\eta \in (0,1)$. Let $I$ be an interval of length $\geq 2$. Let $f\colon I\to \mathbb{C}$ be a function, and suppose that the following hold.
\begin{itemize}
    \item There exists a function $\nu\colon I\to \mathbb{R}_{\geq 0}$ such that $\|\nu-1\|_{U^{2s}(I)}\leq \eta$ and $|f(n)|\leq \nu(n)$.   
    
    \item For any filtered $(s-1)$-step nilmanifold $G/\Gamma$ and any Lipschitz function $F\colon G/\Gamma\to \mathbb{C}$, we have
    \begin{align*}
    \sup_{g\in \Poly(\mathbb{Z}\to G)}\left|\frac{1}{|I|}\sum_{n\in I}f(n)\overline{F}(g(n)\Gamma)\right|\ll_{\|F\|_{\textnormal{Lip}},G/\Gamma} \eta.    
    \end{align*}
\end{itemize}
Then we have the Gowers uniformity estimate
\begin{align*}
 \|f\|_{U^s(I)}=o_{s;\eta\to 0}(1).   
\end{align*}
\end{proposition}

The following lemma from~\cite{MSTT-all} will be important for us.

\begin{lemma}[Pseudorandomness over long intervals implies pseudorandomness over short intervals ]\cite[Lemma 9.5]{MSTT-all}\label{le_pseudo}
Let $\varepsilon \in (0,1)$, $D,k\in \mathbb{N}$ be fixed. Let $C\geq 1$ be large enough in terms of $k$ and $D$.  Let $H\in [X^{\varepsilon},X/2]$ and $\eta\in ((\log X)^{-C},1/2)$, with $X\geq 3$  large enough in terms of $C$. Let $\nu$ be $(D,\eta)$-pseudorandom  at location $0$ and scale $H$. Also let $1\leq A,B\leq \log X$ be integers. 

Suppose that there is an exceptional set $\mathscr{S}\subset \mathbb{Z}$ and a sequence $\lambda_n$ such that 
\begin{align}\label{eqqn2}\begin{split}
\nu(n)&=\sum_{\substack{d\mid An+B\\d\leq X^{\varepsilon/(2D)}}}\lambda_d \quad \textnormal{for}\,\, n\not \in \mathscr{S},\\
|\lambda_n|&\leq (\log X)^{k}d(n)^k\quad \textnormal{for all}\,\, n,\\
|\nu(n)|&\leq (\log X)^kd(An+B)^k \quad \textnormal{for}\,\, n \in \mathscr{S}.
\end{split}
\end{align} 
Also suppose that $\mathscr{S}$ is small in the sense that
\begin{align*}
|\mathscr{S}\cap [y-2DH,y+2DH]|\ll H/(\log X)^{4C}\quad \textnormal{ for }\quad y\in \{0,X\}. 
\end{align*}
Then $\nu$ is $(D,2\eta)$-pseudorandom at location $X$ and scale $H$.
\end{lemma}

We are now ready to state and prove the existence of pseudorandom majorants over short intervals for $W$-tricked versions of our functions of interest. Let us recall that, for any $w\geq 2$,
\begin{align*}
\Lambda_w(n)\coloneqq \frac{W}{\varphi(W)}1_{(n,W)=1},    
\end{align*}
where $W=\prod_{p\leq w}p$.

\begin{lemma}[Pseudorandom majorants over short intervals for $\Lambda-\Lambda_w$, $d_k-d_{k}^{\sharp}$]\label{le_pseudoinshort} Let $\varepsilon>0$ and $D,k\in \mathbb{N}$ be fixed. Let  $X\geq H\geq X^{\varepsilon}\geq 2$. Let $2\leq w\leq w(X)$, where $w(X)$ is a slowly enough growing function of $X$, and denote $W=\prod_{p\leq w}p$. Also let $w\leq \widetilde{w}\leq \exp((\log X)^{1/10})$ and $\widetilde{W}=\prod_{p\leq \widetilde{w}}p$.

\begin{enumerate}
    \item[(i)] There exists a constant $C_0\geq 1$ such that each of the functions 
\begin{align*}
\frac{\varphi(W)}{W}\Lambda(Wn+b)/C_0,\quad \frac{\varphi(W)}{W}\Lambda_{\widetilde{w}}(Wn+b)  \end{align*}
for $1\leq b\leq W$ with $(b,W)=1$, is majorized on $(X, X+H]$ by a $(D,\eta)$-pseudorandom function at location $X$ and scale $H$ for some $\eta=o_{w\to \infty}(1)$. In fact, the latter of the two functions is $(D,\eta)$-pseudorandom at location $X$ and scale $H$.

\item[(ii)] Let $W'$ be such that $W\mid W'\mid W^{\lfloor w\rfloor}$.  There exists a constant $C_k\geq 1$ such that, for all but at most $XH^{-1/2}(\log X)^{O(1)}$ values of $x \in [X, 2X]$, the functions
\begin{align*}
&(\log X)\frac{\varphi(W)}{W}\prod_{w\leq p\leq X}\left(1+\frac{k}{p}\right)^{-1}d_k(W'n+b)/C_k,\\
&(\log X)\frac{\varphi(W)}{W}\prod_{w\leq p\leq X}\left(1+\frac{k}{p}\right)^{-1}d_k^{\sharp}(W'n+b)/C_k
\end{align*}
\end{enumerate}
for $1\leq b\leq W'$ with $(b,W')=1$, are majorized on $(x, x+H]$ by a $(D,\eta)$-pseudorandom function at location $x$ and scale $H$ for some $\eta=o_{w\to \infty}(1)$.  
\end{lemma}

\begin{proof}
Part (i) is precisely~\cite[Lemma 9.6(1)]{MSTT-all}. For part (ii), we follow the construction in the proof of~\cite[Lemma 9.6(2)]{MSTT-all}.    From the definition of $d_k^{\sharp}$, it follows that $d_k^{\sharp}(n)\ll_k d_k(n)$ for all $1\leq n\leq 2X$ (see~\cite[(3.14)]{MSTT-all}), so by Lemma~\ref{le_pseudo} it suffices to show that, for some $C_k'\geq 1$ and all $x \in [X, 2X]$ outside an admissible exceptional set, the function 
$$h(n)\coloneqq (\log X)\frac{\varphi(W)}{W}\prod_{w\leq p\leq X}\left(1+\frac{k}{p}\right)^{-1}d_k(W'n+b)/C_k'$$
is majorized by a function that is $(D,o_{w\to \infty}(1))$-pseudorandom at location $0$ and scale $H$, and which is of the form~\eqref{eqqn2} outside an exceptional set $\mathscr{S}$ satisfying the size bounds
\begin{equation}\label{eqqn3b}
|\mathscr{S}\cap [x-2DH,x+2DH]|\ll H/(\log X)^{4C},\quad |\mathscr{S}\cap [-2DH,2DH]|\ll H/(\log X)^{4C}    
\end{equation}
for some large enough $C\geq 1$. 

We claim that this follows from~\cite[Lemma 9.6]{MSTT-all}. Indeed, by that lemma, the function $h$ is majorized by a $(D,o_{X\to \infty}(1))$-pseudorandom function at location $0$ and scale $H$. The fact that the majorant is of the form~\eqref{eqqn2} follows from the arguments in~\cite[proof of Lemma 9.6]{MSTT-all}, with the exceptional set $\mathscr{S}$ being defined as
\begin{align*}
\mathscr{S}\coloneqq \left\{n\leq 2DX\colon \exists \, p\colon v_p(n)\geq \max\left\{2,8C\frac{\log \log X}{\log p}\right\}\,\, \textnormal{or}\,\, \prod_{p\leq X^{1/(\log \log X)^3}}p^{v_p(n)}\geq X^{\gamma/\log \log X}\right\}    
\end{align*} 
for some constant $\gamma>0$. Also from~\cite[proof of Lemma 9.6]{MSTT-all},  we have the second estimate in~\eqref{eqqn3b}. The remaining task now is to prove the first estimate in~\eqref{eqqn3b}.

We now use Rankin's trick to bound $|\mathscr{S} \cap [x-2DH, x+2DH]|$ for $x \in [X, 2X]$. Denoting $v=X^{1/(\log \log X)^3}$, the number of $n\in [x-2DH, x+2DH]$ that satisfy $\prod_{p\leq X^{1/(\log \log X)^3}}p^{v_p(n)}\geq X^{\gamma/\log \log X}$ is (as in~\cite{MSTT-all})
\begin{align}\label{eq:g(p)}
&\ll \sum_{\substack{ab\in [x-2DH,x+2DH]\\p\mid a\implies p>v\\p\mid b\implies p\leq v\\b\geq X^{\gamma/\log \log X}}}1 \leq \sum_{\substack{ab\in [x-2DH,x+2DH]\\p\mid a\implies p>v\\p\mid b\implies p\leq v}}\left(\frac{b}{X^{\gamma/\log \log X}}\right)^{\frac{10C(\log \log X)^2}{\gamma\log X}}\nonumber\\
&\ll \frac{1}{(\log X)^{10C}}\sum_{n\in [x-2DH,x+2DH]}g(n),    
\end{align}
where $g$ is the completely multiplicative function satisfying
\begin{align*}
g(p)=\begin{cases}1,\quad &\textnormal{if}\quad p>v,\\ p^{\frac{10C(\log \log X)^2}{\gamma\log X}}\quad &\textnormal{if}\quad p\leq v.\end{cases}.   
\end{align*}
By Shiu's bound (\Cref{le:shiu}), we conclude that~\eqref{eq:g(p)} is $\ll H(\log X)^{-4C}$.

We are left with bounding the number of $n\in [x-2DH, x+2DH]$ satisfying $v_p(n)\geq \max\{2,8C\frac{\log \log X}{\log p}\}$ for some $p$. This number is 
\begin{align}\label{eq:p2}
&\ll \sum_{p< (\log X)^{4C}}Hp^{-8C(\log \log X)/(\log p)}+\sum_{(\log X)^{4C}\leq p\leq H^{1/2}}\frac{H}{p^2}\nonumber\\
&\quad +\sum_{H^{1/2}<p\leq (2X)^{1/2}}|\{n \in [x-2DH, x+2DH]\colon p^2\mid n\}|\nonumber \\
& \ll H(\log X)^{-4C}   + \sum_{H^{1/2}<p\leq (2X)^{1/2}}|\{n \in [x-2DH, x+2DH]\colon p^2\mid n\}|.
\end{align}

Define $\mathcal{E}$ to be the set of $x \in [X, 2X]$ such that the sum over $p$ in~\eqref{eq:p2} is $\geq H(\log X)^{-4C}$.
Observe that
\begin{align*}
|\mathcal{E}|H(\log X)^{-4C}&\leq \sum_{\substack{X \leq x \leq 2X\\x\in \mathbb{N}}} \sum_{H^{1/2}<p\leq (2X)^{1/2}}|\{n \in [x-2DH, x+2DH]\colon p^2\mid n\}|\\
&\ll H \sum_{H^{1/2}<p\leq (2X)^{1/2}}\sum_{\substack{X-2DH \leq n \leq 2X+2DH \\ p^2\mid n}} 1\\
&\ll XH\sum_{H^{1/2}<p\leq (2X)^{1/2}}\frac{1}{p^2} \ll XH^{1/2},
\end{align*}
so we have $|\mathcal{E}| \ll XH^{-1/2}(\log X)^{4C}$. We conclude that the desired property~\eqref{eqqn3b} holds for all $x\in [X,2X]$ outside the sufficiently small exceptional set $\mathcal{E}$. 
\end{proof}

\subsection{Gowers uniformity of the von Mangoldt and M\"obius functions}

We are now ready to prove Theorem~\ref{thm_gowers}(i)--(ii).

\begin{proof}[Proof of Theorem~\ref{thm_gowers}(i)--(ii)] 
(i) Let $1\leq a\leq W$ with $(a,W)=1$ (where, as in the theorem, $W=\prod_{p\leq w}p$ with $w$ tending to infinity slowly with $X$). Set $f=\frac{\varphi(W)}{W}(\Lambda-\Lambda^{\sharp})$. By Theorem~\ref{discorrelation-thm}(ii), there is an exceptional set $\mathcal{E}$ of measure $\ll_A X(\log X)^{-A}$ such that for all $x \in [X(\log X)^{-A/2}, 2X]\setminus \mathcal{E}$, $X^{1/3+\varepsilon/2}\leq H\leq X^{1-\varepsilon}$ and for any $F,G/\Gamma$ as in that theorem, we have
\begin{align*}
&\sup_{g \in \Poly(\Z \to G)} \left| \sum_{x < n \leq x+H} f(Wn+a) \overline{F}(g(n)\Gamma) \right|\\
=& \sup_{g \in \Poly(\Z \to G)} \left| \sum_{\substack{Wx+a < n \leq W(x+H)+a\\n\equiv a\pmod W}} f(n) \overline{F}\left(g\left(\frac{n-a}{W}\right)\Gamma\right) \right|\\
\ll&_A H/\log^{A}X, 
\end{align*} 
since there exists a polynomial sequence $\widetilde{g}\colon \mathbb{Z}\to G$ such that $\widetilde{g}(n)=g((n-a)/W)$ for all $n\equiv a\pmod W$. Using Lemma~\ref{le_pseudoinshort}(i), we can find functions  $\nu_j$ for $j\in \{1,2\}$ that are $(2s,o_{w\to \infty}(1))$-pseudorandom at location $x$ and scale $H$ for all $x\in [X(\log X)^{-A/2},2X]\setminus \mathcal{E}$, and such that $\nu_1(n)$ and $\nu_2(n)$ majorize 
$$\frac{\varphi(W)}{W}\Lambda(Wn+a)/C_0,\quad \textnormal{and}\quad  \nu_2(n)=\frac{\varphi(W)}{W}\Lambda^{\sharp}(Wn+a),$$
respectively, where $C_0$ is a suitably chosen constant. In particular, by the triangle inequality for the Gowers norms, we have 
$$
\left\|\frac{\nu_1+\nu_2}{2}-1\right\|_{U^{2s}(x,x+H]}\leq \frac{1}{2}(\|\nu_1-1\|_{U^{2s}(x,x+H]}+\|\nu_2-1\|_{U^{2s}(x,x+H]})=o_{w\to \infty}(1).
$$  
Since $|f(Wn+a)|\ll \nu_1(n)+\nu_2(n)$, the inverse theorem for the Gowers norms (Proposition~\ref{prop_inverse}) tells us that
\begin{align}\label{eq:LambdaW1}
\left\|\frac{\varphi(W)}{W}(\Lambda(W\cdot+a)-\Lambda^{\sharp}(W\cdot+a))\right\|_{U^{s}(x,x+H]}=\|f(W\cdot +a)\|_{U^{s}(x,x+H]}=o_{w\to \infty}(1).    
\end{align}
On the other hand, by Lemma~\ref{le_pseudoinshort}(i), Remark~\ref{rem:Pseudorandom->GowersUniform}, and the definition of $\Lambda^{\sharp}$ we also have for $x\in [X(\log X)^{-A/2},2X]\setminus \mathcal{E}$ that
\begin{align}\label{eq:LambdaW2}
\left\|\frac{\varphi(W)}{W}\Lambda^{\sharp}(W\cdot+a)-1\right\|_{U^{s}(x,x+H]}=o_{w\to \infty}(1).
\end{align}
By~\eqref{eq:LambdaW1},~\eqref{eq:LambdaW2} and the triangle inequality for the Gowers norms, we obtain
\begin{align*}
\left\|\frac{\varphi(W)}{W}\Lambda(W\cdot+a)-1\right\|_{U^{s}(x,x+H]}=o_{w\to \infty}(1).    
\end{align*}
By~\cite[Lemma 9.8]{MSTT-all}, this implies that also
\begin{align*}
 \left\|\Lambda-\Lambda_w\right\|_{U^{s}(x,x+H]}=o_{w\to \infty}(1)  \end{align*}
holds for all $x\in [X(\log X)^{-A/2},2X]\setminus \mathcal{E}$. 

(ii) This part follows from the same argument as part (i), taking $f=\mu$ and $W=1$ and applying Theorem~\ref{discorrelation-thm}(iv) (and now Lemma~\ref{le_pseudoinshort} is not needed). 
\end{proof}

\subsection{Gowers uniformity of the divisor functions}

For the proof of Theorem~\ref{thm_gowers}(iii), we shall need a version of the type $I$ inverse theorem that does not lose any logarithms. Note that Theorem~\ref{inverse}(i) does lose logarithms since we assume there that $\delta<1/\log X$, so that losing powers of $\delta$ corresponds to losing powers of logarithm). 

\begin{proposition}[A log-free type $I$ estimate]\label{prop:sharpinverse}
Let $d,D,k\geq 1$, $X\geq H\geq 2$, and $\delta \in (0, 1/2)$. Let $1\leq M\leq \delta^{K}H$ for some large enough $K=K(d,D)$, and let $G/\Gamma$ be a filtered nilmanifold of degree at most $d$, dimension at most $D$, and complexity at most $1/\delta$. Let $F\colon G/\Gamma\to \mathbb{C}$ be Lipschitz of norm at most $1/\delta$ and mean zero. Let $f\colon \mathbb{N}\to \mathbb{C}$ be an arithmetic function such that
\begin{align*}
\left|\sum_{X<n\leq X+H}f(n)F(g(n)\Gamma)\right|^{*}\geq \delta H(\log X)^{k-1}, \end{align*}
where $g \colon \Z \to G$ is a polynomial map. Suppose that $f=\alpha*\beta$, where $|\alpha(m)|\leq d_{k-1}(m)$, $\alpha$ is supported on $(M,2M]$, and $\|\beta\|_{\TV(\mathbb{N})}\leq \frac{1}{\delta}$. Then we have
\begin{align}\label{eq:fconclusion}
\left|\sum_{X<n\leq X+H}f(n)\right|^{*}\gg_{d,D,k} \delta^{O_{d,D,k}(1)} H(\log X)^{k-1}.    
\end{align}
\end{proposition}

\begin{proof}
In what follows, we let $d,D,k\geq 1$ be fixed; we allow all implied constants to depend on these parameters. Let $A\geq 1$ be large in terms of these fixed parameters. By the factorization theorem for nilsequences~\cite[Theorem 1.19]{green-tao-ratner}, there exist a rational subgroup $G'\subset G$  and $\delta^{-1}\leq Q\leq \delta^{-O_A(1)}$ such that we can factorize $g=\varepsilon g'\gamma$, where the polynomial sequences $\varepsilon,g'\gamma\colon \Z\to G$ have the following properties:
\begin{itemize}
    \item $\varepsilon$ is $(Q,[X,X+H])$-smooth;
    \item $g'$ is totally $Q^{-A}$-equidistributed in $G'/(G'\cap \Gamma)$;
    \item $\gamma$ is $Q$-rational, with $(\gamma(n)\Gamma)_{n\in \mathbb{Z}}$ having period at most $Q$.
\end{itemize}
Moreover, $G'/(G'\cap \Gamma)$ has a Mal'cev basis in which each element is a $Q$-rational combination of elements of the Mal'cev basis of $G$.

By splitting the interval $[X,X+H]$ into arithmetic progressions of length $\asymp H/Q^2$ and modulus equal to the period of $\gamma$, we conclude that for some $Q^{O(1)}$-rational elements $\varepsilon_0,\gamma_0\in G'$ we have
\begin{align*}
\left|\sum_{X<n\leq X+H}f(n)F(\varepsilon_0 g'(n)\gamma_0\Gamma)\right|^{*}\gg \delta^{O(1)} H(\log X)^{k-1}, \end{align*}
Let $G''=\gamma_0^{-1}G'\gamma_0$, $\Gamma''=\Gamma\cap G''$,  $g''(n)=\gamma_0^{-1}g'(n)\gamma_0$, and let $\widetilde{F}\colon G''/\Gamma''\to \mathbb{C}$ be given by $\widetilde{F}(x)=F(\varepsilon_0 \gamma_0 x)$. Then we have 
\begin{align}\label{eq:Fg''}
\left|\sum_{X<n\leq X+H}f(n)\widetilde{F}(g''(n)\Gamma)\right|^{*}\gg \delta^{O(1)} H(\log X)^{k-1}. 
\end{align}
By subtracting off the mean of $\widetilde{F}$, we conclude that either~\eqref{eq:fconclusion} holds, or else~\eqref{eq:Fg''} holds with $\widetilde{F}$ having mean $0$. We may assume that the latter holds.  Performing some nil-algebra as in~\cite[Proof of Claim in Section 2]{gt-mobius}, we see that there is a Mal'cev basis $\mathcal{Y}$ for $G''/\Gamma''$, which is a $Q^{O(1)}$-rational combination of the elements of the Mal'cev basis of $G$, such that  $g''$ is totally $Q^{-cA+O(1)}$-equidistributed on $G''/\Gamma''$ with respect to the metric induced by $\mathcal{Y}$ for some $c>0$ depending only on the fixed parameters, and additionally $\widetilde{F}$ has Lipschitz norm $\ll \delta^{-O(1)}$.  

From the  assumption on $f$ and the triangle inequality, we have
\begin{align*}
\sum_{M<m\leq 2M}d_{k-1}(m)\left|\sum_{X/m < r\leq (X+H)/m}\beta(r)\widetilde{F}(g''(rm)\Gamma)\right|^{*}\gg \delta^{O(1)} H(\log X)^{k-1}.   \end{align*}
Since $\|\beta\|_{\TV(\mathbb{N})}\leq \frac{1}{\delta}$, invoking Lemma~\ref{basic-prop}(ii) we obtain
\begin{align*}
\sum_{M<m\leq 2M}d_{k-1}(m)\left|\sum_{X/m< r\leq (X+H)/m}\widetilde{F}(g''(rm)\Gamma)\right|^{*}\gg \delta^{O(1)} H(\log X)^{k-1}.    
\end{align*}
By the pigeonhole principle, this implies that for some $M_1,\ldots, M_{k-1}\geq 1$ with $M_1\cdots M_{k-1}\asymp M$ we have
\begin{align*}
 &\sum_{M_1<m_1\leq 2M_1}\cdots \sum_{M_{k-1}<m_{k-1}\leq 2M_{k-1}}\left|\sum_{X/(m_1\cdots m_{k-1})< r\leq (X+H)/(m_1\cdots m_{k-1})}\widetilde{F}(g''(rm_1\cdots m_{k-1})\Gamma)\right|^{*}\\
 &\quad \gg \delta^{O(1)} H.
\end{align*}
Since $|F|\leq 1/\delta$ by the assumption that $F$ has Lipschitz norm at most $1/\delta$, from the pigeonhole principle we conclude that there exists a set $\mathcal{M}$ of $\gg \delta^{O(1)} M$ integer tuples $(m_1,\ldots, m_{k-1})\in (M_1,2M_1]\times \cdots \times (M_{k-1},2M_{k-1}]$ such that 
\begin{align*}
\left|\sum_{X/(m_1\cdots m_{k-1})< r\leq (X+H)/(m_1\cdots m_{k-1})}\widetilde{F}(g''(rm_1\cdots m_{k-1})\Gamma)\right|^{*}\gg \delta^{O(1)} \frac{H}{M_1\cdots M_{k-1}}.
\end{align*}

From the quantitative Leibman theorem (Lemma~\ref{qlt}), we deduce that for $(m_1,\ldots, m_{k-1})\in \mathcal{M}$ there exists some non-trivial horizontal character $\eta\colon G''\to \mathbb{R}$ (depending on the tuple) of Lipschitz norm $\ll \delta^{-O(1)}$ such that
\begin{align}\label{eq:gm}
\|(\eta\circ g'')(m_1\cdots m_{k-1}\cdot)\mod 1\|_{C^{\infty}(\frac{X}{m_1\cdots m_{k-1}},\frac{X+H}{m_1\cdots m_{k-1}}]}\ll \delta^{-O(1)}.  \end{align}
By pigeonholing, we can find some such horizontal character $\eta$ such that~\eqref{eq:gm} holds for $\gg \delta^{O(1)}M$ integer tuples $(m_1,\ldots, m_{k-1})\in (M_1,2M_1]\times \cdots \times (M_{k-1},2M_{k-1}]$. Let $\mathcal{M}'$ be the set formed by these tuples. Since $\eta\circ g''$ is a polynomial of degree at most $d$, we can Taylor expand it as 
$$(\eta\circ g'')(n)=\alpha_d(n-X)^{d}+\cdots+\alpha_1(n-X)+\alpha_0\mod 1$$
for some $\alpha_j\in \mathbb{R}/\mathbb{Z}$. 
Therefore, for any $m\geq 1$ we have
$$(\eta\circ g'')(mn)=\sum_{j=0}^d \alpha_j m^j\left(n-\frac{X}{m}\right)^j\mod 1.$$

Now~\eqref{eq:gm} and~\eqref{eq:smoothcomparison} give for any $(m_1,\ldots, m_{k-1})\in \mathcal{M}'$ the bound
\begin{align}\label{eq:gm2}
\max_{1\leq j\leq d}\left(\frac{H}{M_1\cdots M_{k-1}}\right)^j\|\alpha_j(m_1\cdots m_{k-1})^j\|_{\R/\Z}\ll \delta^{-O(1)}.
\end{align}
In particular, there exist $\gg \delta^{O(1)}M_1$ integers $m_1\in (M_1,2M_1]$ for which~\eqref{eq:gm2} holds for $\gg \delta^{O(1)}M_2\cdots M_{k-1}$ integer tuples $(m_2,\ldots, m_{k-1})\in (M_2,2M_2]\times \cdots \times (M_{k-1},2M_{k-1}]$. By Vinogradov's lemma (Lemma~\ref{vin}), this implies that for some integer $1\leq q_1\ll \delta^{-O(1)}$ we have 
\begin{align*}
\max_{1\leq j\leq d}\left(\frac{H}{M_2\cdots M_{k-1}}\right)^j\|q_1\alpha_j(m_2\cdots m_{k-1})^j\|_{\R/\Z}\ll \delta^{-O(1)}
\end{align*}
for $\gg \delta^{O(1)}M_2\cdots M_{k-1}$ integer tuples $(m_2,\ldots, m_{k-1})\in (M_2,2M_2]\times \cdots \times (M_{k-1},2M_{k-1}]$ (by pigeonholing we can take $q_1$ to be independent of the tuple).
Continuing inductively, we see that for some positive integer $q\ll \delta^{-O(1)}$ we have 
\begin{align*}
 \max_{1\leq j\leq d}H^j\|q\alpha_j\|_{\R/\Z}\ll \delta^{-O(1)}.   
\end{align*}
But by the definition of the smoothness norms, this implies that
\begin{align*}
\|(q\eta)\circ g''\mod 1\|_{C^{\infty}(X,X+H]}\ll \delta^{-O(1)}. 
\end{align*}
The horizontal character $q\eta$ has Lipschitz norm $\ll \delta^{-O(1)}$. As noted above, $g''$ is totally $Q^{-cA+O(1)}$-equidistributed on $G''/\Gamma''$. But now we obtain a contradiction from the converse to the quantitative Leibman theorem~\cite[Lemma 5.3]{FH-Fourier} and the assumption that $A$ was large enough. 
\end{proof}

We are now ready to prove Theorem~\ref{thm_gowers}(iii). 

\begin{proof}[Proof of Theorem~\ref{thm_gowers}(iii)] 
We apply Theorem~\ref{discorrelation-thm}(v). Following the same argument as in the proof of Theorem~\ref{thm_gowers}(i), but with 
$$f(n)=\left(\frac{W}{\varphi(W)}\right)^{k-1}(\log^{1-k}X)(d_k(n)-d_k^{\sharp}(n))$$
and using Lemma~\ref{le_pseudoinshort}(ii) in place of Lemma~\ref{le_pseudoinshort}(i), we see that there is some exceptional set $\mathcal{E}$ of measure $\ll_AX/\log^{A}X$ such that for all $x\in [X(\log X)^{-A/2},2X]\setminus \mathcal{E}$ we have
\begin{equation}\label{eq:dkgowers}
\left\|\left(\frac{W}{\varphi(W)}\right)^{k-1}(\log^{1-k}X)(d_k(W'\cdot+a)-d_k^{\sharp}(W'\cdot+a))\right\|_{U^{s}(x,x+H]}=o_{w\to \infty}(1).    
\end{equation}
Arguing as in~\cite[Proof of Theorem 1.5]{MSTT-all} this implies that
\begin{align*}
\|d_k-d_k^{\sharp}\|_{U^s(x,x+H]}=o_{X\to \infty}(\log^{k-1}X)     
\end{align*}
for all $x\in [X,2X]\setminus \mathcal{E}'$, where $\mathcal{E}'\supset \mathcal{E}$ is another exceptional set of measure $\ll_AX/\log^{A}X$. We still need to show that the second estimate in~\eqref{eq:gowersdk} holds outside a small exceptional set.

In view of~\eqref{eq:dkgowers} and the triangle inequality for the Gowers norms, we are left with showing that
\begin{align}\label{eq:dksupbound}
\left\|\left(\frac{W}{\varphi(W)}\right)^{k-1}(\log^{1-k}X)d_k^{\sharp}(W'\cdot+a)-\frac{1}{(k-1)!}\right\|_{U^{s}(x,x+H]}=o_{w\to \infty}(1)  \end{align}
for all $x\in [X,2X]\setminus \mathcal{E}''$, where $\mathcal{E}''\supset \mathcal{E}'$ is another exceptional set of measure $\ll_AX/\log^{A}X$.

By Lemma~\ref{le_pseudoinshort}(ii) and Proposition~\ref{prop_inverse}, it suffices to show that, for any filtered $(s-1)$-step nilmanifold $G/\Gamma$ and any Lipschitz function $F\colon G/\Gamma\to \mathbb{C}$, we have
\begin{align*}
&\sup_{g\in \Poly(\mathbb{Z}\to G)}\left|\sum_{x<n\leq x+H}\left(\left(\frac{W}{\varphi(W)}\right)^{k-1}(\log^{1-k}X)d_k^{\sharp}(W'n+a)-\frac{1}{(k-1)!}\right)\overline{F}(g(n)\Gamma)\right|\\
=&o_{w\to \infty;\|F\|_{\Lip};G/\Gamma}(H)  \end{align*}
for all $x\in [X,2X]$ outside of an exceptional set of measure $\ll_AX/\log^{A}X$.

By Lemma~\ref{le:dkdecompose}(ii), the function inside the sum is a sum of $O(1)$ terms of the form $\alpha_j*\beta_j$, where $\|\beta_j\|_{\TV(\mathbb{N})}\ll 1$ and $|\alpha_j(n)| \leq d_{k-1}(n) 1_{[1,X^{\varepsilon/5}]}(n)$, Hence, from Proposition~\ref{prop:sharpinverse} we conclude 
that~\eqref{eq:dksupbound} holds provided that
\begin{align*}
\left|\sum_{x<n\leq x+H}\left(\left(\frac{W}{\varphi(W)}\right)^{k-1}(\log^{1-k}X)d_k^{\sharp}(W'n+a)-\frac{1}{(k-1)!}\right)\right|^{*}=o_{w\to \infty}(H).    
\end{align*}

From~\eqref{eq:dkgowers} with $s=1$, it follows that
\begin{align*}
&\left|\sum_{x<n\leq x+H}\left(\left(\frac{W}{\varphi(W)}\right)^{k-1}(\log^{1-k}X)(d_k(W'n+a)-d_k^{\sharp}(W'n+a))\right)\right|^{*}\\
=&o_{X\to \infty}(H).        
\end{align*}
for all $x\in [X,2X]\setminus \mathcal{E}$. Hence, it suffices to show that
\begin{align*}
\left|\sum_{x<n\leq x+H}\left(\left(\frac{W}{\varphi(W)}\right)^{k-1}(\log^{1-k}X)d_k(W'n+a)-\frac{1}{(k-1)!}\right)\right|^{*}=o_{X\to \infty}(H)    
\end{align*}
 for all $x\in [X,2X]$ outside of an exceptional set of measure $\ll X/\log^A X$. But since $W'\ll \log \log X$ (say) this follows from Lemma~\ref{le:dkshort}.
\end{proof}

\section{Applications}\label{sec:applications}

\subsection{The polynomial phase case}\label{sec:polyphase}

\begin{proof}[Proof of Corollary~\ref{cor:discorrelation}] 
Part (i) of the corollary is immediate from  Theorem~\ref{discorrelation-thm}(i). From Theorem~\ref{discorrelation-thm}(ii), we either have the conclusion of part (ii) or else 
\begin{align*}
     \left|\sum_{x<n\leq x+H}\Lambda^{\sharp}(n)e(P_x(n))\right|\geq \frac{H}{\log^A X}   
    \end{align*}
for $x \in E$, with $E$ a subset of $[X,2X]$ of measure $\gg_{A,d,\varepsilon} X(\log X)^{-A}$. By Lemma~\ref{le:dkdecompose}(i), this implies that for some sequence $a\colon \mathbb{N}\to \mathbb{C}$ supported on $[1,X^{\varepsilon/5}]$ and satisfying $|a(n)|\ll \log X$ we have
\begin{align*}
     \left|\sum_{x<n\leq x+H}(a*1)(n)e(P_x(n))\right|\geq \frac{H}{2\log^A X}   
    \end{align*}
for all $x \in E$. 
    Applying Theorem~\ref{inverse}(i) (and Remark~\ref{rem:RZhorchar}), we conclude that, for some integer $1\leq q\ll (\log X)^{O_{A,\varepsilon,d}(1)}$, we have
    \begin{align*}
\|qP_x\mod 1\|_{C^{\infty}(x,x+H]}\ll_{d} (\log X)^{O_{A,\varepsilon, d}(1)}
    \end{align*}
    for all $x$ in a subset of $E$ of measure $\gg_{A,d,\varepsilon}X\log^{-A}X$. The desired conclusion follows from~\eqref{eq:smoothcomparison}.
\end{proof}

\subsection{Calculating the main terms}

In order to prove Theorems~\ref{thm:HL} and~\ref{thm:divisor}, we must in particular evaluate the correlations of $\Lambda_w$ and $d_{k,W} \coloneqq \prod_{p \leq w} d_{k,p}$ (where $d_{k,p}$ is defined in Theorem~\ref{thm:divisor}), since these yield the main terms in the correlations asymptotics for $\Lambda$ and $d_k$, respectively. We do this in a more general setting as follows. 

\begin{proposition}[Factorizing correlations of multiplicative functions]\label{prop:simplecorrrelation}
Let $k,\ell\in \mathbb{N}$ be fixed. Let $3\leq w\leq \exp((\log X)^{1/2})$ and $W=\prod_{p\leq w}p$. Let $1\leq h\leq X$ be an integer. Let $g\colon \mathbb{N}\to \mathbb{C}$ be a multiplicative function, and denote by $g_p$ the multiplicative function defined on prime powers by $g_p(p^{a})=g(p^{a})$ and $g_p(q^{a})=1$ if $q\neq p$. Define $g_W\coloneqq \prod_{p\leq w}g_p$. Suppose that $|g_p(p^{a})|\leq d_k(p^{a})$ for all $p\leq w$ and $a\in \mathbb{N}$. Then we have
\begin{align*}
\mathbb{E}_{n\leq X}\prod_{j=0}^{\ell-1}g_W(n+jh)=\prod_{p\leq w} \left(\lim_{Y\to \infty}\mathbb{E}_{n\leq Y}\prod_{j=0}^{\ell-1}g_p(n+jh)\right)+O(\exp(-(\log X)^{1/2}/(30\ell))).
\end{align*}
\end{proposition}

\begin{proof}
 Let $\mathcal{B}$ be the set of $n\leq X$ for which $(n,W^{\infty})>X^{\frac{1}{10\ell}}$. Since every $n\in \mathcal{B}$ can be written uniquely as $e_1e_2$, where $e_1\mid W^{\infty}$ with $e_1>X^{\frac{1}{10\ell}}$ and $(e_2,W)=1$, we have
 \begin{align}\label{eq:Bbound}\begin{split}
 |\mathcal{B}|&= \sum_{\substack{e_1\mid W^{\infty}\\e_1>X^{\frac{1}{10\ell}}}}\sum_{\substack{(e_2,W)=1\\e_2\leq X/e_1}}1 \ll X\sum_{\substack{e_1\mid W^{\infty}\\e_1>X^{\frac{1}{10\ell}}}}\frac{1}{e_1} \ll X^{1-\frac{1}{10\ell\log w}}\sum_{\substack{e_1\mid W^{\infty}\\e_1>X^{\frac{1}{10\ell}}}}\frac{1}{e_1^{1-\frac{1}{\log w}}}\\
 &\ll  X^{1-\frac{1}{10\ell\log w}}\prod_{p\leq w}\left(1+\frac{1}{p^{1-\frac{1}{\log w}}}+O\left(\frac{1}{p^{2(1-\frac{1}{\log w})}}\right)\right)\\
 &\ll X^{1-\frac{1}{10\ell\log w}}\prod_{p\leq w}\left(1+O\left(\frac{1}{p}\right)\right) \ll X^{1-\frac{1}{10\ell\log w}}(\log w)^{O(1)}.
 \end{split}
 \end{align}
 
By M\"{o}bius inversion we can write 
 $$g_W(n) = g((n, W^\infty)) =\sum_{\substack{e\mid (n, W^{\infty})}}\psi(e),$$
 where $\psi\coloneqq \mu*g$ is multiplicative. For any prime power $p^{a}$ we have $\psi(p^{a})=g_W(p^{a})-g_W(p^{a-1})$. In particular, this implies that 
\begin{align}\label{eq:psibound} 
|\psi(p^{a})|\leq 2d_k(p^{a})\cdot 1_{p\leq w}. 
\end{align}
Let us also define the truncated version
 \begin{align*}
 \widetilde{g}_W(n)\coloneqq \sum_{\substack{e\mid (n, W^{\infty}) \\e\leq X^{\frac{1}{10\ell}}}}\psi(e).
 \end{align*}

For $n \notin \mathcal{B}$ we have $g_W(n) = \widetilde{g}_W(n)$. Hence, by the triangle inequality, for some $i_0\in \{0, 1, \dotsc, \ell-1\}$,  we have
 \begin{align}\label{eq:gprod}
 \sum_{n\leq X}\prod_{j=0}^{\ell-1}g_W(n+jh)&=\sum_{\substack{n\leq X\\n+ih\not \in 
 \mathcal{B} \, \forall 0\leq i\leq \ell-1}}\prod_{j=0}^{\ell-1}\widetilde{g}_W(n+jh)+O\left(\max_{0\leq i\leq \ell-1}\sum_{\substack{n\leq X\\n+ih\in 
 \mathcal{B}}}\prod_{j=0}^{\ell-1}|g_W(n+jh)|\right)\nonumber\\
 &=\sum_{n\leq X}\prod_{j=0}^{\ell-1}\widetilde{g}_W(n+jh)+O\left(\sum_{\substack{n\leq X\\n+i_0h\in 
 \mathcal{B}}}\prod_{j=0}^{\ell-1}(|g_W|+|\widetilde{g}_W|)(n+jh)\right).
 \end{align}
Since $|g_W(n)|, |\widetilde{g}_W(n)|\leq d_{k+2}(n)$, the error term here is by H\"older's inequality (with exponents $1/2,1/(2\ell),\ldots, 1/(2\ell)$),~\eqref{eq:Bbound} and Shiu's bound (\Cref{le:shiu})
\begin{align*}
\leq \left(\sum_{\substack{n\leq X\\n+i_0h\in \mathcal{B}}}1\right)^{1/2}\left(\sum_{n\leq X+(\ell-1)h}d_{k+2}^{2\ell}\right)^{1/2}&\ll X(\log X)^{O(1)}\exp\left(-\frac{\log X}{20\ell\log w}\right)\\
&\ll X\exp\left(-\frac{(\log X)^{1/2}}{30\ell}\right). 
\end{align*}

We are now left with estimating the main term in \eqref{eq:gprod}. By the definition of $\widetilde{g}_W$, this can be expanded out as
\begin{align}\label{eq:psisum}
\sum_{\substack{e_0,\ldots, e_{\ell-1}\mid W^{\infty}\\e_0,\ldots, e_{\ell-1}\leq X^{\frac{1}{10\ell}}}}\psi(e_0)\cdots \psi(e_{\ell-1})\sum_{\substack{n\leq X\\e_i\mid n+ih\,\forall 0\leq i\leq \ell-1}}1.     
\end{align}
Let 
\begin{align}\label{eq:alphadef}
\alpha_h(e_0,\ldots, e_{\ell-1})\coloneqq \lim_{Y\to \infty}\mathbb{E}_{n\leq Y}\prod_{j=0}^{\ell-1}1_{e_j\mid n+jh}.    
\end{align}
Then we have
\begin{align*}
\sum_{\substack{n\leq X\\e_i\mid n+ih\,\forall 0\leq i\leq \ell-1}}1=\alpha_h(e_0,\ldots,e_{\ell-1})X+O(1).     
\end{align*}
Now, recalling~\eqref{eq:psibound}, which implies that $|\psi(n)|\ll n^{o(1)}$, the expression~\eqref{eq:psisum} is
\begin{align}\label{eq:psisum2}
X\sum_{\substack{e_0,\ldots, e_{\ell-1}\mid W^{\infty}\\e_0,\ldots, e_{\ell-1}\leq X^{\frac{1}{10\ell}}}}\psi(e_0)\cdots \psi(e_{\ell-1})\alpha_h(e_0,\ldots, e_{\ell-1})+O(X^{1/10+o(1)}).     
\end{align}

By inclusion--exclusion, the main term in~\eqref{eq:psisum2} is, for some $I \subset \{0, 1, \dotsc, \ell-1\}$,
\begin{align}\label{eq:alphasum}\begin{split}
&X \sum_{e_0,\ldots, e_{\ell-1}\mid W^{\infty}}\psi(e_0)\cdots \psi(e_{\ell-1})\alpha_h(e_0,\ldots, e_{\ell-1})\\
+&O\left(X\sum_{\substack{e_0,\ldots, e_{\ell-1}\mid W^{\infty}\\e_{i}> X^{\frac{1}{10\ell}}\,\forall\, i\in I}}\psi(e_0)\cdots \psi(e_{\ell-1})\alpha_h(e_0,\ldots, e_{\ell-1})\right).  
\end{split}
\end{align}
We first bound the error term here. Note that, for any $e_j \in \mathbb{N}$,
\begin{align*}
 \alpha_h(e_0,\ldots, e_{\ell-1})\leq \frac{1}{[e_0,\ldots, e_{\ell-1}]}.
 \end{align*}   
Hence, for some $i_0\in \{0,\ldots, \ell-1\}$, the error term in~\eqref{eq:alphasum} is 
\begin{align*}
&\ll X\sum_{\substack{e_0,\ldots, e_{\ell-1}\mid W^{\infty}\\e_{i_0}>X^{\frac{1}{10\ell}}}}\frac{|\psi(e_0)|\cdots |\psi(e_{\ell-1})|}{[e_0,\cdots,e_{\ell-1}]}\\
&\ll X^{1-\frac{1}{10\ell\log w}} \sum_{e_0,\ldots, e_{\ell-1}\mid W^{\infty}} \frac{|\psi(e_0)|\cdots |\psi(e_{\ell-1})|}{[e_0,\cdots,e_{\ell-1}]} e_{i_0}^{\frac{1}{\log w}} = X^{1-\frac{1}{10\ell\log w}} \prod_{p \leq w} \beta_p,
\end{align*}
where
$$
\beta_p = \sum_{e_0,\ldots, e_{\ell-1}\mid p^{\infty}} \frac{|\psi(e_0)|\cdots |\psi(e_{\ell-1})|}{[e_0,\cdots,e_{\ell-1}]} e_{i_0}^{\frac{1}{\log w}} = \sum_{E\mid p^{\infty}} \frac{1}{E} \sum_{\substack{e_0,\ldots,e_{\ell-1} \\ [e_0,\cdots,e_{\ell-1}]=E}} |\psi(e_0)|\cdots |\psi(e_{\ell-1})| e_{i_0}^{\frac{1}{\log w}}.
$$
For $E = p^a$ with $a \geq 1$, by~\eqref{eq:psibound} the inner sum above is
$$ 
\ll (2d_k(E))^\ell E^{\frac{1}{\log w}} \sum_{\substack{e_0,\ldots,e_{\ell-1} \\ [e_0,\cdots,e_{\ell-1}]=E}} 1 \ll (2(a+1)^k)^{\ell}(a+1)^{\ell} p^{\frac{a}{\log w}}. 
$$
It follows that
$$
\beta_p = 1 + O\left(\sum_{a \geq 1} \frac{(2(a+1))^{(k+1)\ell}}{(p^a)^{1-\frac{1}{\log w}}}\right) = 1+O\left(\frac{1}{p}\right)
$$
for $p \leq w$. Hence the error term in~\eqref{eq:alphasum} is 
$$ \ll X^{1-\frac{1}{10\ell\log w}} \prod_{p \leq w}\left(1 + O\left(\frac{1}{p}\right)\right) \ll X^{1-\frac{1}{10\ell\log w}} (\log w)^{O(1)} \ll X\exp(-(\log X)^{1/2}/(30\ell)). $$

The remaining task is to evaluate the main term in~\eqref{eq:alphasum}. Since the $\ell$-variable function $\alpha_h(e_0,\ldots, e_{\ell-1})$ is multiplicative, we have
\begin{align*}
 \sum_{e_0,\ldots, e_{\ell-1}\mid W^{\infty}}\psi(e_0)\cdots \psi(e_{\ell-1})\alpha_h(e_0,\ldots, e_{\ell-1})=\prod_{p\leq w}\sum_{e_0,\ldots, e_{\ell-1}\mid p^{\infty}}\psi(e_0)\cdots \psi(e_{\ell-1})\alpha_h(e_0,\ldots, e_{\ell-1}).   
\end{align*}
For each prime $p\leq w$, by~\eqref{eq:alphadef} we can write
\begin{align*}
&\sum_{e_0,\ldots, e_{\ell-1}\mid p^{\infty}}\psi(e_0)\cdots \psi(e_{\ell-1})\alpha_h(e_0,\ldots, e_{\ell-1}) \\
&=\sum_{e_0,\ldots, e_{\ell-1}\mid p^{\infty}} \psi(e_0)\cdots \psi(e_{\ell-1}) \lim_{Y\to \infty} \mathbb{E}_{n\leq Y} \prod_{j=0}^{\ell-1} 1_{e_j\mid n+jh}.
\end{align*}
By the dominated convergence theorem together with the convergence of 
\begin{align*}
\sum_{e_0,\ldots, e_{\ell-1}\mid p^{\infty}}\frac{|\psi(e_0)|\cdots |\psi(e_{\ell-1})|}{[e_0,\cdots, e_{\ell-1}]}   
\end{align*}
we can exchange the order of limits and summation. Thus
\begin{align*}
\sum_{e_0,\ldots, e_{\ell-1}\mid p^{\infty}}\psi(e_0)\cdots \psi(e_{\ell-1})\alpha_h(e_0,\ldots, e_{\ell-1}) &=\lim_{Y\to \infty}\mathbb{E}_{n\leq Y}\sum_{e_0,\ldots, e_{\ell-1}\mid p^{\infty}}\prod_{j=0}^{\ell-1}\psi(e_j)1_{e_j\mid n+jh}\\
&=\lim_{Y\to \infty}\mathbb{E}_{n\leq Y}\prod_{j=0}^{\ell-1}g_{p}(n+jh),
\end{align*}
and the claim follows.
\end{proof}

Applying Proposition~\ref{prop:simplecorrrelation} to the primes or the higher order divisor function, we obtain the following corollary. Here and later we write
\[
d_{k, W}(n) := \prod_{p \leq w} d_{k, p}(n)
\]
with $d_{k, p}$ as in Theorem~\ref{thm:divisor}.

\begin{corollary}\label{cor:simplecorrelation}
Let $k,\ell\in \mathbb{N}$ be fixed. Let $\ell< w\leq \exp((\log X)^{1/2})$, $W=\prod_{p\leq w}p$ and $H\geq 2$. 
\begin{enumerate}
    \item[(i)] For $(1-O(w^{-1/2}))H$ values of $1 \leq h \leq H$ we have
    \begin{align*}
    \mathbb{E}_{n\leq X}\Lambda_w(n)\Lambda_w(n+h)\cdots \Lambda_{w}(n+(\ell-1)h)=(1+o_{w\to \infty}(1))\mathfrak{S}(0,h,\ldots, (\ell-1)h),    
    \end{align*}
    where $\mathfrak{S}(h_1,\ldots, h_{\ell})$ is the singular series defined in~\eqref{eq:singular}. 
    \item[(ii)] Let $H \geq 1$. For $(1-O(w^{-1/2}))H$ values of $1 \leq h \leq H$ we have
    \begin{align*}
     \mathbb{E}_{n\leq X}d_{k,W}(n)d_{k,W}(n+h)\cdots d_{k,W}(n+(\ell-1)h)=\left(\frac{W}{\varphi(W)}\right)^{k\ell-\ell} \left(C(h,k,\ell)+o_{w\to \infty}(1))\right),   
    \end{align*}
    where $C(h,k,\ell)$ is as in~\eqref{eq:Chkl}.
\end{enumerate}    
\end{corollary}

\begin{proof} For part (i), we apply Proposition \ref{prop:simplecorrrelation} with $g=1_{(\cdot, P(X^{1/2}))=1}$. This gives
\begin{align*}
 \mathbb{E}_{n\leq X}\prod_{j=0}^{\ell-1}\Lambda_w(n+jh)=\left(\frac{W}{\varphi(W)}\right)^{\ell}\prod_{p\leq w}\lim_{Y\to \infty}\mathbb{E}_{n\leq Y}\prod_{j=0}^{\ell-1}1_{(n+jh,p)=1}+O(\exp(-(\log X)^{1/2}/(30\ell))).   
\end{align*}
The limit inside the product over $p$ equals to $1-\nu_p/p$, where $\nu_p$ is the size of the set $\{0,h,\ldots, (\ell-1)h\}\pmod p$. Now the claim follows by noting that 
\begin{align*}
\prod_{p>w}\left(1-\frac{1}{p}\right)^{-\ell}\left(1-\frac{\nu_p}{p}\right)=\exp\left(O\left(\sum_{p>w}\frac{1}{p^2}+\frac{1_{p\mid h}}{p}\right)\right)=1+O(w^{-1/2})   
\end{align*}
unless
\begin{align*}
\sum_{\substack{p\mid h\\p>w}}\frac{1}{p}>w^{-1/2}.    
\end{align*}
But by Markov's inequality the number of such integers $1\leq h\leq H$ is
\begin{align*}
\leq w^{1/2}\sum_{h\leq H}\sum_{\substack{p\mid h\\p>w}}\frac{1}{p}\ll w^{1/2}\sum_{p>w}\frac{H}{p^2}\ll Hw^{-1/2},    
\end{align*}
so they can be included in the exceptional set.

Let us now turn to part (ii). We have $d_{k, p} = 1_p \ast d_{k-1, p}$, where $1_p$ is the multiplicative function such that $1_p(p^\alpha) = 1$ for every $\alpha \geq 0$ and $1_p(q^\alpha) = 0$ for every prime $q \neq p$ and every $\alpha \geq 1$. Thus we have
\begin{align*}
\lim_{Y \to \infty} \mathbb{E}_{n \leq Y} d_{k, p}(n) &= \lim_{Y \to \infty} \sum_{a \geq 0} d_{k-1, p}(p^a) \frac{1}{Y} \sum_{\substack{n \leq Y \\ p^a \mid n}} 1 = \sum_{a \geq 0} \frac{1}{p^a} \binom{k-2+a}{a} \\
&= \sum_{a \geq 0} \binom{1-k}{a} \left(\frac{-1}{p}\right)^a = \left(1-\frac{1}{p}\right)^{1-k},
\end{align*}
so that
\[
\prod_{p \leq w} \lim_{Y \to \infty} \mathbb{E}_{n \leq Y} d_{k, p}(n) = \left(\frac{W}{\varphi(W)}\right)^{k-1}.
\]
Now we apply Proposition~\ref{prop:simplecorrrelation} with $g=d_k$, obtaining
\begin{align}\label{eq:Cp}
 \mathbb{E}_{n\leq X}\prod_{j=0}^{\ell-1}d_{k,W}(n+jh)&=\left(\frac{W}{\varphi(W)}\right)^{(k-1)\ell} \prod_{p\leq w}C_p(h,k,\ell) +O(\exp(-(\log X)^{1/2}/(30\ell))), 
 \end{align}
 where
 \begin{align*}
C_p(h,k,\ell)\coloneqq \lim_{Y\to \infty}\frac{\mathbb{E}_{n\leq Y}\prod_{j=0}^{\ell-1}d_{k,p}(n+jh)}{(\mathbb{E}_{n\leq Y}d_{k,p}(n))^{\ell}}.
 \end{align*}
 Thus it suffices to show that $\prod_{p>w}C_p(h,k,\ell) = 1+o_{w\rightarrow\infty}(1)$ for all but at most $O(w^{-1/2}H)$ values of $1 \leq h \leq H$.
 
 Since $d_{k,p}=1_p*d_{k-1,p}$, we can write
 $$
\mathbb{E}_{n\leq Y}\prod_{j=0}^{\ell-1}d_{k,p}(n+jh) = \sum_{a_0,\ldots,a_{\ell-1} \geq 0} d_{k-1,p}(p^{a_0}) \cdots d_{k-1,p}(p^{a_{\ell-1}}) \cdot \frac{1}{Y}\sum_{\substack{n \leq Y \\ p^{a_j} \mid n+jh\ \forall j\in [0,\ell-1]}} 1.
 $$
Thus we have the estimate
\begin{align}\label{eq:dkpestimate}
\lim_{Y\rightarrow\infty} \mathbb{E}_{n\leq Y}\prod_{j=0}^{\ell-1}d_{k,p}(n+jh) &= 1 + O\left(\sum_{a \geq 1}\sum_{\substack{a_0,\ldots,a_{\ell-1} \geq 0 \\ \max\{a_0,\ldots,a_{\ell-1}\}=a}} d_{k-1,p}(p^{a_0}) \cdots d_{k-1,p}(p^{a_{\ell-1}}) p^{-a}\right)\nonumber \\
&= 1 + O\left(\sum_{a \geq 1} a^{O(1)} p^{-a}\right) = 1 + O\left(\frac{1}{p}\right),
\end{align}
and for any prime $p > \ell$ such that $p\nmid h$ we can compute more precisely
 \begin{align}\label{eq:dkpcorrelation}
 \lim_{Y\to \infty}\mathbb{E}_{n\leq Y}\prod_{j=0}^{\ell-1}d_{k,p}(n+jh)&=\sum_{a \geq 0}\sum_{\substack{a_0,\ldots, a_{\ell-1}\geq 0\\ \{a_0,\ldots,a_{\ell-1}\}=\{a,0,\ldots,0\}}} d_{k-1,p}(p^{a_0})\cdots d_{k-1,p}(p^{a_{\ell-1}})p^{-a}\nonumber\\
 &=1+\frac{(k-1)\ell}{p}+O\left(\frac{1}{p^2}\right).
 \end{align}
 A special case implies that
 \begin{align}\label{eq:dkpsum}
\left(\lim_{Y\to \infty}\mathbb{E}_{n\leq Y}d_{k,p}(n)\right)^{\ell}=1+\frac{(k-1)\ell}{p}+O\left(\frac{1}{p^2}\right). \end{align}
It follows from~\eqref{eq:dkpestimate} and~\eqref{eq:dkpsum} that $C_p(h,k,\ell) = 1+O(1/p)$, and from~\eqref{eq:dkpcorrelation} and~\eqref{eq:dkpsum} that $C_p(h,k,\ell) = 1+O(1/p^2)$ for $p\nmid h$. Hence
$$
\prod_{p > w} C_p(h,k,\ell) = \prod_{\substack{p > w \\ p\mid h}} C_p(h,k,\ell) \cdot \prod_{\substack{p > w \\ p\nmid h}} \left(1 + O\left(\frac{1}{p^2}\right)\right) = \exp\Biggl(O\Biggl(w^{-1}+\sum_{\substack{p > w \\ p\mid h}} \frac{1}{p}\Biggr)\Biggr).
$$
As in part (i), this is $1+O(w^{-1/2})$ for all but $\ll Hw^{-1/2}$ integers $1\leq h\leq H$, and such integers can be included in the exceptional set. Now the claim follows from~\eqref{eq:Cp}.
\end{proof}

We shall also need an upper bound of the correct order of magnitude for correlations of multiplicative functions.

\begin{lemma}[An upper bound for correlations of multiplicative functions]\label{le:henriot}
Let $\ell\in \mathbb{N}$ be fixed. For $1\leq j\leq \ell$, let $1\leq a_j,h_j\leq \ell X$ be integers with $(a_j,h_j)=1$ and $a_ih_j\neq a_jh_i$ for all $i\neq j$. Let $f_1,\ldots, f_{\ell}\colon \mathbb{N}\to \mathbb{R}_{\geq 0}$ be multiplicative functions with $f_j(ab)\leq f_j(a)f_j(b)$ for all $1\leq j\leq \ell$ and $a,b\in \mathbb{N}$. Suppose that $f_j(n)\ll n^{o(1)}$ and that $f_j(p^{r})\leq e^{O(r)}$ for all $1\leq j\leq \ell$, all primes $p$ and all $r\in \mathbb{N}$. Then we have
\begin{align*}
&\sum_{X^{1/2}\leq n\leq X}\prod_{j=1}^{\ell}f_j(a_jn+h_j)\\
&\ll \Delta_DX\prod_{\ell<p\leq X^{1/2}}\left(1-\frac{|\{t\pmod p\colon \prod_{j=1}^{\ell}(a_jt+h_j)\equiv 0\pmod p\}|}{p}\right)\prod_{j=1}^{\ell}\sum_{\substack{n\leq X^{1/2}\\(n,D)=1}}\frac{f_j(n)}{n},    
\end{align*}
where $D=\prod_{1\leq i<j\leq \ell}(a_ih_j-a_jh_i)^2$ and 
\begin{align*}
\Delta_D=\prod_{p\mid D}\left(1+\sum_{\substack{0\leq \nu_1,\ldots, \nu_{\ell}\leq 1\\ (\nu_1,\ldots, \nu_{\ell})\neq (0,\ldots, 0)}}\frac{|\{n\pmod{p^2}\colon p\mid \mid a_jn+h_j\,\forall 1\leq j\leq \ell\}|}{p^2}\prod_{j=1}^{\ell}f_j(p^{\nu_j})\right).    
\end{align*}
In particular, we have
\begin{align}\label{eq:deltaD}
\Delta_D\ll \prod_{p\mid D}\left(1+O\left(\frac{1}{p}\right)\right).    
\end{align}
\end{lemma}

\begin{proof} 
We will apply Henriot's upper bound for correlations of multiplicative functions~\cite[Corollary 1]{Henriot} with $x=X^{1/2}$, $y=X$, $\delta=1/2$, $Q_j(m)=a_jm+h_j$, $Q(m)=\prod_{j=1}^{\ell}Q_j(m)$ and $F(n_1,\ldots, n_{\ell})=\prod_{j=1}^{\ell}f_j(n)$. Note that the discriminant of $Q$ is $\prod_{1\leq i<j\leq \ell}(a_ih_j-a_jh_i)^2=D$, and that the sum of absolute values of coefficients $\|Q\|$ of $Q$ is $\ll X^{\ell}$. Note also that by Gauss' lemma $Q$ is primitive in the sense that the greatest common divisor of its coefficients is $1$. Now the claim follows from~\cite[Corollary 1]{Henriot}
(noting, as in the end of~\cite[Section 3]{Henriot}, that $\tilde{F}=\tilde{G}=F$ in that corollary) and  the trivial inequality $$\sum_{n_1\cdots n_{\ell}\leq x}a_{n_1,1}\cdots a_{n_{\ell},\ell}\leq \prod_{j=1}^{\ell}\sum_{n\leq x}a_{n,j}$$
for $a_{n,j}\geq 0$.\footnote{Note that~\cite[Theorem 3]{Henriot} contains the assumption that $Q$ has no fixed prime divisor, which is not necessarily satisfied in our case, but ~\cite[Corollary 1]{Henriot} does not assume this.}
\end{proof}

\subsection{Proofs of the correlation asymptotics}

The proofs of Theorems~\ref{thm:HL} and~\ref{thm:divisor} require the Gowers uniformity estimates from Theorem~\ref{thm_gowers} together with the generalized von Neumann theorem which we state in the following form.

\begin{lemma}[Generalized von Neumann theorem]\label{le:gvnthm} 
Let $s, d,t,L\geq 1$ be given, and let $D$ be sufficiently large depending on $s,d,t,L$. Let $\nu$ be $(D, o_{N\to \infty}(1))$-pseudorandom at location $0$ and scale $N$, and let $f_1,\ldots, f_t\colon \mathbb{Z}\to \mathbb{R}$ satisfy $|f_i(x)|\leq \nu(x)$ for all $i\in [t]$ and $x\in [N]$. Let $\Psi=(\psi_1,\ldots, \psi_t)$ be a system of affine-linear forms of complexity $s$ with integer coefficients such that all the linear coefficients of $\psi_i$ are bounded by $L$ in modulus, and such that $|\psi_i(0)|\leq DN$. Let $K\subset [-N,N]^d$ be a convex body with $\Psi(K)\subset (0,N]^d$.   Suppose that for some $\delta>0$ we have   
\begin{align*}
\min_{1\leq i\leq t}\|f_i\|_{U^{s+1}[N]} \leq \delta.  
\end{align*}
Then we have
\begin{align*}
\sum_{\mathbf{n}\in K}\prod_{i=1}^tf_i(\psi_i(\mathbf{n}))=o_{\delta\to 0}(N^d).  
\end{align*}
\end{lemma}

\begin{proof}  If $\Psi$ is in $s$-normal form for some $s$, this follows from \cite[Lemma 10.1]{MSTT-all}.  In general\footnote{We thank Jinseong Kim for stressing the need to cover the non-$s$-normal form case.}, we can first apply \cite[Lemma 4.5]{green-tao} to lift to the $s$-normal form case and then argue exactly as in the end of \cite[\S 4]{green-tao}.
\end{proof}

\begin{proof}[Proof of Theorems~\ref{thm:HL} and~\ref{thm:divisor}]
Let $k,\ell\in \mathbb{N}$ and $\varepsilon>0$ be fixed. Let $f\in \{\Lambda,d_k\}$ and $X^{c_f+\varepsilon}\leq H\leq X^{1-\varepsilon}$, where $c_{f}=1/3$ if $f=\Lambda$ and $c_{f}=0$ if $f=d_k$.  Let $X\geq 3$ and let $w\in \mathbb{N}$ satisfy $2\leq w\leq w(X)$, where $w(X)$ is some function tending to infinity slowly enough. Also let $W=\prod_{p<w}p$. Denote
\begin{align*}
f_w=\begin{cases}\Lambda_w \,\, &\textnormal{if}\quad f=\Lambda\\
 \left(\frac{\varphi(W)}{W}\right)^{k-1} \frac{\log^{k-1} X}{(k-1)!} \cdot d_{k,W} \,\, &\textnormal{if}\quad f=d_k.
\end{cases}
\end{align*}

\textbf{Step 1: Splitting into main term and error term.}
 Writing $f=f_w+(f-f_w)$, we can split
\begin{equation}\label{eq:main_and_error}\begin{split}
 \sum_{n\leq X}\prod_{j=0}^{\ell-1}f(n+jh)&=\sum_{n\leq X}\prod_{j=0}^{\ell-1}f_w(n+jh)+O\big(\max_{\substack{\mathcal{I}\subset \{0,\ldots, \ell-1\}\\\mathcal{I}\neq \emptyset}}|\mathcal{E}_\mathcal{I}(h)|\big),\quad \textnormal{where}\\
 \mathcal{E}_{\mathcal{I}}(h)&=\sum_{n\leq X}\prod_{j\in \mathcal{I}}(f(n+jh)-f_w(n+jh))\prod_{j\in \{0,\ldots, \ell-1\}\setminus \mathcal{I}}f_w(n+jh).\end{split}  \end{equation}
By Corollary~\ref{cor:simplecorrelation}, for $(1-o_{w\to \infty}(1))H$ integers $1\leq h\leq H$, the main term on the right-hand side of~\eqref{eq:main_and_error} produces the desired main term for Theorems~\ref{thm:HL} and~\ref{thm:divisor}. Therefore, we are left with bounding the error term in~\eqref{eq:main_and_error}.

By Markov's inequality, for either $f\in \{\Lambda, d_k\}$ and any $\eta>0$, we have
\begin{align*}
\left|\left\{h\leq H\colon |\mathcal{E}_{\mathcal{I}}(h)|\geq \eta \sum_{n\leq X}\prod_{j=0}^{\ell-1}f_w(n+jh)\right\}\right|\leq \frac{1}{\eta\sum_{n\leq X}\prod_{j=0}^{\ell-1}f_w(n+jh)}\sum_{h\leq H}|\mathcal{E}_{\mathcal{I}}(h)|.    
\end{align*}
Hence, to complete the proof it suffices to show that 
\begin{align}\label{eq:EIh}
\sum_{h\leq H}|\mathcal{E}_{\mathcal{I}}(h)|=o_{w\to \infty}\left(H\sum_{n\leq X}\prod_{j=0}^{\ell-1}f_w(n+jh)\right).    
\end{align}

\textbf{Step 2: Bounding the error term.}
By Corollary~\ref{cor:simplecorrelation}, proving~\eqref{eq:EIh} reduces to showing that, uniformly for all unimodular sequences $c\colon \mathbb{N}\to \mathbb{C}$, we have
\begin{align*}
 \sum_{h\leq H}c(h)\mathcal{E}_{\mathcal{I}}(h)
 =\begin{cases}
o_{w\to \infty}\left(HX\right). &\textnormal{if}\quad f=\Lambda,\\
 o_{w\to \infty}\left( HX (\log X)^{(k-1)\ell} \right) &\textnormal{if}\quad  f=d_k.    
\end{cases}       
\end{align*}
Using for $m\ll X$ the crude bounds $f(m),f_w(m)\ll X^{o(1)}$, we can shift the $n$ sum in the definition of $\mathcal{E}_{\mathcal{I}}(h)$ slightly to obtain
\begin{align}\label{eq:chEI}\begin{aligned}
&\sum_{h\leq H}c(h)\mathcal{E}_{\mathcal{I}}(h)\\
&=\frac{1}{H}\sum_{h\leq H}c(h)\sum_{m\leq X}\sum_{r\leq H}\prod_{j\in \mathcal{I}}(f-f_w)(m+jh+r)\prod_{j\in \{0,\ldots, \ell-1\}\setminus \mathcal{I}}f_w(m+jh+r)+O(H^2X^{o(1)}). 
\end{aligned}\end{align}
Let 
\begin{align*}
W' \coloneqq \begin{cases}W\,\, &\textnormal{if}\quad f=\Lambda\\
W^{w}\,\, &\textnormal{if}\quad f=d_k.
\end{cases}    
\end{align*}
By splitting the variables into residue classes $\pmod{W'}$, the main term on the right-hand side of~\eqref{eq:chEI} becomes
\begin{align*}
&\sum_{0\leq b_1,b_2,b_3< W'}S(b_1,b_2,b_3) + O(W'^3X^{1+o(1)}),
\end{align*}
where
\begin{align*}
&S(b_1,b_2,b_3) \\
&\coloneqq \frac{1}{H}\sum_{h\leq H/W'}c(W'h+b_2)\sum_{m\leq X/W'}\sum_{r\leq H/W'}\prod_{j\in \mathcal{I}}(f-f_w)(W'(m+jh+r)+b_1+jb_2+b_3)\\
&\quad\cdot \prod_{j\in \{0,\ldots, \ell-1\}\setminus \mathcal{I}}f_w(W'(m+jh+r)+b_1+jb_2+b_3).
\end{align*}
Thus, our goal is to show that
\begin{align}\label{eq:Sbgoal}
\sum_{0\leq b_1,b_2,b_3< W'}S(b_1,b_2,b_3)
 =\begin{cases}
o_{w\to \infty}\left(HX\right). &\textnormal{if}\quad f=\Lambda,\\
 o_{w\to \infty}\left( HX (\log X)^{(k-1)\ell} \right) &\textnormal{if}\quad  f=d_k.    
\end{cases}       
\end{align}

Now we bound this separately in the cases $f=\Lambda$ and $f=d_k$.

\textbf{Case $f=\Lambda$.}
Note that $S(b_1,b_2,b_3)\ll (H/W^2)X^{1/2}\log X$ unless $(b_1+jb_2+b_3,W)=1$ for all $0\leq j<\ell$ since if $\Lambda(Wn+r)>0$ and $(W,r)>1$, then $Wn+r$ is a perfect power.  If in turn  $(b_1+jb_2+b_3,W)=1$ for all $0\leq j<\ell$, we can simplify
\begin{align}\label{eq:Sb1b2b3}\begin{split}
&S(b_1,b_2,b_3)\\
&= \frac{1}{H}\sum_{h\leq H/W}c(W'h+b_2)\sum_{m\leq X/W}\sum_{r\leq H/W}\prod_{j\in \mathcal{I}}\left(\frac{\varphi(W)}{W}\Lambda-1\right)(W(m+jh+r)+b_1+jb_2+b_3) \\
&\quad\cdot \prod_{j=0}^{\ell-1}\Lambda_w(b_1+jb_2+b_3). 
\end{split}
\end{align}

Let $D\in \mathbb{N}$ be a large constant. Then by Lemma~\ref{le_pseudoinshort}(i) and Theorem~\ref{thm_gowers}(i) there exists a constant $C_0\geq 1$ and an exceptional set $E\subset [1,X/W]\cap \mathbb{N}$ of size $\ll_A X/\log^{A} X$ such that for any $m\in [1,X/W]\setminus E$ the following hold for any integer $1\leq b\leq W$ with $(b,W)=1$:
\begin{enumerate}
    \item[(i)] The function $(\frac{\varphi(W)}{W}\Lambda(Wn+b)+1)/C_0$ is majorized by a $(D,o_{w\to \infty}(1))$-pseudorandom function at location $m$ and scale $H$.

    \item[(ii)] We have $$\left\|\frac{\varphi(W)}{W}\Lambda(Wn+b)-1\right\|_{U^{D}(m,m+\ell H]}=o_{w\to \infty}(1).$$
\end{enumerate}
Separating on the right-hand side of~\eqref{eq:Sb1b2b3} the terms with $m\not \in E$ and applying Lemma~\ref{le:gvnthm}, it then  follows that
\begin{align}\label{eq:SbE}
S(b_1,b_2,b_3)&=o_{w\to \infty}\left(\frac{HX}{W^3}\prod_{j=0}^{\ell-1}\Lambda_w(b_1+jb_2+b_3)\right)+O\left(\frac{(\log X)^{\ell}}{H}\left(\frac{H}{W}\right)^2\cdot |E|\right).   
\end{align}
Since $|E|\ll_A X/\log^{A} X$, the $O(\cdot)$ error term here is $\ll_A HX/(W^3\log^{100}X)$, which is negligible. 

Observe that for any functions $g_p\colon \mathbb{Z}^3\to \mathbb{C}$ that are $p$-periodic in each coordinate, we have the identity
\begin{align*}
\sum_{0\leq b_1,b_2,b_3<W}\prod_{p<w}g_p(b_1,b_2,b_3)=\prod_{p<w}\sum_{0\leq b_1,b_2,b_3<p}g_p(b_1,b_2,b_3).     
\end{align*}
Applying this with $g_p(b_1,b_2,b_3)=\prod_{j=0}^{\ell-1}\Lambda_p(b_1+jb_2+b_3)$, where $\Lambda_p\coloneqq \frac{p}{p-1}1_{(\cdot, p)=1}$, we obtain
\begin{align*}
\frac{1}{W^3}\sum_{0\leq b_1,b_2,b_3<W}\prod_{j=0}^{\ell-1}\Lambda_w(b_1+jb_2+b_3)&=\prod_{p<w}\frac{1}{p^3}\sum_{0\leq b_1,b_2,b_3<p}\prod_{j=0}^{\ell-1}\Lambda_p(b_1+jb_2+b_3)\\
&= \prod_{p<w}\left(\frac{p}{p-1}\right)^{\ell}\left(\frac{p-1}{p}\left(1-\frac{\min\{\ell, p\}}{p}\right)+\frac{1}{p}\left(1-\frac{1}{p}\right)\right)\\
&\ll \prod_{p<w}\left(1+O\left(\frac{1}{p^2}\right)\right)\ll 1.  
\end{align*}
 Now from~\eqref{eq:SbE} we conclude that
\begin{align*}
\frac{1}{W^3}\sum_{0\leq b_1,b_2,b_3<W}S(b_1,b_2,b_3)=o_{w\to \infty}(HX),    
\end{align*}
as desired.

\textbf{Case $f=d_k$.} Note that $d_{k,W}(W'n+b)=d_{k,W}(b)$ whenever $p^{w}\nmid b$ for all $p\leq w$. Let us say that an integer tuple $(b_1,b_2,b_3)\in [0,W')^3$ is \emph{good} if 
$$p^{w}\nmid b_1+jb_2+b_3 \textnormal{ for all } p\leq w \textnormal { and } j = 0, \dotsc, \ell-1.$$
Recall that we need to prove~\eqref{eq:Sbgoal}. We estimate the sum there separately for good tuples and tuples that are not good.

\textbf{Contribution of good tuples.} For any good $(b_1,b_2,b_3)$, denote for brevity 
$$G_j=G_j(b_1,b_2,b_3)=(W',b_1+jb_2+b_3).$$
Note that with this notation, for any good $(b_1, b_2, b_3)$, we have
\begin{align*}
&d_k(W'(m+jh+r)+b_1+jb_2+b_3) = d_{k}(G_j)d_k\left(\frac{W'(m+jh+r)+b_1+jb_2+b_3}{G_j}\right) \\
&= d_{k, W}(b_1+jb_2+b_3) d_k\left(\frac{W'(m+jh+r)+b_1+jb_2+b_3}{G_j}\right).
\end{align*}
Hence
\begin{align*}
&S(b_1,b_2,b_3)\\
&= \left(\frac{\log^{k-1} X}{(k-1)!} \left(\frac{\varphi(W)}{W}\right)^{k-1}\right)^\ell \frac{1}{H}\sum_{h\leq H/W'}c(W'h+b_2) \prod_{j=0}^{\ell-1}d_{k,W}(b_1+jb_2+b_3) \\
&\sum_{m\leq X/W'}\sum_{r\leq H/W'}\prod_{j\in \mathcal{I}}\Bigg(d_k\Bigg(\frac{W'}{G_j}(m+jh+r)+\frac{b_1+jb_2+b_3}{G_j}\Bigg) \left(\frac{W}{\varphi(W)}\right)^{k-1} \frac{(k-1)!}{\log ^{k-1} X} -1\Bigg).
\end{align*}

Let $D\in \mathbb{N}$ be a large constant. Then by Lemma~\ref{le_pseudoinshort}(ii) and Theorem~\ref{thm_gowers}(iii) there exists a constant $C_0\geq 1$ and an exceptional set $E\subset [1,X/W']\cap \mathbb{N}$ of size $\ll_A X/\log^{A} X$ such that for any $m\in [1,X/W']\setminus E$ the following hold for any  integer $1\leq b\leq W'$ with $(b,W')=1$:
\begin{enumerate}
    \item[(i)] The function 
    \[
    \frac{\log^{1-k} X \left(\frac{W}{\varphi(W)}\right)^{k-1} d_k(W'n+b)+1}{C_0}
    \]
    is majorized by a $(D,o_{w\to \infty}(1))$-pseudorandom function at location $m$ and scale $H$.

    \item[(ii)] We have 
    $$\left\|\left(\frac{W}{\varphi(W)}\right)^{k-1} \frac{(k-1)!}{(\log X)^{k-1}} d_k(W'n+b)-1\right\|_{U^{D}(m,m+\ell H/W']}=o_{w\to \infty}(1).$$
\end{enumerate}
Write $B_j=b_1+jb_2+b_3$ for brevity. From Lemma~\ref{le:gvnthm}, it then follows that
\begin{align}\label{eq:Sb1}\begin{aligned}
S(b_1,b_2,b_3)&=o_{w\to \infty}\left(\frac{HX}{W'^3}(\log X)^{(k-1)\ell} \left(\frac{\varphi(W)}{W}\right)^{(k-1)\ell} \prod_{j=0}^{\ell-1}d_{k,W}(b_1+jb_2+b_3)\right)\\
&\quad +O\left(\frac{(\log X)^{(k-1)\ell}}{H}\sum_{h\leq H/W'}\sum_{\substack{m\leq X/W'\\m\in E}}\sum_{r\leq H/W'}\prod_{j=0}^{\ell-1}d_k(W'(m+jh+r)+B_j)\right).
\end{aligned}\end{align}
Using the inequality $x_0\cdots x_{\ell-1}\leq x_0^{\ell}+\cdots +x_{\ell-1}^{\ell}$ for $x_j\geq 0$ and the estimate
$$\max_{1\leq x\leq 2X}\max_{1\leq b\leq (\ell+1)W'}\sum_{x\leq n\leq x+H}d_k(W'n+b)^{\ell}\ll H(\log X)^{O(1)}$$
coming from Shiu's bound (\Cref{le:shiu}),
and recalling that $|E|\ll_A X/\log^{A} X$, the $O(\cdot)$ term in~\eqref{eq:Sb1} gives a contribution of $\ll HX/(W'^3(\log X)^{100})$, say, which is negligible.

Since the function $d_{k,W}$ is constant on those progressions $b\pmod{W'}$ for which $p^{w}\nmid b$ for all $p\leq w$, we obtain
\begin{align*}
&\left(\frac{\varphi(W)}{W}\right)^{(k-1)\ell}\frac{1}{W'^3}\sum_{\substack{0\leq b_1,b_2,b_3<W'\\b_1+jb_2+b_3 \textnormal{ good}}}\prod_{j=0}^{\ell-1}d_{k,W}(b_1+jb_2+b_3)\\
&\quad =\left(\frac{\varphi(W)}{W}\right)^{(k-1)\ell}\frac{1}{W'^3\lfloor X/W'\rfloor^2}\sum_{\substack{b_1,b_2\leq W'\lfloor X/W'\rfloor\\0\leq b_3< W'\\b_1+jb_2+b_3 \textnormal{ good}}}\prod_{j=0}^{\ell-1}d_{k,W}(b_1+jb_2+b_3)\\
&\quad \ll \left(\frac{\varphi(W)}{W}\right)^{(k-1)\ell}\frac{1}{X^2}\sum_{b_1,b_2\leq X+W'}\prod_{j=0}^{\ell-1}d_{k,W}(b_1+jb_2).    
\end{align*}
By Lemma~\ref{le:henriot}, this is
\begin{align*}
&\ll \left(\frac{\varphi(W)}{W}\right)^{(k-1)\ell}\frac{1}{X}\sum_{b_2\leq 2X}\prod_{p\mid b_2}\left(1+O\left(\frac{1}{p}\right)\right)\prod_{\ell<p\leq (2X)^{1/2}}\left(1-\frac{\ell}{p}\right)\left(\sum_{n\leq (2X)^{1/2}}\frac{d_{k,W}(n)}{n}\right)^{\ell}\\
&\ll \frac{1}{X}\sum_{b_2\leq 2X}\prod_{p\mid b_2}\left(1+O\left(\frac{1}{p}\right)\right)\ll 1. 
\end{align*}
Returning to~\eqref{eq:Sb1}, we obtain
\begin{align*}
\sum_{\substack{0\leq b_1,b_2,b_3<W'\\(b_1,b_2,b_3)\textnormal{ good}}}S(b_1,b_2,b_3)=o_{w\to \infty}(HX(\log X)^{k\ell-\ell}).   
\end{align*}

\textbf{Contribution of non-good tuples.} It remains to bound the contribution of the non-good tuples $(b_1,b_2,b_3)\in [0,W')^3$ to~\eqref{eq:Sbgoal}. We can split such tuples according to the greatest common divisors 
$$G_j=(W',b_1+jb_2+b_3).$$
Let
\begin{align}\label{eq:sigmadef}
\Sigma(G_0, \dotsc, G_{\ell-1})\coloneqq \frac{1}{W'^3}\sum_{\substack{0\leq b_1,b_2,b_3< W'\\(W',b_1+jb_2+b_3)=G_j\, \forall j}}S(b_1,b_2,b_3).    
\end{align}
Note that since $(b_1, b_2, b_3)$ is not good, $p^w\mid G_j$ for some $0\leq j\leq \ell-1$ and $2\leq p\leq w$.

Then it suffices to show that
\begin{align*}
\sum_{\substack{G_0, \dotsc, G_{\ell-1} \mid W^{\infty}\\ \exists j: \, p^{w}\mid G_j \textnormal{ for some } p\leq w}}|\Sigma(G_0, \dotsc, G_{\ell-1})|=o_{w\to \infty}(HX(\log X)^{(k-1)\ell}).      
\end{align*}

Let $(b_1,b_2,b_3)$ be a tuple that is not good. Write $B_j=b_1+jb_2+b_3$ for brevity.  Using the inequalities $d_k(ab)\leq d_k(a)d_k(b)$, $d_{k,W}(ab)\leq d_{k,W}(a)d_{k,W}(b)$ for $a,b\in \mathbb{N}$, we see that the contribution of $S(b_1,b_2,b_3)$ is
\begin{align*}
&\ll \frac{1}{H}\sum_{h\leq H/W'}\sum_{m\leq X/W'}\sum_{r\leq H/W'}\prod_{j=0}^{\ell-1}\left(d_k+(\log^{k-1} X) \left(\frac{\varphi(W)}{W}\right)^{k-1}d_{k,W} \right)(W'(m+jh+r)+B_j)\\
&\ll \frac{1}{W'}\sum_{h\leq H/W'}\sum_{m\leq 2X/W'}\prod_{j=0}^{\ell-1} \left(d_k+(\log^{k-1} X) \left(\frac{\varphi(W)}{W}\right)^{k-1}d_{k,W}\right)(W'(m+jh)+B_j)\\
&\ll \frac{\prod_{j=0}^{\ell-1} d_k(G_j)}{W'}\sum_{h\leq H/W'}\sum_{m\leq 2X/W'}\prod_{j=0}^{\ell-1}\left(d_k+(\log^{k-1} X) \left(\frac{\varphi(W)}{W}\right)^{k-1}d_{k,W}\right)\left(\frac{W'}{G_j}(m+jh)+\frac{B_j}{G_j}\right).
\end{align*}
 We have the crude bound $\prod_{j=0}^{\ell-1} d_k(G_j)\ll (G_0 \dotsm G_{\ell-1})^{1/(100 \ell)}$, say. Hence, for some $\mathcal{J}\subset \{0,\ldots, \ell-1\}$, the above is
\begin{align}\label{eq:G100}\begin{split}
&\ll  \frac{(G_0 \dotsm G_{\ell-1})^{1/(100 \ell)}}{W'}\sum_{h\leq H/W'}\sum_{m\leq 2X/W'}\prod_{j\in \mathcal{J}} (\log^{k-1} X) d_{k,W}\left(\frac{W'}{G_j}(m+jh)+\frac{B_j}{G_j}\right)\\
&\quad \cdot \prod_{j\in \{0,\ldots, \ell-1\}\setminus \mathcal{J}}d_k\left(\frac{W'}{G_j}(m+jh)+\frac{B_j}{G_j}\right).  
\end{split} \end{align}
We now apply Lemma~\ref{le:henriot} to upper bound this. We take $a_j=W'/G_j$, $h_j=(W'/G_j)jh+B_j/G_j$ there and note that $(a_j,h_j)=1$. Denote $D=\prod_{i<j}(a_ih_j-a_jh_i)^2$. Then 
$$p \mid D \implies p \mid W'\prod_{\substack{|i|\leq \ell\\i\neq 0}}\left(W'ih+ib_2\right).$$
Then Lemma~\ref{le:henriot},~\eqref{eq:deltaD} and Shiu's bound (\Cref{le:shiu}) imply that~\eqref{eq:G100} is
\begin{align*}
&\ll \frac{(G_0 \dotsm G_{\ell-1})^{1/(100 \ell)}}{W'}\sum_{h\leq H/W'} \prod_{p\mid W'}\left(1+O\left(\frac{1}{p}\right)\right)\prod_{\substack{|i|\leq \ell\\i\neq 0}}\prod_{p\mid W'ih+ib_2}\left(1+O\left(\frac{1}{p}\right)\right)\\
&\quad \cdot \frac{X}{W'}\prod_{\ell<p\leq (2X/W')^{1/2}}\left(1-\frac{\ell\cdot 1_{p>w}}{p}\right)\prod_{j=0}^{\ell-1}\sum_{n\leq (2X/W')^{1/2}}\frac{\max\{d_k(n),(\log^{k-1} X)d_{k,W}(n)\}}{n}\\
&\ll  \frac{(G_0 \dotsm G_{\ell-1})^{1/(100 \ell)}}{W'}\sum_{h\leq H/W'} (\log w)^{O(1)}\frac{X}{W'}(\log X)^{(k-1)\ell}\prod_{\substack{|i|\leq \ell\\i\neq 0}}\prod_{p\mid W'ih+ib_2}\left(1+O\left(\frac{1}{p}\right)\right).
\end{align*}
Using the inequality $x_1\cdots x_{\ell}\leq \ell \max_{i}x_i^{\ell}$ for $x_j\geq 0$, we see from Shiu's bound that 
\begin{align*}
\sum_{h\leq H/W'}\prod_{\substack{|i|\leq \ell\\i\neq 0}}\prod_{p\mid W'ih+ib_2}\left(1+O\left(\frac{1}{p}\right)\right)&\ll \max_{\substack{|i|\leq \ell\\i\neq 0}}\sum_{h\leq H/W'}\prod_{p\mid W'ih+ib_2}\left(1+O\left(\frac{1}{p}\right)\right)^{\ell}\\
&\ll \max_{\substack{|i|\leq \ell\\i\neq 0}}\sum_{\substack{h'\leq H\ell\\h'\equiv ib_2\pmod{W'}}}\prod_{p\mid h'}\left(1+O\left(\frac{1}{p}\right)\right)^{\ell}\\
&\ll (\log w)^{O(1)}\frac{H}{W'}.
\end{align*}

    Since $(\log w)^{B}\ll_B (\max G_j)^{1/100}$ for any $B>0$ and the number of tuples $(b_1,b_2,b_3)\in [1,W']^3$ with $(W',b_1+jb_2+b_3)=G_j$ for all $j$ is $\ll (W')^3/\max_j G_j$, recalling~\eqref{eq:sigmadef}we conclude that
\begin{align*}
|\Sigma(G_0,\ldots, G_{\ell-1})|\ll \frac{1}{W'^3}(\log^{(k-1)\ell} X)HX(\max_j G_j)^{2/100-1}\ll \frac{(\log^{(k-1)\ell} X)HX}{(\max_j G_j)^{1/2}}.     
\end{align*}
Hence
\begin{align*}
&(\log X)^{-(k-1)\ell}(HX)^{-1}\sum_{\substack{G_0, \dotsc, G_{\ell-1} \mid W^{\infty}\\\exists j\colon p^{w}\mid G_j\textnormal{ for some } p\leq w}}\Sigma(G_0, \dotsc, G_{\ell-1}) \\
&\ll \sum_{\substack{G_0, \dotsc, G_{\ell-1} \mid W^{\infty}\\\exists j\colon p^{w}\mid G_j\textnormal{ for some } p\leq w}} (G_0 \dotsm G_{\ell-1})^{-1/(2\ell)}\\
&\ll 2^{-w/(6\ell)}\sum_{G_0, \dotsc, G_{\ell-1} \mid W^{\infty}}(G_0 \dotsm G_{\ell-1})^{-1/(3\ell)} = 2^{-w/(6\ell)}\prod_{p\leq w}\left(1+O\left(\frac{1}{p^{1/(3\ell)}}\right)\right)^\ell\\
&\ll 2^{-w/(6\ell)}\exp(O(w^{1-1/(3\ell)})) \ll 2^{-w/(10\ell)},
\end{align*}
say. This completes the proof. 
\end{proof}

\section{Fixed Fourier estimate for \texorpdfstring{$d_2$}{d2}}\label{d2-sec}

In this section we establish Theorem~\ref{thm_d2}.

\begin{proof}[Proof of Theorem~\ref{thm_d2}]
Let $Q=X^{\varepsilon/2} \leq H^{1/2}$. By Dirichlet's approximation theorem, there exist some integers $a,q$ with $(a,q)=1$ and $1\leq q\leq Q$ such that $|\alpha-a/q|\leq 1/(qQ)$. We take $H' = q Q^{1/2}$. Then, for all $n\in [x,x+H']$ we have
$$
e(\alpha n)=e(an/q) \cdot e((\alpha-a/q)x)+O(H'/(qQ))) =e(an/q) \cdot e((\alpha-a/q)x)+O(X^{-\varepsilon/4}).
$$
Hence, by splitting into intervals of length $H' \leq H^{3/4}$, it suffices to show that we have
$$
    \Bigg |\sum_{x < n \leq x + H'} (d(n) - d^{\sharp}(n)) e(an/q) \Bigg |^{*}  \leq H'X^{ - \varepsilon/1000}/2.
    $$
for all $x\in [X,2X]$ outside of a set of measure $O_{\varepsilon}(X^{1-\varepsilon/400})$.  

Let $H_1=X^{1-1/100}$. By splitting into progressions modulo $q$ and applying~\cite[Proposition 3.2(ii)]{MSTT-all}, we see that for any $x\in [X,2X]$ we have 
\begin{align*}
 \Bigg |\sum_{x < n \leq x + H_1} (d(n) - d^{\sharp}(n)) e(an/q) \Bigg |^{*}  \leq q  \Bigg |\sum_{x < n \leq x + H_1} (d(n) - d^{\sharp}(n)) \Bigg |^{*}\ll Q\frac{H_1^2}{X}\ll H_1X^{-\varepsilon}.
\end{align*}

Hence, it suffices to show that
\begin{align*}
\max_{P \textnormal{ arithmetic progression}}\left|\frac{1}{H'}\sum_{\substack{x<n\leq x+H'\\n\in P}}d(n)e\left(\frac{an}{q}\right)-\frac{1}{H_1}\sum_{\substack{x<n\leq x+H_1\\n\in P}}d(n)e\left(\frac{an}{q}\right)\right|\ll X^{-\varepsilon/1000}    
\end{align*}
and
\begin{align*}
\max_{P \textnormal{ arithmetic progression}}\left|\frac{1}{H'}\sum_{\substack{x<n\leq x+H'\\n\in P}}d^{\sharp}(n)e\left(\frac{an}{q}\right)-\frac{1}{H_1}\sum_{\substack{x<n\leq x+H_1\\n\in P}}d^{\sharp}(n)e\left(\frac{an}{q}\right)\right|\ll X^{-\varepsilon/1000}  \end{align*}
for all $x\in [X,2X]$ outside of a set of measure $O_{\varepsilon}(X^{1-\varepsilon/400})$. In the maxima above, we may restrict to arithmetic progressions of modulus $\leq X^{\varepsilon/999}$ thanks to the triangle inequality and the divisor bound. Hence, by the union bound, it suffices to show that for any $1\leq b\leq r\leq X^{\varepsilon/999}$ we have
\begin{align*}
\left|\frac{1}{H'}\sum_{\substack{x<n\leq x+H'\\n\equiv b\pmod r}}d(n)e\left(\frac{an}{q}\right)-\frac{1}{H_1}\sum_{\substack{x<n\leq x+H_1\\n\equiv b\pmod r}}d(n)e\left(\frac{an}{q}\right)\right|\leq \frac{1}{10} X^{-\varepsilon/1000}    
\end{align*}
and
\begin{align*}
\left|\frac{1}{H'}\sum_{\substack{x<n\leq x+H'\\n\equiv b\pmod r}}d^{\sharp}(n)e\left(\frac{an}{q}\right)-\frac{1}{H_1}\sum_{\substack{x<n\leq x+H_1\\n\equiv b\pmod r}}d^{\sharp}(n)e\left(\frac{an}{q}\right)\right|\leq \frac{1}{10} X^{-\varepsilon/1000}  \end{align*}
for all $x\in [X,2X]$ outside of a set of measure $O_{\varepsilon}(X^{1-\varepsilon/300})$. By Fourier expanding the summation condition $n\equiv b\pmod r$, we reduce to showing that for any $1\leq q'\leq q X^{\varepsilon/999}$ and $(a',q')=1$ we have  
\begin{align}\label{eq:d2}
\left|\frac{1}{H'}\sum_{x<n\leq x+H'}d(n)e\left(\frac{a'n}{q'}\right)-\frac{1}{H_1}\sum_{x<n\leq x+H_1}d(n)e\left(\frac{a'n}{q'}\right)\right|\leq \frac{1}{10}X^{-\varepsilon/1000}    
\end{align}
and
\begin{align}\label{eq:sharp}
\left|\frac{1}{H'}\sum_{x<n\leq x+H'}d^{\sharp}(n)e\left(\frac{a'n}{q'}\right)-\frac{1}{H_1}\sum_{x<n\leq x+H_1}d^{\sharp}(n)e\left(\frac{a'n}{q'}\right)\right|\leq \frac{1}{10} X^{-\varepsilon/1000}  \end{align}
for all $x\in [X,2X]$ outside of a set of measure $O_{\varepsilon}(X^{1-\varepsilon/300})$.

By Lemma~\ref{le:dkdecompose}(ii), 
the function $d_k^{\sharp}$ is a sum of $O(1)$ terms of the form $a*\psi$, where $|a(m)|\ll d_{k-1}(m)$ and $a$ is supported on $[1,X^{\varepsilon/5}]$ and each function $\psi$ is of the form $\psi(n)=(\log^{\ell} n)/(\log^{\ell} X)$ with $0\leq \ell\leq k-1$ an integer. The estimate~\eqref{eq:sharp} follows from this combined with the triangle inequality and partial summation (splitting into residue classes $\pmod{q'}$).

We are left with showing~\eqref{eq:d2}. By Markov's inequality, it suffices to show that
\begin{align*}
\int_{X}^{2X}\left|\frac{1}{H'}\sum_{x<n\leq x+H'}d(n)e\left(\frac{a'n}{q'}\right)-\frac{1}{H_1}\sum_{x<n\leq x+H_1}d(n)e\left(\frac{a'n}{q'}\right)\right|\,d x\ll X^{1-\varepsilon/100}.    
\end{align*}
Factorizing $n=\ell m$ with $\ell\mid (q')^{\infty}$ and  $(m,q')=1$, and applying the triangle inequality, we reduce to showing that
\begin{align*}
&\sum_{\ell \mid (q')^{\infty}}d(\ell)\int_{X}^{2X}\left|\frac{1}{H'}\sum_{x/\ell<m\leq x/\ell+H'/\ell}d(m)e\left(\frac{a'\ell m}{q'}\right)1_{(m,q')=1}\right.\\
&\left.\quad -\frac{1}{H_1}\sum_{x/\ell<m\leq x/\ell+H_1/\ell}d(m)e\left(\frac{a'\ell m}{q'}\right)1_{(m,q')=1}\right|\,d x\ll X^{1-\varepsilon/100}.  \end{align*}
For any $L\geq 1$, we have the bound
\begin{align*}
\sum_{\substack{\ell \mid (q')^{\infty}\\\ell \geq L}}\frac{1}{\ell}\leq \frac{1}{\sqrt{L}}\sum_{\substack{\ell \mid (q')^{\infty}\\\ell \geq L}}\frac{1}{\ell^{1/2}}\leq\frac{1}{\sqrt{L}}\prod_{p\mid q'}\left(1+\frac{1}{p^{1/2}}+O\left(\frac{1}{p}\right)\right)\ll \frac{1}{\sqrt{L}}\exp(O((\log q')^{1/2})).    
\end{align*}
Combining this with the triangle inequality, we see that it suffices to show that
\begin{align*}
&\int_{X}^{2X}\left|\frac{1}{H'}\sum_{x/\ell<m\leq x/\ell+H'/\ell}d(m)e\left(\frac{a'\ell m}{q'}\right)1_{(m,q')=1}\right.\\
&\left.\quad -\frac{1}{H_1}\sum_{x/\ell<m\leq x/\ell+H_1/\ell}d(m)e\left(\frac{a'\ell m}{q'}\right)1_{(m,q')=1}\right|\,d x\ll X^{1-\varepsilon/30}   
\end{align*}
for any $1\leq \ell\leq X^{\varepsilon/40}$. 

For any integer $b$, we have the Fourier expansion
\begin{align*}
e\left(\frac{bm}{q'}\right)1_{(m,q')=1}=\frac{1}{\varphi(q')}\sum_{\chi\pmod{q'}}\tau(\overline{\chi})\chi(bm),  \end{align*}
where $\tau(\chi)$ is the Gauss sum. Using the bound $|\tau(\chi)|\leq q^{1/2}$ for non-principal $\chi\pmod q$ together with the triangle inequality and Cauchy--Schwarz, we reduce to showing that
\begin{align*}
&\sum_{\chi\pmod{q'}}\int_{X}^{2X}\left|\frac{1}{H'}\sum_{x/\ell<m\leq (x+H')/\ell}d(m)\chi(m )-\frac{1}{H_1}\sum_{x/\ell<m\leq (x+H_1)/\ell}d(m)\chi(m)\right|^2\,d x\ll X^{1-\varepsilon/14}.   
\end{align*}

By Lemma~\ref{le:perron}(i), we reduce to showing that
\begin{align}\label{eq:chisum}
\sum_{\chi\pmod{q'}}\max_{T\geq \frac{X}{H'}}\frac{X/H'}{T}\int_{X^{1/400}\leq |t|\leq T}|D(1+it,\chi)|^2\, dt\ll X^{-\varepsilon/14},
\end{align}
where 
\begin{align*}
D(s,\chi)\coloneqq \sum_{X/8<n\leq 8X}d(n)\chi(n)n^{-s}.
\end{align*}
By the mean value theorem for Dirichlet polynomials and crude estimation, we may restrict the maximum over $T$ to $T\leq X^{1-\varepsilon/1000}$. 

Using Perron's formula,  we have, for $|t| \leq X^{1-\varepsilon/1000}$
\begin{align*}
D(1+it,\chi)=\frac{1}{2\pi i}\int_{1-i|t|/2}^{1+i|t|/2}L(1+it+w,\chi)^2\frac{(8X)^{w}-(X/8)^{w}}{w}\, dw +O\left(\frac{\log X}{|t|}\right).
\end{align*}
Moving the line of integration to $\Re(w)=-1/2$ and applying the triangle inequality, we obtain
\begin{align*}
|D(1+it,\chi)|\ll X^{-1/2} \int_{-|t|/2}^{|t|/2}|L(1/2+i(t+u), \chi)|^2\frac{du}{1+|u|} +O\left(\frac{\log X}{|t|}\right).
\end{align*}
Substituting this to~\eqref{eq:chisum} and applying the Cauchy--Schwarz inequality, we see that it suffices to show that
\begin{align*}
\sum_{\chi\pmod{q'}}\max_{\frac{X}{H'}\leq T\leq X}\frac{X/H'}{T}\int_{X^{1/500}\leq |t|\leq T}|L(1/2+it,\chi)|^4\, dt\ll X^{1-\varepsilon/10}.
\end{align*}
By splitting characters according to their conductors, and using the Euler product formula, we reduce to showing that for all $r\mid q'$ we have
\begin{align*}
\asum_{\chi'\pmod{r}}\max_{\frac{X}{H'}\leq T\leq X}\frac{X/H'}{T}\int_{X^{1/500}\leq |t|\leq T}|L(1/2+it,\chi')|^4\prod_{p\mid \frac{q'}{r}}\left(1+p^{-1/2}\right)^4\, dt\ll X^{1-\varepsilon/9}.
\end{align*}
where $\asum$ denotes a sum over primitive characters. Without loss of generality, we may restrict the maximum over $T$ to powers of two, so, upper bounding the product over $p$ by $x^{\varepsilon/100}$, it suffices to show that
\begin{align}\label{eq:final}
\max_{\frac{X}{H'}\leq T\leq X}\frac{X/H'}{T}\asum_{\chi'\pmod{r}}\int_{X^{1/200}\leq |t|\leq T}|L(1/2+it,\chi')|^4\, dt\ll X^{1-\varepsilon/8}.
\end{align}
But by the fourth moment bound for Dirichlet $L$-functions~\cite[Theorem 10.1]{Montgomery-topics}, the left-hand side of~\eqref{eq:final} is
\begin{align*}
\ll \varphi(r)\frac{X}{H'} \log^4 X.
\end{align*}
The bound~\eqref{eq:final} immediately follows from this since $r \leq q' \leq q X^{\varepsilon/999}$ and $H' = q Q^{1/2} = qX^{\varepsilon/4}$.
\end{proof}

\bibliography{refs}

@article {FH-Fourier,
    AUTHOR = {Frantzikinakis, N. and Host, B.},
     TITLE = {Higher order {F}ourier analysis of multiplicative functions
              and applications},
   JOURNAL = {J. Amer. Math. Soc.},
  FJOURNAL = {Journal of the American Mathematical Society},
    VOLUME = {30},
      YEAR = {2017},
    NUMBER = {1},
     PAGES = {67--157},
      ISSN = {0894-0347},
   MRCLASS = {11N37 (05D10 11B30 11N60 37A45)},
  MRNUMBER = {3556289},
MRREVIEWER = {Vilius Stakenas},
       DOI = {10.1090/jams/857},
       URL = {https://ezproxy-prd.bodleian.ox.ac.uk:4563/10.1090/jams/857},
}

@ARTICLE{MatoTera,
    AUTHOR = {Matom\"aki, K. and Ter\"av\"ainen, J.},
     TITLE = {On the {M}\"obius function in all short intervals},
   JOURNAL = {J. Eur. Math. Soc. (JEMS)},
  FJOURNAL = {Journal of the European Mathematical Society (JEMS)},
    VOLUME = {25},
      YEAR = {2023},
    NUMBER = {4},
     PAGES = {1207--1225},
      ISSN = {1435-9855,1435-9863},
   MRCLASS = {11N37},
  MRNUMBER = {4577962},
MRREVIEWER = {Peter\ Shiu},
       DOI = {10.4171/jems/1205},
       URL = {https://doi.org/10.4171/jems/1205},
}

@article {HeWang,
    AUTHOR = {He, X. and Wang, Z.},
     TITLE = {M\"{o}bius disjointness for nilsequences along short intervals},
   JOURNAL = {Trans. Amer. Math. Soc.},
  FJOURNAL = {Transactions of the American Mathematical Society},
    VOLUME = {374},
      YEAR = {2021},
    NUMBER = {6},
     PAGES = {3881--3917},
      ISSN = {0002-9947},
   MRCLASS = {37A44 (11A25)},
  MRNUMBER = {4251216},
MRREVIEWER = {Jonas Der\'{e}},
       DOI = {10.1090/tran/8176},
       URL = {https://doi.org/10.1090/tran/8176},
}

@article {Henriot,
    AUTHOR = {Henriot, K.},
         TITLE = {Nair-{T}enenbaum bounds uniform with respect to the
	               discriminant},
		          JOURNAL = {Math. Proc. Cambridge Philos. Soc.},
			    FJOURNAL = {Mathematical Proceedings of the Cambridge Philosophical
			                  Society},
					      VOLUME = {152},
					            YEAR = {2012},
						        NUMBER = {3},
							     PAGES = {405--424},
							           ISSN = {0305-0041},
								      MRCLASS = {11A25 (11N37 11N56)},
								        MRNUMBER = {2911138},
									MRREVIEWER = {Olivier Bordell\`es},
									       DOI = {10.1017/S0305004111000752},
									              URL = {https://doi.org/10.1017/S0305004111000752},
										      }

@article {matthiesen-linear,
    AUTHOR = {Matthiesen, L.},
     TITLE = {Linear correlations of multiplicative functions},
   JOURNAL = {Proc. Lond. Math. Soc. (3)},
  FJOURNAL = {Proceedings of the London Mathematical Society. Third Series},
    VOLUME = {121},
      YEAR = {2020},
    NUMBER = {2},
     PAGES = {372--425},
      ISSN = {0024-6115},
   MRCLASS = {11N37 (11B30 11D04 11F30)},
  MRNUMBER = {4093960},
MRREVIEWER = {Olivier Bordell\`es},
       DOI = {10.1112/plms.12309},
       URL = {https://ezproxy-prd.bodleian.ox.ac.uk:2102/10.1112/plms.12309},
}

@book {Montgomery,
    AUTHOR = {Montgomery, H. L.},
         TITLE = {Ten lectures on the interface between analytic number theory
	               and harmonic analysis},
		           SERIES = {CBMS Regional Conference Series in Mathematics},
			       VOLUME = {84},
			        PUBLISHER = {Published for the Conference Board of the Mathematical
				              Sciences, Washington, DC; by the American Mathematical
					                    Society, Providence, RI},
							          YEAR = {1994},
								       PAGES = {xiv+220},
								             ISBN = {0-8218-0737-4},
									        MRCLASS = {11-02 (11Kxx 11L07 11Mxx 11Nxx)},
										  MRNUMBER = {1297543},
										  MRREVIEWER = {John B. Friedlander},
										         DOI = {10.1090/cbms/084},
											        URL = {https://doi.org/10.1090/cbms/084},
												}

@article {MRT,
    AUTHOR = {Matom\"aki, K. and Radziwi{\l\l} , M. and Tao, T.},
     TITLE = {An averaged form of {C}howla's conjecture},
   JOURNAL = {Algebra Number Theory},
  FJOURNAL = {Algebra \& Number Theory},
    VOLUME = {9},
      YEAR = {2015},
    NUMBER = {9},
     PAGES = {2167--2196},
      ISSN = {1937-0652},
   MRCLASS = {11P32},
  MRNUMBER = {3435814},
MRREVIEWER = {Martin Mereb},
       DOI = {10.2140/ant.2015.9.2167},
       URL = {https://doi.org/10.2140/ant.2015.9.2167},
}

@ARTICLE{mr-p,
       author = {{Matom{\"a}ki}, K. and {Radziwi{\l}{\l}}, M.},
        title = "{A note on the Liouville function in short intervals}",
      journal = {arXiv e-prints},
         year = "2015",
        month = "Feb",
          eid = {arXiv:1502.02374},
        pages = {arXiv:1502.02374},
archivePrefix = {arXiv},
       eprint = {1502.02374},
 primaryClass = {math.NT},
       adsurl = {https://ui.adsabs.harvard.edu/abs/2015arXiv150202374M}
}

@article {mrt-correlationsI,
    AUTHOR = {Matom\"{a}ki, K. and Radziwi{\l}{\l}, M. and Tao, T.},
     TITLE = {Correlations of the von {M}angoldt and higher divisor
              functions {I}. {L}ong shift ranges},
   JOURNAL = {Proc. Lond. Math. Soc. (3)},
  FJOURNAL = {Proceedings of the London Mathematical Society. Third Series},
    VOLUME = {118},
      YEAR = {2019},
    NUMBER = {2},
     PAGES = {284--350},
      ISSN = {0024-6115},
   MRCLASS = {11N37 (11A25)},
  MRNUMBER = {3909235},
MRREVIEWER = {Olivier Bordell\`es},
       DOI = {10.1112/plms.12181},
       URL = {https://doi.org/10.1112/plms.12181},
}

@article {matomaki-shao,
    AUTHOR = {Matom\"{a}ki, K. and Shao, X.},
     TITLE = {Discorrelation between primes in short intervals and
              polynomial phases},
   JOURNAL = {Int. Math. Res. Not. IMRN},
  FJOURNAL = {International Mathematics Research Notices. IMRN},
      YEAR = {2021},
    NUMBER = {16},
     PAGES = {12330--12355},
      ISSN = {1073-7928},
   MRCLASS = {11L20 (11N05)},
  MRNUMBER = {4300228},
MRREVIEWER = {Bingrong Huang},
       DOI = {10.1093/imrn/rnz188},
       URL = {https://doi.org/10.1093/imrn/rnz188},
}

@article {mr-annals,
    AUTHOR = {Matom\"aki, K. and Radziwi{\l\l} , M.},
     TITLE = {Multiplicative functions in short intervals},
   JOURNAL = {Ann. of Math. (2)},
  FJOURNAL = {Annals of Mathematics. Second Series},
    VOLUME = {183},
      YEAR = {2016},
    NUMBER = {3},
     PAGES = {1015--1056},
      ISSN = {0003-486X},
   MRCLASS = {11K65 (11N64)},
  MRNUMBER = {3488742},
MRREVIEWER = {Eugenijus Manstavi\v cius},
       DOI = {10.4007/annals.2016.183.3.6},
       URL = {https://doi.org/10.4007/annals.2016.183.3.6},
}

@ARTICLE{MR-II,
       author = {{Matom{\"a}ki}, K. and {Radziwi{\l}{\l}}, M.},
        title = "{Multiplicative functions in short intervals II}",
      journal = {arXiv e-prints},
         year = 2020,
        month = jul,
          eid = {arXiv:2007.04290},
        pages = {arXiv:2007.04290},
archivePrefix = {arXiv},
       eprint = {2007.04290},
 primaryClass = {math.NT},
       adsurl = {https://ui.adsabs.harvard.edu/abs/2020arXiv200704290M}
}

@article {MRTTZ,
    AUTHOR = {Matom\"{a}ki, K. and Radziwi{\l}{\l} , M. and Tao, T. and
              Ter\"{a}v\"{a}inen, J. and Ziegler, T.},
     TITLE = {Higher uniformity of bounded multiplicative functions in short
              intervals on average},
   JOURNAL = {Ann. of Math. (2)},
  FJOURNAL = {Annals of Mathematics. Second Series},
    VOLUME = {197},
      YEAR = {2023},
    NUMBER = {2},
     PAGES = {739--857},
      ISSN = {0003-486X},
   MRCLASS = {11N37 (11B30 37A44)},
  MRNUMBER = {4543441},
       DOI = {10.4007/annals.2023.197.2.3},
       URL = {https://doi.org/10.4007/annals.2023.197.2.3},
}

@incollection {TaoEq,
    AUTHOR = {Tao, T.},
     TITLE = {Equivalence of the logarithmically averaged {C}howla and
              {S}arnak conjectures},
 BOOKTITLE = {Number theory---{D}iophantine problems, uniform distribution
              and applications},
     PAGES = {391--421},
 PUBLISHER = {Springer, Cham},
      YEAR = {2017},
   MRCLASS = {11N37 (11A25)},
  MRNUMBER = {3676413},
MRREVIEWER = {Olivier Bordell\`es},
}

@article {Zhan,
    AUTHOR = {Zhan, T.},
     TITLE = {On the representation of large odd integer as a sum of three
              almost equal primes},
   JOURNAL = {Acta Math. Sinica (N.S.)},
  FJOURNAL = {Acta Mathematica Sinica. New Series},
    VOLUME = {7},
      YEAR = {1991},
    NUMBER = {3},
     PAGES = {259--272},
      ISSN = {1000-9574},
   MRCLASS = {11P32 (11M06)},
  MRNUMBER = {1141240},
MRREVIEWER = {D. R. Heath-Brown},
       DOI = {10.1007/BF02583003},
       URL = {https://doi.org/10.1007/BF02583003},
}

@book {ik,
    AUTHOR = {Iwaniec, H. and Kowalski, E.},
     TITLE = {Analytic number theory},
    SERIES = {American Mathematical Society Colloquium Publications},
    VOLUME = {53},
 PUBLISHER = {American Mathematical Society, Providence, RI},
      YEAR = {2004},
     PAGES = {xii+615},
      ISBN = {0-8218-3633-1},
   MRCLASS = {11-02 (11Fxx 11Lxx 11Mxx 11Nxx)},
  MRNUMBER = {2061214},
MRREVIEWER = {K. Soundararajan},
       DOI = {10.1090/coll/053},
       URL = {https://doi.org/10.1090/coll/053},
}

@article {green-tao,
    AUTHOR = {Green, B. and Tao, T.},
     TITLE = {Linear equations in primes},
   JOURNAL = {Ann. of Math. (2)},
  FJOURNAL = {Annals of Mathematics. Second Series},
    VOLUME = {171},
      YEAR = {2010},
    NUMBER = {3},
     PAGES = {1753--1850},
      ISSN = {0003-486X},
     CODEN = {ANMAAH},
   MRCLASS = {11N13 (11B30 11P32)},
  MRNUMBER = {2680398 (2011j:11177)},
MRREVIEWER = {Tamar Ziegler},
       DOI = {10.4007/annals.2010.171.1753},
       URL = {http://dx.doi.org/10.4007/annals.2010.171.1753},
}

@article {gt-mobius,
    AUTHOR = {Green, B. and Tao, T.},
     TITLE = {The {M}\"{o}bius function is strongly orthogonal to nilsequences},
   JOURNAL = {Ann. of Math. (2)},
  FJOURNAL = {Annals of Mathematics. Second Series},
    VOLUME = {175},
      YEAR = {2012},
    NUMBER = {2},
     PAGES = {541--566},
      ISSN = {0003-486X},
   MRCLASS = {37A45 (11A25)},
  MRNUMBER = {2877066},
MRREVIEWER = {Tamar Ziegler},
       DOI = {10.4007/annals.2012.175.2.3},
       URL = {https://doi.org/10.4007/annals.2012.175.2.3},
}

@article {gtz,
    AUTHOR = {Green, B. and Tao, T. and Ziegler, T.},
     TITLE = {An inverse theorem for the {G}owers {$U^{s+1}[N]$}-norm},
   JOURNAL = {Ann. of Math. (2)},
  FJOURNAL = {Annals of Mathematics. Second Series},
    VOLUME = {176},
      YEAR = {2012},
    NUMBER = {2},
     PAGES = {1231--1372},
      ISSN = {0003-486X},
   MRCLASS = {11B30},
  MRNUMBER = {2950773},
MRREVIEWER = {Julia Wolf},
       DOI = {10.4007/annals.2012.176.2.11},
       URL = {https://doi.org/10.4007/annals.2012.176.2.11},
}

@article {mrt-fourier,
    AUTHOR = {Matom\"{a}ki, K. and Radziwi{\l}{\l} , M. and Tao, T.},
     TITLE = {Fourier uniformity of bounded multiplicative functions in
              short intervals on average},
   JOURNAL = {Invent. Math.},
  FJOURNAL = {Inventiones Mathematicae},
    VOLUME = {220},
      YEAR = {2020},
    NUMBER = {1},
     PAGES = {1--58},
      ISSN = {0020-9910},
   MRCLASS = {11K60 (11N64 42A63)},
  MRNUMBER = {4071407},
MRREVIEWER = {Olivier Bordell\`es},
       DOI = {10.1007/s00222-019-00926-w},
       URL = {https://doi.org/10.1007/s00222-019-00926-w},
}

@book {tao-higher,
    AUTHOR = {Tao, T.},
     TITLE = {Higher order {F}ourier analysis},
    SERIES = {Graduate Studies in Mathematics},
    VOLUME = {142},
 PUBLISHER = {American Mathematical Society, Providence, RI},
      YEAR = {2012},
     PAGES = {x+187},
      ISBN = {978-0-8218-8986-2},
   MRCLASS = {11B30 (11L07 11U07 37A45)},
  MRNUMBER = {2931680},
MRREVIEWER = {David Conlon},
       DOI = {10.1090/gsm/142},
       URL = {https://doi.org/10.1090/gsm/142},
}

@article {Walsh,
    AUTHOR = {Walsh, M. N.},
     TITLE = {Local uniformity through larger scales},
   JOURNAL = {Geom. Funct. Anal.},
  FJOURNAL = {Geometric and Functional Analysis},
    VOLUME = {31},
      YEAR = {2021},
    NUMBER = {4},
     PAGES = {981--991},
      ISSN = {1016-443X},
   MRCLASS = {11N37},
  MRNUMBER = {4317509},
MRREVIEWER = {Peter Shiu},
       DOI = {10.1007/s00039-021-00570-8},
       URL = {https://doi.org/10.1007/s00039-021-00570-8},
}

@article {robert-sargos,
    AUTHOR = {Robert, O. and Sargos, P.},
     TITLE = {Three-dimensional exponential sums with monomials},
   JOURNAL = {J. Reine Angew. Math.},
  FJOURNAL = {Journal f\"{u}r die Reine und Angewandte Mathematik. [Crelle's
              Journal]},
    VOLUME = {591},
      YEAR = {2006},
     PAGES = {1--20},
      ISSN = {0075-4102},
   MRCLASS = {11L03 (11N36 11N37)},
  MRNUMBER = {2212877},
MRREVIEWER = {S. W. Graham},
       DOI = {10.1515/CRELLE.2006.012},
       URL = {https://doi.org/10.1515/CRELLE.2006.012},
}

@book {Montgomery-topics,
    AUTHOR = {Montgomery, H. L.},
     TITLE = {Topics in multiplicative number theory},
    SERIES = {Lecture Notes in Mathematics},
    VOLUME = {Vol. 227},
 PUBLISHER = {Springer-Verlag, Berlin-New York},
      YEAR = {1971},
     PAGES = {ix+178},
   MRCLASS = {10H30},
  MRNUMBER = {337847},
MRREVIEWER = {H.-E.\ Richert},
}

@incollection {ivic-divisor,
    AUTHOR = {Ivi\'{c}, A.},
     TITLE = {The general additive divisor problem and moments of the
              zeta-function},
 BOOKTITLE = {New trends in probability and statistics, {V}ol. 4 ({P}alanga,
              1996)},
     PAGES = {69--89},
 PUBLISHER = {VSP, Utrecht},
      YEAR = {1997},
   MRCLASS = {11N56 (11M06)},
  MRNUMBER = {1653602},
MRREVIEWER = {Gunter Dufner},
}

@article {conrey-gonek,
    AUTHOR = {Conrey, J. B. and Gonek, S. M.},
     TITLE = {High moments of the {R}iemann zeta-function},
   JOURNAL = {Duke Math. J.},
  FJOURNAL = {Duke Mathematical Journal},
    VOLUME = {107},
      YEAR = {2001},
    NUMBER = {3},
     PAGES = {577--604},
      ISSN = {0012-7094},
   MRCLASS = {11M06 (11M26)},
  MRNUMBER = {1828303},
MRREVIEWER = {Cem Y. Y\i ld\i r\i m},
       DOI = {10.1215/S0012-7094-01-10737-0},
       URL = {https://doi.org/10.1215/S0012-7094-01-10737-0},
}

@article {green-tao-ratner,
    AUTHOR = {Green, B. and Tao, T.},
     TITLE = {The quantitative behaviour of polynomial orbits on
              nilmanifolds},
   JOURNAL = {Ann. of Math. (2)},
  FJOURNAL = {Annals of Mathematics. Second Series},
    VOLUME = {175},
      YEAR = {2012},
    NUMBER = {2},
     PAGES = {465--540},
      ISSN = {0003-486X},
   MRCLASS = {37A15},
  MRNUMBER = {2877065},
MRREVIEWER = {Tamar Ziegler},
       DOI = {10.4007/annals.2012.175.2.2},
       URL = {https://doi.org/10.4007/annals.2012.175.2.2},
}

@article {baker-harman-pintz,
    AUTHOR = {Baker, R. C. and Harman, G. and Pintz, J.},
     TITLE = {The difference between consecutive primes. {II}},
   JOURNAL = {Proc. London Math. Soc. (3)},
  FJOURNAL = {Proceedings of the London Mathematical Society. Third Series},
    VOLUME = {83},
      YEAR = {2001},
    NUMBER = {3},
     PAGES = {532--562},
      ISSN = {0024-6115},
   MRCLASS = {11N05 (11N36)},
  MRNUMBER = {1851081},
MRREVIEWER = {D. R. Heath-Brown},
       DOI = {10.1112/plms/83.3.532},
       URL = {https://doi.org/10.1112/plms/83.3.532},
}

@book {harman-book,
    AUTHOR = {Harman, G.},
     TITLE = {Prime-detecting sieves},
    SERIES = {London Mathematical Society Monographs Series},
    VOLUME = {33},
 PUBLISHER = {Princeton University Press, Princeton, NJ},
      YEAR = {2007},
     PAGES = {xvi+362},
      ISBN = {978-0-691-12437-7},
   MRCLASS = {11N36 (11N25 11N35)},
  MRNUMBER = {2331072},
MRREVIEWER = {S. W. Graham},
}

@article {MSTT-all,
    AUTHOR = {Matom\"aki, K. and Shao, X. and Tao, T. and
              Ter\"av\"ainen, J.},
     TITLE = {Higher uniformity of arithmetic functions in short intervals
              {I}. {A}ll intervals},
   JOURNAL = {Forum Math. Pi},
  FJOURNAL = {Forum of Mathematics. Pi},
    VOLUME = {11},
      YEAR = {2023},
     PAGES = {Paper No. e29, 97},
      ISSN = {2050-5086},
   MRCLASS = {11N37 (11B30)},
  MRNUMBER = {4658200},
MRREVIEWER = {Peter\ Shiu},
       DOI = {10.1017/fmp.2023.28},
       URL = {https://doi.org/10.1017/fmp.2023.28},
}

@article {shiu,
    AUTHOR = {Shiu, P.},
     TITLE = {A {B}run-{T}itchmarsh theorem for multiplicative functions},
   JOURNAL = {J. Reine Angew. Math.},
  FJOURNAL = {Journal f\"{u}r die Reine und Angewandte Mathematik. [Crelle's
              Journal]},
    VOLUME = {313},
      YEAR = {1980},
     PAGES = {161--170},
      ISSN = {0075-4102},
   MRCLASS = {10H25},
  MRNUMBER = {552470},
MRREVIEWER = {A. I. Vinogradov},
       DOI = {10.1515/crll.1980.313.161},
       URL = {https://doi.org/10.1515/crll.1980.313.161},
}

@article {mrt-correlations-II,
    AUTHOR = {Matom\"{a}ki, K. and Radziwi{\l}{\l} , M. and Tao, T.},
     TITLE = {Correlations of the von {M}angoldt and higher divisor
              functions {II}: divisor correlations in short ranges},
   JOURNAL = {Math. Ann.},
  FJOURNAL = {Mathematische Annalen},
    VOLUME = {374},
      YEAR = {2019},
    NUMBER = {1-2},
     PAGES = {793--840},
      ISSN = {0025-5831},
   MRCLASS = {11N37 (05A16 11A25)},
  MRNUMBER = {3961326},
MRREVIEWER = {Olivier Bordell\`es},
       DOI = {10.1007/s00208-018-01801-4},
       URL = {https://doi.org/10.1007/s00208-018-01801-4},
}

@article{jutila,
    AUTHOR = {Jutila, M.},
     TITLE = {On the divisor problem for short intervals},
      NOTE = {Studies in honour of Arto Kustaa Salomaa on the occasion of
              his fiftieth birthday},
   JOURNAL = {Ann. Univ. Turku. Ser. A I},
  FJOURNAL = {Annales Universitatis Turkuensis. Series A. I.
              Astronomica-Chemica-Physica-Mathematica},
    NUMBER = {186},
      YEAR = {1984},
     PAGES = {23--30},
      ISSN = {0082-7002},
   MRCLASS = {11N37},
  MRNUMBER = {748516},
MRREVIEWER = {Don Redmond},
}

@article {hilbert-cube,
    AUTHOR = {Hilbert, D.},
     TITLE = {Ueber die {I}rreducibilit\"{a}t ganzer rationaler {F}unctionen mit
              ganzzahligen {C}oefficienten},
   JOURNAL = {J. Reine Angew. Math.},
  FJOURNAL = {Journal f\"{u}r die Reine und Angewandte Mathematik. [Crelle's
              Journal]},
    VOLUME = {110},
      YEAR = {1892},
     PAGES = {104--129},
      ISSN = {0075-4102},
   MRCLASS = {DML},
  MRNUMBER = {1580277},
       DOI = {10.1515/crll.1892.110.104},
       URL = {https://doi.org/10.1515/crll.1892.110.104},
}

@ARTICLE{mangerel,
       author = {{Mangerel}, A.P.},
        title = "{Divisor-bounded multiplicative functions in short intervals}",
      journal = {Res. Math. Sci.},
      volume={10},
         year = {2023},
         number= {12},
}

@misc{tao-computation,
  author = {{Tao}, T.},
  title = {Heuristic computation of correlations of higher order divisor functions},
  howpublished = {Available at \par \texttt{https://terrytao.wordpress.com/2016/08/31/}}
}

@ARTICLE{Walsh2,
       author = {{Walsh}, M. N.},
        title = "{Phase relations and pyramids}",
      journal = {arXiv e-prints},
         year = 2023,
        month = apr,
          eid = {arXiv:2304.09792},
        pages = {arXiv:2304.09792},
          doi = {10.48550/arXiv.2304.09792},
archivePrefix = {arXiv},
       eprint = {2304.09792},
 primaryClass = {math.NT},
       adsurl = {https://ui.adsabs.harvard.edu/abs/2023arXiv230409792W}
}

@article {HL,
    AUTHOR = {Hardy, G. H. and Littlewood, J. E.},
     TITLE = {Some problems of `{P}artitio numerorum'; {III}: {O}n the
              expression of a number as a sum of primes},
   JOURNAL = {Acta Math.},
  FJOURNAL = {Acta Mathematica},
    VOLUME = {44},
      YEAR = {1923},
    NUMBER = {1},
     PAGES = {1--70},
      ISSN = {0001-5962,1871-2509},
   MRCLASS = {99-04},
  MRNUMBER = {1555183},
       DOI = {10.1007/BF02403921},
       URL = {https://doi.org/10.1007/BF02403921},
}

@article {vinogradov-divisor,
    AUTHOR = {Vinogradov, A. I.},
     TITLE = {{${\rm SL}_n$} techniques and the density conjecture},
   JOURNAL = {Zap. Nauchn. Sem. Leningrad. Otdel. Mat. Inst. Steklov.
              (LOMI)},
  FJOURNAL = {Zapiski Nauchnykh Seminarov Leningradskogo Otdeleniya
              Matematicheskogo Instituta imeni V. A. Steklova Akademii Nauk
              SSSR (LOMI)},
    VOLUME = {168},
      YEAR = {1988},
     PAGES = {5--10, 187},
      ISSN = {0373-2703},
   MRCLASS = {11M26 (11F55)},
  MRNUMBER = {982478},
MRREVIEWER = {I.\ Piatetski-Shapiro},
       DOI = {10.1007/BF01303645},
       URL = {https://doi.org/10.1007/BF01303645},
}

@ARTICLE{guth-maynard,
       author = {{Guth}, L. and {Maynard}, J.},
        title = "{New large value estimates for Dirichlet polynomials}",
      journal = {arXiv e-prints},
         year = 2024,
        month = may,
          eid = {arXiv:2405.20552},
        pages = {arXiv:2405.20552},
          doi = {10.48550/arXiv.2405.20552},
archivePrefix = {arXiv},
       eprint = {2405.20552},
 primaryClass = {math.NT},
       adsurl = {https://ui.adsabs.harvard.edu/abs/2024arXiv240520552G}
}

@ARTICLE{walsh3,
       author = {{Walsh}, M. N.},
        title = "{Stability under scaling in the local phases of multiplicative functions}",
      journal = {arXiv e-prints},
         year = 2023,
        month = oct,
          eid = {arXiv:2310.07873},
        pages = {arXiv:2310.07873},
          doi = {10.48550/arXiv.2310.07873},
archivePrefix = {arXiv},
       eprint = {2310.07873},
 primaryClass = {math.NT},
       adsurl = {https://ui.adsabs.harvard.edu/abs/2023arXiv231007873W}
}

@ARTICLE{sun,
       author = {{Sun}, Y.-C.},
        title = "{On divisor bounded multiplicative functions in short intervals}",
      journal = {arXiv e-prints},
         year = 2024,
        month = jan,
          eid = {arXiv:2401.08432},
        pages = {arXiv:2401.08432},
          doi = {10.48550/arXiv.2401.08432},
archivePrefix = {arXiv},
       eprint = {2401.08432},
 primaryClass = {math.NT},
       adsurl = {https://ui.adsabs.harvard.edu/abs/2024arXiv240108432S}
}

@ARTICLE{mengdi-wang,
       author = {{Wang}, M.},
        title = "{Local Fourier uniformity of higher divisor functions on average}",
      journal = {arXiv e-prints},
         year = 2024,
        month = feb,
          eid = {arXiv:2402.18342},
        pages = {arXiv:2402.18342},
          doi = {10.48550/arXiv.2402.18342},
archivePrefix = {arXiv},
       eprint = {2402.18342},
 primaryClass = {math.NT},
       adsurl = {https://ui.adsabs.harvard.edu/abs/2024arXiv240218342W}
}

@ARTICLE{LSS,
       author = {{Leng}, J. and {Sah}, A. and {Sawhney}, M.},
        title = "{Quasipolynomial bounds on the inverse theorem for the Gowers $U^{s+1}[N]$-norm}",
      journal = {arXiv e-prints, \textnormal{arXiv:2402.17994}},
         year = 2024,
        month = feb,
          eid = {arXiv:2402.17994},
          doi = {10.48550/arXiv.2402.17994},
archivePrefix = {arXiv},
       eprint = {2402.17994},
 primaryClass = {math.CO},
       adsurl = {https://ui.adsabs.harvard.edu/abs/2024arXiv240217994L}
}

@book {tao-hilbert,
    AUTHOR = {Tao, T.},
     TITLE = {Hilbert's fifth problem and related topics},
    SERIES = {Graduate Studies in Mathematics},
    VOLUME = {153},
 PUBLISHER = {American Mathematical Society, Providence, RI},
      YEAR = {2014},
     PAGES = {xiv+338},
      ISBN = {978-1-4704-1564-8},
   MRCLASS = {22D05 (11B30 20F65 22E05)},
  MRNUMBER = {3237440},
MRREVIEWER = {Ben\ Joseph\ Green},
       DOI = {10.1090/gsm/153},
       URL = {https://doi.org/10.1090/gsm/153},
}

@article {FI-divisor,
    AUTHOR = {Friedlander, J. B. and Iwaniec, H.},
     TITLE = {The divisor problem for arithmetic progressions},
   JOURNAL = {Acta Arith.},
  FJOURNAL = {Polska Akademia Nauk. Instytut Matematyczny. Acta Arithmetica},
    VOLUME = {45},
      YEAR = {1985},
    NUMBER = {3},
     PAGES = {273--277},
      ISSN = {0065-1036},
   MRCLASS = {11N13 (11N37)},
  MRNUMBER = {808026},
MRREVIEWER = {Julia\ Mueller},
       DOI = {10.4064/aa-45-3-273-277},
       URL = {https://doi.org/10.4064/aa-45-3-273-277},
}
\bibliographystyle{plain}

\end{document}